\documentclass[11pt,UTF-8,reqno]{amsart}
\usepackage{geometry}
\usepackage{amsfonts}
\usepackage{mathrsfs}
\usepackage[greek,english]{babel}
\usepackage{amsmath,amsfonts,amsthm,amssymb,mathtools,amscd,color,xcolor,mathrsfs,verbatim,microtype}
\usepackage{graphicx,eurosym}
\usepackage{amssymb,amsmath,comment}
\usepackage{amsfonts,mathrsfs}
\usepackage{color}
\usepackage{bbm}
\usepackage{bm}
\usepackage{hyperref}
\usepackage{stmaryrd} 
\usepackage{cleveref}
\usepackage[applemac]{inputenc}

\usepackage[cyr]{aeguill}

\usepackage{enumerate, bbm}

\colorlet{darkblue}{blue!50!black}

\hypersetup{
	colorlinks,%
	citecolor=darkblue,%
	filecolor=darkred,%
	linkcolor=blue,%
	urlcolor=blue,%
	pdfnewwindow=true,%
	pdfstartview={FitH}
}

\usepackage{graphicx,amscd,mathrsfs,wrapfig,mathrsfs,lipsum}
\usepackage{eufrak}
\usepackage{float}
\usepackage{tikz}
\usepackage{multicol}
\usepackage{caption}
\usetikzlibrary{arrows}
\usepackage{capt-of}

\colorlet{darkblue}{blue!50!black}

\newcommand{\vertiii}[1]{{\left\vert\kern-0.25ex\left\vert\kern-0.25ex\left\vert #1
		\right\vert\kern-0.25ex\right\vert\kern-0.25ex\right\vert}}

\usepackage{marginnote}

\newcommand{\PPi}{{\boldsymbol{ {\Pi}}}}

\newcommand{\bae}{\begin{equation}\begin{aligned}}
		\newcommand{\eae}{\end{aligned}\end{equation}}
\newcommand{\baee}{\begin{equation*}\begin{aligned}}
		\newcommand{\eaee}{\end{aligned}\end{equation*}}

\newcommand{\EE}{{\cal E}}

\def\EE{\mathbb{E}}

\newcommand{\RR}{{\mathbb R}}

\theoremstyle{plain}

\newtheorem*{lemma*}{Lemma}
\newtheorem{theorem}{Theorem}[section]
\newtheorem{lemma}[theorem]{Lemma}
\newtheorem{proposition}[theorem]{Proposition}

\theoremstyle{definition}
\newtheorem{definition}[theorem]{Definition}
\newtheorem{condition}[theorem]{Condition}

\theoremstyle{remark}

\newtheorem{remark}{Remark}[section]
\newtheorem{example}{Example}[section]
\newtheorem{assumption}{Assumption}[section]

\numberwithin{equation}{section}

\newcommand{\ea}{\end{array}}

\renewcommand{\d}{d}

\newcommand{\bb}[1]{\mathbb{#1}}
\newcommand{\ca}[1]{\mathcal{#1}}

\newcommand{\worknote}[1]{}

\newcommand{\coloneqq}{\mathrel{\mathop:}=}

\allowdisplaybreaks

\begin{document}
	\title[Multi-Scale McKean--Vlasov Diffusions
		with Super-Linear Kernels]
{Asymptotics of Multi-Scale McKean--Vlasov Diffusions
		with Super-Linear Kernels: a Lifted Semigroup Approach$^\dagger$}
\thanks{$\dagger$
This work is supported by National Key R\&D program of China (No.~2023YFA1010101).  W. Hong is supported by  NSFC (No.~12401177) and Basic
Research Program of Jiangsu (No.~BK20241048).  W. Liu is supported by NSFC (No.~12571155). }

	\maketitle
	\centerline{ Wei Hong$^a$,   Shanshan Hu$^{b}$,  Wei Liu$^{a,}$\footnote{Corresponding author: weiliu@jsnu.edu.cn}, Shiyuan Yang$^{b}$ }

\vspace{2mm}
\medskip
 {\footnotesize
\centerline{ $a.$   School of Mathematics and Statistics, Jiangsu Normal University, Xuzhou 221116, China}

\vspace{1mm}
\centerline{ $b.$   Center for Applied Mathematics and KL-AAGDM, Tianjin University, Tianjin 300072, China}}

\begin{abstract}
In this work, we establish the small-noise asymptotic behaviour (namely, the functional law of large numbers and the large deviation principle) for multi-scale McKean--Vlasov diffusions with super-linear kernels. In this setting, the interaction depends on the laws of both the slow component and the fast oscillating process. Consequently, the frozen (parameterized) system exhibits McKean--Vlasov dynamics, forming a nonlinear Markov process and thereby rendering the analysis more complex compared to existing works.

We develop a lifted semigroup argument and employ a generalized Khasminskii time discretization scheme to derive the small-noise limit of the slow variable, providing explicit convergence rates. Furthermore, we introduce the notion of a lifted viable pair and utilize a generalized functional occupation measure approach to establish the Laplace principle, which is equivalent to the large deviation principle. The main results of this work find broad applications in multi-scale models arising in fields such as machine learning and optimization theory. In particular, our results can be employed to analyze the dynamics of multi-scale consensus-based methods for multilevel optimization, where the coefficients typically satisfy local Lipschitz continuity on the interaction kernels.

\end{abstract}

\tableofcontents

\section{Introduction}
In this work, we consider the following  multi-scale McKean--Vlasov (also known as mean-field) dynamics
\begin{subequations}\label{system0}
	\begin{align}
		dX^{\varepsilon}_t &= b(X^{\varepsilon}_t, \mathscr{L}_{X^{\varepsilon}_t}, Y^{\varepsilon}_t,\mathscr{L}_{Y^{\varepsilon}_t})dt+\sqrt{\varepsilon}\sigma(X^{\varepsilon}_t, \mathscr{L}_{X^{\varepsilon}_t}, Y^{\varepsilon}_t,\mathscr{L}_{Y^{\varepsilon}_t})dW^{1}_t,~~X^{\varepsilon}_0=\xi,\label{1.1a} \\
		dY^{\varepsilon}_t &=\frac{1}{\delta}f( \mathscr{L}_{X^{\varepsilon}_t}, Y^{\varepsilon}_t,\mathscr{L}_{Y^{\varepsilon}_t})dt+\frac{1}{\sqrt{\delta}}g(\mathscr{L}_{X^{\varepsilon}_t}, Y^{\varepsilon}_t,\mathscr{L}_{Y^{\varepsilon}_t})dW^{2}_t,~~Y^{\varepsilon}_0=\zeta,\label{1.1b}	
	\end{align}
\end{subequations}
where we denote
\begin{itemize}
	\item $\{W^{1}_t\}_{t\ge0}$ and $\{W^{2}_t\}_{t\ge0}$ are mutually independent $d_1$- and $d_2$-dimensional standard Brownian motions defined on a complete filtered probability space $(\Omega,\mathscr{F},\{\mathscr{F}_t\}_{t\ge0},\mathbb{P})$;
	\item $\xi$ and $\zeta$ are $\mathscr{F}_0$-measurable $\mathbb{R}^n$- and $\mathbb{R}^m$-valued random variables, respectively;
	\item $\varepsilon>0$ is a small parameter describing the noise intensity, while $\delta>0$ (depending on $\varepsilon$) characterizes the time-scale ratio between the slow component $X^{\varepsilon}_t$ and the fast oscillating process $Y^{\varepsilon}_t$.
\end{itemize}

A multi-scale system can be viewed as a coupled dynamical system  where the  variables evolve on distinct time scales. More precisely, certain variables evolve rapidly over short time intervals, while others undergo slow evolution over extended periods. For instance, in the system given by  (\ref{1.1a})-(\ref{1.1b}),  $X_t^{\varepsilon}$ is identified as the slow component, while $Y_t^{\varepsilon}$ is  the fast oscillating component.   The averaging method for systems with different time scales can be traced to the seminal work of Bogoliubov and Mitropolsky \cite{bogoliubov1961asymtotic} on deterministic dynamical systems.  Khasminskii \cite{khas1968averaging} significantly extended this framework to  stochastic dynamical systems, leading to extensive research on various random multi-scale problems.
Existing research includes the averaging principle for stochastic (partial) differential equations (cf.~e.g.~\cite{brehiercharles2012,cerraifreidlin,haierliAOP}), normal deviations (cf.~e.g.~\cite{C09,MR4634338,michaelxie2021}), and large deviation principle (cf.~e.g.~\cite{MR2914778,hong2023multi,MR4022288}).

On the other hand, the interacting particle systems   appear naturally in many fields and have become an active research area in mathematical physics (cf.~\cite{JW18,Se}),  machine learning, and optimization theory (cf.~\cite{CD22, Garbuno-InigoHoffmannLiStuart.2020.SJoADS412,GHV24}).
According to the propagation of chaos theory (cf.~\cite{Sznitman.1991.165}), the empirical measure can be approximated by the macroscopic distribution of a non-interacting particle, leading to the  general formulation of McKean--Vlasov diffusions.
Building on these developments, studying multi-scale interacting particle systems and their McKean--Vlasov counterparts has become an increasingly active area of research.
For example, Delgadino et al.~\cite{DGP} considered the multi-scale interacting particle systems  and demonstrated that the mean-field and homogenization limits fail to commute when the mean-field system undergoes a phase transition, i.e., when multiple steady states are present.
In the context of multi-scale McKean--Vlasov systems with smooth coefficients, \cite{rocknersunxiepoincare} derived optimal strong convergence rates for the averaging principle. Subsequent studies, such as \cite{BS23} and \cite{LWX}, investigated the diffusion approximation for these systems with homogenization term and smooth coefficients. One can see also
\cite{chenghaorocknerbernoulli} on the study for systems with $L^p$-integrable coefficients.

It should be pointed out that in many practical applications, such as machine learning and optimization \cite{borkar1997actor,DNS23,Garbuno-InigoHoffmannLiStuart.2020.SJoADS412, CBOPinnau}, the interaction kernel is not globally Lipschitz but only locally Lipschitz, and it admits super-linear growth. Handling  interaction kernels that are merely locally Lipschitz and of polynomial growth has been a longstanding and challenging problem. As stated in Section 3.1.2 in recent review \cite{CD22}: ``There is no real hope for better results at this level of generality, many directions have been explored to weaken the hypotheses in specific cases.'' Motivated by these observations, our goal is to develop a general framework for studying the multi-scale McKean--Vlasov system \eqref{system0} involving  super-linear interaction kernels.  As a typical example, we apply our results to  the bi-level (or multilevel) consensus-based optimization (CBO)  introduced in \cite{hertyM3MA} (see also \cite{hajinkim,hakangkim} and reference therein), which approximates the minimizer through a weighted average over the particle systems, thereby avoiding the costly computation of gradients. 

Given cost functions $\ell:\RR^n\times \RR^m\to\RR$ and $h:\RR^m\to\RR$, we consider a bi-level minimization problem
\vspace{-2mm}
\begin{align*}
	&\min_{x\in\RR^n} \ell(x,y)\\
	{\rm{s.t.}}\,\,&y\in \text{argmin}_{y\in \RR^m} h(y)\,,
\end{align*}
where a lower-level problem $\min_{y} h(y)$ is embedded within an upper-level problem $\min_{x} \ell(x,y)$, which motivates the application  of two-time-scales collective dynamics as follows. Define two weight functions
$$\omega_{\alpha}^\ell(x,y)\coloneqq \exp (-\alpha \ell(x,y)),\,\,\,\,\,\,\omega_{\beta}^h(y)\coloneqq \exp (-\beta h(y)),$$ then the weight functions assign larger weights to those  directions $x\in\RR^n, y\in\RR^m$ for which the value of cost function $\ell,  h$ take smaller values. Let us denote the particles systems and the corresponding empirical measures by
\begin{align*}
	X&\coloneqq (X^1, \ldots, X^N)\in (\RR^n)^N,\,\,\,\,\,\,\,\,\,\, \eta_X\coloneqq \frac{1}{N}\sum_{i=1}^N\delta_{X^i},\\
	Y&\coloneqq (Y^1,\ldots, Y^M)\in (\RR^m)^M, \,\,\,\,\,\,\,\,\,\,\eta_{Y}\coloneqq \frac{1}{M}\sum_{i=1}^M\delta_{Y^i}.
\end{align*}
Intuitively speaking, the weighted averages
\begin{align*}
	{\rm{M}}_{\beta}^h(\eta_{Y})\coloneqq& \sum_{k=1}^M Y^k \frac{\omega_{\beta}^h(Y^k)}{\sum_{k=1}^M \omega_{\beta}^h(Y^k)}\approx\text{argmin}_{Y^k,k=1,\ldots M} h(Y^k),\,\\
	{\rm{M}}_{\alpha}^\ell(\eta_X,\eta_Y)\coloneqq &\sum_{k=1}^N X^k \frac{\omega_{\alpha}^\ell(X^k,{\rm{M}}_{\beta}^h(\eta_{Y}))}{\sum_{k=1}^M \omega_{\alpha}^\ell(X^k,{\rm{M}}_{\beta}^h(\eta_{Y}))}\approx\text{argmin}_{X^k,k=1,\ldots N} \ell(X^k,{\rm{M}}_{\beta}^h(\eta_{Y})).
\end{align*}
Namely, the particles will approximately meet at consensus points,   where the cost functions $\ell, h$ attain the  minimum.

It is suggested in \cite{hertyM3MA} that $x$- and $y$- particles are not at the same decision level to deal with the hierarchical structure in the bi-level optimization problem. The insight lies in the fact that while the $y$-particles perform optimization, the $x$-particles are  frozen waiting for the $y$-particles to reach their equilibrium and then the $x$-particles are updated.  It was proposed in \cite{hertyM3MA} that   $x$- and $y$ particles can be updated via the following  interacting particle systems for given $(X_0^i, Y_0^j), i=1,\ldots, N, j=1,\ldots, M$,

\vspace{-3mm}
\begin{subequations}
	\begin{align}
		dX^{i}_t &= -\lambda_1(X^{i}_t- {\rm{M}}_{\alpha}^\ell(\eta_{X},\eta_{Y}))dt+\sigma_1 D_1\big(X_t^i, {\rm{M}}_{\alpha}^\ell(\eta_{X},\eta_{Y})\big)dW^{i}_t\,,\label{cbo1}\\
		dY^{j}_t &=-\lambda_2\frac{1}{\delta}(Y^{j}_t- {\rm{M}}_{\beta}^h(\eta_{Y}))dt+\frac{1}{\sqrt{\delta}}\sigma_2 D_2\big(Y_t^j,{\rm{M}}_{\beta}^h(\eta_{Y})\big)dW^{j}_t\,,\label{cbo2}
	\end{align}
\end{subequations}
where  $W_t^i, W_t^j$ denote independent standard Brownian motions. Here, $D_i$ is the anisotropic diffusion defined by
$$D(u,m):={\rm{diag}}(u-m)~~\text{or}~~D(u,m):={\rm{diag}}(\chi_{R_0}(u-m)),$$
where $\chi_{R_0}$ is a cut-off function with radius $R_0$.
As one may notice,  for probability measures $\mu$ and $\nu$,  ${\rm{M}}_{\beta}(\mu)- {\rm{M}}_{\beta}(\nu)$  does not exhibit global Lipschitz continuity with respect to the Wasserstein metric, but it is locally Lipschitz continuous. In fact,  only a truncated version of the original multi-scale CBO model was investigated in \cite{hertyM3MA}, as stated in the concluding
paragraph of page 2214: ``Although the averaging principle has been intensively studied since decades, to the best of our knowledge, the CBO system  does not satisfy the
assumptions in the very recent developments of the averaging principle.''
\subsection{Comments on assumptions and applications}
Motivated by the aforementioned illustrative example, we impose the following conditions of locally Lipschitz type on the coefficients of the multi-scale McKean--Vlasov system (\ref{1.1a})-(\ref{1.1b}).
\begin{assumption}\label{ass1}
	(see \eqref{con4.1} in $(\mathbf{A}_{\text{fast}})$) There exist constants $C>0$, $ K_1>K_2>0$ such that for all $\mu_1,\mu_2\in \mathscr{P}_{\kappa}(\RR^n), y_1,y_2\in\RR^m,\nu_1,\nu_2\in \mathscr{P}_{\kappa}(\RR^m) $,
	\begin{align*}
		&2\langle f(\mu_1, y_1,\nu_1)-f( \mu_2, y_2,\nu_2),y_1-y_2\rangle+\|g(\mu_1, y_1,\nu_1)-g(\mu_2, y_2,\nu_2)\|^2\nonumber\\
		&\leq -K_1|y_1-y_2|^2+K_2\mathbb{W}_2(\nu_1, \nu_2)^2\nonumber\\
		&~~~~~+C\big(1+\rho_{\kappa}(0, 0, y_1, y_2)+\ca{M}_{\kappa}(\mu_1,\mu_2, \nu_1, \nu_2)\big)\mathbb{W}_2(\mu_1, \mu_2)^2.
	\end{align*}
\end{assumption}
\noindent Note that Assumption \ref{ass1}  is related to the ergodic behavior of the  fast component $Y_t^{\varepsilon}$. We  assume the following condition on the slow variable $X_t^{\varepsilon}$.
\begin{assumption}  (see \eqref{continuityofb}
	in $(\mathbf{A}_{\text{slow}})$)
	\noindent For all $x_1,x_2\in\RR^n, \mu_1,\mu_2\in \mathscr{P}_{\kappa}(\RR^n), y_1,y_2\in\RR^m, \nu_1,\nu_2\in \mathscr{P}_{\kappa}(\RR^m) $,
	\begin{align*}
		\begin{split}
			&|b(x_1, \mu_1, y_1,\nu_1)-b(x_2, \mu_2, y_2,\nu_2)|^2+\|\sigma(x_1,\mu_1, y_1,\nu_1)-\sigma(x_2,\mu_2, y_2,\nu_2)\|^2\\
			&\lesssim \big(1+\ca{M}_{\kappa}(\mu_1,\mu_2,\nu_1,\nu_2)\big)\big(|x_1-x_2|^2+|y_1-y_2|^2\big)
			\\
			& ~~~~+ \big(1+\rho_{\kappa}(x_1, x_2,y_1,y_2)+\ca{M}_{\kappa}(\mu_1, \mu_2, \nu_1, \nu_2)\big)\big(\mathbb{W}_2(\mu_1, \mu_2)^2+\mathbb{W}_2(\nu_1, \nu_2)^2\big)\,.
		\end{split}
	\end{align*}
\end{assumption}

\begin{remark}
	\begin{enumerate}[(i)]
		\item Compared with existing works such as \cite{BS23,rocknersunxiepoincare, SXW22}, the coefficients
		in our setting depend simultaneously on the slow component, the fast oscillating process, and their  laws. Moreover,   a distinguishing feature of our framework, in contrast to \cite{BS23,HLX24,LWX,lixie2023}, is that we do not require any smoothness conditions on the coefficients.
		\item  Different from the classical SDE case, allowing the  coefficients  to be merely locally Lipschitz continuous is highly non-trivial in the McKean--Vlasov setting (see \cite{honghuliuAAP,HLL25}), due to the inherent nonlocality of the distribution. For example, a standard technique to handle local Lipschitz coefficients in stochastic analysis is the stopping time argument. However, consider the stopped processes $\tilde{X}_t^{\varepsilon}:= X^{\varepsilon}_{t\wedge\tau}$ and $\tilde{Y}_t^{\varepsilon}:= Y^{\varepsilon}_{t\wedge\tau}$ associated with  (\ref{1.1a})-(\ref{1.1b}), where $\tau$ is a certain stopping time. Then, $\tilde{X}_t^{\varepsilon}$ satisfies
		$$\tilde{X}^{\varepsilon}_t =\int_0^{t\wedge\tau } b(\tilde{X}^{\varepsilon}_s, \mathscr{L}_{X^{\varepsilon}_s}, \tilde{Y}^{\varepsilon}_s,\mathscr{L}_{Y^{\varepsilon}_s})ds+\sqrt{\varepsilon}\int_0^{t\wedge\tau }\sigma(\tilde{X}^{\varepsilon}_s, \mathscr{L}_{X^{\varepsilon}_s}, \tilde{Y}^{\varepsilon}_s,\mathscr{L}_{Y^{\varepsilon}_s})dW^{1}_s,$$
		in which the laws of  $X^{\varepsilon}_t,Y^{\varepsilon}_t$ are still involved in the system, rather than only $\tilde{X}^{\varepsilon}_t,\tilde{Y}^{\varepsilon}_t$. As a result, the classical localization argument cannot be directly applied.
		
	\end{enumerate}
\end{remark}

Building upon the general structure, our main results can apply to the following multi-scale McKean--Vlasov systems.
\begin{example}\label{ex11}
	The multi-scale CBO models
	\	\begin{subequations}
		\begin{align*}
			d X^{\varepsilon}_t &= -\big(X^{\varepsilon}_t-{\rm{M}}_{\alpha}^{\ell}(\mathscr{L}_{X^{\varepsilon}_t}, \mathscr{L}_{Y^{\varepsilon}_t})\big) d t+\sqrt{\varepsilon}D_1\big(X^{\varepsilon}_t,{\rm{M}}_{\alpha}^{\ell}( \mathscr{L}_{X^{\varepsilon}_t}, \mathscr{L}_{Y^{\varepsilon}_t})\big) d W^{1}_t, \\
			d Y^{\varepsilon}_t &=-\frac{1}{\delta}\big(\lambda_1Y^{\varepsilon}_t-\lambda_2{\rm{M}}_{\beta}^h(\mathscr{L}_{Y^{\varepsilon}_t})\big)d t+\frac{\sigma_1}{\sqrt{\delta}}D_2\big(\lambda_1Y^{\varepsilon}_t,\lambda_2{\rm{M}}_{\beta}^h(\mathscr{L}_{Y^{\varepsilon}_t})\big)d W^{2}_t,
		\end{align*}
	\end{subequations}
	where the constants $\lambda_1,\lambda_2,\sigma_1>0$ satisfy certain conditions.
\end{example}

\begin{example}\label{ex12}
	The multi-scale ensemble Kalman
	sampler type models
	\begin{subequations}
		\begin{align*}
			d X^{\varepsilon}_t &= \big(-{\rm{Cov}}(\mathscr{L}_{Y^{\varepsilon}_t})X^{\varepsilon}_t+Y^{\varepsilon}_t\big)\,d t+\sqrt{\varepsilon}\Big({\sqrt{{\rm{Cov}}(\mathscr{L}_{X^{\varepsilon}_t})}}+{\sqrt{{\rm{Cov}}(\mathscr{L}_{Y^{\varepsilon}_t})}}\Big)d W^{1}_t,
			\\
			d Y^{\varepsilon}_t &=\frac{1}{\delta}\big(-2Y^{\varepsilon}_t-|Y_t^{\varepsilon}|^2Y_t^{\varepsilon}-{\rm{Cov}}(\mathscr{L}_{X^{\varepsilon}_t})Y^{\varepsilon}_t\big)d t+\frac{1}{\sqrt{\delta}}\sqrt{{\rm{Cov}}(\mathscr{L}_{Y^{\varepsilon}_t})}\,d W^{2}_t.
		\end{align*}
	\end{subequations}
\end{example}

We point out that the coefficients of the multi-scale systems in Examples \ref{ex11} and \ref{ex12} depend on the measure via weighted means and covariances, respectively. The coefficients are  locally Lipschitz continuous in measure and exhibit super-linear growth. See Section \ref{secapp} for more details.
\subsection{Functional law of large numbers}
We informally present the  averaging principle  for multi-scale McKean--Vlasov system (\ref{1.1a})-(\ref{1.1b}), which can be  regarded as a functional  law of large numbers. A precise statement is given in Theorem \ref{theo1}.
\begin{theorem}\label{th01}
	Under suitable assumptions on the coefficients and initial values $\xi, \zeta$, we have
	\begin{align*}
		\mathbb{E}\Big[\sup_{t\in [0, T]}|X_{t}^{\varepsilon}-\bar{X}_{t}|^{2}\Big]\lesssim_T (\varepsilon+\delta^{1/3}).
	\end{align*}
	Here $\bar{X}_t$ is the solution of the averaged equation
	\begin{equation*} \frac{\d\bar{X}_{t}}{dt}=\bar{b}(\bar{X}_t,\mathscr{L}_{\bar{X}_t})\,,\;\;\;\bar{X}_{0}=\xi\,.
	\end{equation*}
\end{theorem}

\begin{remark}
	It should be noted that if the propagation of chaos holds for the multi-scale interacting particle system associated with  (\ref{1.1a})-(\ref{1.1b}), then the averaging principle (i.e., Theorem \ref{th01}) suggests a simplified dynamics obtained
	by passing to the joint limit $\varepsilon\to 0$ and $N\to \infty$, where $N$ denotes the number of particles. This result will be done in our forthcoming work.
	
	In particular, when implementing the algorithm for bi-level optimization problem associated with the multi-scale interacting particle system (\ref{cbo1})-(\ref{cbo2}) on a  computer, it is not necessary  to evolve the $y$-dynamics for a sufficiently long time. This is because  the errors in approximating the consensus of the
	$y$-particles are effectively averaged out during the computation of the consensus of the
	$x$-particles (cf.~\cite[Remark 3.4]{hertyM3MA}).
\end{remark}

We note that an explicit convergence rate of the slow variable is also  derived in Theorem \ref{th01}, which does not appear sharp with respect to the time scales, see Remark \ref{rem2.5} for more details. We now outline the main difficulties and proof strategies. Since we do not impose  smoothness assumptions and only assume local Lipschitz continuity  on the coefficients, we adopt a generalized Khasminskii's time discretization technique rather than  the Poisson equation approach used in \cite{LWX,lixie2023}. To analyze the dynamics (\ref{1.1a})-(\ref{1.1b}) and address the challenges arising from the dependence on $\mathscr{L}_{Y_t^{\varepsilon}}$, we employ a lifted semigroup argument. Note that the Markovian property and long-time behaviour of the pointwise semigroup
\begin{align*}
	P_t \varphi(y)\coloneqq \EE\varphi(Y_t^y),
\end{align*}
with $Y_t^y$ denoting the process starting from $y\in\mathbb{R}^m$, play an essential role in the classical Khasminskii time discretization scheme (cf.~e.g.~\cite{brehiercharles2012,MR4374850,khas1968averaging}). However, in the McKean--Vlasov setting, the operator $P_t$ fails to be  a semigroup (cf.~\cite{WFY}). Instead, the operator defined by
\begin{align*}
	(P_t^* \gamma)(A)\coloneqq \mathscr{L}_{Y_t^{\gamma}}(A),
\end{align*}
where $Y_t^{\gamma}$ is the process starting from the measure $\gamma$, forms   a nonlinear semigroup in the sense that
\begin{align*}
	P_t^*\gamma\neq \int_{\RR^m} (P_t^*\delta_x) d\gamma.
\end{align*}

The key idea of this paper is to utilize a lifted semigroup, which was introduced in \cite{RRW},  associated with the following coupled system
\begin{subequations}
	\begin{align}
		dY_{t}^{\mu,\gamma}&=f(\mu,Y_{t}^{\mu,\gamma},\mathscr{L}_{Y_t^{\mu,\gamma}})dt+g(\mu,Y_{t}^{\mu,\gamma},\mathscr{L}_{Y_t^{\mu,\gamma}})dB_{t}\,, \,\,\,\,\,\,\,\,\,\,\,Y_0^{\mu,\gamma}=\zeta\sim \gamma\,,\label{e12}\\
		dY^{\mu,\gamma,y}_t&=f(\mu,Y^{\mu,\gamma, y}_t,\mathscr{L}_{Y_t^{\mu,\gamma}})dt+g(\mu,Y^{\mu,\gamma,y}_t,\mathscr{L}_{Y_t^{\mu,\gamma}})dB_{t}\,,\,\,\,\,\,Y^{\mu,\gamma,y}_0=y.\label{e13}
	\end{align}
\end{subequations}
The lifted semigroup, acting on $\RR^m\times \mathscr{P}(\RR^d)$, is  defined by
\begin{equation}\label{semi01}
	\tilde{\mathbf{P}}^{\mu}_t\varphi(y,\gamma)\coloneqq\int_{\mathbb{R}^m\times\mathscr{P}(\mathbb{R}^m)} \varphi(z,\eta)\big(\mathscr{L}_{Y^{\mu,\gamma,y}_t}\times \delta_{\mathscr{L}_{Y_{t}^{\mu,\gamma}}}\big)(dz\d\eta),
\end{equation}
which is  Markovian (see \cref{frozen} for more details).
The lifted semigroup is constructed as a linearization of the nonlinear semigroup $P_t^*$, by incorporating the distribution of $Y_t^{\mu,\gamma}$ into the dynamics.
In this way, the lifted semigroup $\tilde{\mathbf{P}}^{\mu}_t$ captures both the evolution of $Y_t^{\mu,\gamma}$ and its distribution, considered as a process on $\RR^m \times \mathscr{P}(\RR^m)$.
Based on this construction, we develop a generalized time discretization technique (see (\ref{e11}) for details) associated with the coupled system (\ref{e12})-(\ref{e13}) in order to simultaneously control the discretized process and the averaged process $\bar{X}_t$.

\subsection{Large deviation principle}
In this paper, we further investigate the  large deviation principle (LDP) in the small-noise limit for the multi-scale McKean-Vlasov system (\ref{1.1a})-(\ref{1.1b}) as $\varepsilon\to 0$.
The LDP provides a powerful framework for characterizing the asymptotic exponential decay of probabilities of rare events, with broad applications in fields such as statistical mechanics and information theory (cf.~\cite{DemboZeitounibook}). Moreover, by the classical Varadhan's theorem (cf.~\cite{V66}), the LDP implies the Laplace principle, which characterizes a minimization problem with nonnegative relative entropy costs. In particular, the Laplace principle is  crucial in the context of the  CBO problem, as it ensures the convergence of weighted empirical measures toward the global minimizer (cf.~\cite{ananalyticalMMM,hajinkim,hakangkim}). This also motivates our present study of the LDP for the multi-scale system (\ref{1.1a})-(\ref{1.1b}), particularly, the multi-scale CBO models (see subsection \ref{seccbo} below for more details).

We now  informally state the LDP result for the multi-scale McKean--Vlasov dynamics (\ref{1.1a})-(\ref{1.1b}), where the precise statement  is presented in Theorem  \ref{t2.1}.

\begin{theorem}
	Assume $\varepsilon=o(\delta)$. Under certain assumptions on  coefficients, for given initial values $x\in\mathbb{R}^n, y\in\mathbb{R}^m$,  $\{X^\varepsilon\}_{\varepsilon\in(0,1)}$ satisfies the LDP with the good rate function $I$ given by
	\begin{align*}
		{I(\varphi)=\begin{cases}\frac{1}{2}\int_0^T\big|Q_2^{-1/2}(\varphi_{t},\mathscr{L}_{\bar{X}_t},\nu^{\mathscr{L}_{\bar{X}_t}})(\dot{\varphi}_t-\bar{b}(\varphi_t,\mathscr{L}_{\bar{X}_t}))\big|^2dt,&\varphi \in \mathscr{H} ,\\[3mm]
				+\infty,&{\rm{otherwise}},
		\end{cases}}
	\end{align*}
	where
	$\mathscr{H}:=\{ \varphi : \varphi_0=x, \varphi~\text{is absolutely continuous}\}$, and
	\begin{align*}			Q_{2}(\varphi_{t},\mathscr{L}_{\bar{X}_{t}},\nu^{\mathscr{L}_{\bar{X}_t}}):=\int_{\mathbb{R}^m}\sigma(\varphi_{t},\mathscr{L}_{\bar{X}_{t}},y,\nu^{\mathscr{L}_{\bar{X}_t}})  \pi_1 \pi_1^*\sigma ^{*}(\varphi_{t},\mathscr{L}_{\bar{X}_{t}},y,\nu^{\mathscr{L}_{\bar{X}_t}}) \nu^{\mathscr{L}_{\bar{X}_t}}(dy).
	\end{align*}
\end{theorem}

The weak convergence approach, systematically introduced by Budhiraja, Dupuis, and Ellis \cite{MR1785237, MR2435853, MR1431744}, provides a powerful framework for establishing LDP. Its theoretical foundation relies on the variational representation of the Laplace transform for bounded continuous functionals. Within the multi-scale setting, Dupuis and Spiliopoulos \cite{MR2914778}  studied the  LDP for locally periodic SDEs with small noise and fast oscillating coefficients, introducing a viable pair framework and employing the weak convergence approach. Building upon this foundation, Hu et al.~\cite{MR4022288} further established the LDP for multi-scale stochastic reaction-diffusion equations. Hong et al.~\cite{MR4634338} proved the LDP for multi-scale McKean--Vlasov diffusions without  dependence on the law of the fast oscillating process by directly justifying the weak convergence criterion. Recently, the authors in \cite{hong2023multi} established the moderate deviation principle for the same model as in \cite{MR4634338}, based on the classical viable pair framework.

We point out that the weak convergence approach  typically proceeds by establishing the weak convergence of a controlled slow process towards its deterministic limit (see, e.g.,~\cite{MR4634338}).  However, in the present work, the slow component of the dynamical system (\ref{1.1a}) is {\it{fully dependent}},  and the fast oscillating process exhibits McKean--Vlasov dynamics. It is highly nontrivial to characterize the limit  of the controlled slow process $X_t^{\varepsilon,h^{\varepsilon}}$ (see (\ref{2.7}) below). As a result,  we cannot directly apply the weak convergence criterion  or the classical viable pair framework.

To address this issue, we  introduce  the notion of a lifted viable pair (cf.~Definition \ref{d2.4} below), extending the definition given in \cite{MR2914778,hong2023multi,MR4022288}, which is more effective within the McKean--Vlasov setting. This lifted viable pair  combines a trajectory and a measure to characterize both the limiting averaging dynamics of the controlled slow process and the invariant measure of the controlled fast process. Moreover,
we employ a generalized functional occupation measure associated with the control variable $h_t^{\varepsilon}$, the controlled fast process $Y_t^{\varepsilon,h^{\varepsilon}}$, and the law of the original fast process $\mathscr{L}_{Y^{\varepsilon}_t}$, i.e.,
\begin{equation}\label{1.2}
	\PPi^{\varepsilon,\Delta}(A_1\times A_2\times A_3\times A_4):=\int_{A_4}\dfrac{1}{\Delta}\int_t^{t+\Delta}\mathbf{1}_{A_1}(h_s^{\varepsilon})\mathbf{1}_{A_2}(Y_s^{\varepsilon,h^{\varepsilon}})\mathbf{1}_{A_3}(\mathscr{L}_{Y_{s}^{\varepsilon}})dsdt,
\end{equation}
and utilize the lifted semigroup argument developed in this work to deal with the controlled dynamics (see (\ref{2.7})-(\ref{2.8}) below).
This  is an effective tool, while the fast process does not converge pathwise to a specific outcome, the occupation measure (\ref{1.2}) converges to a limiting measure. This convergence is achieved by combining the ergodicity of the lifted semigroup (\ref{semi01}) with the classical Birkhoff ergodic theorem (see Lemma \ref{l4.5} below for details).

We will establish the upper and lower bounds for the Laplace principle  via the variational representation of Brownian motion functionals, which yields the LDP. Notably, the proof of the upper bound is more complicate than that of the lower bound. To address this, we  construct a feedback-form control   and employ a mollification argument to achieve the desired bounds.

\subsection{Organization of the paper}
In Section \ref{sec2}, we present the assumptions on the coefficients and the main results concerning the functional  law of large numbers and the LDP. In Section \ref{secapp}, we apply our results to some specific multi-scale examples. In Section \ref{frozen}, we introduce the lifted semigroup and present some apriori estimates asscociated with the frozen McKean--Vlasov system. Section \ref{sec5} is devoted to proving the functional law of large numbers. In Section \ref{sec6}, we prove the large deviation principle. Finally, in the appendix section,  we provide several supplementary proofs for readers' convenience.

\section{Assumptions and main results}\label{sec2}
In this section, we first present some notations and definitions on the large deviation principle and main assumptions on the coefficients of the system (\ref{1.1a})-(\ref{1.1b}). Then we state the main results including the functional type law of large numbers with explicit convergence  rates and the large deviation principle.
\subsection{Definitions}
We recall some standard notations which will be used frequently throughout this paper. Denote by $|\cdot|$ the Euclidean vector norm, $\langle\cdot, \cdot\rangle$ the Euclidean inner product and $\|\cdot\|$ the Hilbert--Schimidt  norm of the matrix.

\vspace{0.1cm}
For $n\in\mathbb{N}$, define $\mathscr{P}(\RR^n)$ the set of all probability measures on $(\RR^n, \mathscr{B}(\RR^n))$ and $\mathscr{P}_{\kappa}(\RR^n), \kappa\geq 1,$ by
$$
\mathscr{P}_{\kappa}(\RR^n)\coloneqq \Big\{\mu\in \mathscr{P}(\RR^n): M_{\kappa}(\mu)\coloneqq\int_{\RR^n}|x|^{\kappa}\,\mu(dx)<\infty\Big\}\,,
$$
then $\mathscr{P}_{\kappa}(\RR^n)$ is a Polish space under the $L^{\kappa}$-Wasserstein distance
$$
\mathbb{W}_{\kappa}(\mu_1,\mu_2):=\inf_{\pi\in \mathscr{C}_{\mu_1,\mu_2}}\left[\int_{\RR^n\times \RR^n}|x-y|^{\kappa}\pi(dx, dy)\right]^{\frac{1}{\kappa}}, \quad \mu_1,\mu_2\in\mathscr{P}_{\kappa}(\RR^n)\,,
$$
where $\mathscr{C}_{\mu_1,\mu_2}$ is the set of all couplings for $\mu_1$ and $\mu_2$.

Let's define the following sets  frequently used in the theory of LDP,
\begin{equation*}
	\mathcal{A}:=\left\{h: h~ \text{is}~\mathbb{R}^{d_1+d_2}\text{-valued}~\mathscr{F}_t\text{-predictable process and}\int_0^T|h_s|^2ds<\infty,\mathbb{P}\text{-a.s.}\right\},
\end{equation*}
\begin{equation*}
	S_M:=\left\{h\in L^2([0,T];\mathbb{R}^{d_1+d_2}):\int_0^T|h_s|^2ds\leq M\right\},
\end{equation*}
and
\begin{equation*}
	\mathcal{A}_{M}:=\Big\{h\in\mathcal{A}:h(\omega)\in S_{M},\mathbb{P}\text{-a.s.}\Big\}.
\end{equation*}

Now we recall the definition of LDP and Laplace principle.  Consider the family of random variables $\{X^\varepsilon\}_{\varepsilon>0}$ defined on a probability space $(\Omega,\mathscr{F},\mathbb{P})$ and taking values in a Polish space $\mathcal{E}$. The rate function of LDP is defined as follows.

\begin{definition}\label{2.1}(Rate function)
	A function $I:\mathcal{E} \to [0,+\infty]$ is called a rate function if $I$ is
	lower semicontinous. Moreover, a rate function $I$ is called a good rate function if for each constant $K$, the level set $\{x\in \mathcal{E}: I(x) \leq K\}$ is a compact subset of $\mathcal{E}$.
\end{definition}

\begin{definition}\label{2.2}(LDP)
	The random variable family $\{X^\varepsilon\}_{\varepsilon\in(0,1)}$ is said to
	satisfy the LDP on $\mathcal{E}$ with rate function $I$ if the following two conditions hold:
	
	(i) (LDP lower bound) For any open set $G \subset \mathcal{E}$,
	
	\begin{equation*}
		\liminf_{\varepsilon\to0}\varepsilon\log\mathbb{P}(X^{\varepsilon}\in G)\geq-\inf_{x\in G}I(x),
	\end{equation*}
	
	(ii) (LDP upper bound) For any closed set $ F \subset \mathcal{E}$,
	\begin{equation*}
		\limsup_{\varepsilon\to0}\varepsilon\log\mathbb{P}(X^\varepsilon\in F)\leq-\inf_{x\in F}I(x).
	\end{equation*}
\end{definition}

\begin{definition}\label{2.3}(Laplace principle)
	The sequence  $\{X^\varepsilon\}_{\varepsilon\in(0,1)}$ is said to be satisfied the Laplace
	principle upper bound (respectively, lower bound) on $\mathcal{E}$ with a rate function $I$ if for each bounded continuous real-valued function $\Lambda$ : $\mathcal{E} \to \mathbb{R}$
	\begin{equation*}
		\limsup_{\varepsilon\to0}-\varepsilon\log\mathbb{E}\Big\{\exp[-\frac{1}{\varepsilon}\Lambda(X^\varepsilon)]\Big\}\leq\inf_{x\in\mathcal{E}}\big(\Lambda(x)+I(x)\big),
	\end{equation*}
	(respectively,
	\begin{align*}
		\liminf_{\varepsilon\to0}-\varepsilon\log\mathbb{E}\Big\{\exp[-\frac{1}{\varepsilon}\Lambda(X^\varepsilon)]\Big\}\geq\inf_{x\in\mathcal{E}}\big(\Lambda(x)+I(x)\big)).
	\end{align*}
\end{definition}	

\begin{remark}
	It is known that the LDP and Laplace principle are equivalent if $\mathcal{E}$ is a Polish space and $I$ is a good rate function (cf.~e.g.~\cite{MR1431744}).
\end{remark}

\subsection{Assumptions}
We suppose that there exist constants $q, \kappa\geq 2$ such that the maps
\begin{align}\nonumber
	&b: \RR^n\times\mathscr{P}(\RR^n)\times\RR^m \times\mathscr{P}(\RR^m)\rightarrow \RR^{n}\,;\\\nonumber
	& \sigma: \RR^n\times\mathscr{P}(\RR^n)\times\RR^m \times\mathscr{P}(\RR^m)\rightarrow \RR^{n}\otimes\RR^{d_1}\,;\\\nonumber
	&f:\mathscr{P}(\RR^n)\times\RR^m \times\mathscr{P}(\RR^m)\rightarrow \RR^{m}\,;\\\nonumber
	&g:\mathscr{P}(\RR^n)\times\RR^m \times\mathscr{P}(\RR^m)\rightarrow \RR^{m}\otimes\RR^{d_2} \nonumber
\end{align}
satisfy the following conditions, in which for all $x_1,x_2\in\RR^n, \mu_1,\mu_2\in \mathscr{P}_{\kappa}(\RR^n), y_1,y_2\in\RR^m, \nu_1,\nu_2\in \mathscr{P}_{\kappa}(\RR^m) $
we denote
\begin{equation*}
	\rho_{\kappa}(x_1, x_2,y_1,y_2)\coloneqq|x_1|^{\kappa}+|x_2|^{\kappa}+|y_1|^{\kappa}+|y_2|^{\kappa},
\end{equation*}
\begin{equation*}
	\ca{M}_{\kappa}(\mu_1,\mu_2,\nu_1,\nu_2)\coloneqq M_{\kappa}(\mu_1)+M_{\kappa}(\mu_2)
	+M_{\kappa}(\nu_1)+M_{\kappa}(\nu_2).
\end{equation*}

\vspace{2mm}
\noindent$(\mathbf{A}_{\text{fast}})$
\noindent (i) There exist constants $C>0$, $ K_1>K_2>0$ such that for all $\mu_1,\mu_2\in \mathscr{P}_{\kappa}(\RR^n), y_1,y_2\in\RR^m,\nu_1,\nu_2\in \mathscr{P}_{\kappa}(\RR^m) $,
\begin{align}\label{con4.1}
		\begin{split}
	&2\langle f(\mu_1, y_1,\nu_1)-f( \mu_2, y_2,\nu_2),y_1-y_2\rangle+\|g(\mu_1, y_1,\nu_1)-g(\mu_2, y_2,\nu_2)\|^2\\
	&\leq -K_1|y_1-y_2|^2+K_2\mathbb{W}_2(\nu_1, \nu_2)^2\\
	&~~~~~+C\big(1+\rho_{\kappa}(0, 0, y_1, y_2)+\ca{M}_{\kappa}(\mu_1,\mu_2, \nu_1, \nu_2)\big)\mathbb{W}_2(\mu_1, \mu_2)^2.
	\end{split}
\end{align}

\noindent{(ii)} There exist constants $K_1>K_2> 0$, $K_0\geq 0$, $C>0$ such that for all $\mu\in \mathscr{P}_2(\RR^n), y\in\RR^m,\nu\in \mathscr{P}_2(\RR^m) $,
\begin{align}\label{inneroffplusg}
	2\langle f(\mu, y,\nu), y\rangle\!+\|g(\mu, y,\nu)\|^2
	\leq & -K_0|y|^{q}-K_1|y|^2+K_2M_2(\nu)+ C.
\end{align}

\noindent{(iii)} For all $\mu\in \mathscr{P}_{\kappa}(\RR^n), y\in\RR^m,\nu\in \mathscr{P}_{\kappa}(\RR^m) $,
\begin{equation}
	|f(\mu, y,\nu)|^2
	\lesssim  \,|y|^{\kappa}+M_{\kappa}(\mu)
	+M_{\kappa}(\nu)+1,
\end{equation}

and there is a constant $\gamma\in [0,1)$ such that
\begin{align}\label{growthofg}
	\|g(\mu, y,\nu)\|^2
	\lesssim  |y|^{2\gamma}+M_{2}(\nu)+1.
\end{align}
$(\mathbf{A}_{\text{slow}})$
\noindent (i) For all $x_1,x_2\in\RR^n, \mu_1,\mu_2\in \mathscr{P}_{\kappa}(\RR^n), y_1,y_2\in\RR^m, \nu_1,\nu_2\in \mathscr{P}_{\kappa}(\RR^m) $,
\begin{align}\label{continuityofb}
	\begin{split}
		&|b(x_1, \mu_1, y_1,\nu_1)-b(x_2, \mu_2, y_2,\nu_2)|^2+\|\sigma(x_1,\mu_1, y_1,\nu_1)-\sigma(x_2,\mu_2, y_2,\nu_2)\|^2\\
		&\lesssim \big(1+\ca{M}_{\kappa}(\mu_1,\mu_2,\nu_1,\nu_2)\big)\big(|x_1-x_2|^2+|y_1-y_2|^2\big)
		\\
		& ~~~~+ \big(1+\rho_{\kappa}(x_1, x_2,y_1,y_2)+\ca{M}_{\kappa}(\mu_1, \mu_2, \nu_1, \nu_2)\big)\big(\mathbb{W}_2(\mu_1, \mu_2)^2+\mathbb{W}_2(\nu_1, \nu_2)^2\big).
	\end{split}
\end{align}

\vspace{2mm}
\noindent (ii) There exist constants $C>0$, $0\leq\tilde{K}_0\leq K_0$ such that for all $x\in\RR^n, \mu\in \mathscr{P}_2(\RR^n), y\in\RR^m,\nu\in \mathscr{P}_2(\RR^m) $,
\begin{equation}\label{innerofb}
	2\langle b(x,\mu, y,\nu), x\rangle\leq \tilde{K}_0|y|^{q}+C(1+|x|^2+M_{2}(\mu)
	+M_{2}(\nu)).
\end{equation}

\noindent (iii) For all $x\in\RR^n, \mu\in \mathscr{P}_{\kappa}(\RR^n), y\in\RR^m,\nu\in \mathscr{P}_{\kappa}(\RR^m) $,
\begin{equation}\label{growthb}
	|b(x, \mu, y,\nu)|^2
	\lesssim |x|^{\kappa}+|y|^{\kappa}+M_{\kappa}(\mu)
	+M_{\kappa}(\nu)+1,
\end{equation}
\begin{equation}\label{growthsigma}
	\|\sigma(x,\mu, y,\nu)\|^2
	\lesssim |x|^{2}+|y|^{2}+M_{2}(\mu)
	+M_{2}(\nu)+1.
\end{equation}

\begin{theorem}\label{thm:wellposedness}
	Under conditions $(\mathbf{A}_{\text{slow}})$ and $(\mathbf{A}_{\text{fast}})$, the system $(\ref{1.1a})$-$(\ref{1.1b})$ is well-posed in the sense that for any initial values $\xi\in L^p(\Omega;\mathbb{R}^n)$, $\zeta\in L^p(\Omega;\mathbb{R}^m)$ with $p>\kappa$, there exists a unique solution $(X_t^{\varepsilon}, Y_t^{\varepsilon})_{t\geq 0}$ to  $(\ref{1.1a})$-$(\ref{1.1b})$.
\end{theorem}
\begin{proof}
	The proof is left in Appendix \ref{appen1}.
\end{proof}

We give some comments on the assumptions $(\mathbf{A}_{\text{slow}})$ and $(\mathbf{A}_{\text{fast}})$ as follows.
\begin{remark}
	In view of the assumption $(\mathbf{A}_{\text{slow}})$ $(i)$ $($see also $(\mathbf{A}_{\text{fast}})$ $(i)$$)$, we assume the locally Lipschitz condition on the coefficients. It seems that instead of $(\mathbf{A}_{\text{slow}})$ $(i)$  a more natural assumption should be
	\begin{align*}
		\begin{split}
			&|b(x_1, \mu_1, y_1,\nu_1)-b(x_2, \mu_2, y_2,\nu_2)|^2+\|\sigma(x_1,\mu_1, y_1,\nu_1)-\sigma(x_2,\mu_2, y_2,\nu_2)\|^2\\
			&\lesssim \big(1+\rho_{\kappa}(x_1, x_2,y_1,y_2)+\ca{M}_{\kappa}(\mu_1,\mu_2,\nu_1,\nu_2)\big)\big(|x_1-x_2|^2+|y_1-y_2|^2
			\\
			& ~~~~+ \mathbb{W}_2(\mu_1, \mu_2)^2+\mathbb{W}_2(\nu_1, \nu_2)^2\big).
		\end{split}
	\end{align*}
	Nevertheless, in contrast to the classical SDEs scenario, we note that there are counterexamples $($cf.~\cite{hong2025meanfieldstochasticpartial,SCHEUTZOW}$)$ in the study of McKean--Vlasov dynamics, e.g.,
	\begin{equation}\label{eqcoun1}
		\frac{dX_t}{dt}=b(X_t)+\mathbb{E}X_t,~X_0=\xi.
	\end{equation}
	This example shows  that if the drift $b$ is merely locally Lipschitz continuous, the pathwise uniqueness, or even uniqueness in distribution, of solutions to Eq.~$(\ref{eqcoun1})$  does not hold in general.
	Therefore, imposing a stronger structural condition on the coefficients $($e.g.~$(\mathbf{A}_{\text{slow}})$ $(i)$$)$ is both reasonable and necessary for ensuring the uniqueness of solutions.
\end{remark}
\begin{remark}
	$(i)$ Since the fast equation $(\ref{1.1b})$ in this work  depends on the law of the fast oscillating process, we adopt the assumption $(\mathbf{A}_{\text{fast}})$ $(i)$ to  guarantee the existence and uniqueness of invariant measures to the frozen McKean--Vlasov dynamics $($i.e., Eq.~$(\ref{eqfrozen})$ below$)$, see \cite[Theorem 3.1]{WFY}. Moreover, because we employ the lifted semigroup argument in this work, the assumption $(\mathbf{A}_{\text{fast}})$ $(i)$ also ensures the existence and uniqueness of
	invariant measures associated with the lifted semigroup; see section \ref{frozen} for details.
	
	$(ii)$ In this work, we consider coefficients that are locally Lipschitz continuous and exhibit sup-linear growth $($i.e.,~$(\mathbf{A}_{\text{slow}})$ $(iii)$ and $(\mathbf{A}_{\text{fast}})$ $(iii)$$)$. This is motivated by the growing interest in multi-scale models within the fields such as
	machine learning and data science.
	
	For example, Herty et al.~\cite{hertyM3MA} recently introduced a multi-scale CBO algorithm to address bi- and tri-level optimization problems, in which the underlying multi-scale model is only locally Lipschitz continuous. In fact, the work \cite{hertyM3MA} only focused on a  modified version of the multi-scale CBO model because the original system, as mentioned in the concluding paragraph of \cite[page 2214]{hertyM3MA}, fails to satisfy the
	assumptions required by recent advancements in  the averaging principle $($cf.~\cite{BS23,MR4374850,LWX,lixie2023,rocknersunxiepoincare}$)$. To address this problem, we propose a more general framework for multi-scale McKean--Vlasov diffusions in the present study, enabling us to apply the results to the  general CBO models.

\end{remark}

\subsection{Main results}
Define an averaged process $\bar{X}$ that satisfies the following averaged equation
\begin{equation}\label{eqav}
	\frac{\d\bar{X}_{t}}{dt}=\bar{b}(\bar{X}_t,\mathscr{L}_{\bar{X}_t})\,,\;\;\;\bar{X}_{0}=\xi\,,
\end{equation}
where the (averaged) coefficient $\bar{b}$ is given by
$$\bar{b}(x,\mu)=\int_{\RR^m\times\mathscr{P}(\RR^m)}b(x,\mu, z,\eta)\,(\nu^{\mu}\times \delta_{\nu^{\mu}})(dz\d\eta),$$
and $\nu^{ \mu}$ denoting the unique invariant measure of the following frozen equation (for fixed  $\mu$)
\begin{equation}\label{eqfrozen}
	dY_{t}=f(\mu, Y_{t},\mathscr{L}_{Y_{t}})dt+g(\mu, Y_{t},\mathscr{L}_{Y_{t}})dB_{t},~Y_{0}=\zeta.
\end{equation}
Here, $\{B_{t}\}_{t\geq 0}$ is a $d_2$-dimensional standard Brownian motion on a complete filtration  probability space $(\tilde{\Omega}, \tilde{\mathscr{F}},  \{\tilde{\mathscr{F}}_{t}\}_{t\geq0},\tilde{\mathbb{P}})$.

The following is the first main result considering the functional type law of large numbers with explicit rates.
\begin{theorem}\label{theo1}
	Suppose that the assumptions $(\mathbf{A}_{\text{slow}})$ and $(\mathbf{A}_{\text{fast}})$  hold. Then for any initial values $\xi\in L^p(\Omega;\mathbb{R}^n)$, $\zeta\in L^p(\Omega;\mathbb{R}^m)$ with $p>\kappa$, it holds that
	\begin{align}\label{t2.155555}
		\mathbb{E}\Big[\sup_{t\in [0, T]}|X_{t}^{\varepsilon}-\bar{X}_{t}|^{2}\Big]\lesssim_T (\varepsilon+\delta^{1/3}).
	\end{align}
\end{theorem}

\begin{remark}
	Compared with previous works such as \cite{BS23,MR4374850,rocknersunxiepoincare}, we allow the coefficients to depend on the law of the fast oscillating process, which play an important role in applications. See the next section for detailed applications on the multi-scale CBO models and the multi-scale ensemble Kalman sampler type models. Moreover, a  contribution here also includes developing the lifted semigroup argument to deal with the dependence on the law of the fast oscillating process, which might be employed to treat other multi-scale problems in the future.
\end{remark}
\begin{remark}\label{rem2.5}
	Note that the convergence rate established here is not optimal for the time scale $\delta$ compared to recent works \cite{LWX,lixie2023,rocknersunxiepoincare}. This is mainly due to the fact that   we only assume the coefficients are locally Lipschitz continuous and do not impose any additional smoothness assumptions. A further interesting  problem is to impose suitable growth conditions on the derivative of the coefficients to archive the optimal rate,  which would require appropriate notions of derivatives on the distribution,  such as the Lions derivative or the linear functional derivative.
\end{remark}

In the context of the large deviation principle,  fix the initial values
$$X^\varepsilon_0=x\in\mathbb{R}^n,\,\,\,\,\,\,\,\,\,Y^\varepsilon_0=y\in\mathbb{R}^m.$$
As a consequnce, the solution $\bar{X}_t$ to averaged eqution \eqref{eqav} starts from $\bar{X}_0=x$. Let $d=d_1+d_2$. Denote by
$$
\pi_1:\mathbb{R}^{d}\to \mathbb{R}^{d_1},\qquad
\pi_2:\mathbb{R}^{d}\to \mathbb{R}^{d_2}
$$
the canonical projection mappings. Consider the mapping
\begin{align}
	\Theta:&\,\,\,\mathbb{R}^n\times\mathscr{P}_2(\mathbb{R}^n)\times\mathbb{R}^{d}\times\mathbb{R}^m\times\mathscr{P}_2(\mathbb{R}^m)\to\mathbb{R}^n\nonumber\\[1mm]
	&	(x,\mu,h,y,\rho)\mapsto \bar{b}(x,\mu)+\sigma(x,\mu,y,\rho)\pi_1 h .\label{2.14}
\end{align}
In order to investigate the large deviation principle for the slow variable $X_t^{\varepsilon}$ from its averaged equation $\bar X_t$, we introduce, for the first time, the notion of a lifted viable pair, which generalized the definition introduced in e.g.~\cite{hong2023multi, MR4022288} to  adapt the dependency of the law of fast oscillating process in the  McKean--Vlasov setting.
\begin{definition} (Lifted viable pair)\label{d2.4}
	A pair $$(\varphi,\PPi)\in C([0,T];\mathbb{R}^n)\times\mathscr{P}(\mathbb{R}^d\times\mathbb{R}^m\times\mathscr{P}_2(\RR^m)\times[0,T])$$
	is called viable w.r.t.~$(\Theta,\nu^{\mathscr{L}_{\bar{X}}})$ and write $( \varphi, \PPi ) \in \mathscr{V} _{(x,\bar{X},\Theta,\nu^{\mathscr{L}_{\bar{X}}})}$, if the following statements hold:
	
	(i) The measure $\PPi$ admits the following moment estimate
	
	\begin{equation}\label{2.15}
		\int_{\mathbb{R}^d\times\mathbb{R}^m\times\mathscr{P}_2(\RR^m)\times[0,T]}\big[|h|^{2}+|y|^{4}+M_4(\rho)\big]\PPi(dhdyd\rho ds)<\infty.
	\end{equation}

	(ii) For all $t\in [ 0, T] , $
	
	\begin{equation}\label{2.16}
		\varphi_{t}=x+\int_{\mathbb{R}^d\times\RR^m\times\mathscr{P}_2(\RR^m)\times[0,t]}\Theta(\varphi_{s},\mathscr{L}_{\bar{X}_{s}},h,y,\rho)\PPi(dhdyd\rho ds).
	\end{equation}
	
	(iii) For all $A_1\times A_2\times A_3\times A_4\in\mathscr{B}(\mathbb{R}^d\times\mathbb{R}^m\times\mathscr{P}_2(\RR^m)\times[0,T])$,
	
	\begin{equation}\label{2.17}
		\PPi(A_{1}\times A_{2}\times A_{3}\times A_{4})=\int_{A_{4}}\int_{A_{2}\times A_{3}}\eta(A_{1}|y,\rho,t)(\nu^{\mathscr{L}_{\bar{X}_t}}\times \delta_{\nu^{\mathscr{L}_{\bar{X}_t}}})(dyd\rho)dt,	
	\end{equation}
	where $\eta$ denotes a stochastic kernel (see \cite[Appendix B.2]{MR3967100}) conditioned on $\mathbb{R}^m\times\mathscr{P}_2(\RR^m)\times[0,T]$. This ensures that the marginal of $\PPi$ on $[0,T]$ is Lebesgue measure  and in particular for all $t\in[0,T]$,
	\begin{equation}\label{2.18}
		\PPi(\mathbb{R}^d\times\mathbb{R}^m\times\mathscr{P}_2(\RR^m)\times[0,t])=t.
	\end{equation}
\end{definition}

In addition to the assumptions considered in Theorem \ref{theo1}, in the study of the large deviation principle we further assume  the following assumptions.\\

\noindent$(\mathbf{H}_1)$
The constant $q \geq \max\{4,\kappa\}$ and $\tilde{K}_0>0$.

\vspace{1mm}

\noindent$(\mathbf{H}_2)$ $g$ is bounded in $\mu$ and $y$, i.e.,
$$\|g(\mu, y,\nu)\|^2
\lesssim  M_{2}(\nu)+1.$$

\vspace{1mm}

\noindent$(\mathbf{H}_3)$ There exists $c_1>0$ such that for all $ x\in\RR^n, \mu\in \mathscr{P}_2(\RR^n), y\in\RR^m,\nu\in \mathscr{P}_2(\RR^m)$ and $\eta\in \mathbb{R} ^n $,
$$
\langle\sigma\sigma^{*}(x,\mu,y,\nu)\eta,\eta\rangle\geq c_{1}|\eta|^{2}.
$$

We are in the position to state the large deviation result.
\begin{theorem}\label{t2.1}
	Suppose that the assumptions $(\mathbf{A}_{\text{slow}})$, $(\mathbf{A}_{\text{fast}})$, $(\mathbf{H}_1)$ and $(\mathbf{H}_2)$ hold and that $\delta = o(\varepsilon)$.
	
	\vspace{1mm}
	\textbf{Case 1.} $[$$\sigma$ is independent of $y$ $($i.e.~$\sigma(x,\mu,y,\nu)\equiv\sigma(x,\mu,\nu)$$)]$.  $\{X^\varepsilon\}_{\varepsilon\in(0,1)}$  satisfies the LDP with the good rate function $I$ given by
	\begin{equation}\label{sulv}
		I(\varphi):=\inf_{(\varphi,\PPi)\in\mathscr{V}_{(x,\bar{X},\Theta,\nu^{\mathscr{L}_{\bar{X}}})}}\left\{\frac{1}{2}\int_{\mathbb{R}^d\times\mathbb{R}^m\times\mathscr{P}_2(\RR^m)\times[0,T]}|h|^{2} \PPi(dhdyd\rho dt)\right\},
	\end{equation}
	with the convention that $\inf\emptyset= \infty$.
	
	\vspace{1mm}
	\textbf{Case 2.} $[$$\sigma$ depends on $y$$]$. Suppose further that the assumption $(\mathbf{H}_3)$ holds. Then $\{X^\varepsilon\}_{\varepsilon\in(0,1)}$  satisfies the LDP with the good rate function $I$ given by $(\ref{sulv})$.
	Furthermore, the rate function $I$ admits the following explicit representation
	\begin{align}\label{z2.20}
		{I(\varphi)=\begin{cases}\frac{1}{2}\int_0^T|Q_2^{-1/2}(\varphi_{t},\mathscr{L}_{\bar{X}_t},\nu^{\mathscr{L}_{\bar{X}_t}})(\dot{\varphi}_t-\bar{b}(\varphi_t,\mathscr{L}_{\bar{X}_t}))|^2dt,&\varphi \in \mathscr{H} ,\\+\infty,&otherwise,\end{cases}}
	\end{align}
	where
	$\mathscr{H}:=\big\{ \varphi : \varphi_0=x, \varphi~\text{is absolutely continuous}\big\}$, and
	\begin{align}\label{z2.21}
		Q_{2}(\varphi_{t},\mathscr{L}_{\bar{X}_{t}},\nu^{\mathscr{L}_{\bar{X}_t}}):=\int_{\mathbb{R}^m}\sigma(\varphi_{t},\mathscr{L}_{\bar{X}_{t}},y,\nu^{\mathscr{L}_{\bar{X}_t}})  \pi_1 \pi_1^*\sigma ^{*}(\varphi_{t},\mathscr{L}_{\bar{X}_{t}},y,\nu^{\mathscr{L}_{\bar{X}_t}}) \nu^{\mathscr{L}_{\bar{X}_t}}(dy).
	\end{align}
	
\end{theorem}

\begin{remark}
	(i) Under the assumption $(\mathbf{H}_3)$, since $\pi_1:\mathbb{R}^{d_1 + d_2} \to \mathbb{R}^{d_1}$ is a projection map, it follows directly that $(\sigma \pi_1)(\sigma \pi_1)^{*}$ is uniformly positive definite, and consequently so is $Q_{2}$.
	
	(ii) In fact, the non-degeneracy condition on  $\sigma\sigma^*$ assumed in \textbf{Case 2} in Theorem \ref{t2.1} is  only required for establishing the upper bound of the Laplace principle. More specifically, it is needed  to transfer the control into a feedback form and to derive the explicit expression of the rate function (see subsection \ref{sub5.5} for details). In contrast, the derivation of the lower bound of the Laplace principle remains valid  even for degenerate noise and does not require  this non-degeneracy condition.
\end{remark}

\begin{remark}
	(i) To the best of our knowledge, the large deviation results for the multi-scale system $(\ref{1.1a})$-$(\ref{1.1b})$ seem to be new even in the case where the coefficients are sufficiently smooth.
	In the existing works (cf.~\cite{MR4634338,SZW, WHY}), the LDP for multi-scale McKean--Vlasov systems that do not depend on the law of the fast oscillating process was  studied by employing the powerful weak convergence criterion directly. In addition, a very recent work \cite{hong2023multi} established a moderate deviation principle for the same model as in \cite{MR4634338} by employing the classical viable pair framework introduced in \cite{MR2914778,MR4022288}. Notice that all of the aforementioned works (i.e.,~\cite{hong2023multi,MR4634338,SZW, WHY}) also impose a strong structural restriction:  the diffusion term $\sigma$ does not depend on the fast oscillating process. In fact, if the coefficient $\sigma$ depends on $Y^{\varepsilon}_t$, the weak convergence criterion cannot be applied directly.
	
	(ii) To address the dependence on the fast oscillating process and its time-marginal law, we develop a notion of lifted viable pair that extends the definitions in \cite{MR2914778,MR4022288} to govern the dynamics of the control, the  fast process, and  its  law. Moreover, we combine the generalized functional
	occupation measure approach with the lifted semigroup argument to characterize the limit of the controlled slow process, thereby deriving the large deviation result (see subsection \ref{sub5.3} below for details).
\end{remark}

\section{Applications}\label{secapp}
\subsection{Multi-scale  consensus-based optimization}\label{seccbo}
The consensus-based optimization is a use of opinion dynamics and consensus formation in global optimization. As explained in the part of Introduction, the multi-scale CBO model  is applied to treat the  bi-optimization problem
\begin{align*}
	&\min\,\,\ell(x,y)\,,\\
	s.t. \,\,\,\,&y\in \arg\min\,\,h\,.
\end{align*}

Consider the large $N$ (and $M$) limit of the multi-scale CBO interacting systems, we study the following multi-scale McKean--Vlasov dynamics on $\RR^n\times \RR^m$,	 where the measure dependency comes from the weighted mean,
\begin{subequations}
	\begin{align}
		d X^{\varepsilon}_t &= -\big(X^{\varepsilon}_t-{\rm{M}}_{\alpha}^{\ell}(\mathscr{L}_{X^{\varepsilon}_t}, \mathscr{L}_{Y^{\varepsilon}_t})\big) d t+\sqrt{\varepsilon}D_1\big(X^{\varepsilon}_t,{\rm{M}}_{\alpha}^{\ell}( \mathscr{L}_{X^{\varepsilon}_t}, \mathscr{L}_{Y^{\varepsilon}_t})\big) d W^{1}_t\,,\label{eq:examCBOsystem1} \\
		d Y^{\varepsilon}_t &=-\frac{1}{\delta}\big(\lambda_1Y^{\varepsilon}_t-\lambda_2{\rm{M}}_{\beta}^h(\mathscr{L}_{Y^{\varepsilon}_t})+\lambda_3\nabla\Phi(Y^{\varepsilon}_t)\big)d t+\frac{\sigma_1}{\sqrt{2\delta}}D_2\big(\lambda_1Y^{\varepsilon}_t,\lambda_2{\rm{M}}_{\beta}^h(\mathscr{L}_{Y^{\varepsilon}_t})\big)d W^{2}_t\,,\label{eq:examCBOsystem2}
	\end{align}
\end{subequations}
where  the constants $\lambda_1,\lambda_2,\lambda_3,\sigma_1\geq 0$ satisfy $2\lambda_1>\sigma_1^2\lambda_1^2+1$ with small enough $\lambda_2$,  $D_1,D_2$ are the anisotropic diffusion given by
$$D_1(x, a):={\rm{diag}}(x-a),~~D_2(y, b):={\rm{diag}}(\chi_{R_0}(y-b)),$$
with the cut-off function $\chi_{R_0}:\mathbb{R}^m\to \mathbb{R}^m$ defined by
$$\chi_{R_0}(u)=\begin{cases} u,~~~~~~~~~~~~|u|\leq R_0,&\quad\\
	(u/|u|)R_0,~|u|>R_0,&\quad\end{cases}$$
and for every  $\mu\in\mathscr{P}(\RR^n), \nu\in\mathscr{P}(\RR^m)$, the weighted means ${\rm{M}}_{\beta}^h$, ${\rm{M}}_{\alpha}^{\ell}$ are defined by
$${\rm{M}}_{\beta}^h(\nu)\footnote{Note the different notations of ${\rm{M}}$, $\mathscr{M}$ and $M$.}\coloneqq \frac{\int_{\RR^m} y {\rm{e}}^{-\beta h(y)}d \nu}{\int_{\RR^m}  {\rm{e}}^{-\beta h(y)}d \nu}\,,$$
$${\rm{M}}_{\alpha}^{\ell}(\mu, \nu)\coloneqq \frac{\int_{\RR^n} x {\rm{e}}^{-\alpha \ell(x, {\rm{M}}_{\beta}^h(\nu))}d \mu}{\int_{\RR^n}  {\rm{e}}^{-\alpha \ell(x, {\rm{M}}_{\beta}^h(\nu))}d \mu}\,,$$
where $\alpha, \beta>0$.
The potential $\Phi$ satisfies the following conditions
\begin{enumerate}
	\item [$(\mathbf{A}^1_{\Phi})$]  The mapping $y\mapsto\nabla\Phi(y)$ is continuous;
	
	\vspace{1mm}
	\item [$(\mathbf{A}^2_{\Phi})$]  There exist $q\geq 4,C_{\Phi}>0$ such that for any $y\in\mathbb{R}^m$,
	$$\langle\nabla\Phi(y),y\rangle\geq C_{\Phi} |y|^q;$$
	
	\vspace{1mm}
	\item [$(\mathbf{A}^3_{\Phi})$]  For any $y_1,y_2\in\mathbb{R}^m$,
	$$\langle\nabla\Phi(y_1)-\nabla\Phi(y_2),y_1-y_2\rangle \geq 0;$$

	\vspace{1mm}
	\item [$(\mathbf{A}^4_{\Phi})$]  For any $y\in\mathbb{R}^m$,
	$$|\nabla\Phi(y)|\lesssim(1+|y|^{q-1}),$$
	where $q$ is the same as in  $(\mathbf{A}^2_{\Phi})$.
\end{enumerate}

\begin{remark}
	A typical example of the potential $\nabla\Phi$ is $$\nabla\Phi(y):=|y|^2y,$$ where we can choose $q=4$ in $(\mathbf{A}^2_{\Phi})$.
\end{remark}

\begin{theorem}\label{examCBO}
	Suppose that the assumptions $(\mathbf{A}^1_{\Phi})$-$(\mathbf{A}^4_{\Phi})$ hold. The  functional type law of large numbers holds for the multi-scale CBO model $(\ref{eq:examCBOsystem1})$-$(\ref{eq:examCBOsystem2})$, i.e.,
	\begin{align*}
		\mathbb{E}\Big[\sup_{t\in [0, T]}|X_{t}^{\varepsilon}-\bar{X}_{t}|^{2}\Big]\lesssim_T (\varepsilon+\delta^{1/3}),
	\end{align*}
	provided that $\lambda_2$ is sufficiently small, where $\bar{X}_t$ is the solution of the corresponding averaged equation. Moreover,
	if $\delta = o(\varepsilon)$ and $\lambda_3> 0$, then $\{X^{\varepsilon}\}_{\varepsilon\in (0,1)}$ satisfies the LDP.
\end{theorem}
Before proving Theorem \ref{examCBO}, we give the following lemma as preparation, which should be  of independent interest in the study of the consensus-based optimization problem.
\begin{lemma}\label{lem:weightedmean}
	Assume that there exist positive constants ${\rm{Lip}}(\ell)$, ${\rm{Lip}}(h)$, $ c_u, c_l$ such that for all $x, \hat{x} \in \RR^n, y, \hat{y}\in\RR^m,$
	\begin{align*}
		|h(y)-h(\hat{y})|\leq\,\,& {\rm{Lip}}(h)\big(|y|+|\hat{y}|\big)|y-\hat{y}|\,,\\[1mm]
		c_l\leq\, h(y)-\underline{h}\leq\,\,& c_u\,,\\[1mm]
		|\ell(x, y)-\ell(\hat{x}, \hat{y})|\leq\,\,& {\rm{Lip}}(\ell)\big(|x|+|y|+|\hat{x}|+|\hat{y}|\big)\big(|x-\hat{x}|+|y-\hat{y}|\big)\,,\\[1mm]
		\ell(x, y)-\underline{\ell}\leq\,\,& c_u\big(1+|x|^2+|y|^2\big)\,,
	\end{align*}
	where $\underline{h}\coloneqq \inf_{x\in\RR^m} h(y)>-\infty$ and  $\underline{\ell}\coloneqq \inf_{x\in \RR^n, y\in\RR^m} \ell(x, y)>-\infty$. Then, for any $\nu,  \hat{\nu}\in\mathscr{P}_4(\RR^m)$, $\mu,  \hat{\mu}\in\mathscr{P}_4(\RR^n)$, it holds that	
	\begin{align}
		&|{\rm{M}}_{\beta}^h(\nu)|^2\leq  \,\, {\rm{e}}^{-2\beta(c_l-c_u)}M_2(\nu)\,,\label{ass:hlowerbound}\\
		&|{\rm{M}}_{\beta}^h(\nu)-{\rm{M}}_{\beta}^h(\hat{\nu})|\leq {\rm{e}}^{\alpha c_u} (1+\alpha {\rm{Lip}}(h)(1+ {\rm{e}}^{\alpha c_u})p_k)\,\, \bb{W}_2(\nu,\hat{\nu})\,,\label{ass:hLip}\\
		&|{\rm{M}}_{\alpha}^{\ell}(\mu, \nu)-{\rm{M}}_{\alpha}^{\ell}(\hat{\mu}, \hat{\nu})|\nonumber\\
&\lesssim_{\alpha, {\rm{Lip}}(\ell)}\,\, \big(M_4(\mu)+M_4(\hat{\mu})+M_4(\nu)+M_4(\hat{\nu})\big)\big(\bb{W}_2(\mu,\hat{\mu})+\bb{W}_2(\nu,\hat{\nu})\big)\,.\label{ass:ellLip}
	\end{align}
	If in addition, there is $M>0$ such that
	for any $y_0\in\RR^m$,
	\begin{align}
		\ell(x, y_0)-\underline{\ell}\geq\,\,& c_l|x|^2\,,\,\,\,\,\,\,\text{for all} \,\,\,|x|> M+|y_0|,\label{ass:elllowerbound}
	\end{align}
	then it holds that
	\begin{align}
		|{\rm{M}}^{\ell}_{\alpha}(\mu, \nu)|^2\lesssim_{\alpha, M, c_u, c_l}\,\,& (1+M_2(\mu)+ M_2(\nu))\,.\label{ass:ellupbound}
	\end{align}
\end{lemma}

\begin{proof}
	The estimate $\eqref{ass:hlowerbound}$ follows from the fact that $h$ is bounded from below and above. As for the globally Lipschitz continuity \eqref{ass:hLip},
	it follows from  \cite[Lemma 3.2]{ananalyticalMMM} line by line by noticing that $\frac{{\rm{e}}^{-\beta \underline h}}{\|\omega^{h}\|_{L^1(\mu)}}$ is bounded from above.
	
	It remains to prove  \eqref{ass:ellLip} and \eqref{ass:ellupbound}. In the following, we omit the weights $\alpha, \beta$ and integral domains for conciseness. For any $\mu, \hat{\mu}\in \mathscr{P}_4(\RR^n), \nu, \hat{\nu}\in \mathscr{P}_4(\RR^m)$, we have the following  decomposition
	\begin{align}
		&\frac{x\,\omega^{\ell}(x, {\rm{M}}^h(\nu))}{\|\omega^{\ell}(\cdot, {\rm{M}}^h(\nu))\|_{L^1(\mu)}}-\frac{\hat{x}\,\omega^{\ell}(\hat{x}, {\rm{M}}^h(\hat{\nu}))}{\|\omega^{\ell}(\cdot, {\rm{M}}^h(\hat{\nu}))\|_{L^1(\hat{\mu})}}\nonumber\\[1mm]
	=\,\,& \bigg(\frac{x\,\omega^{\ell}(x, {\rm{M}}^h(\nu))- \hat{x}\,\omega^{\ell}(\hat{x}, {\rm{M}}^h(\nu))}{\|\omega^{\ell}(\cdot, {\rm{M}}^h(\nu))\|_{L^1(\mu)}}\bigg)+\bigg(\frac{\hat{x}\,\omega^{\ell}(\hat{x}, {\rm{M}}^h(\nu))- \hat{x}\,\omega^{\ell}(\hat{x}, {\rm{M}}^h(\hat{\nu}))}{\|\omega^{\ell}(\cdot, {\rm{M}}^h(\nu))\|_{L^1(\mu)}}\bigg)\nonumber\\[1mm]
	&+\bigg(\frac{\hat{x}\,\omega^{\ell}(\hat{x}, {\rm{M}}^h(\hat{\nu}))}{\|\omega^{\ell}(\cdot, {\rm{M}}^h(\nu))\|_{L^1(\mu)}}-\frac{\hat{x}\,\omega^{\ell}(\hat{x}, {\rm{M}}^h(\hat{\nu}))}{\|\omega^{\ell}(\cdot, {\rm{M}}^h(\hat{\nu}))\|_{L^1(\hat{\mu})}}\bigg)\nonumber\\[1mm]
	\eqqcolon\,\,&\rm{I}_1+\rm{I}_2+\rm{I}_3\,.	\label{es:elllipI}
	\end{align}
	Let $K\coloneqq \int |x|^4 \,d \mu+\int |x|^4 \,d \nu$ and $\hat{K}\coloneqq \int |x|^4 \,d \hat{\mu}+\int |x|^4 \,d \hat{\nu}$.  Using Jensen's inequality, we obtain that for any fixed $y_0\in\RR^d$,
	\begin{align*}
		\int {\rm{e}}^{-\alpha ( \ell(x, y_0)- \underline{\ell})} \,d \mu\,\geq {\rm{e}}^{-\alpha c_{u}\int 1+|x|^2+|y_0|^2\, d\mu}\,\geq {\rm{e}}^{-\alpha c_{u}(1+K)}{\rm{e}}^{-\alpha c_{u}|y_0|^2}\,.
	\end{align*}
	Hence, it follows that
	\begin{align*}
		\frac{{\rm{e}}^{-\alpha \underline{\ell}}}{\int {\rm{e}}^{-\alpha  \ell(x, y_0)} d\mu}=\frac{1}{\int {\rm{e}}^{-\alpha ( \ell(x, y_0)- \underline{\ell})} \,d \mu}\leq {\rm{e}}^{\alpha c_{u}(1+K)}{\rm{e}}^{\alpha c_{u}|y_0|^2}\eqqcolon c_K{\rm{e}}^{\alpha c_{u}|y_0|^2}\,.
	\end{align*}
	For the term $\rm{I}_1$, we see that
	\begin{align}
		{\rm{I}}_1\leq\,\,& |x-\hat{x}|c_K+\frac{|\hat{x}|\alpha {\rm{e}}^{-\alpha \underline{\ell}}|\ell(x, {\rm{M}}^h(\nu))-\ell(\hat{x}, {\rm{M}}^h(\nu))|}{\|\omega^{\ell}(\cdot, {\rm{M}}^h(\nu))\|_{L^1(\mu)}}\nonumber\\[1mm]
\leq\,\,&|x-\hat{x}|c_K{\rm{e}}^{\alpha c_{u}|{\rm{M}}^h(\nu)|^2}+\alpha |\hat{x}|c_K{\rm{e}}^{\alpha c_{u}|{\rm{M}}^h(\nu)|^2}{\rm{Lip}}(\ell)(|x|+|\hat{x}|+|{\rm{M}}^h(\nu)|)|x-\hat{x}|\,.	\label{es:elllipI1}
	\end{align}
	Analogously, it is easy to see that the  term $\rm{I}_2$ is bounded by
	\begin{align}\label{es:elllipI2}
		|\hat{x}|\alpha c_K{\rm{e}}^{\alpha c_{u}|{\rm{M}}^h(\nu)|^2}{\rm{Lip}}(\ell)(|\hat{x}|+|{\rm{M}}^h(\nu)|+|{\rm{M}}^h(\hat{\nu})|)({\rm{M}}^h(\nu)-{\rm{M}}^h(\hat{\nu}))\,.
	\end{align}
	As for the term $\rm{I}_3$, we estimate it as follows. For any coupling $\pi_{\mu, \hat{\mu}}\in \mathscr{C}(\mu, \hat{\mu})$,
	\begin{align}
	{\rm{I}}_3=\,\,& \frac{\hat{x}\,\omega^{\ell}(\hat{x}, {\rm{M}}^h(\hat{\nu})) \big(\|\omega^{\ell}(\cdot, {\rm{M}}^h(\hat{\nu}))\|_{L^1(\hat{\mu})}-\|\omega^{\ell}(\cdot, {\rm{M}}^h(\nu))\|_{L^1(\mu)}\big)}{\|\omega^{\ell}(\cdot, {\rm{M}}^h(\nu))\|_{L^1(\mu)}\|\omega^{\ell}(\cdot, {\rm{M}}^h(\hat{\nu}))\|_{L^1(\hat{\mu})}}\nonumber\\[1mm]
\leq\,\,&|\hat{x}| c_K{\rm{e}}^{\alpha c_{u}|{\rm{M}}^h(\nu)|^2} \frac{\big|\iint {\rm{e}}^{-\alpha \ell(\hat{x}, {\rm{M}}^h(\hat{\nu}))}-{\rm{e}}^{-\alpha \ell(x, {\rm{M}}^h(\nu)) }d \pi_{\mu, \hat{\mu}}\big|}{\|\omega^{\ell}(\cdot, {\rm{M}}^h(\hat{\nu}))\|_{L^1(\hat{\mu})}}\nonumber\\[1mm]
\leq\,\,&|\hat{x}| c_K{\rm{e}}^{\alpha c_{u}|{\rm{M}}^h(\nu)|^2} \nonumber\\
\,\,&\cdot
\frac{\alpha{\rm{Lip}}(\ell){\rm{e}}^{-\alpha  \underline {\ell}} \iint (|x|+|\hat{x}|+{\rm{M}}^h(\nu)+{\rm{M}}^h(\hat{\nu}))\big(|x-\hat{x}|+|{\rm{M}}^h(\nu)-{\rm{M}}^h(\hat{\nu})|\big)d \pi_{\mu, \hat{\mu}}}{\|\omega^{\ell}(\cdot, {\rm{M}}^h(\hat{\nu}))\|_{L^1(\hat{\mu})}}\nonumber\\[1mm]
\leq\,\,&\alpha{\rm{Lip}}(\ell) c_K{\rm{e}}^{\alpha c_{u}|{\rm{M}}^h(\nu)|^2} c_{\hat{K}}{\rm{e}}^{\alpha c_{u}|{\rm{M}}^h(\hat{\nu})|^2}|\hat{x}|\nonumber\\[1mm]
&\,\,\,\cdot\iint (|x|+|\hat{x}|+{\rm{M}}^h(\nu)+{\rm{M}}^h(\hat{\nu}))\big(|x-\hat{x}|+|{\rm{M}}^h(\nu)-{\rm{M}}^h(\hat{\nu})|\big)d \pi_{\mu, \hat{\mu}}\,.\label{es:elllipI3}
		\end{align}
	By plugging the estimates \eqref{es:elllipI1}, \eqref{es:elllipI2}, \eqref{es:elllipI3} into \eqref{es:elllipI},  we arrive at
	\begin{align*}
		&\,\,\,\,\,\,\,\,\,\frac{x\,\omega^{\ell}(x, {\rm{M}}^h(\nu))}{\|\omega^{\ell}(\cdot, {\rm{M}}^h(\nu))\|_{L^1(\mu)}}-\frac{\hat{x}\,\omega^{\ell}(\hat{x}, {\rm{M}}^h(\hat{\nu}))}{\|\omega^{\ell}(\cdot, {\rm{M}}^h(\hat{\nu}))\|_{L^1(\hat{\mu})}}\nonumber\\[1mm]
	&\lesssim_{K, \hat{K},\alpha, {\rm{Lip}}(\ell)}\,\,|x-\hat{x}|+|\hat{x}|(1+|x|+|\hat{x}|)(|x-\hat{x}|)+(|\hat{x}|+|\hat{x}|^2)\mathbb{W}_2(\nu, \hat{\nu})\nonumber\\
	&\,\,\,\,\,\,\,\,\,+|\hat{x}|\iint(|x|+|\hat{x}|)(|x-\hat{x}|) d\pi_{\mu, \hat{\mu}}+|\hat{x}|\mathbb{W}_2(\nu, \hat{\nu})\iint(|x|+|\hat{x}|) d\pi_{\mu, \hat{\mu}}\nonumber\\
	&\,\,\,\,\,\,\,\,\,+|\hat{x}|\iint|x-\hat{x}| d\pi_{\mu, \hat{\mu}}+|\hat{x}|\mathbb{W}_2(\nu, \hat{\nu})\nonumber\\
	&\lesssim_{K, \hat{K},\alpha, {\rm{Lip}}(\ell)}\,\, \big(1+|x||\hat{x}|+|\hat{x}|^2+|\hat{x}|\big)|x-\hat{x}|+(|\hat{x}|+|\hat{x}|^2)\big(\mathbb{W}_2(\mu, \hat{\mu})+\mathbb{W}_2(\nu, \hat{\nu})\big)\,.		
		\end{align*}
	Hence, it follows that for any coupling $\pi_{\mu, \hat{\mu}}\in \mathscr{C}({\mu, \hat{\mu}})$,
	\begin{align*}
		|{\rm{M}}_{\ell}^{\alpha}(\mu, \nu)-{\rm{M}}_{\ell}^{\alpha}(\hat{\mu}, \hat{\nu})|
&=\,\,\iint \frac{x\,\omega^{\ell}(x, {\rm{M}}^h(\nu))}{\|\omega^{\ell}(\cdot, {\rm{M}}^h(\nu))\|_{L^1(\mu)}}-\frac{\hat{x}\,\omega^{\ell}(\hat{x}, {\rm{M}}^h(\hat{\nu}))}{\|\omega^{\ell}(\cdot, {\rm{M}}^h(\hat{\nu}))\|_{L^1(\hat{\mu})}}d\pi_{\mu, \hat{\mu}}\nonumber\\
&\lesssim\,_{K, \hat{K},\alpha, {\rm{Lip}}(\ell)}\,\, \mathbb{W}_2(\mu, \hat{\mu})+\mathbb{W}_2(\nu, \hat{\nu})\,.			
	\end{align*}
	
	In order to estimate the second moment of ${\rm{M}}^{\ell}(\mu, \nu)$, we first define a probability measure
	\begin{align*}
		\eta^{y_0}\coloneqq \frac{{\rm{e}}^{-\alpha \ell(\cdot, y_0)}d\mu}{\|{\rm{e}}^{-\alpha \ell(\cdot, y_0)}\|_{L^1(\mu)}}\,,
	\end{align*}
	where $y_0\in\RR^d$, and then link $M_2(\eta^{y_0})$ to the tail estimates of $\mu$.  It is clear that for any $R_0\geq M+|y_0|$,
	\begin{align*}
		\int |x|^2 \frac{{\rm{e}}^{-\alpha \ell(x, y_0)}d\mu}{\|{\rm{e}}^{-\alpha \ell(\cdot, y_0)}\|_{L^1(\mu)}}=\,\,&\int |x|^2 d \eta^{y_0}\\
		\leq \,\,&R_0^2+\int_{B(R_0)^c} (|x|^2-R_0^2 )\,d\eta^{y_0}\\
		= \,\,& R_0^2+\sum_{i=1}^{\infty}\int_{A_i}(|x|^2-R_0^2 )\,d\eta^{y_0}\\
		\leq\,\,&R_0^2+R_0^2\sum_{i=1}^{\infty}(2^i-1) \eta^{y_0}(A_i)\,,
	\end{align*}
	where $A_i=\{2^{(i-1)/2} R_0\leq |x|\leq 2^{i/2}R_0\}$. Since in each $A_i$, $|x|\geq R_0\geq M+|y_0|$, it follows from \eqref{ass:elllowerbound} that
	$$A_i\subset \big\{\ell(x, y_0)-\underline{\ell}\geq c_l2^{i-1}R_0^2 \big\}\,,$$
	hence,
	\begin{align}
		\int |x|^2 \frac{{\rm{e}}^{-\alpha \ell(x, y_0)}d\mu}{\|{\rm{e}}^{-\alpha \ell(\cdot, y_0)}\|_{L^1(\mu)}} 	\leq\,\,&R_0^2+R_0^2\sum_{i=1}^{\infty}(2^i-1) \eta^{y_0}\big(\big\{\ell(x, y_0)-\underline{\ell}\geq c_l 2^{i-1}R_0^2 \big\}\big)\,.\label{es:secondmomemteta}
	\end{align}
	Notice that for any $K\geq 0$, the definition of $\eta^{y_0}$ yields that
	\begin{align}
		\eta^{y_0}\big(\big\{\ell(x, y_0)-\underline{\ell}\geq K\big\}\big)=\,\,&\frac{1}{\|{\rm{e}}^{-\alpha\ell(\cdot, y_0)}\|_{L^1(\mu)}}\int_{\big\{\ell(x, y_0)-\underline{\ell}\geq K\big\}}{\rm{e}}^{-\alpha\ell(x, y_0)}\,d \mu\nonumber\\
\leq\,\,& \frac{1}{\|{\rm{e}}^{-\alpha\ell(\cdot, y_0)}\|_{L^1(\mu)}}{\rm{e}}^{-\alpha(\underline{\ell}+K)}\mu\big(\{\ell(x, y_0)-\underline{\ell}\geq K\}\big)\,.\label{es:tail1}
	\end{align}
	Moreover, for any $l\geq c_l(1+|y_0|^2+M_2(\nu))$,
	\begin{align*}
		\|{\rm{e}}^{-\alpha\ell(\cdot, y_0)}\|_{L^1(\mu)}\geq\,\,& \int_{\big\{\ell(x, y_0)-\underline{\ell}\leq l\big\}}{\rm{e}}^{-\alpha\ell(x, y_0)} d\mu\\[1mm]
		\geq\,\,&{\rm{e}}^{-\alpha(\underline{\ell}+l)}\mu\big(\{\ell(x, y_0)-\underline{\ell}\leq l\} \big)\,,	
		\end{align*}
	substituting it into \eqref{es:tail1} leads to
	\begin{align*}
		\eta^{y_0}\big(\big\{\ell(x, y_0)-\underline{\ell}\geq K\big\}\big)\leq\,\,& {\rm{e}}^{-\alpha(K-l)}\frac{\mu\big(\{\ell(x, y_0)-\underline{\ell}\geq K\}\big)}{\mu\big(\{\ell(x, y_0)-\underline{\ell}\leq l\} \big)}\\
		\leq\,\,&{\rm{e}}^{-\alpha(K-l)}\frac{\frac{1}{K}\int c_u(1+|x|^2+|y_0|^2) d\mu}{1-\frac{1}{l}\int c_u(1+|x|^2+|y_0|^2) d\mu },
	\end{align*}
	where we use the Chebyshev's inequality in the second inequality to estimate the tail part.
	
	By choosing $R_0^2=(M+|y_0|)^2+2\frac{c_u}{c_l} (1+|y_0|^2+M_2(\mu))> (M+|y_0|)^2$, $K=c_l 2^{i-1}R_0^2$ and $l=c_l 2^{i-2}R_0^2$ and putting the estimate into \eqref{es:secondmomemteta}, it holds that
	\begin{align}
		&\int |x|^2 \frac{{\rm{e}}^{-\alpha \ell(x, y_0)}d\mu}{\|{\rm{e}}^{-\alpha \ell(\cdot, y_0)}\|_{L^1(\mu)}}\nonumber\\
		\leq\,\,& R_0^2+R_0^2\sum_{i=1}^{\infty}(2^i-1) {\rm{e}}^{-\alpha c_l2^{i-2}R_0^2}\frac{\frac{1}{c_l 2^{i-1}R_0^2}c_u(1+|y_0|^2+M_2(\mu))}{1-\frac{1}{c_l 2^{i-2}R_0^2}c_u(1+|y_0|^2+M_2(\mu)) }\nonumber\\
		\leq\,\,&R_0^2+2R_0^2\frac{c_u}{c_l}\frac{1+|y_0|^2+M_2(\mu)}{(M+|y_0|)^2 }\sum_{i=1}^{\infty}(1-2^{-i}) {\rm{e}}^{-\alpha c_l2^{i-2}R_0^2}\nonumber\\
		\leq\,\,&R_0^2+2R_0^2\frac{c_u}{c_l}\frac{1+|y_0|^2+M_2(\mu)}{(M+|y_0|)^2 }\int_{1/2}^{\infty} {\rm{e}}^{-R_0^2s}\,d s\nonumber\\
		\leq\,\,&(M+|y_0|)^2+2\frac{c_u}{c_l} (1+|y_0|^2+M_2(\mu))+\frac{2c_u}{ c_l}\frac{1+|y_0|^2+M_2(\mu)}{(M+|y_0|)^2 }\,,\nonumber
	\end{align}
	where in the second-to-last inequality, we choose  $\alpha$ large enough such that $\alpha c_l\geq 2$.	Hence, \eqref{ass:ellupbound}
	follows by  choosing $y_0={\rm{M}}^h(\nu)$ and using $|{\rm{M}}^h(\nu)|^2\leq M_2(\nu)$.
\end{proof}

According to the  estimates established in Lemma \ref{lem:weightedmean}, it suffices to verify  $(\mathbf{A}_{\text{fast}})$, $(\mathbf{A}_{\text{slow}})$,  $(\mathbf{H}_1)$ and $(\mathbf{H}_2)$, then Theorem \ref{examCBO} follows  directly  from  Theorem \ref{theo1} and Case 1 of Theorem \ref{t2.1}.

\vspace{3mm}
\noindent{\textbf{Proof of Theorem \ref{examCBO}.}}
For any $x\in \RR^n,y\in \RR^m, \mu\in\mathscr{P}(\RR^n)$, and $\nu\in \mathscr{P}(\RR^m)$, we set
$$f(y, \nu):=-(\lambda_1y-\lambda_2{\rm{M}}^h(\nu)+\lambda_3\nabla\Phi(y)),\,\,\,\,\,\,\,\,\,\,\,\,g(y, \nu):=\frac{\sigma_1}{\sqrt{2}}D_2(\lambda_1y-\lambda_2{\rm{M}}^h(\nu)),$$
$$b(x,\mu,\nu):=-(x-{\rm{M}}^{\ell}(\mu, \nu)),\,\,\,\,\,\,\,\ \sigma(x,\mu,\nu):=D_1(x-{\rm{M}}^{\ell}(\mu, \nu))\,.$$
We only verify the assumptions $(\mathbf{A}_{\text{fast}})$ (i) and (ii) and $(\mathbf{A}_{\text{slow}})$ (i), whereas the others are more straightforward.

Firstly,  by $(\mathbf{A}^3_{\Phi})$ and using Young's inequality, we have
\begin{align}
	&2\langle f(y_1, \nu_1)-f(y_2, \nu_2),y_1-y_2\rangle+\|g(y_1, \nu_1)-g(y_2, \nu_2)\|^2\nonumber\\
	= \,\,& 2\langle -\lambda_1(y_1-y_2)+\lambda_2({\rm{M}}^h(\nu_1)-{\rm{M}}^h(\nu_2))-\lambda_3(\nabla\Phi(y_1)-\nabla\Phi(y_2)),y_1-y_2\rangle\nonumber\\
	&
	+\frac{\sigma_1^2}{2}|\chi_{R_0}(\lambda_1y_1-\lambda_2{\rm{M}}^h(\nu_1))-\chi_{R_0}(\lambda_1y_2-\lambda_2{\rm{M}}^h(\nu_2))|^2\nonumber\\
	\leq \,\,& 2\langle -\lambda_1(y_1-y_2)+\lambda_2({\rm{M}}^h(\nu_1)-{\rm{M}}^h(\nu_2))-\lambda_3(\nabla\Phi(y_1)-\nabla\Phi(y_2)),y_1-y_2\rangle\nonumber\\
	&
	+\frac{\sigma_1^2}{2}|\lambda_1(y_1-y_2)-\lambda_2({\rm{M}}^h(\nu_1)-{\rm{M}}^h(\nu_2))|^2\nonumber\\
	\leq\,\,& -(2\lambda_1-\sigma_1^2\lambda_1^2)|y_1-y_2|^2+2\lambda_2|{\rm{M}}^h(\nu_1)-{\rm{M}}^h(\nu_2)|\cdot|y_1-y_2|\nonumber\\
&
+\sigma_1^2\lambda_2^2|{\rm{M}}^h(\nu_1)-{\rm{M}}^h(\nu_2)|^2\nonumber\\[1mm]
	\leq\,\,&-(2\lambda_1-\sigma_1^2\lambda_1^2-1)|y_1-y_2|^2+(1+\sigma_1^2)\lambda_2^2|{\rm{M}}^h(\nu_1)-{\rm{M}}^h(\nu_2)|^2\nonumber\\
	\leq\,\,&  -(2\lambda_1-\sigma_1^2\lambda_1^2)|y_1-y_2|^2+2\lambda_2|{\rm{M}}^h(\nu_1)-{\rm{M}}^h(\nu_2)|\cdot|y_1-y_2|\nonumber\\
&
+\sigma_1^2\lambda_2^2|{\rm{M}}^h(\nu_1)-{\rm{M}}^h(\nu_2)|^2\nonumber\\[1mm]
	\leq\,\,&-(2\lambda_1-\sigma_1^2\lambda_1^2-1)|y_1-y_2|^2+(1+\sigma_1^2)\lambda_2^2C_{\alpha, {\rm{Lip}}(h)}\bb{W}_2(\nu_1, \nu_2)^2\,,	\label{es:ex1local}
	\end{align}
Thus, $(\mathbf{A}_{\text{fast}})$ (i) holds due to the assumption $2\lambda_1>\sigma_1^2\lambda_1^2+1$ with small enough $\lambda_2$.

Similarly, by $(\mathbf{A}^2_{\Phi})$, we obtain
\begin{align}
	&2\langle f(y, \nu),y\rangle+\|g(y, \nu)\|^2\nonumber\\
	\leq\,\,&-2\lambda_1|y|^2+2\lambda_2|{\rm{M}}^h(\nu)|\cdot|y|-2\lambda_3C_{\Phi}|y|^q+C_{R_0}\nonumber\\[1mm]
	\leq\,\,&-(2\lambda_1-1)|y|^2-2\lambda_3C_{\Phi}|y|^q+\lambda_2^2|{\rm{M}}^h(\nu)|^2+C_{R_0}\nonumber\\[1mm]
	\leq\,\,&-(2\lambda_1-1)|y|^2-2\lambda_3C_{\Phi}|y|^q+\lambda_2^2{\rm{e}}^{-4\beta (c_l-c_u)}M_2(\nu)^2+C_{R_0}.\label{es:ex1local}
\end{align}
Then $(\mathbf{A}_{\text{fast}})$ (ii) holds.

Finally, a direct calculation yields that
\begin{align*}
	&|b(x_1, \mu_1, \nu_1)-b(x_2, \mu_2, \nu_2)|^2+\|\sigma(x_1,\mu_1, \nu_1)-\sigma(x_2,\mu_2, \nu_2)\|^2\\[1mm]
	\leq \,\,&4|x_1-x_2|^2+4|{\rm{M}}^{\ell}(\mu_1, \nu_1)-{\rm{M}}^{\ell}(\mu_2, \nu_2)|^2\\[1mm]
	\lesssim\,\,&_{\alpha, {\rm{Lip}}(\ell)}|x_1-x_2|^2+\mathcal{M}_4(\mu_1, \mu_2, \nu_1, \nu_2)\bb{W}_2(\mu_1,\mu_2)^2+\bb{W}_2(\nu_1,\nu_2)^2\,,
\end{align*}
where we use the estimate \eqref{ass:ellLip} in the last inequality. Then $(\mathbf{A}_{\text{slow}})$ (i) follows. \hspace{\fill}$\Box$

\subsection{Multi-scale  ensemble Kalman sampler}
Recently, the ensemble Kalman sampler (cf.~\cite{Garbuno-InigoHoffmannLiStuart.2020.SJoADS412,vaes2024sharp}) is exploited to find approximately i.i.d. samples for the target distribution. The essential insight of designing  ensemble system is that a large number of particles are sampled firstly from an easy-to-sample distribution, such as Gaussian or uniform, and then move around according to certain dynamics, expecting in large time they reconstruct the target distribution.

In the following example, we consider a multi-scale McKean--Vlasov systems on $\RR^n\times \RR^n$ based on the ensemble Kalman sampler,	 where the measure dependency is through the covariance matrix,
\begin{subequations}\label{ex2:system}
	\begin{align}
		d X^{\varepsilon}_t &= \big(-{\rm{Cov}}(\mathscr{L}_{Y^{\varepsilon}_t})X^{\varepsilon}_t+Y^{\varepsilon}_t\big)\,d t+\sqrt{\varepsilon}\Big({\sqrt{{\rm{Cov}}(\mathscr{L}_{X^{\varepsilon}_t})}}+{\sqrt{{\rm{Cov}}(\mathscr{L}_{Y^{\varepsilon}_t})}}\Big)\,d W^{1}_t\,,\label{ex2:systema}
		\\
		d Y^{\varepsilon}_t &=\frac{1}{\delta}\big(-2Y^{\varepsilon}_t-|Y_t^{\varepsilon}|^2Y_t^{\varepsilon}-{\rm{Cov}}(\mathscr{L}_{X^{\varepsilon}_t})Y^{\varepsilon}_t\big)\,d t+\frac{1}{\sqrt{\delta}}\sqrt{{\rm{Cov}}(\mathscr{L}_{Y^{\varepsilon}_t})}\,d W^{2}_t\,,	\label{ex2:systemb}
	\end{align}
\end{subequations}
where for any probability measure $\mu\in\mathscr{P}(\RR^n)$, the covariance operator $${\rm{Cov}}(\mu)\coloneqq\int_{\RR^n} (x-m(\mu))\otimes (x-m(\mu)) \,d \mu(x)$$  with $m(\mu)=\int_{\RR^n} x\,d \mu(x)$.

\begin{theorem}
	The  functional type law of large numbers holds for the multi-scale ensemble Kalman sampler model $(\ref{ex2:systema})$-$(\ref{ex2:systemb})$, i.e.,
	\begin{align*}
		\mathbb{E}\Big[\sup_{t\in [0, T]}|X_{t}^{\varepsilon}-\bar{X}_{t}|^{2}\Big]\lesssim_T (\varepsilon+\delta^{1/3}),
	\end{align*}
	where $\bar{X}_t$ is the solution of the corresponding averaged equation. Moreover,
	if $\delta = o(\varepsilon)$, then $\{X^{\varepsilon}\}_{\varepsilon\in (0,1)}$ satisfies the LDP.
\end{theorem}

\begin{proof} For any $x, y\in \RR^n$ $\mu, \nu\in \mathscr{P}(\RR^n)$, define
	\begin{align*}
		f(\mu, y)&:=-2y-|y|^2y-{\rm{Cov}}(\mu)y,\,\,\,\,\,\,\,\,\,\,\,\,\,g(\nu):=\sqrt{{\rm{Cov}}(\nu)}\,,\\
		b(x, y, \nu)&:=-{\rm{Cov}}(\nu)x+y,\,\,\,\,\,\,\,\,\,\,\,\,\,\,\,\,\,\,\sigma(\mu,\nu):=\sqrt{{\rm{Cov}}(\mu)}+{\sqrt{{\rm{Cov}}(\nu)}}.
	\end{align*}
	We only verify the assumptions $(\mathbf{A}_{\text{fast}})$ (i) and (ii) and $(\mathbf{A}_{\text{slow}})$ (i). Thanks to \cite[Lemma 2]{vaes2024sharp} and \cite[(5.2)]{gess2025random}, we have that for all $\mu,\nu\in\mathscr{P}_2(\RR^n),$
	\begin{align}
		\big\|{\rm{Cov}}(\mu)-{\rm{Cov}(\nu)}\big\|\leq &\,\,2\big(M_2(\mu)+M_2(\nu)\big)\bb{W}_2(\mu,\nu), \label{es:sqrtcov2}\\
		\big\|\sqrt{{\rm{Cov}}(\mu)}-\sqrt{{\rm{Cov}}(\nu)}\big\|\leq &\,\,\sqrt{2}\bb{W}_2(\mu,\nu).\label{es:sqrtcov3}
	\end{align}
	As for the assumption $(\mathbf{A}_{\text{fast}})$ (i), 	due to (\ref{es:sqrtcov2})-(\ref{es:sqrtcov3}), the semi-positivity of the covariance matrix leads that	\begin{align*}
		&2\langle f(\mu_1, y_1)-f( \mu_2, y_2),y_1-y_2\rangle+\|g(\nu_1)-g(\nu_2)\|^2\\\
		\leq \,\,&-4|y_1-y_2|^2+2\bb{W}_2(\nu_1,\nu_2)+2\|{\rm{Cov}}(\mu_1)-{\rm{Cov}}(\mu_2)\|\cdot |y_1|\cdot|y_1-y_2|\\
		\leq \,\,&-3|y_1-y_2|^2+2\bb{W}_2(\nu_1,\nu_2)+ 4|y_1|^2	\big(M_2(\mu_1)^2+M_2(\mu_2)^2\big)\bb{W}_2(\mu_1,\mu_2)^2\,.
	\end{align*}
	In regards of the assumption $(\mathbf{A}_{\text{fast}})$ (ii), it follows from (\ref{es:sqrtcov3}) and the semi-positivity of covariance matrix that
	\begin{align}
		2\langle f(\mu,y), y\rangle+\|g(\nu)\|^2\leq\,\, &2\langle-2y-|y|^2y-{\rm{Cov}}(\mu)y, y\rangle +2M_2(\nu)\nonumber\\
\leq\,\,&-4|y|^2-2|y|^4+2M_2(\nu)\,.\label{es:ex2coercivity}
	\end{align}
	
	Now, we  verify $(\mathbf{A}_{\text{slow}})$ (i). It follows from \eqref{es:sqrtcov2}-\eqref{es:sqrtcov3} that
	\begin{align*}
		&|b(x_1, y_1,\nu_1)-b(x_2,  y_2,\nu_2)|^2+\|\sigma(\mu_1, \nu_1)-\sigma(\mu_2, \nu_2)\|^2\\
		\leq &\,\,\,2|\big({\rm{Cov}}(\nu_1)-{\rm{Cov}}(\nu_2)\big)x_1+{\rm{Cov}}(\nu_2)(x_1-x_2)|^2+|y_1-y_2|^2\\\
		&+\big\|\sqrt{{\rm{Cov}}(\mu_1)}+{\sqrt{{\rm{Cov}}(\nu_1)}}-\sqrt{{\rm{Cov}}(\mu_2)}-{\sqrt{{\rm{Cov}}(\nu_2)}}\big\|^2\\[1mm]
		\lesssim&\,\,\,(1+M_2(\nu_2))(|x_1-x_2|^2+|y_1-y_2|^2)\\
		&\,+(1+M_2(\nu_1)^2|x_1|^2+M_2(\nu_2)^2|x_1|^2)(\bb{W}_2(\mu_1,\mu_2)^2+\bb{W}_2(\nu_1,\nu_2)^2)\,.
	\end{align*}
	We complete the proof.
\end{proof}

\section{Lifted semigroup and frozen McKean--Vlasov systems}\label{frozen}
In this section, we study the following {\it{frozen}} McKean--Vlasov equations related to the fast process $Y_t^{\varepsilon}$. Namely, for any fixed $\mu\in\mathscr{P}_{\kappa}(\RR^n)$, consider
\begin{align}\label{Yfro}
	dY_{t}=f(\mu,Y_{t},\mathscr{L}_{Y_t})dt+g(\mu,Y_{t},\mathscr{L}_{Y_t})dB_{t},~~~~Y_{0}=\zeta\thicksim \gamma,
\end{align}
where $\{B_{t}\}_{t\geq 0}$ is a $d_2$-dimensional standard Brownian motion on a complete filtration  probability space $(\tilde{\Omega}, \tilde{\mathscr{F}},  \{\tilde{\mathscr{F}}_{t}\}_{t\geq0},\tilde{\mathbb{P}})$. It follows from \cite{WFY} that under the assumption (\ref{con4.1}), the system \eqref{Yfro}  has a unique invariant measure $\nu^{\mu}$ associated to the (nonlinear) semigroup $\tilde{P}^{\mu*}_t$  given by
$$\tilde{P}^{\mu*}_t\gamma:=\mathscr{L}_{Y_{t}}, $$
and there exists $\beta>0$ independent of $t\geq 0$ such that for any initial law $\gamma\in\mathscr{P}_{\kappa}(\RR^m)$,
\begin{equation}\label{es:expentialergodicity}
	\mathbb{W}_2(\tilde{P}^{\mu*}_t\gamma,\nu^{\mu})\lesssim e^{-\beta t}\mathbb{W}_2(\gamma,\nu^{\mu}).
\end{equation}

For fixed $\mu\in\mathscr{P}_{\kappa}(\RR^n)$, $y\in\RR^m$, consider
\begin{subequations}\label{eq:fro}
	\begin{align}
		dY_{t}^{\mu,\gamma}&=f(\mu,Y_{t}^{\mu,\gamma},\mathscr{L}_{Y_t^{\mu,\gamma}})dt+g(\mu,Y_{t}^{\mu,\gamma},\mathscr{L}_{Y_t^{\mu,\gamma}})dB_{t}\,, \,\,\,\,\,\,\,\,\,\,\,Y_0^{\mu,\gamma}=\zeta\sim \gamma\,,\label{fro1}\\[1mm]
		dY^{\mu,\gamma,y}_t&=f(\mu,Y^{\mu,\gamma, y}_t,\mathscr{L}_{Y_t^{\mu,\gamma}})dt+g(\mu,Y^{\mu,\gamma,y}_t,\mathscr{L}_{Y_t^{\mu,\gamma}})dB_{t}\,,\,\,\,\,\,Y^{\mu,\gamma,y}_0=y.\label{fro2}
	\end{align}
\end{subequations}
Thanks to the well-posedness of \eqref{fro1},
$\big(\mathscr{L}_{Y_t^{\mu,\gamma}}\big)_{t\in [0, T]}$ is a known law. Furthermore, we have the following estimates.
\begin{lemma}\label{lem:estimateoffro}
	For any $p\geq 2$, there exist constants $\beta, C_p>0$ such that for any $\mu\in\mathscr{P}_{\kappa}(\RR^n)$, $\gamma\in\mathscr{P}_p(\RR^m),y\in\RR^m$,
	\begin{equation}\label{es7}
		\sup_{t\geq 0}\tilde \EE|Y_{t}^{\mu,\gamma}|^p\lesssim_p(1+M_p(\gamma)),
	\end{equation}	
	\begin{align}\label{es:Yxmuxiy}
		\tilde \EE|Y_{t}^{\mu,\gamma,y}|^p\leq e^{-\beta t}|y|^p+C_{p}(1+M_p(\gamma)).
	\end{align}
\end{lemma}
\begin{proof}
	By It\^{o}'s formula, we obtain
	\begin{equation}
		\tilde{\mathbb{E}}|Y_{t}^{\mu,\gamma}|^{2}= |\zeta|^2+\tilde{\mathbb{E}}\int_0^t\Big[2\langle Y_{s}^{\mu,\gamma},f(\mu,Y_{s}^{\mu,\gamma},\mathscr{L}_{Y_s^{\mu,\gamma}})\rangle
		+\|g(\mu,Y_{s}^{\mu,\gamma},\mathscr{L}_{Y_s^{\mu,\gamma}})\|^2\Big]ds\nonumber
	\end{equation}
	and
	\begin{eqnarray}
		\tilde{\mathbb{E}}|Y_{t}^{\mu,\gamma}|^{p}= \!\!\!\!\!\!\!\!&&|\zeta|^p+p\,\tilde{\mathbb{E}}\int_0^t|Y_{s}^{\mu,\gamma}|^{p-2}\langle Y_{s}^{\mu,\gamma},f(\mu,Y_{s}^{\mu,\gamma},\mathscr{L}_{Y_s^{\mu,\gamma}})\rangle ds
		\nonumber\\
		\!\!\!\!\!\!\!\!&&+\frac{p(p-2)}{2}\tilde{\mathbb{E}}\int_0^t|Y_{s}^{\mu,\gamma}|^{p-4}|g^*(\mu,Y_{s}^{\mu,\gamma},\mathscr{L}_{Y_s^{\mu,\gamma}})
		\cdot Y_{s}^{\mu,\gamma}|^2ds
		\nonumber\\
		\!\!\!\!\!\!\!\!&&
		+\frac{p}{2}\,\tilde{\mathbb{E}}\int_0^t|Y_{s}^{\mu,\gamma}|^{p-2}\|g(\mu,Y_{s}^{\mu,\gamma},\mathscr{L}_{Y_s^{\mu,\gamma}})
		\|^2ds.\nonumber
	\end{eqnarray}
	Due to the assumption (\ref{inneroffplusg}), it follows that
	\begin{equation}
		\frac{\d}{dt}\,\tilde{\mathbb{E}}|Y_{t}^{\mu,\gamma}|^{2}\leq -(K_1-K_2)\tilde{\mathbb{E}}|Y_{t}^{\mu,\gamma}|^{2}+C,
		\nonumber
	\end{equation}
	where the constant $C>0$ is independent of $t$. Gronwall's lemma implies that
		\begin{equation}\label{e03}
			\sup_{t\geq 0}\tilde{\mathbb{E}}|Y_{t}^{\mu,\gamma}|^{2}\lesssim (1+M_2(\gamma)).
		\end{equation}
		
		Furthermore, using the conditions (\ref{inneroffplusg}) and (\ref{growthofg}), there is a constant $\beta\in (0,\frac{K_1p}{2})$,
		\begin{align}
				\frac{\d}{dt}\,\tilde{\mathbb{E}}|Y_{t}^{\mu,\gamma}|^{p}\leq  &-\frac{K_1p}{2}\tilde{\mathbb{E}}|Y_{t}^{\mu,\gamma}|^{p}+\frac{K_2p}{2}\tilde{\mathbb{E}}|Y_{t}^{\mu,\gamma}|^{p-2}\tilde{\mathbb{E}}|Y_{t}^{\mu,\gamma}|^{2}\nonumber\\
	&
+C_p\tilde{\mathbb{E}}|Y_{t}^{\mu,\gamma}|^{p-2}+C_p\,\tilde{\mathbb{E}}|Y_{t}^{\mu,\gamma}|^{p-2+2\gamma}
	\nonumber\\[1mm]
	\leq &-\beta\tilde{\mathbb{E}}|Y_{t}^{\mu,\gamma}|^{p}+C_p(1+M_p(\gamma)),\label{e04}
		\end{align}
		where we have used (\ref{e03}) in the last inequality, which yields (\ref{es7}) by making use of Gronwall's lemma.

		On the other hand, by It\^{o}'s formula  we also obtain
		\begin{eqnarray} 
\tilde{\mathbb{E}}|Y_{t}^{\mu,\gamma,y}|^{p}= \!\!\!\!\!\!\!\!&&|y|^p+p\,\tilde{\mathbb{E}}\int_0^t|Y_{s}^{\mu,\gamma,y}|^{p-2}\langle Y_{s}^{\mu,\gamma,y},f(\mu,Y_{s}^{\mu,\gamma,y},\mathscr{L}_{Y_s^{\mu,\gamma}})\rangle ds
			\nonumber\\
			\!\!\!\!\!\!\!\!&&+\frac{p(p-2)}{2}\tilde{\mathbb{E}}\int_0^t|Y_{s}^{\mu,\gamma,y}|^{p-4}|g^*(\mu,Y_{s}^{\mu,\gamma,y},\mathscr{L}_{Y_s^{\mu,\gamma}})
			\cdot Y_{s}^{\mu,\gamma,y}|^2ds
			\nonumber\\
			\!\!\!\!\!\!\!\!&&
			+\frac{p}{2}\,\tilde{\mathbb{E}}\int_0^t|Y_{s}^{\mu,\gamma,y}|^{p-2}\|g(\mu,Y_{s}^{\mu,\gamma,y},\mathscr{L}_{Y_s^{\mu,\gamma}})
			\|^2ds.\nonumber
		\end{eqnarray}
		Similar to the proof of (\ref{e04}), using the estimate  \eqref{es7}, there is a constant $\beta\in (0,\frac{K_1p}{2})$,
		\begin{eqnarray}
			\frac{\d}{dt}\,\tilde{\mathbb{E}}|Y_{t}^{\mu,\gamma,y}|^{p}\leq  		\!\!\!\!\!\!\!\!&&-\frac{K_1p}{2}\tilde{\mathbb{E}}|Y_{t}^{\mu,\gamma,y}|^{p}+C_p\,\tilde{\mathbb{E}}|Y_{t}^{\mu,\gamma,y}|^{p-2}(1+\tilde{\mathbb{E}}|Y_{t}^{\mu,\gamma}|^{2})
			+C_p\,\tilde{\mathbb{E}}|Y_{t}^{\mu,\gamma,y}|^{p-2+2\gamma}
			\nonumber\\[1mm]
			\leq  		\!\!\!\!\!\!\!\!&&-\beta\tilde{\mathbb{E}}|Y_{t}^{\mu,\gamma,y}|^{p}+C_p(1+\tilde{\mathbb{E}}|Y_{t}^{\mu,\gamma}|^{p}).\nonumber
		\end{eqnarray}
		Applying Gronwall's lemma, we have
		\begin{align}
			\tilde{\mathbb{E}}|Y_{t}^{\mu,\gamma,y}|^{p}\leq e^{-\beta t}|y|^p+C_{p}(1+M_p(\gamma)),\nonumber
		\end{align}
		which gives (\ref{es:Yxmuxiy}).
	\end{proof}

	\begin{lemma}
		There exists a constant $C>0$ such that for any $y_i\in\RR^m$, $\mu_i\in\mathscr{P}_{\kappa}(\RR^n)$, $\zeta_i\sim\gamma_i\in\mathscr{P}_{\kappa}(\RR^m)$, $i=1,2$, we have
		\begin{align}
		&\tilde{\mathbb{E}} |Y^{\mu_1,\gamma_1}_{t}\!-\!Y^{\mu_2,\gamma_2}_{t} |^2
\nonumber	\\
		\leq\,\,\, & e^{-(K_1-K_2) t}\EE|\zeta_1-\zeta_2|^2+ C\big(1+M_{\kappa}(\gamma_1)+M_{\kappa}(\gamma_2)
	\nonumber	\\
		&+M_{\kappa}(\mu_1)+M_{\kappa}(\mu_2)\big)\mathbb{W}_2(\mu_1,\mu_2)^2 \label{ergo1}
			\end{align}
		and
		\begin{align}
			&\tilde{\mathbb{E}} |Y^{\mu_1,\gamma_1,y_1}_{t}\!-\!Y^{\mu_2,\gamma_2,y_2}_{t} |^2
\nonumber	\\
		\leq \,\,\,& C\,\big(1+|y_1|^{\kappa}+|y_2|^{\kappa}+M_{\kappa}(\gamma_1)+M_{\kappa}(\gamma_2)+M_{\kappa}(\mu_1)+M_{\kappa}(\mu_2)\big)\mathbb{W}_2(\mu_1,\mu_2)^2
\nonumber	\\
		&+e^{-K_1 t}|y_1-y_2|^2+e^{-(K_1-K_2) t}\EE|\zeta_1-\zeta_2|^2.\label{ergo2}
		\end{align}
	\end{lemma}
	
	\begin{proof}
		By It\^{o}'s formula, due to  the condition \eqref{con4.1}, we have
		\begin{align}
			&\frac{\d}{dt}\,\EE|Y^{\mu_1,\gamma_1}_{t}\!-\!Y^{\mu_2,\gamma_2}_{t} |^2\nonumber\\
			=\,\,\,& \,\EE\Big[2\langle Y^{\mu_1,\gamma_1}_{t}\!-\!Y^{\mu_2,\gamma_2}_{t} , f(\mu_1, Y^{\mu_1,\gamma_1}_{t},\mathscr{L}_{Y^{\mu_1,\gamma_1}_{t}})-f(\mu_2, Y^{\mu_2,\gamma_2}_{t},\mathscr{L}_{Y^{\mu_2,\gamma_2}_t})\rangle\nonumber\\
			&+\,\|g(\mu_1, Y^{\mu_1,\gamma_1}_{t},\mathscr{L}_{Y^{\mu_1,\gamma_1}_{t}})-g(\mu_2, Y^{\mu_2,\gamma_2}_{t},\mathscr{L}_{Y^{\mu_2,\gamma_2}_{t}})\|^2\Big]\nonumber\\
			\leq&-(K_1-K_2)\,\EE|Y^{\mu_1,\gamma_1}_{t}\!-\!Y^{\mu_2,\gamma_2}_{t}|^2\nonumber\\
			&+C \,\big( 1+
			\EE|Y_t^{\mu_1,\gamma_1}|^{\kappa}+\EE|Y_t^{\mu_2,\gamma_2}|^{\kappa}+M_{\kappa}(\mu_1)+M_{\kappa}(\mu_2)\big)\mathbb{W}_2(\mu_1,\mu_2)^2.\nonumber
		\end{align}
		By Gronwall' lemma and the estimate \eqref{es7},
		\begin{align}
			&\EE |Y^{\mu_1,\gamma_1}_{t}\!-\!Y^{\mu_2,\gamma_2}_{t} |^2\nonumber\\
			\leq\,\,\, & e^{-(K_1-K_2) t}\,\EE|\zeta_1-\zeta_2|^2+ C(1+M_{\kappa}(\gamma_1)+M_{\kappa}(\gamma_2)+M_{\kappa}(\mu_1)+M_{\kappa}(\mu_2))\mathbb{W}_2(\mu_1,\mu_2)^2.\nonumber
		\end{align}
		Likewise,
		\begin{align}
			&\frac{\d}{dt}\,\EE|Y^{\mu_1,\gamma_1,y_1}_{t}-Y^{\mu_2,\gamma_2,y_2}_{t} |^2\nonumber\\
			\leq& -K_1\,\EE|Y^{\mu_1,\gamma_1, y_1}_{t}-Y^{\mu_2,\gamma_2, y_2}_{t}|^2+K_2\,\bb{W}_2(\mathscr{L}_{Y_t^{\mu_1,\gamma_1}}, \mathscr{L}_{Y_t^{\mu_2,\gamma_2}})^2
			\nonumber\\ &+C\,\big(1+\EE|Y_t^{\mu_1,\gamma_1,y_1}|^{\kappa}+\EE|Y_t^{\mu_2,\gamma_2,y_2}|^{\kappa}+M_{\kappa}(\mu_1)+M_{\kappa}(\mu_2)\big)\mathbb{W}_2(\mu_1,\mu_2)^2
			\nonumber\\[1mm]
			\leq& -K_1\,\EE|Y^{\mu_1,\gamma_1, y_1}_{t}-Y^{\mu_2,\gamma_2, y_2}_{t}|^2+K_2e^{-(K_1-K_2) t}\,\EE|\zeta_1-\zeta_2|^2
			\nonumber\\ &+C\,\big(1+|y_1|^{\kappa}+|y_2|^{\kappa}+M_{\kappa}(\gamma_1)+M_{\kappa}(\gamma_2)+M_{\kappa}(\mu_1)+M_{\kappa}(\mu_2)\big)\mathbb{W}_2(\mu_1,\mu_2)^2.\nonumber
		\end{align}
		Hence, by Gronwall's lemma and the estimates \eqref{es7} and \eqref{es:Yxmuxiy} we have
		\begin{align}
			&\EE|Y^{\mu_1,\gamma_1,y_1}_{t}\!-\!Y^{\mu_2,\gamma_2,y_2}_{t} |^2\nonumber\\
			\leq \,\,\,& e^{-K_1 t}|y_1-y_2|^2+e^{-(K_1-K_2) t}\EE|\zeta_1-\zeta_2|^2
			\nonumber\\ &+C\,\big(1+|y_1|^{\kappa}+|y_2|^{\kappa}+M_{\kappa}(\gamma_1)+M_{\kappa}(\gamma_2)+M_{\kappa}(\mu_1)+M_{\kappa}(\mu_2)\big)\mathbb{W}_2(\mu_1,\mu_2)^2.\nonumber
		\end{align}
		We complete the proof.
	\end{proof}

	Now, we introduce the lifted transition probability associated to the system (\ref{fro1})-(\ref{fro2}),
	\begin{equation}\label{tran}
		\tilde{\mathbf{P}}^{\mu}_t(y,\gamma;dz\d\eta)\coloneqq\big(\mathscr{L}_{Y^{\mu,\gamma,y}_t}|_{\tilde{\mathbb{P}}}\times \delta_{\mathscr{L}_{Y_{t}^{\mu,\gamma}}|_{\tilde{\mathbb{P}}}}\big)(dz\d\eta)\,.
	\end{equation}
	Then for  any bounded Borel-measurable function $\varphi:\mathbb{R}^m\times\mathscr{P}(\mathbb{R}^m)\to \RR$, we define the lifted semigroup by
	\begin{equation}\label{semig}
		\tilde{\mathbf{P}}^{\mu}_t\varphi(y,\gamma)\coloneqq\int_{\mathbb{R}^m\times\mathscr{P}(\mathbb{R}^m)} \varphi(z,\eta)\,\tilde{\mathbf{P}}^{\mu}_t(y,\gamma;dz\d\eta) .
	\end{equation}
	We shall use $\tilde{\mathbf{E}}$ to denote the expectation on the product space $\tilde{\Omega}\times \mathscr{P}(\RR^m)$. Then it is straightforward that
	\begin{equation}
		\tilde{\mathbf{P}}^{\mu}_t\varphi(y,\gamma)=\tilde{\mathbf{E}}\,\varphi(Y^{\mu,\gamma,y}_t,\mathscr{L}_{Y_{t}^{\mu,\gamma}})=\tilde{\EE}\,\varphi(Y^{\mu,\gamma,y}_t,\mathscr{L}_{Y_{t}^{\mu,\gamma}}).\nonumber
	\end{equation}
	
	\begin{lemma}\label{lem1}
		$\{\tilde{\mathbf{P}}^{\mu}_t\}_{t\geq 0}$ is  a
		time-homogenous Markov semigroup for which   the unique invariant  measure takes the form of
		$$\nu^{\mu}\times \delta_{\nu^{\mu}},$$
		where $\nu^{\mu}$ is the unique invariant  measure of the frozen McKean--Vlasov equation \eqref{Yfro}.
	\end{lemma}
	
	\begin{proof}
		This result refers to Proposition 4.8 and Theorem 6.1 in \cite{RRW}.
	\end{proof}
	
	\begin{remark}
		We point out that the lifted semigroup $\tilde{\mathbf{P}}^{\mu}_t$ essentially governs the joint dynamics and long-time behaviours of both $Y_t$ and its time-marginal law, which effectively
		addresses the problem for the system $(\ref{1.1a})$-$(\ref{1.1b})$.
		
	\end{remark}
	
	\section{Proof of functional law of large numbers}\label{sec5}
	In this section, we prove Theorem \ref{theo1} by combining the time discretization method and the lifted semigroup technique. More precisely, in Subsection \ref{sub4.1} we give  some necessary a priori estimates associated to  the multi-scale McKean--Vlasov system \eqref{1.1a}-\eqref{1.1b}. In Subsection \ref{sub4.2}, we present some useful properties for the averaged drift coefficient $\bar{b}$ and derive the uniform  estimates for the averaged
	process $\bar{X}$.    In Subsection \ref{sub4.3}, we aim to prove the main theorem.

	\subsection{Uniform estimates}\label{sub4.1}
	In this subsection, we first  provide some a priori estimates for the solution to the system \eqref{1.1a}-\eqref{1.1b}.

	\begin{lemma} \label{lem:uniformestimateofXY}
		$(i)$ For some $p\geq 2$, we have
		\begin{align}
			&\sup_{\varepsilon,\delta\in(0,1)}\mathbb{E}\Big[\sup_{t\in [0, T]}|X_{t}^{\varepsilon}|^{p}\Big]\lesssim_{p,T}(1+\mathbb{E}|\xi|^p+\mathbb{E}|\zeta|^{\frac{pq}{2}})\,, \label{esX}\\
			&\sup_{\varepsilon,\delta\in(0,1)}\sup_{t\geq 0}\mathbb{E}|Y_{t}^{\varepsilon}|^{p}\lesssim_{p}(1+\mathbb{E}|\zeta|^p).\label{esY}
		\end{align}
		
		$(ii)$ For any $T>0$, $0\leq t< t+h\leq T$ and $\varepsilon,\delta\in(0,1)$, we have
		\begin{align}
			\mathbb{E}|X^{\varepsilon}_{t+h}-X^{\varepsilon}_{t}|^2\lesssim_T(1+\mathbb{E}|\xi|^{\kappa}+\mathbb{E}|\zeta|^{\frac{\kappa q}{2}})h^2+(1+\mathbb{E}|\xi|^{2}+\mathbb{E}|\zeta|^{q})\varepsilon h.\label{F3.10}
		\end{align}
	\end{lemma}
	\begin{proof}
		(i)
		Applying It\^{o}'s formula,
		\begin{align}
			|X_{t}^{\varepsilon}|^{p}= &|\xi|^p+p\int_0^t|X_{s}^{\varepsilon}|^{p-2}\langle X_{s}^{\varepsilon},b(X_{s}^{\varepsilon},\mathscr{L}_{X_{s}^{\varepsilon}},Y_{s}^{\varepsilon},\mathscr{L}_{Y_{s}^{\varepsilon}})\rangle ds
	\nonumber	\\
		\!\!\!\!\!\!\!\!&
		+\frac{p}{2}\int_0^t|X_{s}^{\varepsilon}|^{p-2}\|\sigma^*(X_{s}^{\varepsilon},\mathscr{L}_{X_{s}^{\varepsilon}},Y_{s}^{\varepsilon},\mathscr{L}_{Y_{s}^{\varepsilon}})\|^2ds
		\nonumber	\\
		\!\!\!\!\!\!\!\!&
		+\frac{p(p-2)}{2}\int_0^t|X_{s}^{\varepsilon}|^{p-4}|\sigma^*(X_{s}^{\varepsilon},\mathscr{L}_{X_{s}^{\varepsilon}},Y_{s}^{\varepsilon},\mathscr{L}_{Y_{s}^{\varepsilon}})\cdot X_{s}^{\varepsilon}|^2ds
		\nonumber	\\
		\!\!\!\!\!\!\!\!&
		+p\varepsilon^{\frac{p}{2}}\int_0^t|X_{s}^{\varepsilon}|^{p-2}\langle X_{s}^{\varepsilon}, \sigma(X_{s}^{\varepsilon},\mathscr{L}_{X_{s}^{\varepsilon}},Y_{s}^{\varepsilon},\mathscr{L}_{Y_{s}^{\varepsilon}})dW^1_s
		\rangle
		\nonumber	\\
		=: &|\xi|^p+I^{\varepsilon}(t)+II^{\varepsilon}(t)+III^{\varepsilon}(t)+IV^{\varepsilon}(t).	\label{es3}
		\end{align}
		From the condition (\ref{innerofb}), we deduce that
		\begin{eqnarray}\label{es4}
			\!\!\!\!\!\!\!\!&&~~~~~\mathbb{E}\Big[\sup_{t\in[0,T]}I^{\varepsilon}(t)\Big]
			\nonumber\\
			\!\!\!\!\!\!\!\!&&\lesssim_p \mathbb{E}\int_0^T \big(1+|X_{t}^{\varepsilon}|^{p}\big)dt+\mathbb{E}\int_0^T|X_{t}^{\varepsilon}|^{p-2}|Y_{t}^{\varepsilon}|^{q}dt
\nonumber\\
\!\!\!\!\!\!\!\!&&+\int_0^T\mathbb{E}|X_{t}^{\varepsilon}|^{p-2}\mathbb{E}|Y_{t}^{\varepsilon}|^{2}dt+\int_0^T\mathbb{E}|X_{t}^{\varepsilon}|^{p-2}\mathbb{E} |X_{t}^{\varepsilon}|^{2}dt
			\nonumber\\
			\!\!\!\!\!\!\!\!&&\lesssim_p \mathbb{E}\int_0^T \big(1+|X_{t}^{\varepsilon}|^{p}\big)dt+\mathbb{E}\int_0^T |Y_{t}^{\varepsilon}|^{p}dt+\mathbb{E}\int_0^T|Y_{t}^{\varepsilon}|^{\frac{pq}{2}}dt
			\nonumber\\
			\!\!\!\!\!\!\!\!&&\lesssim_p \mathbb{E}\int_0^T \big(1+|X_{t}^{\varepsilon}|^{p}\big)dt+\mathbb{E}\int_0^T|Y_{t}^{\varepsilon}|^{\frac{pq}{2}}dt,
		\end{eqnarray}
		where we used  Young's inequality in the second and the last steps. By the condition (\ref{growthsigma}), we obtain
		\begin{eqnarray}\label{es5}
			\mathbb{E}\Big[\sup_{t\in[0,T]}\big(II^{\varepsilon}(t)+III^{\varepsilon}(t)\big)\Big]\lesssim_p  &&\mathbb{E}\int_0^T\big(1+|X_{t}^{\varepsilon}|^{p}\big)dt+\mathbb{E}\int_0^T|Y_{t}^{\varepsilon}|^{p}dt.
		\end{eqnarray}
		By Burkholder--Davis--Gundy's inequality, we have
		\begin{align}
			 &~~~~~\mathbb{E}\Big[\sup_{t\in[0,T]}IV^{\varepsilon}(t)\Big]
	\nonumber	\\
		 &\lesssim_p \mathbb{E}\bigg[\sup_{t\in[0,T]}|X_{t}^{\varepsilon}|^{p}\cdot\int_0^T|X_{t}^{\varepsilon}|^{p-2}\|\sigma(X_{t}^{\varepsilon},\mathscr{L}_{X_{t}^{\varepsilon}},Y_{t}^{\varepsilon},\mathscr{L}_{Y_{t}^{\varepsilon}})\|^2dt\bigg]^{\frac{1}{2}}
			\nonumber	\\
		 &\leq \frac{1}{2}\mathbb{E}\Big[\sup_{t\in[0,T]}|X_{t}^{\varepsilon}|^{p}\Big]+C_p\mathbb{E}\int_0^T \big(1+|X_{t}^{\varepsilon}|^{p}\big)dt+C_p\mathbb{E}\int_0^T|Y_{t}^{\varepsilon}|^{p}dt. \label{es6}
				\end{align}
		Combining (\ref{es3})-(\ref{es6}), it yields that
		\begin{align}\label{es8}
			\mathbb{E}\Big[\sup_{t\in[0,T]}|X_{t}^{\varepsilon}|^{p}\Big]&\lesssim_{p,T}(1+\mathbb{E}|\xi|^p)+\mathbb{E}\int_0^T |X_{t}^{\varepsilon}|^{p}dt+\mathbb{E}\int_0^T|Y_{t}^{\varepsilon}|^{\frac{pq}{2}}dt.
		\end{align}
		
		On the other hand, by It\^{o}'s formula and taking expectation, we obtain
		\begin{align}\label{es0}
			\mathbb{E}|Y_{t}^{\varepsilon}|^{2}=\,\,\,&\mathbb{E}|\zeta|^2+\frac{1}{\delta}\mathbb{E}\int_0^t\big(2\langle Y_{s}^{\varepsilon},f(\mathscr{L}_{X_{s}^{\varepsilon}},Y_{s}^{\varepsilon},\mathscr{L}_{Y_{s}^{\varepsilon}})\rangle+\|g(\mathscr{L}_{X_{s}^{\varepsilon}},Y_{s}^{\varepsilon},\mathscr{L}_{Y_{s}^{\varepsilon}})
			\|^2\big) ds.
		\end{align}
		Taking differential on both sides of (\ref{es0}) and  by (\ref{inneroffplusg}), it follows that
		\begin{align}
			\frac{\d}{dt}\mathbb{E}|Y_{t}^{\varepsilon}|^{2}\leq -\frac{(K_1-K_2)}{\delta}\mathbb{E}|Y_{t}^{\varepsilon}|^{2}+\frac{C}{\delta}.
		\end{align}
		Using Gronwall's lemma, we have
		\begin{align}
			\sup_{t\geq 0}\mathbb{E}|Y_{t}^{\varepsilon}|^{2}\lesssim(1+\mathbb{E}|\zeta|^2).\label{es2}
		\end{align}
		By It\^{o}'s formula again for any $\lambda\geq 2$, we obtain
		\begin{align}
			\mathbb{E}|Y_{t}^{\varepsilon}|^{\lambda}=\,\,\,&\mathbb{E}|\zeta|^{\lambda}+\frac{\lambda}{\delta}\mathbb{E}\int_0^t|Y_{s}^{\varepsilon}|^{\lambda-2}\langle Y_{s}^{\varepsilon},f(\mathscr{L}_{X_{s}^{\varepsilon}},Y_{s}^{\varepsilon},\mathscr{L}_{Y_{s}^{\varepsilon}})\rangle ds \nonumber
			\\
			&+\frac{\lambda(\lambda-2)}{2\delta}\mathbb{E}\int_0^t|Y_{s}^{\varepsilon}|^{\lambda-4}|g^*(\mathscr{L}_{X_{s}^{\varepsilon}},Y_{s}^{\varepsilon},\mathscr{L}_{Y_{s}^{\varepsilon}})
			\cdot Y_{s}^{\varepsilon}|^2ds
			\nonumber\\
			&
			+\frac{\lambda}{2\delta}\mathbb{E}\int_0^t|Y_{s}^{\varepsilon}|^{\lambda-2}\|g(\mathscr{L}_{X_{s}^{\varepsilon}},Y_{s}^{\varepsilon},\mathscr{L}_{Y_{s}^{\varepsilon}})
			\|^2ds.\nonumber
		\end{align}
		In view of the conditions (\ref{inneroffplusg}) and (\ref{growthofg}),  using Young's inequality,
		\begin{align}
			\frac{\d}{dt}\mathbb{E}|Y_{t}^{\varepsilon}|^{\lambda}\leq &-\frac{K_1\lambda}{2\delta}\mathbb{E}|Y_{t}^{\varepsilon}|^{\lambda}+\frac{K_2\lambda}{2\delta}\mathbb{E}|Y_{t}^{\varepsilon}|^{\lambda-2}\mathbb{E}|Y_{t}^{\varepsilon}|^{2}
\nonumber\\
&+\frac{C_\lambda}{\delta}\mathbb{E}|Y_{t}^{\varepsilon}|^{\lambda-2+2\gamma}+\frac{C_\lambda}{\delta}\mathbb{E}|Y_{t}^{\varepsilon}|^{\lambda-2}\mathbb{E}|Y_{t}^{\varepsilon}|^{2}
\nonumber\\[1mm]
\leq &-\frac{K_1\lambda}{2\delta}\mathbb{E}|Y_{t}^{\varepsilon}|^{\lambda}+\frac{C_\lambda}{\delta}\mathbb{E}|Y_{t}^{\varepsilon}|^{\lambda-2}(1+\mathbb{E}|\zeta|^2)
\nonumber\\
\leq &-\frac{\beta}{\delta}\mathbb{E}|Y_{t}^{\varepsilon}|^{\lambda}+\frac{C_{\lambda}}{\delta}(1+\mathbb{E}|\zeta|^{\lambda}), \label{es:Yp}
		\end{align}
		where the constant $\beta\in(0,\frac{K_1\lambda}{2})$. Applying Gronwall's lemma leads to
		\begin{align}\label{es1}
			\sup_{t\geq 0}\mathbb{E}|Y_{t}^{\varepsilon}|^{\lambda}\lesssim_{\lambda}(1+\mathbb{E}|\zeta|^{\lambda}).
		\end{align}
		
		Now substituting (\ref{es1}) with $\lambda=\frac{pq}{2}$ into the estimate (\ref{es8}), we derive
		\begin{align}
			\mathbb{E}\Big[\sup_{t\in[0,T]}|X_{t}^{\varepsilon}|^{p}\Big]\lesssim_{p,T}(1+\mathbb{E}|\xi|^p+\mathbb{E}|\zeta|^{\frac{pq}{2}})+\mathbb{E}\int_0^T |X_{t}^{\varepsilon}|^{p}dt. \nonumber
		\end{align}
		Finally, the Gronwall's lemma gives
		\begin{align}
			\mathbb{E}\Big[\sup_{t\in[0,T]}|X_{t}^{\varepsilon}|^{p}\Big]\lesssim_{p,T}(1+\mathbb{E}|\xi|^p+\mathbb{E}|\zeta|^{\frac{pq}{2}}).\nonumber
		\end{align}
		
		\noindent (ii)
		Recall
		$$X^{\varepsilon}_{t+h}-X^{\varepsilon}_{t}=\int_t^{t+h}b(X^{\varepsilon}_{s},\mathscr{L}_{X^{\varepsilon}_{s}},Y^{\varepsilon}_{s},\mathscr{L}_{Y^{\varepsilon}_{s}})ds
		+\sqrt{\varepsilon}\int_t^{t+h}\sigma(X^{\varepsilon}_{s},\mathscr{L}_{X^{\varepsilon}_{s}},Y^{\varepsilon}_{s},\mathscr{L}_{Y^{\varepsilon}_{s}})dW^1_s.$$
		Due to the growth condition $(\mathbf{A}_{\text{slow}})$ (iii) and the uniform estimates \eqref{esX}-(\ref{esY}), we observe from It\^{o}'s isometry that
		\begin{align}
			&\mathbb{E}|X^{\varepsilon}_{t+h}-X^{\varepsilon}_{t}|^2 \nonumber\\
\lesssim & \mathbb{E}\Big|\int_t^{t+h}b(X^{\varepsilon}_{s},\mathscr{L}_{X^{\varepsilon}_{s}},Y^{\varepsilon}_{s},\mathscr{L}_{Y^{\varepsilon}_{s}})ds\Big|+\varepsilon\mathbb{E}\Big|\int_t^{t+h}\sigma(X^{\varepsilon}_{s},\mathscr{L}_{X^{\varepsilon}_{s}},Y^{\varepsilon}_{s},\mathscr{L}_{Y^{\varepsilon}_{s}})dW^1_s\Big|^2\,\nonumber
			\\
			\lesssim& h\int_t^{t+h}\mathbb{E}|b(X^{\varepsilon}_{s},\mathscr{L}_{X^{\varepsilon}_{s}},Y^{\varepsilon}_{s},\mathscr{L}_{Y^{\varepsilon}_{s}})|^2ds+\varepsilon\int_t^{t+h}\mathbb{E}\|\sigma(X^{\varepsilon}_{s},\mathscr{L}_{X^{\varepsilon}_{s}},Y^{\varepsilon}_{s},\mathscr{L}_{Y^{\varepsilon}_{s}})\|^2ds\nonumber
			\\
			\lesssim& h\int_t^{t+h}\big(1+\mathbb{E}|X_{s}^{\varepsilon}|^{\kappa}+\mathbb{E}|Y_{s}^{\varepsilon}|^{\kappa}\big)ds+\varepsilon\int_t^{t+h}\big(1+\mathbb{E}|X_{s}^{\varepsilon}|^{2}+\mathbb{E}|Y_{s}^{\varepsilon}|^{2}\big)ds\nonumber
			\\
			\lesssim&_T(1+\mathbb{E}|\xi|^{\kappa}+\mathbb{E}|\zeta|^{\frac{\kappa q}{2}})h^2+(1+\mathbb{E}|\xi|^{2}+\mathbb{E}|\zeta|^{q})\varepsilon h.\nonumber
		\end{align}
		We complete the proof. \hfill
	\end{proof}

	\subsection{Averaged equation}\label{sub4.2}
	In this part, we investigate the averaged equation (\ref{eqav}) associated to the slow equation \eqref{1.1a}.
	
	We first present some crucial properties for the averaged drift coefficient $\bar{b}$.
	\begin{proposition}\label{prop:barb}
		$(i)$ $($Coercivity$)$ For any $x\in \RR^n, \mu\in \mathscr{P}_{\kappa}(\RR^n)$, we have
		\begin{align}\label{es:innerofb}
			\langle \bar{b}(x,\mu), x\rangle \lesssim \big(1+|x|^2+M_2(\mu)\big).
		\end{align}
		$(ii)$ $($Exponential ergodicity$)$ There exist  positive constants $C_{\kappa}, \beta$ such that for any $x\in \RR^n, \mu\in \mathscr{P}_{\kappa}(\RR^n), \gamma\in \mathscr{P}_{\kappa}(\RR^m), z\in \RR^m$,
		\begin{align}\label{es:differencebandbarb}
			\big|\tilde{\mathbb{E}}\,b\big(x,\mu,Y_t^{\mu,\gamma,z},\mathscr{L}_{Y_t^{\mu,\gamma}}\big)-\bar{b}(x,\mu)\big|\lesssim\big(1+|x|^{\kappa}+|z|^{\kappa}+M_{\kappa}(\mu)+M_{\kappa}(\gamma)\big)e^{-\beta t}.
		\end{align}
		$ (iii)$ $($Local Lipschitz$)$ For any $x_1,x_2\in\RR^n$, $\mu, \nu\in\mathscr{P}_{\kappa}(\RR^n)$, it holds that
		\begin{align}
			&|\bar{b}(x_1,\mu)-\bar{b}(x_2,\nu)| \nonumber\\
		\lesssim&\big(1+M_{\kappa}^{\frac{1}{2}}(\mu)+M_{\kappa}^{\frac{1}{2}}(\nu)\big)\big|x_1-x_2\big| \nonumber\\
		&+\big(1+|x_1|^{\frac{\kappa}{2}}+ |x_2|^{\frac{\kappa}{2}}+M_{\kappa}^{\frac{1}{2}}(\mu)+M_{\kappa}^{\frac{1}{2}}(\nu)\big)\big(1+M_{\kappa}^{\frac{1}{2}}(\mu)+M_{\kappa}^{\frac{1}{2}}(\nu)\big)  \mathbb{W}_2(\mu,\nu).\label{continuityofbarb}
			\end{align}
		In particular, we obtain the growth of $\bar{b}$ in the sense that for any $x\in \RR^n, \mu\in \mathscr{P}_{\kappa}(\RR^n)$, it holds that
		\begin{align}\label{es:growthofbarb}
			|\bar{b}(x,\mu)|\lesssim \big(1+M_{\kappa}^{\frac{1}{2}}(\mu)\big)|x|
			+\big(1+|x|^{\frac{\kappa}{2}}+M_{\kappa}^{\frac{1}{2}}(\mu)\big)\big(1+M_{\kappa}^{\frac{1}{2}}(\mu)\big)  M_2^{\frac{1}{2}}(\mu).
		\end{align}
	\end{proposition}
	
	\begin{proof}
		We first show a moment estimate for invariant probability measure $\nu^{\mu}$, namely, for some $p>0$ there exists a constant $C_{p}>0$ such that
		\begin{align}\label{es:IPM}
			M_{p}(\nu^{\mu})\leq C_{p}.
		\end{align}
		Indeed, because of the exponential ergodicity \eqref{es:expentialergodicity} and the fact that $\bb{W}_2$-convergence implies the weak convergence,  it follows from  Fatou's lemma that
		\begin{align}
			\int_{\RR^m} |y|^{p} \,\nu^{\mu}(dy)&\,=\int_{\RR^m} \liminf_{R \to \infty}\big(|y|^{p}\wedge R\big)  \,\nu^{\mu}(dy)\nonumber\\
			&\leq\, \liminf_{R \to \infty} \liminf_{t\to \infty}\int_{\RR^m}\big(|y|^{p}\wedge R\big)  \big((\tilde{P}_t^{\mu})^*\delta_0\big)(dy)\nonumber\\
			&\leq\, \sup_{t\geq 0}\EE|Y_{t}^{\mu,\delta_0}|^{p}\nonumber\\
			&\leq\, C_{p}\,,\nonumber
		\end{align}
		where the last inequality is due to the estimate  \eqref{es7}.
		
		\vspace{1mm}
		\noindent(i) It follows from the condition \eqref{innerofb} that
\begin{align}
			\langle \bar{b}(x,\mu), x\rangle &\,=\int_{\RR^m\times\mathscr{P}(\RR^m)}\langle b(x,\mu, z,\eta), x\rangle (\nu^{\mu}\times \delta_{\nu^{\mu}})(dz \d\eta)\nonumber\\
			&\lesssim \int_{\RR^m\times\mathscr{P}(\RR^m)} \big(1+|x|^2+|z|^{q}+M_2(\mu)+M_2(\eta)\big)  (\nu^{\mu}\times \delta_{\nu^{\mu}})(dz\d\eta)\nonumber\\
			&\lesssim\big(1+|x|^2+M_2(\mu)+M_2(\nu^{\mu})+M_q(\nu^{\mu})\big)\nonumber\\
			&\lesssim_q\big(1+|x|^2+M_2(\mu)\big).\nonumber
		\end{align}
		
		\noindent(ii) By the invariance of the measure $\nu^{\mu}\times \delta_{\nu^{\mu}}$, the local Lipschitz condition \eqref{continuityofb} for $b $, and the estimates \eqref{es7}-\eqref{ergo2}, we have
		\begin{align}
			&|\tilde{\mathbb{E}}\,b(x,\mu,Y_t^{\mu,\gamma,z},\mathscr{L}_{Y_t^{\mu,\gamma}})-\bar{b}(x,\mu)|\nonumber\\
			=\,\,\,&\Big|\tilde{\mathbf{E}}\,b(x,\mu,(Y_t^{\mu,\gamma,z},\mathscr{L}_{Y_t^{\mu,\gamma}}))-\int_{\RR^m\times\mathscr{P}(\RR^m)}b(x,\mu, y,\eta)(\nu^{ \mu}\times \delta_{\nu^{\mu}})(dy \d\eta)\Big|\nonumber
			\\
			=\,\,\,&\bigg|\int_{\RR^m\times\mathscr{P}(\RR^m)}\big[\tilde{\mathbf{E}}\,b(x,\mu,(Y_t^{\mu,\gamma,z},\mathscr{L}_{Y_t^{\mu,\gamma}}))-\tilde{\mathbf{E}}\,b(x,\mu,(Y_t^{\mu,\eta,y},\nu^{ \mu}))\big](\nu^{ \mu}\times \delta_{\nu^{\mu}})(dy \d\eta)\bigg|\nonumber
			\\
			\lesssim \,\,\,& \big(1+\ca{M}_{\kappa}^{\frac{1}{2}}(\mu,\mu,\mathscr{L}_{Y_t^{\mu,\gamma}},\nu^{\mu})\big)\int_{\RR^m}\tilde{\EE}\big|Y_t^{\mu,\gamma,z}-Y_t^{\mu,\nu^{\mu},y}\big|\,\nu^{ \mu}(dy)+\bb{W}_2(\mathscr{L}_{Y_t^{\mu,\gamma}}, \nu^{\mu})\nonumber
			\\
			&\cdot\int_{\RR^m}\Big[1+\EE\rho_{\kappa}^{\frac{1}{2}}(x, x, Y_t^{\mu,\gamma,z}, Y_t^{\mu,\nu^{\mu},y})+\ca{M}_{\kappa}^{\frac{1}{2}}(\mu,\mu,\mathscr{L}_{Y_t^{\mu,\gamma}},\nu^{\mu})\Big]\nu^{ \mu}(dy)\nonumber
			\\
			\lesssim\,\,\,& \big(1+M_{\kappa}^{\frac{1}{2}}(\mu)+M_{\kappa}^{\frac{1}{2}}(\gamma)+M_{\kappa}^{\frac{1}{2}}(\nu^{\mu})\big)\int_{\RR^m}e^{-\beta t}\big[|z-y|+\bb{W}_2(\gamma, \nu^{\mu})\big]\,\nu^{\mu}(dy)\nonumber
			\\
			&+ \int_{\RR^m}\Big[1+|x|^{\frac{\kappa}{2}}+e^{-\beta t}|z|^{\frac{\kappa}{2}}+e^{-\beta t}|y|^{\frac{\kappa}{2}}+M_{\kappa}^{\frac{1}{2}}(\mu)+M_{\kappa}^{\frac{1}{2}}(\gamma)+M_{\kappa}^{\frac{1}{2}}(\nu^{\mu})\Big]\nu^{\mu}(dy)\nonumber
			\\
			&\cdot e^{-\beta t}\bb{W}_2(\gamma, \nu^{\mu})\nonumber
			\\
			\lesssim \,\,\,& \big(1+|x|^{\kappa}+|z|^{\kappa}+M_{\kappa}(\mu)+M_{\kappa}(\gamma)\big)e^{-\beta t}\,,\nonumber
		\end{align}
		where $\beta$ is a suitable positive constant derived from the estimates \eqref{es7}-\eqref{ergo2}, and  we used (\ref{es:IPM}) in the last step.
		
		\vspace{1mm}
		\noindent(iii) It is straightforward to see from (ii) that
		\begin{align}
			&|\bar{b}\,(x_1,\mu)-\bar{b}\,(x_2,\nu)|\nonumber\\[1mm]
			\leq\,\,\, &|\bar{b}(x_1,\mu)-\mathbb{E}\,b(x_1,\mu,Y_t^{\mu,\gamma,z},\mathscr{L}_{Y_t^{\mu,\gamma}})|+|\mathbb{E}\,b(x_2,\nu,Y_t^{\nu,\gamma,z},\mathscr{L}_{Y_t^{\nu,\gamma}})-\bar{b}(x_2,\nu)|\nonumber
			\\[1mm]
			&+|\mathbb{E}\,b(x_1,\mu,Y_t^{\mu,\gamma,z},\mathscr{L}_{Y_t^{\mu,\gamma}})-\mathbb{E}\,b(x_2,\nu,Y_t^{\nu,\gamma,z},\mathscr{L}_{Y_t^{\nu,\gamma}})|
			\nonumber\\[1mm]
			\lesssim\,\,\, &\big(1+|x_1|^{\kappa}+|z|^{\kappa}+M_{\kappa}(\mu)+M_{\kappa}(\gamma)\big)e^{-\beta t}\nonumber\\
			&+\big(1+|x_2|^{\kappa}+|z|^{\kappa}+M_{\kappa}(\nu)+M_{\kappa}(\gamma)\big)e^{-\beta t}
			\nonumber\\
			&+\EE\,|b(x_1,\mu,Y_t^{\mu,\gamma,z},\mathscr{L}_{Y_t^{\mu,\gamma}})-b(x_2,\nu,Y_t^{\nu,\gamma,z},\mathscr{L}_{Y_t^{\nu,\gamma}})|\nonumber
			\\
			\lesssim \,\,\,&\big(1+|x_1|^{\kappa}+|x_2|^{\kappa}+|z|^{\kappa}+M_{\kappa}(\mu)+M_{\kappa}(\nu)+M_{\kappa}(\gamma)\big)e^{-\beta t}\nonumber\\
			&+\big(1+M_{\kappa}^{\frac{1}{2}}(\mu)+M_{\kappa}^{\frac{1}{2}}(\nu)+M_{\kappa}^{\frac{1}{2}}(\gamma) \big)|x_1-x_2|\nonumber\\
			&+\big(1+M_{\kappa}^{\frac{1}{2}}(\mu)+M_{\kappa}^{\frac{1}{2}}(\nu)+M_{\kappa}^{\frac{1}{2}}(\gamma) \big)\cdot\big(1+|z|^{\frac{\kappa}{2}}+M_{\kappa}^{\frac{1}{2}}(\gamma)+M_{\kappa}^{\frac{1}{2}}(\mu)+M_{\kappa}^{\frac{1}{2}}(\nu)\big)\bb{W}_2(\mu,\nu)\nonumber\\
			&+\big(1+|x_1|^{\frac{\kappa}{2}}+ |x_2|^{\frac{\kappa}{2}}+e^{-\beta t}|z|^{\frac{\kappa}{2}}+M_{\kappa}^{\frac{1}{2}}(\gamma)+M_{\kappa}^{\frac{1}{2}}(\mu)+M_{\kappa}^{\frac{1}{2}}(\nu)\big)\bb{W}_2(\mu,\nu)\nonumber\\
			&+\big(1+|x_1|^{\frac{\kappa}{2}}+ |x_2|^{\frac{\kappa}{2}}+e^{-\beta t}|z|^{\frac{\kappa}{2}}+M_{\kappa}^{\frac{1}{2}}(\gamma)+M_{\kappa}^{\frac{1}{2}}(\mu)+M_{\kappa}^{\frac{1}{2}}(\nu)\big)
			\nonumber \\
			&\cdot\big(1+M_{\kappa}^{\frac{1}{2}}(\mu)+M_{\kappa}^{\frac{1}{2}}(\nu)+M_{\kappa}^{\frac{1}{2}}(\gamma)\big) \mathbb{W}_2(\mu,\nu).\nonumber
		\end{align}
		Taking limit $t\to \infty$ and $\gamma=\delta_0$, $z=0$, we arrive at
		\begin{align}
			&|\bar{b}(x_1,\mu)-\bar{b}(x_2,\nu)|\nonumber\\
			\lesssim \,\,\,&\big(1+M_{\kappa}^{\frac{1}{2}}(\mu)+M_{\kappa}^{\frac{1}{2}}(\nu)\big)\big|x_1-x_2\big|\nonumber\\
			&+\big(1+|x_1|^{\frac{\kappa}{2}}+ |x_2|^{\frac{\kappa}{2}}+M_{\kappa}^{\frac{1}{2}}(\mu)+M_{\kappa}^{\frac{1}{2}}(\nu)\big)\big(1+M_{\kappa}^{\frac{1}{2}}(\mu)+M_{\kappa}^{\frac{1}{2}}(\nu)\big)  \mathbb{W}_2(\mu,\nu)\,,\nonumber
		\end{align}
		which completes the proof.  \hfill
	\end{proof}
	
	Building upon Proposition \ref{prop:barb}, we can derive the following estimates for the averaged process $\bar{X}$.
	\begin{lemma}\label{lem:estimateofaverage}
		For any $T>0,p\geq 2$, we have
		\begin{equation}\label{es:expectationbarX}
			\EE\Big[\sup_{t\in [0,T]}|\bar{X}_t|^p\Big]\lesssim_{p,T}(1+\EE|\xi|^p),
		\end{equation}
		\begin{equation}
			\sup_{t\in[0,T]}\EE|\bar{X}_t-\bar{X}_{t(\Delta)}|^2\lesssim_{T,\xi}\Delta^{2}.\label{es:differencebarX}
		\end{equation}
	\end{lemma}
	
	\begin{proof}
		We obtain from \eqref{es:innerofb} that
		\begin{align}
				\EE\Big[\sup_{t\in [0,T]}|\bar{X}_t|^{p}\Big]&=\,\EE|\xi|^{p}+p\EE\int_0^t |\bar{X}_s|^{p-2} \langle \bar{b}(\bar{X}_s,\mathscr{L}_{\bar{X}_s}), \bar{X}_s\rangle ds \nonumber\\
			&\lesssim_p\, \EE|\xi|^p+\int_0^T(1+\EE|\bar{X}_s|^{p})ds.\label{es9}
					\end{align}
		Thus, the desired estimate \eqref{es:expectationbarX} follows by (\ref{es9}) and Gronwall's lemma.
		
		Moreover, by H$\rm{\ddot{o}}$lder's inequality and the estimates  \eqref{es:growthofbarb} and \eqref{es:expectationbarX},
		\begin{align}
			&\EE|\bar{X}_t-\bar{X}_{t(\Delta)}|^2\nonumber
			\\
			\lesssim& \Delta\int_{t(\Delta)}^t\EE|\bar{b}(\bar{X}_s,\mathscr{L}_{\bar{X}_s})|^2ds\nonumber\\
			\lesssim &\Delta\int_{t(\Delta)}^t \Big(\big(1+M_{\kappa}(\mathscr{L}_{\bar{X}_s})\big) \EE|\bar{X}_s|^2+(1+\EE|\bar{X}_s|^{\kappa}+M_{\kappa}(\mathscr{L}_{\bar{X}_s}))^2M_{2}(\mathscr{L}_{\bar{X}_s})\Big)ds\nonumber\\
			\lesssim&_{T} \Delta^{2}(1+\EE|\xi|^{\kappa})^2\EE|\xi|^{2}.\nonumber
		\end{align}
		We complete the proof.  \hfill
	\end{proof}
	\subsection{Proof of main result}\label{sub4.3}
	Before proving the main result, we introduce an auxiliary system $(\hat{Y}_{t}^{\varepsilon}, \hat{Y}_{t}^{\varepsilon,y})$ associated with the fast process $Y_{t}^{\varepsilon}$ to \eqref{1.1b}. More precisely, we consider
	\begin{subequations}
		\begin{align}
			\d\hat{Y}_{t}^{\varepsilon}&=\frac{1}{\delta}f(\mathscr{L}_{X^{\varepsilon}_{t(\Delta)}},\hat{Y}_{t}^{\varepsilon},\mathscr{L}_{\hat{Y}_{t}^{\varepsilon}})dt+\frac{1}{\sqrt{\delta}}g(\mathscr{L}_{X^{\varepsilon}_{t(\Delta)}},\hat{Y}_{t}^{\varepsilon},\mathscr{L}_{\hat{Y}_{t}^{\varepsilon}})dW_t^2,~\hat{Y}_{0}^{\varepsilon}=\zeta\sim \gamma\,,\label{4.4a}\\
			\d\hat{Y}_{t}^{ \varepsilon, y }&=\frac{1}{\delta}
			f(\mathscr{L}_{X^{\varepsilon}_{t(\Delta)}},\hat{Y}_{t}^{ \varepsilon, y}, \mathscr{L}_{\hat{Y}_t^{\varepsilon}})dt+\frac{1}{\sqrt{\delta}}g(\mathscr{L}_{X^{\varepsilon}_{t(\Delta)}},\hat{Y}_{t}^{\varepsilon, y},\mathscr{L}_{\hat{Y}_t^{\varepsilon}})dW_t^2,~\hat{Y}_{0}^{\varepsilon,y}=y\,,\label{4.4b}
		\end{align}
	\end{subequations}
	where $y\in \RR^m, t(\Delta):=[\frac{t}{\Delta}]\Delta$ and $[t]$ refers to the integer part of $t$. Notice that due to the solvability of (\ref{4.4a}), $(\mathscr{L}_{\hat{Y}_t^{\varepsilon}})_{t\in [0, T]}$ is a known law, which leads that \eqref{4.4b} is no longer a distribution-dependent equation in the sense that the measure term in \eqref{4.4b} is not the law of the solution $\hat{Y}_t^{\varepsilon,y}$.

	We also mention that \eqref{4.4a} is  equivalent  to the following integral equation, i.e., for any $t\in\big[k\Delta,\min\{(k+1)\Delta,T\}\big]$, $k\in \mathbb{N}$,
	\begin{align}
		\hat{Y}_{t}^{\varepsilon}&=\hat{Y}_{k\Delta}^{\varepsilon}+\frac{1}{\delta}\int_{k\Delta}^{t}
		f(\mathscr{L}_{X^{\varepsilon}_{k\Delta}},\hat{Y}_{s}^{\varepsilon},\mathscr{L}_{\hat{Y}_{s}^{\varepsilon}})ds+\frac{1}{\sqrt{\delta}}\int_{k\Delta}^{t}g( \mathscr{L}_{X^{\varepsilon}_{k\Delta}},\hat{Y}_{s}^{\varepsilon},\mathscr{L}_{\hat{Y}_{s}^{\varepsilon}})dW_s^2.\label{4.6a}
	\end{align}

	\vspace{1mm}
	\begin{lemma} \label{lem:uniformestimateofauxiliary}
		For any $T>0,p\geq 2$, we have
		\begin{equation}
			\sup_{\varepsilon,\delta\in(0,1)}\sup_{t\geq 0}\mathbb{E}|\hat{Y}_{t}^{\varepsilon}|^p \lesssim_{p}(1+\EE|\zeta|^{p}), \label{3.13a}
		\end{equation}
		\begin{equation}
			\sup_{t\in[0,T]}\mathbb{E}|Y^{\varepsilon}_t-\hat{Y}_{t}^{\varepsilon}|^2\lesssim_{\xi,\zeta,T}(\Delta^2+\Delta\varepsilon). \label{3.14}
		\end{equation}
	\end{lemma}
	
	\begin{proof}
		The proof of (\ref{3.13a}) follows  from similar argument as (\ref{esY}), we omit the details.
		Recall that the process $Y^{\varepsilon}_t-\hat{Y}_{t}^{\varepsilon}$ satisfies the following equation
		\begin{align}\label{e77}
			\left\{ \begin{aligned}
				\d(Y^{\varepsilon}_t-\hat{Y}_{t}^{\varepsilon})=&\frac{1}{\delta}\big[f(\mathscr{L}_{X^{\varepsilon}_{t}},{Y}_{t}^{\varepsilon},\mathscr{L}_{Y^{\varepsilon}_{t}})-f(\mathscr{L}_{X^{\varepsilon}_{t(\Delta)}},\hat{Y}_{t}^{\varepsilon},\mathscr{L}_{\hat{Y}_{t}^{\varepsilon}})\big]dt\\ &+\frac{1}{\sqrt{\delta}}\big[g(\mathscr{L}_{X^{\varepsilon}_{t}},{Y}_{t}^{\varepsilon},\mathscr{L}_{Y^{\varepsilon}_{t}})-g(\mathscr{L}_{X^{\varepsilon}_{t(\Delta)}},\hat{Y}_{t}^{\varepsilon},\mathscr{L}_{\hat{Y}_{t}^{\varepsilon}})\big]dW_t^2\,,\\
				Y_{0}^{\varepsilon}-\hat{Y}_{0}^{\varepsilon}=0.&
			\end{aligned}\right.
		\end{align}
		Applying It\^{o}'s formula and taking expectation, we have
		\begin{align}
				\!\!\!\!\!\!\!\!&\frac{\d}{dt}\EE|Y^{\varepsilon}_t-\hat{Y}_{t}^{\varepsilon}|^2
			\nonumber	 \\= &\frac{1}{\delta}\,\EE\Big(2\langle f(\mathscr{L}_{X^{\varepsilon}_{t}},Y_{t}^{\varepsilon},\mathscr{L}_{Y_{t}^{\varepsilon}})-f(\mathscr{L}_{X^{\varepsilon}_{t(\Delta)}},\hat{Y}_{t}^{\varepsilon},\mathscr{L}_{\hat{Y}_{t}^{\varepsilon}}),Y^{\varepsilon}_t-\hat{Y}_{t}^{\varepsilon}\rangle
			\nonumber	 \\
			 &+\|g(\mathscr{L}_{X^{\varepsilon}_{t}},{Y}_{t}^{\varepsilon},\mathscr{L}_{Y_{t}^{\varepsilon}})-g(\mathscr{L}_{X^{\varepsilon}_{t(\Delta)}},\hat{Y}_{t}^{\varepsilon},\mathscr{L}_{\hat{Y}_{t}^{\varepsilon}})\|^2\Big)
			\nonumber \\
			\leq &-\frac{K_1-K_2}{\delta}\EE|Y^{\varepsilon}_t-\hat{Y}_{t}^{\varepsilon}|^2+C\big(1+\EE|Y_{t}^{\varepsilon}|^{\kappa}+\EE|\hat{Y}_{t}^{\varepsilon}|^{\kappa}+\EE|X^{\varepsilon}_{t}|^{\kappa}+\EE|X^{\varepsilon}_{t(\Delta)}|^{\kappa}\big) 	\nonumber\\
		 &\cdot\EE|X^{\varepsilon}_{t}-X^{\varepsilon}_{t(\Delta)}|^2
				\nonumber \\
			\leq &-\frac{K_1-K_2}{\delta}\EE|Y^{\varepsilon}_t-\hat{Y}_{t}^{\varepsilon}|^2
			+\frac{C_{T}}{\delta}(1+\mathbb{E}|\xi|^{\kappa}+\mathbb{E}|\zeta|^{\frac{\kappa q}{2}})^2\Delta^2
		\nonumber \\
		 &+\frac{C_{T}}{\delta}(1+\mathbb{E}|\xi|^{2}+\mathbb{E}|\zeta|^{q})(1+\mathbb{E}|\xi|^{\kappa}+\mathbb{E}|\zeta|^{\frac{\kappa q}{2}})\varepsilon \Delta.\label{f1}
				\end{align}
		By Gronwall's lemma, we deduce that
		\begin{align}
			\mathbb{E}|Y^{\varepsilon}_t-\hat{Y}_{t}^{\varepsilon}|^2&\lesssim_T (1+\mathbb{E}|\xi|^{\kappa}+\mathbb{E}|\zeta|^{\frac{\kappa q}{2}})^2\Delta^2+(1+\mathbb{E}|\xi|^{2}+\mathbb{E}|\zeta|^{q})\cdot(1+\mathbb{E}|\xi|^{\kappa}+\mathbb{E}|\zeta|^{\frac{\kappa q}{2}})\varepsilon \Delta.\nonumber
		\end{align}
		The proof is complete.  \hfill
	\end{proof}

	\vspace{2mm}
	\noindent\textbf{Proof of Theorem \ref{theo1}.}  ~\noindent {\bf{Step 1.}} Recall the difference process $X^{\varepsilon}_t -\bar{X}_t$ satisfying
	\begin{align}\left\{\begin{array}{l}
			\displaystyle
			d(X^{\varepsilon}_t -\bar{X}_t)= \left[b(X^{\varepsilon}_t, \mathscr{L}_{X^{\varepsilon}_t}, Y^{\varepsilon}_t,\mathscr{L}_{Y^{\varepsilon}_t})-\bar{b}(\bar{X}_t,\mathscr{L}_{\bar{X}_t})\right]dt\\
			~~~~~~~~~~~~~~~~~~~~~~~+\sqrt{\varepsilon}\sigma(X^{\varepsilon}_t, \mathscr{L}_{X^{\varepsilon}_t}, Y^{\varepsilon}_t,\mathscr{L}_{Y^{\varepsilon}_t})dW^{1}_t,\\
			X_0^{\varepsilon}-\bar{X}_0=0.
		\end{array}\right.
	\end{align}
	Applying It\^{o}'s formula yields that
	\begin{align}
		|X_t^{\varepsilon}-\bar{X}_t|^2=\,\,&2\int_0^t \langle X_s^{\varepsilon}-\bar{X}_s, b(X^{\varepsilon}_s, \mathscr{L}_{X^{\varepsilon}_s}, Y^{\varepsilon}_s,\mathscr{L}_{Y^{\varepsilon}_s})-\bar{b}(\bar{X}_s,\mathscr{L}_{\bar{X}_s})\rangle ds\nonumber\\
		&+\varepsilon\int_0^t \|\sigma(X^{\varepsilon}_s, \mathscr{L}_{X^{\varepsilon}_s}, Y^{\varepsilon}_s,\mathscr{L}_{Y^{\varepsilon}_s})\|^2ds\nonumber\\
		&+2\sqrt{\varepsilon}\int_0^t \langle X_s^{\varepsilon}-\bar{X}_s, \sigma(X^{\varepsilon}_t, \mathscr{L}_{X^{\varepsilon}_t}, Y^{\varepsilon}_t,\mathscr{L}_{Y^{\varepsilon}_t})\rangle dW^{1}_s\nonumber\\
		\eqqcolon \,\,& I^{\varepsilon}(t)+II^{\varepsilon}(t)+III^{\varepsilon}(t)\nonumber.
	\end{align}
	By the growth condition \eqref{growthsigma} of $\sigma$ and the estimate \eqref{esX}, we have
	\begin{align}
		\EE\Big[\sup_{t\in [0,T]} II^{\varepsilon}(t)\Big]
&\leq  \varepsilon\int_0^T \EE\|\sigma(X^{\varepsilon}_s, \mathscr{L}_{X^{\varepsilon}_s}, Y^{\varepsilon}_s,\mathscr{L}_{Y^{\varepsilon}_s})\|^2ds \nonumber \\
&\lesssim \varepsilon\int_0^T \EE\big(1+|X^{\varepsilon}_s|^2+|Y^{\varepsilon}_s|^2+M_2(\mathscr{L}_{X^{\varepsilon}_s})+M_2(\mathscr{L}_{Y^{\varepsilon}_s})\big)ds  \nonumber\\
&\lesssim_{\xi,\zeta,T} \varepsilon .		\label{es:I2}
	\end{align}
	By Burkholder-Davis-Gundy's inequality,  growth condition \eqref{growthsigma} of $\sigma$ and the estimate \eqref{esX}, it holds that
	\begin{align}
		&\EE\Big[\sup_{t\in [0,T]} III^{\varepsilon}(t)\Big] \nonumber\\
		\lesssim\,\,\, & \sqrt{\varepsilon}\,   \EE\bigg[\bigg(\sup_{s\in [0,T]}|X_s^{\varepsilon}-\bar{X}_s|^2\cdot\int_0^T\|\sigma(X^{\varepsilon}_s, \mathscr{L}_{X^{\varepsilon}_s}, Y^{\varepsilon}_s,\mathscr{L}_{Y^{\varepsilon}_s})\|^2ds\bigg)^{\frac{1}{2}}\bigg] \nonumber\\
		\leq\,\,\, & \frac{1}{2}\,\EE\Big[\sup_{s\in [0,T]}|X_s^{\varepsilon}-\bar{X}_s|^2\Big]+C_T\varepsilon \big(1+\EE|\xi|^2+\EE|\zeta|^{q}\big).\label{es:I3}
			\end{align}
	Considering the term $I^{\varepsilon}(t)$, we reformulate it as
	\begin{align}
		I^{\varepsilon}(t)=\,\,\,&2\int_0^t \langle X_s^{\varepsilon}-\bar{X}_s, b(X^{\varepsilon}_s, \mathscr{L}_{X^{\varepsilon}_s}, Y^{\varepsilon}_s,\mathscr{L}_{Y^{\varepsilon}_s})-b(X_{s(\Delta)}^{\varepsilon},\mathscr{L}_{X_{s(\Delta)}^{\varepsilon}}, \hat{Y}_s^{\varepsilon},\mathscr{L}_{\hat{Y}_s^{\varepsilon}})\rangle ds\nonumber\\
		&+2\int_0^t \langle X_s^{\varepsilon}-\bar{X}_s, b(X_{s(\Delta)}^{\varepsilon},\mathscr{L}_{X_{s(\Delta)}^{\varepsilon}}, \hat{Y}_s^{\varepsilon},\mathscr{L}_{\hat{Y}_s^{\varepsilon}})-\bar{b}(X_{s(\Delta)}^{\varepsilon},\mathscr{L}_{X_{s(\Delta)}^{\varepsilon}})\rangle ds\nonumber\\
		&+2\int_0^t \langle X_s^{\varepsilon}-\bar{X}_s, \bar{b}(X_{s(\Delta)}^{\varepsilon},\mathscr{L}_{X_{s(\Delta)}^{\varepsilon}})-\bar{b}(X_s^{\varepsilon},\mathscr{L}_{X_s^{\varepsilon}})\rangle ds\nonumber\\
		&+2\int_0^t \langle X_s^{\varepsilon}-\bar{X}_s, \bar{b}(X_s^{\varepsilon},\mathscr{L}_{X_s^{\varepsilon}})-\bar{b}(\bar{X}_s,\mathscr{L}_{\bar{X}_s})\rangle ds\nonumber\\
		\eqqcolon \,\,\,&I_{1}^{\varepsilon}(t)+I_{2}^{\varepsilon}(t)+I_{3}^{\varepsilon}(t)+I_{4}^{\varepsilon}(t).\nonumber
	\end{align}
	Due to the local Lipschitz continuity \eqref{continuityofb} and \eqref{continuityofbarb} of $b$ and $\bar{b}$, respectively, the terms $I_{1}^{\varepsilon}(t), I_{3}^{\varepsilon}(t)$ and $I_{4}^{\varepsilon}(t) $ can be estimated as follows,
	\begin{align}\label{es:J_1+J_3}
		&\EE\Big[\sup_{t\in [0,T]} \big(I_{1}^{\varepsilon}(t)+I_{3}^{\varepsilon}(t)\big)\Big] \nonumber\\
	\lesssim& \EE\int_0^{T} |X_s^{\varepsilon}-\bar{X}_s|\big(1+\ca{M}_{\kappa}^{\frac{1}{2}}(\mathscr{L}_{X_s^{\varepsilon}},\mathscr{L}_{X_{s(\Delta)}^{\varepsilon}},\mathscr{L}_{Y^{\varepsilon}_s},  \mathscr{L}_{\hat{Y}_s^{\varepsilon}})\big)\nonumber\\
&
\cdot
	\big( |X_s^{\varepsilon}-X_{s(\Delta)}^{\varepsilon}|+|Y_s^{\varepsilon}-\hat{Y}_s^{\varepsilon}|\big)ds
 \nonumber\\
	& +\EE\int_0^{T} |X_s^{\varepsilon}-\bar{X}_s|\big(1+\rho_{\kappa}^{\frac{1}{2}}(X^{\varepsilon}_s, X_{s(\Delta)}^{\varepsilon}, Y^{\varepsilon}_s,  \hat{Y}_s^{\varepsilon})+\ca{M}_{\kappa}^{\frac{1}{2}}(\mathscr{L}_{X_s^{\varepsilon}}, \mathscr{L}_{Y^{\varepsilon}_s}, \mathscr{L}_{X_{s(\Delta)}^{\varepsilon}},\mathscr{L}_{\hat{Y}_s^{\varepsilon}})\big) \nonumber\\
	&\,\,\,\,\,\,\,\,\,\,\cdot\big(\bb{W}_2(\mathscr{L}_{X_s^{\varepsilon}}, \mathscr{L}_{X_{s(\Delta)}^{\varepsilon}})+\bb{W}_2(\mathscr{L}_{Y_s^{\varepsilon}}, \mathscr{L}_{\hat{Y}_{s}^{\varepsilon}})\big)ds \nonumber\\
	&+\EE\int_0^{T} |X_s^{\varepsilon}-\bar{X}_s| \Big\{\big(1+M_{\kappa}^{\frac{1}{2}}(\mathscr{L}_{X_{s(\Delta)}^{\varepsilon}})+M_{\kappa}^{\frac{1}{2}}(\mathscr{L}_{X_{s}^{\varepsilon}})\big)|X_{s(\Delta)}^{\varepsilon}-X^{\varepsilon}_s|
 \nonumber\\
	&\,\,\,\,\,\,\,\,\,\,+\big(1+|X_{s(\Delta)}^{\varepsilon}|^{\frac{\kappa}{2}}+|X_s^{\varepsilon}|^{\frac{\kappa}{2}}+M_{\kappa}^{\frac{1}{2}}(\mathscr{L}_{X_{s(\Delta)}^{\varepsilon}})+M_{\kappa}^{\frac{1}{2}}(\mathscr{L}_{X_{s}^{\varepsilon}})\big)
 \nonumber\\
	&\,\,\,\,\,\,\,\,\,\,\cdot
	\big(1+M_{\kappa}^{\frac{1}{2}}(\mathscr{L}_{X_{s(\Delta)}^{\varepsilon}})+M_{\kappa}^{\frac{1}{2}}(\mathscr{L}_{X_{s}^{\varepsilon}})\big)\bb{W}_2(\mathscr{L}_{X^{\varepsilon}_s},\mathscr{L}_{X_{s(\Delta)}^{\varepsilon}} )\Big\}ds
	 \nonumber\\
	\lesssim&_{\xi,\zeta,T}(\Delta\varepsilon+\Delta^2)+\int_0^T\EE|X_s^{\varepsilon}-\bar{X}_s|^2 ds,		
		\end{align}
	and
	\begin{align}
		\EE\Big[\sup_{t\in [0,T]} I_{4}^{\varepsilon}(t)\Big]
		&\lesssim\,\,\, \EE\int_0^{T} |X_s^{\varepsilon}-\bar{X}_s|\cdot \Big[\big(1+M_{\kappa}(\mathscr{L}_{X_{s}^{\varepsilon}})+M_{\kappa}(\mathscr{L}_{\bar{X}_{s}})\big)|X_{s}^{\varepsilon}-\bar{X}_s|
	\nonumber \\
		&~~~+\big(1+|X_{s}^{\varepsilon}|^{\frac{\kappa}{2}}+|\bar{X}_{s}|^{\frac{\kappa}{2}}+M_{\kappa}^{\frac{1}{2}}(\mathscr{L}_{X_{s}^{\varepsilon}})+M_{\kappa}^{\frac{1}{2}}(\mathscr{L}_{\bar{X}_{s}})\big)
\nonumber	 \\
		&~~~\cdot\big(1+M_{\kappa}^{\frac{1}{2}}(\mathscr{L}_{X_{s}^{\varepsilon}})+M_{\kappa}^{\frac{1}{2}}(\mathscr{L}_{\bar{X}_{s}})\big)\bb{W}_2(\mathscr{L}_{X_{s}^{\varepsilon}}, \mathscr{L}_{\bar{X}_s})\Big]ds
	\nonumber	 \\
		&\lesssim_{\xi,\zeta,T}\int_0^T\EE|X_s^{\varepsilon}-\bar{X}_s|^2 ds.\label{es:J4}
			\end{align}
	\noindent{\bf{Step 2.}} It remains to estimate the  term $I_{2}^{\varepsilon}(t)$. To do this, we rewrite it as
	\begin{align*}
		I_{2}^{\varepsilon}(t)=\,\,&2\int_0^t \big\langle X_s^{\varepsilon}-X_{s(\Delta)}^{\varepsilon}, b(X_{s(\Delta)}^{\varepsilon},\mathscr{L}_{X_{s(\Delta)}^{\varepsilon}}, \hat{Y}_s^{\varepsilon},\mathscr{L}_{\hat{Y}_s^{\varepsilon}})-\bar{b}(X_{s(\Delta)}^{\varepsilon},\mathscr{L}_{X_{s(\Delta)}^{\varepsilon}})\big\rangle ds\nonumber\\
	&+2\int_0^t \big\langle X_{s(\Delta)}^{\varepsilon}-\bar{X}_{s(\Delta)}, b(X_{s(\Delta)}^{\varepsilon},\mathscr{L}_{X_{s(\Delta)}^{\varepsilon}}, \hat{Y}_s^{\varepsilon},\mathscr{L}_{\hat{Y}_s^{\varepsilon}})-\bar{b}(X_{s(\Delta)}^{\varepsilon},\mathscr{L}_{X_{s(\Delta)}^{\varepsilon}})\big\rangle ds\nonumber\\
	&+2\int_0^t \big\langle \bar{X}_{s(\Delta)}-\bar{X}_s, b(X_{s(\Delta)}^{\varepsilon},\mathscr{L}_{X_{s(\Delta)}^{\varepsilon}}, \hat{Y}_s^{\varepsilon},\mathscr{L}_{\hat{Y}_s^{\varepsilon}})-\bar{b}(X_{s(\Delta)}^{\varepsilon},\mathscr{L}_{X_{s(\Delta)}^{\varepsilon}})\big\rangle ds\nonumber\\
	\eqqcolon\,\,&\mathscr{J}_1^{\varepsilon}(t)+\mathscr{J}_2^{\varepsilon}(t)+\mathscr{J}_3^{\varepsilon}(t).
		\end{align*}
	Notice that the terms $\mathscr{J}_1^{\varepsilon}(t)$ and $\mathscr{J}_3^{\varepsilon}(t)$ vanish, as $\Delta\to 0$, in view of the difference $X_s^{\varepsilon}-X_{s(\Delta)}^{\varepsilon}$ and $\bar{X}_{s(\Delta)}-\bar{X}_s$. More precisely, by employing the growth conditions \eqref{growthb} for  $b$ and \eqref{es:growthofbarb}
	for $\bar{b}$,  it follows from H\"{o}lder's inequality and \eqref{F3.10} that
	\begin{align}
		&\EE\Big[\sup_{t\in [0,T]} \mathscr{J}_1^{\varepsilon}(t)\Big]
 \nonumber\\
	\lesssim& \bigg(\int_0^T\EE|X_s^{\varepsilon}-X_{s(\Delta)}^{\varepsilon}|^2ds\bigg)^{\frac{1}{2}}
\nonumber \\
&\cdot\bigg(\int_0^T \EE|b(X_{s(\Delta)}^{\varepsilon},\mathscr{L}_{X_{s(\Delta)}^{\varepsilon}}, \hat{Y}_s^{\varepsilon},\mathscr{L}_{\hat{Y}_s^{\varepsilon}})|^2+\EE|\bar{b}(X_{s(\Delta)}^{\varepsilon},\mathscr{L}_{X_{s(\Delta)}^{\varepsilon}})|^2ds\bigg)^{\frac{1}{2}} \nonumber\\[1mm]
	\lesssim&_{\xi,\zeta,T}(\Delta+\Delta^{\frac{1}{2}}\varepsilon^{\frac{1}{2}}).\label{es:h1}
		\end{align}
	Likewise, we also deduce from \eqref{es:differencebarX} that
	\begin{align}
		&\EE\Big[\sup_{t\in [0,T]} \mathscr{J}_3^{\varepsilon}(t)\Big]
 \nonumber\\
	\lesssim&\bigg(\EE\int_0^T |\bar{X}_{s(\Delta)}-\bar{X}_s|^2ds\bigg)^{\frac{1}{2}}
 \nonumber\\
&\cdot\bigg(\EE\int_0^T |b(X_{s(\Delta)}^{\varepsilon},\mathscr{L}_{X_{s(\Delta)}^{\varepsilon}}, \hat{Y}_s^{\varepsilon},\mathscr{L}_{\hat{Y}_s^{\varepsilon}})-\bar{b}(X_{s(\Delta)}^{\varepsilon},\mathscr{L}_{X_{s(\Delta)}^{\varepsilon}})|^2ds\bigg)^{\frac{1}{2}} \nonumber\\
	\lesssim&_{\xi,\zeta,T}\Delta. \label{es:h3}
		\end{align}
	Once we can derive that
	\begin{align}\label{h_2}
		\EE\Big[\sup_{t\in [0,T]} \mathscr{J}_{2}^{\varepsilon}(t)\Big]\leq\frac{1}{4}\EE\Big[\sup_{t\in [0,T]} |X_t^{\varepsilon}-\bar{X}_t|^2\Big]+C_{\xi,\zeta,T}\Big(\Delta^{\frac{1}{2}}+\frac{\delta}{\Delta}+\frac{\delta^2}{\Delta^2}\Big),
	\end{align}
	then in combination of \eqref{es:h1} and \eqref{es:h3},  it holds that
	\begin{align}
		\EE\Big[\sup_{t\in [0,T]} I_{2}^{\varepsilon}(t)\Big]
		\leq\frac{1}{4}\EE\Big[\sup_{t\in [0,T]} |X_t^{\varepsilon}-\bar{X}_t|^2\Big]+C_{\xi,\zeta,T}\Big(\Delta^{\frac{1}{2}}+\frac{\delta}{\Delta}+\frac{\delta^2}{\Delta^2}\Big).\label{es:I12}
	\end{align}
	
	\vspace{2mm}
	\noindent{\bf{Step 3.}} Now, we employ  a modified time discretization technique, which relies on the lifted semigroup introduced in section \ref{frozen},  to deal with the term
	$\mathscr{J}_2^{\varepsilon}(t)$. Note that
	\begin{align}
		&|\mathscr{J}_2^{\varepsilon}(t)|\nonumber\\
		=&2\Bigg|\sum_{k=0}^{[t/\Delta]-1}\int_{k\Delta}^{(k+1)\Delta} \langle X_{s(\Delta)}^{\varepsilon}-\bar{X}_{s(\Delta)}, b(X_{s(\Delta)}^{\varepsilon},\mathscr{L}_{X_{s(\Delta)}^{\varepsilon}}, \hat{Y}_s^{\varepsilon},\mathscr{L}_{\hat{Y}_s^{\varepsilon}})-\bar{b}(X_{s(\Delta)}^{\varepsilon},\mathscr{L}_{X_{s(\Delta)}^{\varepsilon}})\rangle ds\nonumber\\
		&\,\,\,+\int_{t(\Delta)}^{t} \langle X_{s(\Delta)}^{\varepsilon}-\bar{X}_{s(\Delta)}, b(X_{s(\Delta)}^{\varepsilon},\mathscr{L}_{X_{s(\Delta)}^{\varepsilon}}, \hat{Y}_s^{\varepsilon},\mathscr{L}_{\hat{Y}_s^{\varepsilon}})-\bar{b}(X_{s(\Delta)}^{\varepsilon},\mathscr{L}_{X_{s(\Delta)}^{\varepsilon}})\rangle ds\Bigg|\nonumber\\
		\lesssim& \sum_{k=0}^{[t/\Delta]-1}\Bigg|\int_{k\Delta}^{(k+1)\Delta} \langle X_{s(\Delta)}^{\varepsilon}-\bar{X}_{s(\Delta)}, b(X_{s(\Delta)}^{\varepsilon},\mathscr{L}_{X_{s(\Delta)}^{\varepsilon}}, \hat{Y}_s^{\varepsilon},\mathscr{L}_{\hat{Y}_s^{\varepsilon}})-\bar{b}(X_{s(\Delta)}^{\varepsilon},\mathscr{L}_{X_{s(\Delta)}^{\varepsilon}})\rangle ds\Bigg|\nonumber\\
		&\,\,\,+\Bigg|\int_{t(\Delta)}^{t} \langle X_{s(\Delta)}^{\varepsilon}-\bar{X}_{s(\Delta)}, b(X_{s(\Delta)}^{\varepsilon},\mathscr{L}_{X_{s(\Delta)}^{\varepsilon}}, \hat{Y}_s^{\varepsilon},\mathscr{L}_{\hat{Y}_s^{\varepsilon}})-\bar{b}(X_{s(\Delta)}^{\varepsilon},\mathscr{L}_{X_{s(\Delta)}^{\varepsilon}})\rangle ds\Bigg|\nonumber\\
		=&: \mathscr{J}_{21}^{\varepsilon}(t)+\mathscr{J}_{22}^{\varepsilon}(t).\nonumber
	\end{align}
	It is clear  that $\mathscr{J}_{22}^{\varepsilon}(t)$ vanishes, as $\Delta\to 0$, in view of the integral interval $[t(\Delta), t]$ and by means  of the growth conditions \eqref{growthb}, \eqref{es:growthofbarb} for $b$, $\bar{b}$, and uniform estimates \eqref{esX}, \eqref{es:expectationbarX} for $X_t^{\varepsilon}$ and $\bar{X}_t$. More precisely, we can calculate that
	\begin{align}
		&\EE\Big[\sup_{t\in [0,T]} \mathscr{J}_{22}^{\varepsilon}(t)\Big]\nonumber\\
&\lesssim  \Delta^{\frac{1}{2}}\Big[\EE\sup_{t\in [0,T]}|X_{t}^{\varepsilon}-\bar{X}_{t}|^2\Big]^{\frac{1}{2}} \nonumber\\
	&\,\,\,\cdot\bigg(\EE\int_0^T\big|b(X_{s(\Delta)}^{\varepsilon},\mathscr{L}_{X_{s(\Delta)}^{\varepsilon}}, \hat{Y}_s^{\varepsilon},\mathscr{L}_{\hat{Y}_s^{\varepsilon}})-\bar{b}(X_{s(\Delta)}^{\varepsilon},\mathscr{L}_{X_{s(\Delta)}^{\varepsilon}}) \big|^2ds\bigg)^{\frac{1}{2}} \nonumber\\
	&\lesssim_{\xi,\zeta,T}\Delta^{\frac{1}{2}}. \label{es:l2}
		\end{align}
	As for the term $\mathscr{J}_{21}^{\varepsilon}(t)$, we see that
	\begin{align*}
			&\EE\Big[\sup_{t\in [0,T]} \mathscr{J}_{21}^{\varepsilon}(t)\Big] \\
	\leq& \sum_{k=0}^{[T/\Delta]-1}\EE\Bigg|\int_{k\Delta}^{(k+1)\Delta} \langle X_{k\Delta}^{\varepsilon}-\bar{X}_{k\Delta}, b(X_{k\Delta}^{\varepsilon},\mathscr{L}_{X_{k\Delta}^{\varepsilon}}, \hat{Y}_s^{\varepsilon},\mathscr{L}_{\hat{Y}_s^{\varepsilon}})-\bar{b}(X_{k\Delta}^{\varepsilon},\mathscr{L}_{X_{k\Delta}^{\varepsilon}})\rangle ds\Bigg| \\
	\leq &[T/\Delta] \max\limits_{\{0\leq k\leq [T/\Delta]-1\}}\EE\Bigg|\int_{k\Delta}^{(k+1)\Delta} \langle X_{k\Delta}^{\varepsilon}-\bar{X}_{k\Delta}, b(X_{k\Delta}^{\varepsilon},\mathscr{L}_{X_{k\Delta}^{\varepsilon}}, \hat{Y}_s^{\varepsilon},\mathscr{L}_{\hat{Y}_s^{\varepsilon}})
\\
&-\bar{b}(X_{k\Delta}^{\varepsilon},\mathscr{L}_{X_{k\Delta}^{\varepsilon}})\rangle ds\Bigg|.
\end{align*}
By Young's inequality, we have
	\begin{align}
&\EE\Big[\sup_{t\in [0,T]} \mathscr{J}_{21}^{\varepsilon}(t)\Big] \nonumber\\
		\leq & \frac{1}{4}\EE\Big[\sup_{t\in [0,T]} |X_t^{\varepsilon}-\bar{X}_t|^2\Big] \nonumber\\
	& +\frac{C_T}{\Delta^2} \max\limits_{\{0\leq k\leq [T/\Delta]-1\}}\EE\Bigg|\int_{k\Delta}^{(k+1)\Delta}b(X_{k\Delta}^{\varepsilon},\mathscr{L}_{X_{k\Delta}^{\varepsilon}}, \hat{Y}_s^{\varepsilon},\mathscr{L}_{\hat{Y}_s^{\varepsilon}})-\bar{b}(X_{k\Delta}^{\varepsilon},\mathscr{L}_{X_{k\Delta}^{\varepsilon}})ds\Bigg|^2 \nonumber\\[1mm]
	\leq & \frac{1}{4}\EE\Big[\sup_{t\in [0,T]} |X_t^{\varepsilon}-\bar{X}_t|^2\Big] \nonumber\\
	& +\frac{C_T\delta^2}{\Delta^2} \max\limits_{\{0\leq k\leq [T/\Delta]-1\}}\EE\Bigg|\int_{0}^{\frac{\Delta}{\delta}}b(X_{k\Delta}^{\varepsilon},\mathscr{L}_{X_{k\Delta}^{\varepsilon}}, \hat{Y}_{s\delta+k\Delta}^{\varepsilon},\mathscr{L}_{\hat{Y}_{s\delta+k\Delta}^{\varepsilon}})-\bar{b}(X_{k\Delta}^{\varepsilon},\mathscr{L}_{X_{k\Delta}^{\varepsilon}})ds\Bigg|^2. \label{es:l_1}
	\end{align}
	On the other hand, we can rewrite that
	\begin{align}\label{es100}
		&\EE\Bigg|\int_{0}^{\frac{\Delta}{\delta}}b(X_{k\Delta}^{\varepsilon},\mathscr{L}_{X_{k\Delta}^{\varepsilon}}, \hat{Y}_{s\delta+k\Delta}^{\varepsilon},\mathscr{L}_{\hat{Y}_{s\delta+k\Delta}^{\varepsilon}})-\bar{b}(X_{k\Delta}^{\varepsilon},\mathscr{L}_{X_{k\Delta}^{\varepsilon}})ds\Bigg|^2=2\int_0^{\frac{\Delta}{\delta}}\int_r^{\frac{\Delta}{\delta}}\Phi_k(s,r)dsdr,
	\end{align}
	where
	\begin{align}
		&\Phi_k(s,r)= \EE\Big[ \big\langle  b(X_{k\Delta}^{\varepsilon},\mathscr{L}_{X_{k\Delta}^{\varepsilon}}, \hat{Y}_{s\delta+k\Delta}^{\varepsilon},\mathscr{L}_{\hat{Y}_{s\delta+k\Delta}^{\varepsilon}})-\bar{b}(X_{k\Delta}^{\varepsilon},\mathscr{L}_{X_{k\Delta}^{\varepsilon}}),\nonumber\\
		&\;\;\;\;\;\;\;\;\;\;\;\;\;\;\;\;\;\;\;\;\;\;b(X_{k\Delta}^{\varepsilon},\mathscr{L}_{X_{k\Delta}^{\varepsilon}}, \hat{Y}_{r\delta+k\Delta}^{\varepsilon},\mathscr{L}_{\hat{Y}_{r\delta+k\Delta}^{\varepsilon}})-\bar{b}(X_{k\Delta}^{\varepsilon},\mathscr{L}_{X_{k\Delta}^{\varepsilon}})\big\rangle
		\Big].\nonumber
	\end{align}
	In order to estimate (\ref{es100}), we introduce the following auxiliary systems
	\begin{subequations}
		\begin{align}
			\d\tilde{Y}_t^{\gamma}&=\frac{1}{\delta}\,f(\mu, \tilde{Y}_t^{\gamma}, \mathscr{L}_{\tilde{Y}_t^{\gamma}})dt+\frac{1}{\sqrt{\delta}}\,g(\mu, \tilde{Y}_t^{\gamma}, \mathscr{L}_{\tilde{Y}_t^{\gamma}})dW_t^2,\;\;\;\;\tilde{Y}_s^{\gamma}=Y\sim\gamma\,,\label{fro3}\\
			\d\tilde{Y}_t^{\gamma,y}&=\frac{1}{\delta}\,f(\mu, \tilde{Y}_t^{\gamma,y},\mathscr{L}_{\tilde{Y}_t^{\gamma}})dt+\frac{1}{\sqrt{\delta}}\,g(\mu, \tilde{Y}_t^{\gamma, y}, \mathscr{L}_{\tilde{Y}_t^{\gamma}})dW_t^2,\;\;\;\;\tilde{Y}_s^{\gamma, y}=y\,,\label{fro4}
		\end{align}
	\end{subequations}
	and denote by $\tilde{Y}_t^{\gamma, \varepsilon, s, \mu, Y}$ and $\tilde{Y}_t^{\gamma, \varepsilon, s, \mu, y}$ the unique solutions to \eqref{fro3} and \eqref{fro4}, respectively, where we recall that the constant $\delta$ depends on $\varepsilon$.  We also recall the auxiliary systems (\ref{4.4a})-(\ref{4.4b}). Then for any $t\in [k\Delta, (k+1)\Delta]$, we know
	\begin{align}
		\hat{Y}_t^{\varepsilon}=\tilde{Y}_t^{\gamma, \varepsilon, k\Delta, \eta, \hat{Y}_{k\Delta}^{\varepsilon}}~~\text{and}~~
		\hat{Y}_t^{\varepsilon, y}=\tilde{Y}_t^{\gamma, \varepsilon, k\Delta, \eta, y},\nonumber
	\end{align}
	where $
	\eta:=\mathscr{L}_{X_{k\Delta}^{\varepsilon}}$, $\gamma:=\mathscr{L}_{\hat{Y}_{k\Delta}^{\varepsilon}}$.
	Moreover, we have for any $s\geq0$ that
	\begin{align}
&\tilde{Y}_{s\delta+k\Delta}^{\gamma,\varepsilon, k\Delta, \eta, \hat{Y}_{k\Delta}^{\varepsilon}} \nonumber\\
=\,\,\,&\hat{Y}_{k\Delta}^{\varepsilon}+\frac{1}{\delta}\int_{k\Delta}^{s\delta+k\Delta}f(\eta, \tilde{Y}_u^{\gamma,\varepsilon, k\Delta, \eta, \hat{Y}_{k\Delta}^{\varepsilon}}, \mathscr{L}_{\tilde{Y}_u^{\gamma,\varepsilon, k\Delta, \eta, \hat{Y}_{k\Delta}^{\varepsilon}}})du
\nonumber \\
&\,\,\,\,\,\,\,\,\,\,+\frac{1}{\sqrt{\delta}}\int_{k\Delta}^{s\delta+k\Delta}g(\eta, \tilde{Y}_u^{\gamma,\varepsilon, k\Delta, \eta, \hat{Y}_{k\Delta}^{\varepsilon}}, \mathscr{L}_{\tilde{Y}_u^{\gamma,\varepsilon, k\Delta, \eta, \hat{Y}_{k\Delta}^{\varepsilon}}})dW_u^2 \nonumber\\
=\,\,\,&\hat{Y}_{k\Delta}^{\varepsilon}+\frac{1}{\delta}\int_{0}^{s\delta}f(\eta, \tilde{Y}_{u+k\Delta}^{\gamma,\varepsilon, k\Delta, \eta, \hat{Y}_{k\Delta}^{\varepsilon}}, \mathscr{L}_{\tilde{Y}_{u+k\Delta}^{\gamma,\varepsilon, k\Delta, \eta, \hat{Y}_{k\Delta}^{\varepsilon}}})du
\nonumber \\
&\,\,\,\,\,\,\,\,\,\,+\frac{1}{\sqrt{\delta}}\int_{0}^{s\delta}g(\eta, \tilde{Y}_{u+k\Delta}^{\gamma,\varepsilon, k\Delta, \eta, \hat{Y}_{k\Delta}^{\varepsilon}}, \mathscr{L}_{\tilde{Y}_{u+k\Delta}^{\gamma,\varepsilon, k\Delta, \eta, \hat{Y}_{k\Delta}^{\varepsilon}}})dW_u^{2,k\Delta} \nonumber\\
=\,\,\,&\hat{Y}_{k\Delta}^{\varepsilon}+\int_{0}^{s}f(\eta, \tilde{Y}_{u\delta+k\Delta}^{\gamma,\varepsilon, k\Delta, \eta, \hat{Y}_{k\Delta}^{\varepsilon}}, \mathscr{L}_{\tilde{Y}_{u\delta+k\Delta}^{\gamma,\varepsilon, k\Delta, \eta, \hat{Y}_{k\Delta}^{\varepsilon}}})du
\nonumber \\
&\,\,\,\,\,\,\,\,\,\,+\int_{0}^{s}g(\eta, \tilde{Y}_{u\delta+k\Delta}^{\gamma,\varepsilon, k\Delta, \eta, \hat{Y}_{k\Delta}^{\varepsilon}}, \mathscr{L}_{\tilde{Y}_{u\delta+k\Delta}^{\gamma,\varepsilon, k\Delta, \eta, \hat{Y}_{k\Delta}^{\varepsilon}}})dW_u^{2,k\Delta,\delta},\label{eq:ffor1}
			\end{align}
	where we denote $W_u^{2,k\Delta}:=W_{u+k\Delta}^2-W_{k\Delta}^2$ and $W_u^{2,k\Delta,\delta}:=\frac{1}{\sqrt{\delta}}W_{u\delta}^{2,k\Delta}$, and
	\begin{align}
		&\tilde{Y}_{s\delta+k\Delta}^{\gamma, \varepsilon, k\Delta, \eta, y}\nonumber \\
		=\,\,\,&y+\frac{1}{\delta}\int_{k\Delta}^{s\delta+k\Delta}f(\eta, \tilde{Y}_u^{\gamma,\varepsilon, k\Delta, \eta, y},\mathscr{L}_{\tilde{Y}_u^{\gamma,\varepsilon, k\Delta, \eta, \hat{Y}_{k\Delta}^{\varepsilon}}})du
	\nonumber	\\
		&\,\,\,\,+\frac{1}{\sqrt{\delta}}\int_{k\Delta}^{s\delta+k\Delta}g(\eta, \tilde{Y}_u^{\gamma, \varepsilon, k\Delta, \eta, y}, \mathscr{L}_{\tilde{Y}_u^{\gamma,\varepsilon, k\Delta, \eta, \hat{Y}_{k\Delta}^{\varepsilon}}})dW_u^2 \nonumber\\
		=\,\,\,&y+\int_{0}^{s}f(\eta, \tilde{Y}_{u\delta+k\Delta}^{\gamma,\varepsilon, k\Delta, \eta, y},\mathscr{L}_{\tilde{Y}_{u\delta+k\Delta}^{\gamma,\varepsilon, k\Delta, \eta, \hat{Y}_{k\Delta}^{\varepsilon}}})du
	\nonumber	\\
		&\,\,\,\,+\int_{0}^{s}g(\eta, \tilde{Y}_{u\delta+k\Delta}^{\gamma, \varepsilon, k\Delta, \eta, y}, \mathscr{L}_{\tilde{Y}_{u\delta+k\Delta}^{\gamma,\varepsilon, k\Delta, \eta, \hat{Y}_{k\Delta}^{\varepsilon}}})dW_u^{2,k\Delta,\delta}.\label{eq:ffor2}
			\end{align}
	Compared the above two equations \eqref{eq:ffor1} and \eqref{eq:ffor2} with the formulation of frozen systems
	\eqref{fro1} and \eqref{fro2}, we can see that
	the pair
	\begin{align}
		\big(\tilde{Y}_{s\delta+k\Delta}^{\gamma, \varepsilon, k\Delta, \eta, y} ,  \mathscr{L}_{\tilde{Y}_{s\delta+k\Delta}^{\gamma, \varepsilon, k\Delta, \eta, \hat{Y}_{k\Delta}^{\varepsilon}}}\big)_{0\leq s\leq \frac{\Delta}{\delta}}\nonumber
	\end{align}
	is identical in distribution with the pair
	\begin{align}
		\big(Y_s^{\eta, \gamma,y}, \mathscr{L}_{Y_s^{\eta,\gamma}}\big)_{0\leq s\leq \frac{\Delta}{\delta}},\nonumber
	\end{align}
	due to the weak uniqueness of the frozen systems \eqref{eq:ffor1}-\eqref{eq:ffor2}, in which we recall $\gamma=\mathscr{L}_{\hat{Y}_{k\Delta}^{\varepsilon}}, \eta=\mathscr{L}_{X_{k\Delta}^{\varepsilon}}$.
	
	Now, we are ready to estimate $\Phi_{k}(s,r)$. Since $X_{k\Delta}^{\varepsilon}, \hat{Y}_{k\Delta}^{\varepsilon}$ are $\mathscr{F}_{k\Delta}$-measurable, the property of conditional expectation implies that
	\begin{align}\label{e10}
		&\Phi_k(s, r) \nonumber\\
		=\,\,\,& \EE\bigg[\EE\Big[ \big\langle  b(X_{k\Delta}^{\varepsilon},\eta, \tilde{Y}_{s\delta+k\Delta}^{\gamma,\varepsilon, k\Delta, \eta, \hat{Y}_{k\Delta}^{\varepsilon}},\mathscr{L}_{\tilde{Y}_{s\delta+k\Delta}^{\gamma,\varepsilon, k\Delta, \eta, \hat{Y}_{k\Delta}^{\varepsilon}}})-\bar{b}(X_{k\Delta}^{\varepsilon},\eta), \nonumber\\
		&\;\;\;\;\;\;\;\;\;\;\;b(X_{k\Delta}^{\varepsilon},\eta, \tilde{Y}_{r\delta+k\Delta}^{\gamma,\varepsilon, k\Delta, \eta, \hat{Y}_{k\Delta}^{\varepsilon}},\mathscr{L}_{\tilde{Y}_{r\delta+k\Delta}^{\gamma,\varepsilon, k\Delta, \eta, \hat{Y}_{k\Delta}^{\varepsilon}}})-\bar{b}(X_{k\Delta}^{\varepsilon},\eta)\big\rangle
		\big|\mathscr{F}_{k\Delta}\Big]\bigg] \nonumber\\
		=\,\,\,& \EE\bigg[\EE\Big[ \big\langle  b(x,\eta, \tilde{Y}_{s\delta+k\Delta}^{\gamma,\varepsilon, k\Delta, \eta, y},\mathscr{L}_{\tilde{Y}_{s\delta+k\Delta}^{\gamma,\varepsilon, k\Delta, \eta, \hat{Y}_{k\Delta}^{\varepsilon}}})-\bar{b}(x,\eta),\nonumber \\
		&\;\;\;\;\;\;\;\;\;\;\;\;b(x,\eta, \tilde{Y}_{r\delta+k\Delta}^{\gamma,\varepsilon, k\Delta, \eta, y},\mathscr{L}_{\tilde{Y}_{r\delta+k\Delta}^{\gamma,\varepsilon, k\Delta, \eta, \hat{Y}_{k\Delta}^{\varepsilon}}})-\bar{b}(x,\eta)\big\rangle
		\Big]\Big|_{x=X_{k\Delta}^{\varepsilon}, y=\hat{Y}_{k\Delta}^{\varepsilon}}\bigg],
			\end{align}
	where we notice that $\tilde{Y}_{r\delta+k\Delta}^{\gamma,\varepsilon, k\Delta, \eta, y}$ is the solution of Eq.~\eqref{eq:ffor2}.    Furthermore, thanks to  the Markovian property and time-homogeneous property of the lifted semigroup $\tilde{\bf{P}}^{\eta}_t$ defined by (\ref{semig}), it holds that
	\begin{align}
&\Phi_k(s,r)\nonumber\\
=\,\,\,& \EE\bigg[\EE\Big[ \big\langle  b(x,\eta, \tilde{Y}_{s\delta+k\Delta}^{\gamma,\varepsilon, k\Delta, \eta, y},\mathscr{L}_{\tilde{Y}_{s\delta+k\Delta}^{\gamma,\varepsilon, k\Delta, \eta, \hat{Y}_{k\Delta}^{\varepsilon}}})-\bar{b}(x,\eta), \nonumber\\
&\;\;\;\;\;\;\;\;\;\;\;b(x,\eta, \tilde{Y}_{r\delta+k\Delta}^{\gamma,\varepsilon, k\Delta, \eta, y},\mathscr{L}_{\tilde{Y}_{r\delta+k\Delta}^{\gamma,\varepsilon, k\Delta, \eta, \hat{Y}_{k\Delta}^{\varepsilon}}})-\bar{b}(x,\eta)\big\rangle
\Big]\Big|_{x=X_{k\Delta}^{\varepsilon}, y=\hat{Y}_{k\Delta}^{\varepsilon}}\bigg] \nonumber\\
=\,\,\,& \EE\bigg[\tilde{\EE}\Big[ \big\langle  b(x,\eta,
Y_{s}^{\eta,\gamma,y}, \mathscr{L}_{Y_{s}^{\eta, \gamma}})-\bar{b}(x,\eta), \nonumber\\
&\;\;\;\;\;\;\;\;\;\;\;b(x,\eta, Y_{r}^{\eta,\gamma,y}, \mathscr{L}_{Y_{r}^{\eta, \gamma}})-\bar{b}(x,\eta)\big\rangle
\Big]\Big|_{x=X_{k\Delta}^{\varepsilon}, y=\hat{Y}_{k\Delta}^{\varepsilon}}\bigg]\nonumber \\
=\,\,\,& \EE\bigg[\tilde{\EE}\Big[\tilde{\bf{E}}\big[ \big\langle  b(x,\eta,
(Y_{s}^{\eta,\gamma,y}, \mathscr{L}_{Y_{s}^{\eta, \gamma}}))-\bar{b}(x,\eta), \nonumber\\
&\;\;\;\;\;\;\;\;\;\;\;b(x,\eta, Y_{r}^{\eta,\gamma,y}, \mathscr{L}_{Y_{r}^{\eta, \gamma}})-\bar{b}(x,\eta)\big\rangle
\big]\big|\mathscr{F}_r\Big]\Big|_{x=X_{k\Delta}^{\varepsilon}, y=\hat{Y}_{k\Delta}^{\varepsilon}}\bigg] \nonumber\\
=\,\,\,& \EE\bigg[\tilde{\EE}\Big[ \big\langle{\tilde{\bf{E}}}\big[ b(x,\eta,
(Y_{s}^{\eta,\gamma,y}, \mathscr{L}_{Y_{s}^{\eta, \gamma}}))-\bar{b}(x,\eta)\big|\mathscr{F}_r\big], \nonumber\\
&\;\;\;\;\;\;\;\;\;\;\;\;\;\;\;\;\;b(x,\eta, Y_{r}^{\eta,\gamma,y}, \mathscr{L}_{Y_{r}^{\eta, \gamma}})-\bar{b}(x,\eta)\big\rangle
\Big]\Big|_{x=X_{k\Delta}^{\varepsilon}, y=\hat{Y}_{k\Delta}^{\varepsilon}}\bigg] \nonumber\\
=\,\,\,& \EE\bigg[\tilde{\EE}\Big[ \langle \tilde{{\bf{P}}}_{s-r}^{\eta}\big[ b(x,\eta, ( Y_r^{\eta,\gamma, y}, \mathscr{L}_{Y_r}^{\eta,\gamma}))-\bar{b}(x,\eta)\big], \nonumber\\
&\;\;\;\;\;\;\;\;\;\;\;\;\;\;\;\;\;\;\;\;b(x,\eta, Y_{r}^{\eta,\gamma,y}, \mathscr{L}_{Y_{r}^{\eta, \gamma}})-\bar{b}(x,\eta)\rangle
\Big]\Big|_{x=X_{k\Delta}^{\varepsilon}, y=\hat{Y}_{k\Delta}^{\varepsilon}}\bigg] \nonumber\\
=\,\,\,& \EE\bigg[\tilde{\EE}\Big[ \langle \tilde{\EE}\big[ b(x,\eta, Y_{s-r}^{\eta, \mathscr{L}_{Y_r}^{\eta,\gamma}, z}, \mathscr{L}_{Y_{s-r}^{\eta, \mathscr{L}_{Y_r}^{\eta,\gamma}}})-\bar{b}(x,\eta)\big]\mathbf{1}_{\{z=Y_r^{\eta,\gamma,y}\}}, \nonumber\\
&\;\;\;\;\;\;\;\;\;\;\;\;\;\;\;\;b(x,\eta, Y_{r}^{\eta,\gamma,y}, \mathscr{L}_{Y_{r}^{\eta, \gamma}})-\bar{b}(x,\eta)\rangle
\Big]\Big|_{x=X_{k\Delta}^{\varepsilon}, y=\hat{Y}_{k\Delta}^{\varepsilon}}\bigg]. \label{e11}
			\end{align}
	
Hence, in view of \Cref{prop:barb} (ii) and the growth conditions \eqref{growthb} and \eqref{es:growthofbarb} for $b$ and $\bar{b}$, we derive that
\begin{align}
	\Phi_k(s, r)&= \EE\bigg[ \tilde{\EE}\Big[\,\, \Big|\tilde{\EE}\big[ b(x,\eta, Y_{s-r}^{\eta, \mathscr{L}_{Y_r}^{\eta,\gamma}, z}, \mathscr{L}_{Y_{s-r}^{\eta, \mathscr{L}_{Y_r}^{\eta,\gamma}}})-\bar{b}(x,\eta)\big]\mathbf{1}_{\{z=Y_r^{\eta,\gamma,y}\}}\Big|\nonumber\\
	&\,\,\,\cdot\big|b(x,\eta, Y_{r}^{\eta,\gamma,y}, \mathscr{L}_{Y_{r}^{\eta, \gamma}})-\bar{b}(x,\eta)\big|
	\Big]\Big|_{x=X_{k\Delta}^{\varepsilon}, y=\hat{Y}_{k\Delta}^{\varepsilon}}\bigg]\nonumber\\
	&\lesssim \EE\bigg[ \tilde{\EE}\Big[\big(1+|x|^{\kappa}+\big|Y_r^{\eta,\gamma,y}\big|^{\kappa}+M_{\kappa}(\eta)+M_{\kappa}(\mathscr{L}_{Y_r}^{\eta,\gamma})\big)e^{-\beta(s-r)}\nonumber\\
	&\,\,\,\cdot \Big(\big(1+|x|^{\frac{\kappa}{2}}+|Y_r^{\eta,\gamma,y}|^{\frac{\kappa}{2}}+M_{\kappa}^{\frac{1}{2}}(\eta)+M_{\kappa}^{\frac{1}{2}}(\mathscr{L}_{Y_r^{\eta,\gamma}})\big)\nonumber\\
	&\,\,\,+\big( 1+M_{\kappa}^{\frac{1}{2}}(\eta)\big)|x|+(1+|x|^{\frac{\kappa}{2}}+M_{\kappa}^{\frac{1}{2}}(\eta))(1+M_{\kappa}^{\frac{1}{2}}(\eta))^2\Big)\Big]\Big|_{x=X_{k\Delta}^{\varepsilon}, y=\hat{Y}_{k\Delta}^{\varepsilon}}\bigg]\nonumber\\
	&\lesssim_{\xi,\zeta,T}e^{-\beta(s-r)}.\nonumber
\end{align}
Returning to the estimate \eqref{es:l_1}, we obtain that
\begin{align}
	&\EE\Big[\sup_{t\in [0,T]} \mathscr{J}_{21}^{\varepsilon}(t)\Big]\nonumber\\
	\leq\,\,\,&\frac{1}{4}\EE\Big[\sup_{t\in [0,T]} |X_t^{\varepsilon}-\bar{X}_t|^2\Big]+\frac{C_T\delta^2}{\Delta^2} \max\limits_{\{0\leq k\leq [T/\Delta]-1\}}\int_0^{\frac{\Delta}{\delta}}\int_r^{\frac{\Delta}{\delta}}\Phi_k(s,r)dsdr\nonumber\\
	\leq\,\,\,&\frac{1}{4}\EE\Big[\sup_{t\in [0,T]} |X_t^{\varepsilon}-\bar{X}_t|^2\Big]+C_{\xi,\zeta,T}\Big(\frac{\delta}{\Delta}+\frac{\delta^2}{\Delta^2}\Big).\nonumber
\end{align}
Therefore, taking all the estimates \eqref{es:I2}-\eqref{es:J4}, \eqref{es:I12} into account and using Gronwall's inequality, we arrive at
\begin{align*}
	\EE\Big[\sup_{t\in [0,T]}|X_t^{\varepsilon}-\bar{X}_t|^2\Big]
	&\,\lesssim_{\xi,\zeta,T}\Big(\varepsilon+\Delta^{\frac{1}{2}}+\int_0^T\EE|X_s^{\varepsilon}-\bar{X}_s|^2ds+\frac{\delta}{\Delta}+\frac{\delta^2}{\Delta^2}\Big)\\
	&\,\lesssim_{\xi,\zeta,T}\Big(\varepsilon+\Delta^{\frac{1}{2}}+\frac{\delta}{\Delta}+\frac{\delta^2}{\Delta^2}\Big)\,.
\end{align*}
By taking $\Delta=\delta^{\frac{2}{3}}$, we deduce that
\begin{align}
	\mathbb{E}\Big[\sup_{t\in [0, T]}|X_{t}^{\varepsilon}-\bar{X}_{t}|^{2}\Big]\lesssim_{\xi,\zeta,T} \Big(\varepsilon+\Delta^{\frac{1}{3}}\Big),\nonumber
\end{align}
which completes the proof of \Cref{theo1}.  \hfill

\section{Proof of large deviation principle}\label{sec6}
In this section, we devote to proving Theorem \ref{t2.1} based on the weak convergence analysis and the functional occupation measure technique.  In Subsection \ref{sub5.1}, we give an  overview of the powerful weak convergence method developed by Budhiraja et al.~\cite{MR1785237,MR3967100}. Then we construct a suitable functional occupation measure ${\PPi}^{\varepsilon,\Delta}$ associated to the controlled multi-scale dynamics.   In Subsection \ref{sub5.2}, we present some uniform estimates for the controlled processes $(X^{\varepsilon,h^{\varepsilon}},Y^{\varepsilon,h^{\varepsilon}})$ and show the tightness of the family $\{(X^{\varepsilon,h^{\varepsilon}},{\PPi}^{\varepsilon,\Delta})\}_{\varepsilon\in(0,1)}$. In Subsection \ref{sub5.3}, we characterize the weak limit of $(X^{\varepsilon,h^{\varepsilon}},{\PPi}^{\varepsilon,\Delta})$, which is shown to be a lifted viable pair in the sense of Definition \ref{d2.4}. In Subsections \ref{sub5.4}-\ref{sub5.6}, we devote to proving the main theorem.

\subsection{Weak convergence analysis}\label{sub5.1}
This section presents a concise overview of the weak convergence method, which is an effective framework for investigating LDP for stochastic dynamical systems. Building upon foundational work in \cite{MR3967100, MR1431744}, this  approach  relies on two critical components: the equivalence between LDP and the Laplace principle, and the application of variational representations for exponential functionals of Brownian motions as demonstrated  in \cite{MR1675051, MR1785237}.

Let $\mathcal{E}$  be the space consisting of all  $\mathbb{R}^d$-valued continuous functions over the interval $[0,T]$.  As mentioned before, our objective  is to show that the family of processes $\{X^\varepsilon\}_{\varepsilon\in(0,1)}$ satisfies the Laplace principle with speed $\varepsilon$, namely,  for every bounded and continuous functional $\Lambda:\mathcal{E}\to\mathbb{R}$,
\begin{equation}\label{2.5}
	\lim\limits_{\varepsilon\to0}-{\varepsilon}\log\mathbb{E}\left[\exp\left\{-\frac{1}{\varepsilon}\Lambda(X^\varepsilon)\right\}\right]=\inf\limits_{x\in\mathcal{E}}\left[I(x)+\Lambda(x)\right].
\end{equation}
Consider a  bounded and measurable real-valued functional $F(\cdot)$ defined on $\mathcal{E}$. Following from  \cite{MR1675051,MR1785237}, the following identity holds
\begin{equation}\label{es13}
	-\log\mathbb{E}\Big[\exp\Big\{-F(W)\Big\}\Big]=\inf_{h\in\mathcal{A}}\mathbb{E}\left[\frac{1}{2}\int_{0}^{T}|h_{s}|^{2}ds+F\left(W+\int_{0}^{\cdot}h_{s}ds\right)\right],
\end{equation}
where $W$ denotes a standard $d$-dimensional Brownian motion.

In the context of the present work, we take $W:= (W^1, W^2)$ and $d:= d_1 + d_2$. Following from Theorem \ref{thm:wellposedness} and employing a standard decoupled argument (as detailed in Section 4.1 of \cite{MR4634338}), there exists a measurable mapping $\mathcal{G}^{\varepsilon}: C([0,T]; \mathbb{R}^{d}) \to C([0,T]; \mathbb{R}^n)$ such that
$$X^{\varepsilon}=\mathcal{G}^{\varepsilon}(W.).$$

For any control $h^\varepsilon=(h^{1,\varepsilon},h^{2,\varepsilon})\in\mathcal{A}_M$, let us define
\begin{equation}\label{2.6}
	X^{\varepsilon,h^{\varepsilon}}:=\mathcal{G}^{\varepsilon}\Big(W.+\frac{1}{\sqrt{\varepsilon}}\int_{0}^{\cdot}h_{s}^{\varepsilon}ds\Big).
\end{equation}
Then it is the solution associated to the following stochastic control problem
\begin{subequations}
	\begin{align}
		dX_t^{\varepsilon,h^{\varepsilon}}=&b(X_t^{\varepsilon,h^{\varepsilon}},\mathscr{L}_{X^{\varepsilon}_t},Y_t^{\varepsilon,h^{\varepsilon}},\mathscr{L}_{Y^{\varepsilon}_t})dt+\sigma(X_t^{\varepsilon,h^{\varepsilon}},\mathscr{L}_{X^{\varepsilon}_t},Y_t^{\varepsilon,h^{\varepsilon}},\mathscr{L}_{Y^{\varepsilon}_t})h_t^{1,\varepsilon}dt\label{2.7}
		\\
		&+\sqrt{\varepsilon}\sigma(X_t^{\varepsilon,h^{\varepsilon}},\mathscr{L}_{X^{\varepsilon}_t},Y_t^{\varepsilon,h^{\varepsilon}},\mathscr{L}_{Y^{\varepsilon}_t})dW_t^1,
		~~X_0^{\varepsilon,h^{\varepsilon}}=x,\nonumber
		\\dY_t^{\varepsilon,h^{\varepsilon}}=&\frac{1}{\delta}f(\mathscr{L}_{X^{\varepsilon}_t},Y_t^{\varepsilon,h^{\varepsilon}},\mathscr{L}_{Y^{\varepsilon}_t})dt+\frac{1}{\sqrt{\delta\varepsilon}}g(\mathscr{L}_{X^{\varepsilon}_t},Y_t^{\varepsilon,h^{\varepsilon}},\mathscr{L}_{Y^{\varepsilon}_t})h_t^{2,\varepsilon}dt\label{2.8}
		\\
		&+\frac{1}{\sqrt{\delta}}g(\mathscr{L}_{X^{\varepsilon}_t},Y_t^{\varepsilon,h^{\varepsilon}},\mathscr{L}_{Y^{\varepsilon}_t})dW_t^2,~~Y_0^{\varepsilon,h^{\varepsilon}}=y.\nonumber
	\end{align}
\end{subequations}
By setting $F(W)=\frac{1}\varepsilon\Lambda(X^\varepsilon)$ in (\ref{es13}) and rescaling the controls by $\sqrt{\varepsilon}$, we derive the following representation
\begin{equation}\label{es14}
	-\varepsilon\log\mathbb{E}\left[\exp\left\{-\frac{1}{\varepsilon}\Lambda(X^\varepsilon)\right\}\right] =\inf_{h\in\mathcal{A}}\mathbb{E}\left[\frac{1}{2}\int_0^T|h_s|^2ds+\Lambda(X^{\varepsilon,h})\right],
\end{equation}
where $X^{\varepsilon,h}$ is defined by (\ref{2.6}) with $h$ replacing $h^\mathrm{\varepsilon}$.
The next key step is to  analyze the limiting behavior of the controlled process $X^{\varepsilon,h^\varepsilon}$ in the context of the weak convergence method.

Before doing that, we give some necessary notations that will be used in the functional occupation measure approach. Let $A_1,A_2,A_3,A_4$ be Borel subsets of $\mathbb{R}^d, \mathbb{R}^m, \mathscr{P}_2(\mathbb{R}^m), [0,T]$, respectively. Let $\Delta:=\Delta(\varepsilon)$ be a separation parameter between the slow and fast time scales, such that
\begin{align}\label{2.11}
	\Delta(\varepsilon)\to0,~\frac{\delta}{\varepsilon\Delta} \to 0 ~\mathrm{as~} \varepsilon\to0.
\end{align}

To analyze the joint perturbation for the control and the fast oscillating process in the multi-scale systems, we construct the occupation measure ${\PPi}^{\varepsilon,\Delta}$ as follows
\begin{equation}\label{2.12}
	\PPi^{\varepsilon,\Delta}(A_1\times A_2\times A_3\times A_4):=\int_{A_4}\dfrac{1}{\Delta}\int_t^{t+\Delta}\mathbf{1}_{A_1}(h_s^{\varepsilon})\mathbf{1}_{A_2}(Y_s^{\varepsilon,h^{\varepsilon}})\mathbf{1}_{A_3}(\mathscr{L}_{Y_{s}^{\varepsilon}})dsdt,
\end{equation}
where the controlled processes $(X^{\varepsilon,h^\varepsilon},Y^{\varepsilon,h^\varepsilon})$ are governed by (\ref{2.7})-(\ref{2.8}). Here, we make a  convention that $h_t= h_t^\varepsilon= 0$ if $t> T$. It is worth noting  that for any bounded continuous function $\Theta:\mathbb{R}^n\times\mathscr{P}_2(\RR^n)\times\mathbb{R}^{d}\times\mathbb{R}^m\times\mathscr{P}_2(\mathbb{R}^m)\to\mathbb{R}^n$, we have
\begin{align}
	&\int_{\mathbb{R}^d\times\mathbb{R}^m\times\mathscr{P}_2(\mathbb{R}^m)\times[0,T]}\Theta(X_{t}^{\varepsilon,h^{\varepsilon}},\mathscr{L}_{X_{t}^{\varepsilon}},h,y,\rho ){\PPi}^{\varepsilon,\Delta}(dhdyd\rho dt)  \nonumber\\ &=\int_{0}^{T}\frac{1}{\Delta}\int_{t}^{t+\Delta}\Theta(X_{t}^{\varepsilon,h^{\varepsilon}},\mathscr{L}_{X_{t}^{\varepsilon}},h_{s}^{\varepsilon},Y_{s}^{\varepsilon,h^{\varepsilon}},\mathscr{L}_{Y_{s}^{\varepsilon}})dsdt.\label{2.13}
\end{align}

\subsection{Uniform estimates and tightness}\label{sub5.2}
We begin this subsection by the following a priori estimates of the controlled process $(X^{\varepsilon,h^\varepsilon},Y^{\varepsilon,h^\varepsilon})$.
\begin{lemma}\label{ll3.1}
	Let $M>0$.	For any $\{h^{\varepsilon}\}_{\varepsilon\in(0,1)}\subset\mathcal{A}_{M}$ and $p \geq 1$,
	\begin{equation}\label{3.1}
		\mathbb{E}\Big[\sup\limits_{t\in[0,T]}|X_t^{\varepsilon,h^\varepsilon}|^{2p}\Big]\lesssim_{p,M,T}(1+|x|^{2p}+|y|^{pq}),
	\end{equation}
	\begin{equation}\label{3.2}
		\mathbb{E}\bigg[\Big(\int_0^T|Y_t^{\varepsilon,h^\varepsilon}|^qdt\Big)^p\bigg]\lesssim_{p,M,T}(1+|y|^{2p}).
	\end{equation}
\end{lemma}
\begin{proof}
	Firstly, using It\^o's formula for $|Y_{t}^{\varepsilon,h^{\varepsilon}}|^{2}$, we have
	\begin{align}
		|Y_{t}^{\varepsilon,h^{\varepsilon}}|^{2} =&|y|^{2}+\frac{1}{\delta}\int_{0}^{t}\Big[2\langle f(\mathscr{L}_{X^{\varepsilon}_t},Y_t^{\varepsilon,h^{\varepsilon}},\mathscr{L}_{Y^{\varepsilon}_t}),Y_{s}^{\varepsilon,h^{\varepsilon}}\rangle+\|g(\mathscr{L}_{X^{\varepsilon}_t},Y_t^{\varepsilon,h^{\varepsilon}},\mathscr{L}_{Y^{\varepsilon}_t})\|^{2}\Big]ds \nonumber \\
	&+\frac{2}{\sqrt{\varepsilon\delta}}\int_{0}^{t}\langle g(\mathscr{L}_{X^{\varepsilon}_t},Y_t^{\varepsilon,h^{\varepsilon}},\mathscr{L}_{Y^{\varepsilon}_t})h_{s}^{2,\varepsilon},Y_{s}^{\varepsilon,h^{\varepsilon}}\rangle ds+M_{t}, \label{3.3}
		\end{align}
	where we denote
	\begin{equation*}
		M_t:=\frac{2}{\sqrt{\delta}}\int_0^t\langle Y_s^{\varepsilon,h^\varepsilon},g(\mathscr{L}_{X^{\varepsilon}_t},Y_t^{\varepsilon,h^{\varepsilon}},\mathscr{L}_{Y^{\varepsilon}_t})dW_s^2\rangle.
	\end{equation*}
	
	By the assumption $(\mathbf{H}_2)$ and the estimate (\ref{esY}), we  have
	\begin{align}
		\frac{2}{\sqrt{\varepsilon\delta}}\langle g(\mathscr{L}_{X^{\varepsilon}_t},Y_t^{\varepsilon,h^{\varepsilon}},\mathscr{L}_{Y^{\varepsilon}_t})h_{s}^{2,\varepsilon},Y_{s}^{\varepsilon,h^{\varepsilon}}\rangle  \lesssim&\frac{1}{\sqrt{\varepsilon\delta}}(\mathbb{E}|Y^{\varepsilon}_t|^2)^{\frac{1}{2}}|h_{s}^{2,\varepsilon}||Y_{s}^{\varepsilon,h^{\varepsilon}}|  \nonumber\\
	\leq&\frac{C(1+|y|^2)}{\varepsilon}|h_{s}^{2,\varepsilon}|^{2}+\frac{\tilde{C}}{\delta}|Y_{s}^{\varepsilon,h^{\varepsilon}}|^{2}, \label{3.4}
		\end{align}
	where in the last step we applied Young's inequality with a small constant $\tilde{C}\in(0,K_1)$, and $K_1$ is defined in (\ref{inneroffplusg}). Combining (\ref{inneroffplusg}) and (\ref{3.3})-(\ref{3.4}), it follows that
	\begin{align}
		&\frac{\kappa_0}{\delta}\int_0^T|Y_s^{\varepsilon,h^\varepsilon}|^2ds+	\frac{K_0}{\delta}\int_0^T|Y_s^{\varepsilon,h^\varepsilon}|^qds\nonumber\\
		\leq&|y|^2+\frac{C_T}{\delta}+\frac{C(1+|y|^2)}{\varepsilon}\int_0^T|h_s^{2,\varepsilon}|^2ds+\sup_{t\in[0,T]}|M_t|+\frac{\tilde{K}_2}{\delta}\int_0^T\mathbb{E}|Y_s^{\varepsilon}|^2ds,\nonumber
	\end{align}
	where $\kappa_0:=K_1 -\tilde{C}>0$. Then by the estimate (\ref{esY}),
	\begin{align}
&\mathbb{E}\Big[\Big(\int_{0}^{T}|Y_{s}^{\varepsilon,h^{\varepsilon}}|^{2}ds\Big)^{p}\Big]\nonumber\\
\lesssim&_{p,T} (1+(1+\delta^{p})|y|^{2p})+{\frac{\delta^{p}(1+|y|^2)}{\varepsilon^{p}}}\mathbb{E}\Big[\Big(\int_{0}^{T}|h_{s}^{2,\varepsilon}|^{2}ds\Big)^{p}\Big]+\delta^{p}\mathbb{E}\Big[\Big(\sup_{t\in[0,T]}|M_{t}|\Big)^{p}\Big]\nonumber \\
\leq&\frac{1}{2}\mathbb{E}\Big[\Big(\int_{0}^{T}|Y_{s}^{\varepsilon,h^{\varepsilon}}|^{2}ds\Big)^{p}\Big]+C_{p,T}(1+|y|^{2p})+\frac{C_{p,M,T}\delta^{p}}{\varepsilon^{p}}, \label{es11}
	\end{align}
	where we utilized the fact that $ h^{\varepsilon}\in\mathcal{A}_{M}$ and the following estimate in the second step
	\begin{align}
	C_{p,T}\delta^{p}\mathbb{E}\bigg[\Big(\operatorname*{sup}_{t\in[0,T]}|M_{t}|\Big)^{p}\bigg]\leq& C_{p,T}\delta^{\frac{p}{2}}\mathbb{E}\Big[\Big(\int_{0}^{T}|Y_{s}^{\varepsilon,h^{\varepsilon}}|^{2}\cdot\mathbb{E}|Y_s^{\varepsilon}|^2ds\Big)^{\frac{p}{2}}\Big]   \nonumber\\
	\leq&\frac{1}{2}\mathbb{E}\Big[\Big(\int_{0}^{T}|Y_{s}^{\varepsilon,h^{\varepsilon}}|^{2}ds\Big)^{p}\Big]+C_{p,T}(1+|y|^{2p}). \label{3.211111}
		\end{align}
	Note that  in view of the scale condition $\delta = o(\varepsilon)$, without loss of generality we can assume $\frac{\delta}{\varepsilon} < 1$. Thus, we derive from (\ref{es11}) that
	\begin{equation}\label{eee2.9}
		\mathbb{E}\Big[\Big(\int_0^T|Y_t^{\varepsilon,h^\varepsilon}|^2dt\Big)^p\Big]\lesssim_{p,M,T}(1+|y|^{2p}).
	\end{equation}
	On the other hand,  we can  get
	\begin{align}
&\mathbb{E}\Big[\Big(\int_{0}^{T}|Y_{s}^{\varepsilon,h^{\varepsilon}}|^{q}ds\Big)^{p}\Big]
\nonumber \\
\lesssim&_{p,T} (1+(1+\delta^{p})|y|^{2p})+{\frac{\delta^{p}}{\varepsilon^{p}}}\mathbb{E}\Big[\Big(\int_{0}^{T}|h_{s}^{2,\varepsilon}|^{2}ds\Big)^{p}\Big]+\delta^{p}\mathbb{E}\Big[\Big(\sup_{t\in[0,T]}|M_{t}|\Big)^{p}\Big]
 \nonumber	\\
\lesssim&_{p,M,T}\mathbb{E}\Big[\Big(\int_{0}^{T}|Y_{s}^{\varepsilon,h^{\varepsilon}}|^{2}ds\Big)^{p}\Big]+(1+|y|^{2p}).\label{es12}
		\end{align}
	Substituting (\ref{eee2.9}) into  (\ref{es12}) gives
	\begin{equation}\label{eeee2.9}
		\mathbb{E}\Big[\Big(\int_0^T|Y_t^{\varepsilon,h^\varepsilon}|^qdt\Big)^p\Big]\lesssim_{p,M,T}(1+|y|^{2p}).
	\end{equation}
	
	Using It\^o's formula for $|X_{t}^{\varepsilon,h^{\varepsilon}}|^{2}$, we have
	\begin{align}
		|X_{t}^{\varepsilon,h^{\varepsilon}}|^{2}
		=&|x|^{2}+2\int_{0}^{t}\langle b(X_{s}^{\varepsilon,h^{\varepsilon}},\mathscr{L}_{X_{s}^{\varepsilon}},Y_{s}^{\varepsilon,h^{\varepsilon}},\mathscr{L}_{Y_{s}^{\varepsilon}}),X_{s}^{\varepsilon,h^{\varepsilon}}\rangle ds\nonumber\\
		&+2\int_{0}^{t}\langle\sigma(X_{s}^{\varepsilon,h^{\varepsilon}},\mathscr{L}_{X_{s}^{\varepsilon}},Y_{s}^{\varepsilon,h^{\varepsilon}},\mathscr{L}_{Y_{s}^{\varepsilon}})h_{s}^{1,\varepsilon},X_{s}^{\varepsilon,h^{\varepsilon}}\rangle ds \nonumber \\
		&+\varepsilon\int_{0}^{t}\|\sigma(X_{s}^{\varepsilon,h^{\varepsilon}},\mathscr{L}_{X_{s}^{\varepsilon}},Y_{s}^{\varepsilon,h^{\varepsilon}},\mathscr{L}_{Y_{s}^{\varepsilon}})\|^{2}ds\nonumber \\
		&+2\sqrt{\varepsilon}\int_{0}^{t}\langle\sigma(X_{s}^{\varepsilon,h^{\varepsilon}},\mathscr{L}_{X_{s}^{\varepsilon}},Y_{s}^{\varepsilon,h^{\varepsilon}},\mathscr{L}_{Y_{s}^{\varepsilon}})dW_{s}^{1},X_{s}^{\varepsilon,h^{\varepsilon}}\rangle.\nonumber
	\end{align}
	It follows from the condition $(\mathbf{A}_{\text{slow}})$ that
	\begin{align}
	&\sup_{t\in[0,T]}|X_{t}^{\varepsilon,h^{\varepsilon}}|^{2}  \nonumber\\ \leq&|x|^{2}+C\int_{0}^{T}\big(1+|X_{t}^{\varepsilon,h^{\varepsilon}}|^{2}+\mathscr{L}_{X_{t}^{\varepsilon}}(|\cdot|^{2})+|Y_{t}^{\varepsilon,h^{\varepsilon}}|^{2}+|Y_{t}^{\varepsilon,h^{\varepsilon}}|^{q}+\mathscr{L}_{Y_{t}^{\varepsilon}}(|\cdot|^{2})\big)dt   \nonumber\\
	&+2\int_{0}^{T}\big|\langle\sigma(X_{t}^{\varepsilon,h^{\varepsilon}},\mathscr{L}_{X_{t}^{\varepsilon}},Y_{t}^{\varepsilon,h^{\varepsilon}},\mathscr{L}_{Y_{t}^{\varepsilon}})h_{t}^{1,\varepsilon},X_{t}^{\varepsilon,h^{\varepsilon}}\rangle\big|dt \nonumber \\
	&+2\sqrt{\varepsilon}\sup_{t\in[0,T]}\Big|\int_{0}^{t}\langle\sigma(X_{s}^{\varepsilon,h^{\varepsilon}},\mathscr{L}_{X_{s}^{\varepsilon}},Y_{s}^{\varepsilon,h^{\varepsilon}},\mathscr{L}_{Y_{t}^{\varepsilon}})dW_{s}^{1},X_{s}^{\varepsilon,h^{\varepsilon}}\rangle\Big|
 \nonumber\\
	=:& |x|^{2}+ I^{\varepsilon}(T)+II^{\varepsilon}(T)+III^{\varepsilon}(T)
	. \label{3.5}
		\end{align}
	Due to the condition $(\mathbf{H}_3)$,	the term $II^{\varepsilon}(T)$ can be estimated as follows,
\begin{align}
	II^{\varepsilon}(T)
\leq&\frac{1}{4}\sup_{t\in[0,T]}|X_{t}^{\varepsilon,h^{\varepsilon}}|^{2}+C\int_{0}^{T}\|\sigma(X_{t}^{\varepsilon,h^{\varepsilon}},\mathscr{L}_{X_{t}^{\varepsilon}},Y_{t}^{\varepsilon,h^{\varepsilon}},\mathscr{L}_{Y_{t}^{\varepsilon}})\|^{2}dt\cdot\int_{0}^{T}|h_{t}^{1,\varepsilon}|^{2}dt \nonumber \\
\leq&\frac{1}{4}\sup_{t\in[0,T]}|X_{t}^{\varepsilon,h^{\varepsilon}}|^{2}+C_{M,T}+C_{M}\bigg\{\int_{0}^{T}|X_{t}^{\varepsilon,h^{\varepsilon}}|^{2}dt +\int_{0}^{T}\mathbb{E}|X_{t}^{\varepsilon}|^{2}dt\nonumber \\
&+\int_{0}^{T}\mathbb{E}|Y_{t}^{\varepsilon}|^{2}dt+\int_{0}^{T}|Y_{t}^{\varepsilon,h^{\varepsilon}}|^{2}dt\bigg\}. \label{3.6}
\end{align}
Using Burkholder-Davis-Gundy's inequality, we get
\begin{align}
	III^{\varepsilon}(T)
\leq&8\sqrt{\varepsilon}\mathbb{E}\bigg[ \Big(\int_{0}^{T}\|\sigma(X_{t}^{\varepsilon,h^{\varepsilon}},\mathscr{L}_{X_{t}^{\varepsilon}},Y_{t}^{\varepsilon,h^{\varepsilon}},\mathscr{L}_{Y_{t}^{\varepsilon}})\|^{2}|X_{t}^{\varepsilon,h^{\varepsilon}}|^{2}dt\Big)^{\frac{1}{2}} \bigg]    \nonumber\\
\leq&\frac{1}{4}\mathbb{E}\Big[\sup_{t\in[0,T]}|X_{t}^{\varepsilon,h^{\varepsilon}}|^{2}\Big]+C_{T}+C\bigg\{\mathbb{E}\int_{0}^{T}|X_{t}^{\varepsilon,h^{\varepsilon}}|^{2}dt  \nonumber\\
&+\int_{0}^{T}\mathbb{E}|X_{t}^{\varepsilon}|^{2}dt+\int_{0}^{T}\mathbb{E}|Y_{t}^{\varepsilon}|^{2}dt+\mathbb{E}\int_{0}^{T}|Y_{t}^{\varepsilon,h^{\varepsilon}}|^{2}dt\bigg\}, \label{3.7}
\end{align}
where in the last step we used Young's inequality and (\ref{growthsigma}).
	Combining (\ref{3.5})-(\ref{3.7}) yields that
\begin{align*}
	&\mathbb{E}\Big[\sup_{t\in[0,T]}|X_{t}^{\varepsilon,h^{\varepsilon}}|^{2p}\Big]   \\
\lesssim&_{p,M,T}(1+|x|^{2p})+\mathbb{E}\bigg[\Big(\int_{0}^{T}|X_{t}^{\varepsilon,h^{\varepsilon}}|^{2}dt\Big)^p\bigg] +\Big(\int_{0}^{T}\mathbb{E}|X_{t}^{\varepsilon}|^{2}dt\Big)^p \\
&~~~+\Big(\int_{0}^{T}\mathbb{E}|Y_{t}^{\varepsilon}|^{2}dt\Big)^p+\mathbb{E}\bigg[\Big(\int_{0}^{T}|Y_{t}^{\varepsilon,h^{\varepsilon}}|^{2}dt\Big)^p\bigg] +\mathbb{E}\bigg[\Big(\int_{0}^{T}|Y_{t}^{\varepsilon,h^{\varepsilon}}|^{q}dt\Big)^p\bigg]  \\
\lesssim&_{p,M,T} (1+|x|^{2p}+|y|^{qp})+\mathbb{E}\int_{0}^{T}|X_{t}^{\varepsilon,h^{\varepsilon}}|^{2p}dt, \\
\lesssim&_{p,M,T}\big(1+|x|^{2p}+|y|^{pq}\big).
\end{align*}
where we used the estimates (\ref{esX}), (\ref{esY}), (\ref{eee2.9}) and  (\ref{eeee2.9}) in the last step, which completes the proof.
\end{proof}

The following lemma provides a time H{\"o}lder continuity estimate for the controlled process $X_{t}^{\varepsilon,h^{\varepsilon}}$, which is crucial in the proof of the existence of viable pairs later.
\begin{lemma}\label{l3.3}
For any $t_1,t_2\in[0,T]$ and $p \geq 2$,
\begin{equation}\label{3.12}
	\mathbb{E}|X_{t_1}^{\varepsilon,h^\varepsilon}-X_{t_2}^{\varepsilon,h^\varepsilon}|^p \lesssim_{x,y,p,M,T}|t_1-t_2|^{\frac{p}{2}}+|t_{1}-t_{2}|^{\frac{p(q-2)}{2q}}.
\end{equation}
\end{lemma}
\begin{proof}For reader's convenience, we recall
\begin{align} X_t^{\varepsilon,h^{\varepsilon}}=&~x+\int_{0}^{t}b(X_s^{\varepsilon,h^{\varepsilon}},\mathscr{L}_{X_s^{\varepsilon}},Y_s^{\varepsilon,h^{\varepsilon}},\mathscr{L}_{Y_s^{\varepsilon}})ds+\int_{0}^{t}\sigma(X_s^{\varepsilon,h^{\varepsilon}},\mathscr{L}_{X_s^{\varepsilon}},Y_s^{\varepsilon,h^{\varepsilon}},\mathscr{L}_{Y_s^{\varepsilon}})h_s^{1,\varepsilon}ds \nonumber \\
&+\sqrt{\varepsilon}\int_{0}^{t}\sigma(X_s^{\varepsilon,h^{\varepsilon}},\mathscr{L}_{X_s^{\varepsilon}},Y_s^{\varepsilon,h^{\varepsilon}},\mathscr{L}_{Y_s^{\varepsilon}})dW_s^1 \nonumber \\
=:&~x+\sum_{i=1}^3\mathcal{K}_i^\varepsilon(t).\label{4.1}
\end{align}
As for $\mathcal{K}_1^\varepsilon(t)$, for any $t_{1}, t_{2} \in [0,T]$, it follows from the assumptions (\ref{growthb}) and $(\mathbf{H}_1)$ that
\begin{align}
	&~~~~\mathbb{E}|\mathcal{K}_{1}^{\varepsilon}(t_{1})-\mathcal{K}_{1}^{\varepsilon}(t_{2})|^{p}  \nonumber\\
&\lesssim |t_{1}-t_{2}|^{\frac{p}{2}}\mathbb{E}\bigg[\Big(\int_{t_{2}}^{t_{1}}\big(1+|X_{s}^{\varepsilon,h^{\varepsilon}}|^{\kappa}+\mathbb{E}|X_{s}^{\varepsilon}|^{\kappa}+|Y_{s}^{\varepsilon,h^{\varepsilon}}|^{\kappa}+\mathbb{E}|Y_{s}^{\varepsilon}|^{\kappa}\big)ds\Big)^{\frac{p}{2}}\bigg] \nonumber\\
&\lesssim_{x,y,p,M,T}|t_{1}-t_{2}|^{\frac{p}{2}}, \label{4.3}
\end{align}
where we used (\ref{esX}), (\ref{esY}), ({\ref{3.1}}) and ({\ref{3.2}}) in the last step. Similarly, as for  $\mathcal{K}_2^\varepsilon(t)$, due to the assumption $(\mathbf{H}_3)$ and the fact that $h^{\varepsilon}\in\mathcal{A}_{M},$ we have
\begin{align}
	&~~~~\mathbb{E}|\mathcal{K}_{2}^{\varepsilon}(t_{1})-\mathcal{K}_{2}^{\varepsilon}(t_{2})|^{p}
 \nonumber \\
&\leq\mathbb{E}\bigg[\Big(\int_{t_{2}}^{t_{1}}\|\sigma\big(X_s^{\varepsilon,h^{\varepsilon}},\mathscr{L}_{X_s^{\varepsilon}},Y_s^{\varepsilon,h^{\varepsilon}},\mathscr{L}_{Y_s^{\varepsilon}}\big)\|^2ds\Big)^{\frac{p}{2}}\cdot \Big(\int_{t_2}^{t_1}|h_s^{1,\varepsilon}|^2ds\Big)^{\frac{p}{2}}\bigg]  \nonumber\\
& \lesssim_{x,y,p,M,T}|t_{1}-t_{2}|^{\frac{p}{2}}+\mathbb{E}\bigg[\Big(\int_{t_{2}}^{t_{1}}|Y_{s}^{\varepsilon,h^{\varepsilon}}|^{2}ds\Big)^{\frac{p}{2}}\bigg] \nonumber\\
& \lesssim_{x,y,p,M,T}|t_{1}-t_{2}|^{\frac{p}{2}}+|t_{1}-t_{2}|^{\frac{p(q-2)}{2q}}. \label{4.4}
\end{align}
By Burkholder--Davis--Gundy's inequality we also obtain
\begin{equation} \mathbb{E}|\mathcal{K}_{3}^{\varepsilon}(t_1)-\mathcal{K}_{3}^{\varepsilon}(t_2)|^{p}\lesssim_{x,y,p,M,T}\varepsilon\Big(|t_{1}-t_{2}|^{\frac{p}{2}}+|t_{1}-t_{2}|^{\frac{p(q-2)}{2q}}\Big). \label{4.5}
\end{equation}
Combining (\ref{4.3})-(\ref{4.5}) implies (\ref{3.12}) holds. We complete the proof.
\end{proof}

\vspace{5mm}
In the following lemma, we prove the tightness of the family $\{(X^{\varepsilon,h^{\varepsilon}},\mathcal{K}_{1}^{\varepsilon},\mathcal{K}_{2}^{\varepsilon},{\PPi}^{\varepsilon,\Delta})\}_{\varepsilon\in(0,1)}$
in
$$\mathcal{Z}_T:=C([0,T];\mathbb{R}^n)^{\otimes3}\times \mathscr{P}(\mathbb{R}^d\times\mathbb{R}^m\times \mathscr{P}_2(\mathbb{R}^m)\times[0,T]), $$

\vspace{2mm}
\noindent where $C([0,T];\mathbb{R}^n)^{\otimes3}:=C([0,T];\mathbb{R}^n)\times C([0,T];\mathbb{R}^n)\times C([0,T];\mathbb{R}^n)$. Moreover, let $[{\PPi}^{\varepsilon,\Delta}]_{1,2,4}$ (resp.~$[{\PPi}^{\varepsilon,\Delta}]_{3}$) denote the marginal measure of ${\PPi}^{\varepsilon,\Delta}$ on $\mathbb{R}^d\times\mathbb{R}^m\times[0,T]$ (resp.~$\mathscr{P}_2(\mathbb{R}^m)$).

\begin{lemma}\label{p4.1}
	Fix $M<\infty$. Suppose that $\{h^{\varepsilon}\}_{\varepsilon\in(0,1)}\subset\mathcal{A}_{M}$, the family $$\{(X^{\varepsilon,h^{\varepsilon}},\mathcal{K}_{1}^{\varepsilon},\mathcal{K}_{2}^{\varepsilon},{\PPi}^{\varepsilon,\Delta})\}_{\varepsilon\in(0,1)}$$  is tight in $\mathcal{Z}_T$.
\end{lemma}

\begin{proof}\noindent$\textbf{Step 1.}~( \mathbf{Tightness}~\mathbf{of}~\{(X^{\varepsilon,h^{\varepsilon}},\mathcal{K}_1^\varepsilon,\mathcal{K}_2^\varepsilon)\}_{\varepsilon\in(0,1)}$)
Recall the uniform moment estimate (\ref{3.1}), due to the criterion of tightness (cf.~\cite[Theorem 7.3]{MR1700749}), it suffices to prove that for any positive $\theta$, $\eta$, there exist constants $\delta_{0}$ such that
	\begin{equation}\label{4.2} \sup\limits_{\varepsilon\in(0,1)}\mathbb{P}\Big(\sup\limits_{t_1,t_2\in[0,T],|t_1-t_2|<\delta_0}|X_{t_1}^{\varepsilon,h^\varepsilon}-X_{t_2}^{\varepsilon,h^\varepsilon}|\geq\theta\Big)\leq\eta.
	\end{equation}
In light of Lemma \ref{l3.3} by taking $p=6$, due to the assumption $(\mathbf{H}_1)$, it is clear that (\ref{4.2}) holds for $\{X^{\varepsilon,h^{\varepsilon}}\}_{\varepsilon\in(0,1)}$ by making use of the Kolmogorov's continuity criterion. Therefore, the family  $\{X^{\varepsilon,h^{\varepsilon}}\}_{\varepsilon\in(0,1)}$  is tight.	
Similarly, following from (\ref{4.3}) and (\ref{4.4}), we can also get the tightness of  the family $\{(\mathcal{K}_1^\varepsilon,\mathcal{K}_2^\varepsilon)\}_{\varepsilon\in(0,1)}$.

\vspace{1mm}
\noindent$\textbf{Step 2.}~(  \mathbf{Tightness}~\mathbf{of}~\{{\PPi}^{\varepsilon,\Delta}\}_{\varepsilon\in(0,1)})$  In this part, we prove the tightness of the functional occupation measure $\{{\PPi}^{\varepsilon,\Delta}\}_{\varepsilon\in(0,1)}$ as $\mathscr{P}(\mathbb{R}^d\times\mathbb{R}^m\times \mathscr{P}_2(\mathbb{R}^m)\times[0,T])$-valued random variables by verifying the tightness of the marginals $[{\PPi}^{\varepsilon,\Delta}]_{1,2,4}$ and $[{\PPi}^{\varepsilon,\Delta}]_{3}$.

We first claim that the function
\begin{equation*}
	\Psi(\Pi):=\int_{\mathbb{R}^d\times\mathbb{R}^m\times[0,T]}\left[|h|^2+|y|^2\right]\Pi(dhdy dt)
\end{equation*}
is a tightness function on $\mathscr{P}(\mathbb{R}^d\times\mathbb{R}^m\times[0,T])$ by the fact that it is nonnegative and that the level set
\begin{equation*}
	\mathcal{R}_k:=\left\{\Pi\in\mathscr{P}(\mathbb{R}^d\times\mathbb{R}^m\times[0,T]):\Psi(\Pi)\leq k\right\}
\end{equation*}
is relatively compact in $\mathscr{P}(\mathbb{R}^d\times\mathbb{R}^m\times[0,T])$, for every $k < \infty$. To prove the relative compactness, we observe that
by Chebyshev's inequality
\begin{equation}
	\sup\limits_{\Pi\in\mathcal{R}_k}\Pi\Big(\Big\{(h,y,t)\in\mathbb{U}_N\times[0,T]\Big\}\Big)\leq\sup\limits_{\Pi\in\mathcal{R}_k}\frac{\Psi(\Pi)}{N^2}\leq\frac{k}{N^2}. \label{level1}
\end{equation}
Hence, $\mathcal{R}_k$ is tight and thus relatively compact as a subset of $\mathscr{P}(\mathbb{R}^d\times\mathbb{R}^m\times[0,T])$.

Since $\Psi$ is a tightness function, by Theorem A.3.17 in \cite{MR1431744} the tightness of  $\{[{\PPi}^{\varepsilon,\Delta}]_{1,2,4}\}_{\varepsilon\in(0,1)}$ holds if we can prove that $$\sup_{\varepsilon\in(0,1)}\mathbb{E}\big[\Psi([{\PPi}^{\varepsilon,\Delta}]_{1,2,4})\big]<\infty.$$ Indeed,
\begin{align}
	\sup_{\varepsilon\in(0,1)}\mathbb{E}\big[\Psi([{\PPi}^{\varepsilon,\Delta}]_{1,2,4})\big] =&\operatorname*{sup}_{\varepsilon\in(0,1)}\mathbb{E}\int_{\mathbb{R}^d\times\mathbb{R}^m\times[0,T]}\left[|h|^{2}+|y|^{2}\right][{\PPi}^{\varepsilon,\Delta}]_{1,2,4}(dhdy dt)  \nonumber\\
=&\sup_{\varepsilon\in(0,1)}\mathbb{E}\int_{0}^{T}\frac{1}{\Delta}\int_{t}^{t+\Delta}\left[|h_{s}^{\varepsilon}|^{2}+|Y_{s}^{\varepsilon,h^{\varepsilon}}|^{2}\right]dsdt \nonumber \\
\lesssim& \sup_{\varepsilon\in(0,1)}\mathbb{E}\int_{0}^{T+\Delta}\Big[|h_{s}^{\varepsilon}|^{2}+|Y_{s}^{\varepsilon,h^{\varepsilon}}|^{2}\Big]ds<\infty. \label{4.6}
\end{align}

As for the marginal $[{\PPi}^{\varepsilon,\Delta}]_{3}$,  we first claim that the function
\begin{equation*}
	\tilde{\Psi}(\rho):=\int_{\mathbb{R}^m}|y|^4\rho(dy),
\end{equation*}
is a tightness function on $\mathscr{P}_2(\mathbb{R}^m)$. Firstly, following from the same argument as in (\ref{level1}), we know that the level set
\begin{equation*}
	\tilde{\mathcal{R}}_k:=\left\{\rho\in\mathscr{P}_2(\mathbb{R}^m):\tilde{\Psi}(\rho)\leq k\right\}
\end{equation*}
is relatively compact in $\mathscr{P}(\mathbb{R}^m)$, for every $k < \infty$. Therefore, any
sequence in $\tilde{\mathcal{R}}_k $ has a weakly convergent subsequence with limit in $\mathscr{P}(\mathbb{R}^m)$.   Let $\{\rho_n\}_{n\in\mathbb{N}}\subset \tilde{\mathcal{R}}_k$ be such that $\rho_n\Rightarrow\rho^*$ for some $\rho^*\in \mathscr{P}(\mathbb{R}^m)$, where ``$\Rightarrow$" represents the weak convergence.
It remains to show that $\rho^*$ has finite second moment and the second moment of $\rho_n$ converges to that of $\rho^*$.

Applying H\"{o}lder's inequality and Fatou's lemma
\begin{equation*}
	\int_{\mathbb{R}^m}|y|^2\rho^*(dy)\leq \liminf_{n\to\infty} \int_{\mathbb{R}^m}|y|^2\rho_n(dy)\leq \sqrt{k}.
\end{equation*}
Let $N>0$. Using H\"{o}lder's inequality we have for all $\rho\in \tilde{\mathcal{R}}_k$,
$$\int_{\{y\in\mathbb{R}^m:|y|>N\}}|y|^2\rho(dy)\leq (\tilde{\Psi}(\mu))^{\frac{1}{2}}\frac{k^{\frac{1}{2}}}{N^2}\leq \frac{k}{N^2}.   $$
Therefore, using the weak convergence and  by reverse Fatou's Lemma we can get
\begin{align*}
	\limsup_{n\to\infty} \int_{\mathbb{R}^m}|y|^2\rho_n(dy)\leq&\frac{k}{N^2}+\int_{\{y\in\mathbb{R}^m:|y|\leq N\}}|y|^2\rho^*(dy)
 \\
	\leq &\frac{k}{N^2}+\int_{\mathbb{R}^m}|y|^2\rho^*(dy).
\end{align*}
Since $N$ may be arbitrarily large, we have
$$\lim_{n\to\infty}\int_{\mathbb{R}^m}|y|^2\rho_n(dy)=\int_{\mathbb{R}^m}|y|^2\rho^*(dy).$$
Thus  $\tilde{\Psi}$ is a tightness function on $\mathscr{P}_2(\mathbb{R}^m)$. Now we define a function $\Upsilon:\mathscr{P}(\mathscr{P}_2(\mathbb{R}^m))\to [0,\infty]$ by
$$\Upsilon(\Pi):=\int_{\mathscr{P}_2(\mathbb{R}^m)}\tilde{\Psi}(\rho)\Pi(\rho). $$
Then $\Upsilon$ is a tightness function on $\mathscr{P}(\mathscr{P}_2(\mathbb{R}^m))$. Thus in order to prove tightness of the marginal $[{\PPi}^{\varepsilon,\Delta}]_{3}$, it is sufficient  to show that
$$\sup_{\varepsilon\in(0,1)}\mathbb{E}[\Upsilon([{\PPi}^{\varepsilon,\Delta}]_{3})]<\infty.$$
However, this follows immediately from (\ref{esY}), since  by the definition of $\Upsilon$ and $[{\PPi}^{\varepsilon,\Delta}]_{3}$,
\begin{align*}
	\sup_{\varepsilon\in(0,1)}\mathbb{E}[\Upsilon([{\PPi}^{\varepsilon,\Delta}]_{3})]&=\sup_{\varepsilon\in(0,1)}\mathbb{E}\int_{\mathscr{P}_2(\mathbb{R}^m)}\int_{\mathbb{R}^m}|y|^4\rho(dy) [{\PPi}^{\varepsilon,\Delta}]_{3}(d\rho)
	 \\
	&=\sup_{\varepsilon\in(0,1)}\mathbb{E}\int_{0}^{T}\frac{1}{\Delta}\int_{t}^{t+\Delta}|Y_{s}^{\varepsilon}|^{4}dsdt  \\
	&\lesssim_T \sup_{\varepsilon\in(0,1)}\sup_{t\geq 0}\mathbb{E}|Y_{t}^{\varepsilon}|^{4}<\infty.
\end{align*}	
	The proof is complete.
	\end{proof}
	
Building upon Lemma \ref{p4.1} and using Prokhorov's theorem,
there exists a subsequence, still denoted by $\{\varepsilon\}$, such that
$$(X^{\varepsilon,h^{\varepsilon}},\mathcal{K}_1^\varepsilon,\mathcal{K}_2^\varepsilon,{\PPi}^{\varepsilon,\Delta})\Rightarrow(X,\mathcal{K}_1,\mathcal{K}_2,{\PPi})~~\text{in}~\mathcal{Z}_T,\quad \text{as } \varepsilon\to0.$$
Applying the Skorokhod representation theorem, there is a probability space along with random variables $((\hat{\Omega},\hat{\mathscr{F}},\hat{\mathbb{P}}),\hat{X}^{\varepsilon,h^{\varepsilon}},\hat{\mathcal{K}}_1^\varepsilon,\hat{\mathcal{K}}_2^\varepsilon,\hat{{\PPi}}^{\varepsilon,\Delta})$ such that
\begin{align}
	&(i)~~(\hat{X}^{\varepsilon,h^\varepsilon},\hat{\mathcal{K}}_1^\varepsilon,\hat{\mathcal{K}}_2^\varepsilon,\hat{{\PPi}}^{\varepsilon,\Delta})\to(\hat{X},\hat{\mathcal{K}}_1,\hat{\mathcal{K}}_2,\hat{{\PPi}})~~\text{in}~\mathcal{Z}_T~~\hat{\mathbb{P}}\text{-a.s.},~\text{as}~\varepsilon\to0;\label{444.7}
	\\
	&(ii)~~\mathscr{L}_{(X^{\varepsilon,h^\varepsilon},\mathcal{K}_1^\varepsilon,\mathcal{K}_2^\varepsilon,{\PPi}^{\varepsilon,\Delta})}|_{\mathbb{P}}=\mathscr{L}_{(\hat{X}^{\varepsilon,h^\varepsilon},\hat{\mathcal{K}}_1^\varepsilon,\hat{\mathcal{K}}_2^\varepsilon,\hat{{\PPi}}^{\varepsilon,\Delta})}|_{\hat{\mathbb{P}}},~
	\mathscr{L}_{(X,\mathcal{K}_1,\mathcal{K}_2,{\PPi})}|_{\mathbb{P}}=\mathscr{L}_{(\hat{X},\hat{\mathcal{K}}_1,\hat{\mathcal{K}}_2,\hat{{\PPi}})}|_{\hat{\mathbb{P}}},\label{444.16}
\end{align}
where we denote by $\mathscr{L}_{X}|_{\mathbb{P}}$ the law of random variable $X$ under the probability measure $\mathbb{P}$.

\subsection{Characterization of viable pairs}\label{sub5.3}
With these preparations,  in this section, we verify that the limit pair $(X,\PPi)\in \mathscr{V}_{(x,\bar{X},\Theta,\nu^{\mathscr{L}_{\bar{X}}})}$ defined in Definition \ref{d2.4}. Note that by Fatou's lemma, we have
\begin{align}
	&\mathbb{E}\int_{\mathbb{R}^d\times\mathbb{R}^m\times\mathscr{P}_2(\mathbb{R}^m)\times[0,T]}\left[|h|^{2}+|y|^{4}+M_4(\rho)\right]\mathbf{{\PPi}}(dhdyd\rho dt)  \nonumber \\
\leq&\liminf_{\varepsilon\to0}\liminf_{N\to\infty}\mathbb{E}\int_{\mathbb{R}^d\times\mathbb{R}^m\times\mathscr{P}_2(\mathbb{R}^m)\times[0,T]}\Big\{\big[|h|^{2}+|y|^{4}+M_4(\rho)\big]\wedge N\Big\}{\PPi}^{\varepsilon,\Delta}(dhdyd\rho dt)  \nonumber\\
\leq&\sup_{\varepsilon\in(0,1)}\mathbb{E}\int_{0}^{T}\frac{1}{\Delta}\int_{t}^{t+\Delta}\big[|h_{s}^{\varepsilon}|^{2}+|Y_{s}^{\varepsilon,h^{\varepsilon}}|^{4}+\mathbb{E}|Y_{s}^{\varepsilon}|^{4}\big]dsdt \nonumber\\
\lesssim& \sup_{\varepsilon\in(0,1)}\mathbb{E}\int_{0}^{T+\Delta}\big[1+|h_{s}^{\varepsilon}|^{2}+|Y_{s}^{\varepsilon,h^{\varepsilon}}|^{q}+\mathbb{E}|Y_{s}^{\varepsilon}|^{4}\big]ds<\infty, \label{e01}
\end{align}
which implies that
$$
\int_{\mathbb{R}^d \times\mathbb{R}^m\times\mathscr{P}_2(\mathbb{R}^m)\times[0,T]}\left[|h|^{2}+|y|^{4}+M_4(\rho)\right]\mathbf{{\PPi}}(dhdyd\rho dt)<\infty~~{\mathbb{P}}\text{-a.s..}
$$
Therefore, (\ref{2.15}) in Definition \ref{d2.4} holds. It remains to show that $(X,{\PPi})$ satisfies (\ref{2.16})-(\ref{2.17}).

\vspace{3mm}
\noindent\textbf{Proof of (\ref{2.18}).} We first show (\ref{2.18}), which will help us to  prove (\ref{2.16}). By the fact that ${\PPi}^{\varepsilon,\Delta}(\mathbb{R}^d\times\mathbb{R}^m\times\mathscr{P}_2(\mathbb{R}^m)\times[0,t])=t,$ along with $\mathbf{{\PPi}}(\mathbb{R}^d\times\mathbb{R}^m\times\mathscr{P}_2(\mathbb{R}^m)\times{t})=0$ and the continuity of the mapping $t \mapsto \mathbf{{\PPi}}(\mathbb{R}^d\times\mathbb{R}^m\times\mathscr{P}_2(\mathbb{R}^m)\times[0,t])$
to deal with null sets, we can conclude (\ref{2.18}) holds.\hspace{\fill}$\Box$

\vspace{3mm}
\noindent\textbf{Proof of (\ref{2.16}).} Recall the equality (\ref{4.1})
\begin{align*}
	X_t^{\varepsilon,h^{\varepsilon}}=x+\sum_{i=1}^3\mathcal{K}_i^\varepsilon(t),
\end{align*}
where
\begin{align*}
	&\mathcal{K}_1^\varepsilon(t)=\int_{0}^{t}b(X_s^{\varepsilon,h^{\varepsilon}},\mathscr{L}_{X_s^{\varepsilon}},Y_s^{\varepsilon,h^{\varepsilon}},\mathscr{L}_{Y_s^{\varepsilon}})ds, \\ &\mathcal{K}_2^\varepsilon(t)=\int_{0}^{t}\sigma(X_s^{\varepsilon,h^{\varepsilon}},\mathscr{L}_{X_s^{\varepsilon}},Y_s^{\varepsilon,h^{\varepsilon}},\mathscr{L}_{Y_s^{\varepsilon}})h_s^{1,\varepsilon}ds, \\
	&
	\mathcal{K}_3^\varepsilon(t)=\sqrt{\varepsilon}\int_{0}^{t}\sigma(X_s^{\varepsilon,h^{\varepsilon}},\mathscr{L}_{X_s^{\varepsilon}},Y_s^{\varepsilon,h^{\varepsilon}},\mathscr{L}_{Y_s^{\varepsilon}})dW_s^1 .
\end{align*}

Our next objective is to examine the convergence of the terms $\mathcal{K}_{i}^{\varepsilon}(t),i=1,2,3$. Firstly, it is straightforward that the term
$\mathcal{K}_{3}^{\varepsilon}\to 0$ in probability in $C([0,T];\mathbb{R}^{n})$, as $\varepsilon \to 0$, by making use of   the assumption $(\mathbf{H}_3)$.  Therefore, we focus on verifying the convergence of the remaining terms, where the detailed statements are presented in Lemmas $\ref{l4.1}$ and $\ref{l4.2}$.
\begin{lemma}\label{l4.1}
	Fix $t\in[0,T]$.	The following limit holds in distribution:
	\begin{equation*}
		\mathcal{K}_{2}^{\varepsilon}(t)-\int_{\mathbb{R}^d\times\mathbb{R}^m\times\mathscr{P}_2(\mathbb{R}^m)\times[0,t]}\sigma(X_{s},\mathscr{L}_{\bar{X}_{s}},y,\rho)\pi_1h{\PPi}(dhdyd\rho ds)\xrightarrow{\varepsilon\to0}0.
	\end{equation*}
\end{lemma}
\begin{proof}
	Firstly, notice that we have the following composition
	\begin{align*}
		&\int_{0}^{t}\sigma(X_{s}^{\varepsilon,h^{s}},\mathscr{L}_{X_{s}^{\varepsilon}},Y_{s}^{\varepsilon,h^{\varepsilon}},\mathscr{L}_{Y_{s}^{\varepsilon}})h_{s}^{1,\varepsilon}ds
  \\
&-\int_{\mathbb{R}^d\times\mathbb{R}^m\times\mathscr{P}_2(\mathbb{R}^m)\times[0,t]}\sigma(X_{s},\mathscr{L}_{\bar{X}_{s}},y,\rho)h^{1}{\PPi}(dhdyd\rho ds)  \\[1mm]
		=&\bigg(\int_{0}^{t}\sigma(X_{s}^{\varepsilon,h^{\varepsilon}},\mathscr{L}_{X_{s}^{\varepsilon}},Y_{s}^{\varepsilon,h^{\varepsilon}},\mathscr{L}_{Y_{s}^{\varepsilon}})h_{s}^{1,\varepsilon}ds \\
		&-\int_{\mathbb{R}^d\times\mathbb{R}^m\times\mathscr{P}_2(\mathbb{R}^m)\times[0,t]}\sigma(X_s^{\varepsilon,h^{\varepsilon}},\mathscr{L}_{X_s^\varepsilon},y,\rho)h^1{\PPi}^{\varepsilon,\Delta}(dhdyd\rho ds)\bigg)  \\
		&+\Bigg(\int_{\mathbb{R}^d\times\mathbb{R}^m\times\mathscr{P}_2(\mathbb{R}^m)\times[0,t]}\sigma(X_{s}^{\varepsilon,h^{\varepsilon}},\mathscr{L}_{X_{s}^{\varepsilon}},y,\rho)h^{1}{\PPi}^{\varepsilon,\Delta}(dhdyd\rho ds)  \\
		&-\int_{\mathbb{R}^d\times\mathbb{R}^m\times\mathscr{P}_2(\mathbb{R}^m)\times[0,t]}\sigma(X_s,\mathscr{L}_{\bar{X}_s},y,\rho)h^1{\PPi}(dhdyd\rho ds)\bigg)  \\
		=:&\sum_{i=1}^{2}\mathcal{O}_{i}^{\varepsilon}(t). r
	\end{align*}
	\noindent$\textbf{Step 1.}$ In this part, we estimate the term $\mathcal{O}_1^\varepsilon(t)$. By (\ref{2.12}), we observe that
	\begin{align*}
		&\int_{\mathbb{R}^d\times\mathbb{R}^m\times\mathscr{P}_2(\mathbb{R}^m)\times[0,t]}\sigma(X_s^{\varepsilon,h^\varepsilon},\mathscr{L}_{X_s^\varepsilon},y,\rho)h^1{\PPi}^{\varepsilon,\Delta}(dhdyd\rho ds)  \\
		=& \Bigg(\int_{0}^{t}\frac{1}{\Delta}\int_{s}^{s+\Delta}\sigma(X_{s}^{\varepsilon,h^{\varepsilon}},\mathscr{L}_{X_{s}^{\varepsilon}},Y_{r}^{\varepsilon,h^{\varepsilon}},\mathscr{L}_{Y_{r}^{\varepsilon}})h_{r}^{1,\varepsilon}drds \\
		&-\int_0^t\frac1\Delta\int_s^{s+\Delta}\sigma(X_r^{\varepsilon,h^\varepsilon},\mathscr{L}_{X_r^\varepsilon},Y_r^{\varepsilon,h^\varepsilon},\mathscr{L}_{Y_{r}^{\varepsilon}})h_r^{1,\varepsilon}drds\Bigg)  \\
		&+\int_0^t\frac1\Delta\int_s^{s+\Delta}\sigma(X_r^{\varepsilon,h^\varepsilon},\mathscr{L}_{X_r^\varepsilon},Y_r^{\varepsilon,h^\varepsilon},\mathscr{L}_{Y_{r}^{\varepsilon}})h_r^{1,\varepsilon}drds \\[1mm]
		=:&I^\varepsilon(t)+II^\varepsilon(t).
	\end{align*}
	
	As for $I^\varepsilon(t)$, due to the local Lipschitz continuity (\ref{continuityofb}) of $\sigma$ and the estimates (\ref{esX}) and (\ref{esY}), we can deduce that
	\begin{align}
		&\mathbb{E}\Big[\sup_{t\in[0,T]}|I^\varepsilon(t)|^2\Big]
	\nonumber\\		\lesssim&\frac{1}{\Delta^{2}}\mathbb{E}\Bigg\{\int_{0}^{T}\int_{s}^{s+\Delta}\Big[\big(1+|X_{s}^{\varepsilon,h^{\varepsilon}}|^{\kappa}+|X_{r}^{\varepsilon,h^{\varepsilon}}|^{\kappa}+|Y_{r}^{\varepsilon,h^{\varepsilon}}|^{\kappa}
	\nonumber\\
&+\mathbb{E}|X_{s}^{\varepsilon}|^{\kappa}+\mathbb{E}|X_{r}^{\varepsilon}|^{\kappa}+\mathbb{E}|Y_{r}^{\varepsilon}|^{\kappa}\big)\mathbb{E}|X_{s}^{\varepsilon}-X_{r}^{\varepsilon}|^{2} \nonumber\\
	&+\big(1+\mathbb{E}|X_{s}^{\varepsilon}|^{\kappa}+\mathbb{E}|X_{r}^{\varepsilon}|^{\kappa}+\mathbb{E}|Y_{r}^{\varepsilon}|^{\kappa}\big)|X_{s}^{\varepsilon,h^{\varepsilon}}-X_{r}^{\varepsilon,h^{\varepsilon}}|^{2}\Big]drds\cdot\int_{0}^{T}\int_{s}^{s+\Delta}|h_{r}^{1,\varepsilon}|^{2}drds\Bigg\}
\nonumber	\\
	\lesssim&_{M}\frac{1}{\Delta}\int_{0}^{T}\int_{s}^{s+\Delta}\Big[\big(1+\mathbb{E}|X_{s}^{\varepsilon,h^{\varepsilon}}|^{\kappa}+\mathbb{E}|X_{r}^{\varepsilon,h^{\varepsilon}}|^{\kappa}\big)
	\mathbb{E}|X_{s}^{\varepsilon}-X_{r}^{\varepsilon}|^{2}
\nonumber	\\
	&+\mathbb{E}|X_{s}^{\varepsilon,h^{\varepsilon}}-X_{r}^{\varepsilon,h^{\varepsilon}}|^{2}\Big]drds+\frac{1}{\Delta}\mathbb{E}\Bigg\{ \int_{0}^{T}\int_{s}^{s+\Delta}|Y_{r}^{\varepsilon,h^{\varepsilon}}|^{\kappa}\mathbb{E}|X_{s}^{\varepsilon}-X_{r}^{\varepsilon}|^{2}drds\Bigg\}
\nonumber	\\
	\lesssim&_{x,y,M,T}\frac{1}{\Delta}\int_{0}^{T}\int_{s}^{s+\Delta}\big(|s-r|^2+|s-r|+|s-r|^{\frac{q-2}{q}}\big)drds
\nonumber\\
&+\frac{1}{\Delta}\mathbb{E} \int_{0}^{T}\int_{s}^{s+\Delta}|Y_{r}^{\varepsilon,h^{\varepsilon}}|^{\kappa}|s-r|drds
\nonumber	\\
	\lesssim&_{x,y,M,T}\Delta+\Delta^{\frac{q-2}{q}}+\mathbb{E} \int_{0}^{T}\int_{s}^{s+\Delta}|Y_{r}^{\varepsilon,h^{\varepsilon}}|^{\kappa}drds
\nonumber	\\
	\lesssim&_{x,y,M,T}\Delta+\Delta^{\frac{q-2}{q}}\to0,~~\mathrm{as}~\varepsilon\to0, \label{4.9}
		\end{align}
	where we used  the estimates (\ref{F3.10}), (\ref{3.1}) and (\ref{3.12}) in the third step,  the estimates (\ref{3.2}) in the last step. In addition, the second step in (\ref{4.9}) (analogously for the last inequality) is due to the fact that for any $t \in [0, T],$
	\begin{align*}
		&\int_{0}^{t}\int_{s}^{s+\Delta}|h_{r}^{1,\varepsilon}|^{2}drds  \\=&\int_{0}^{\Delta}\int_{0}^{r}|h_{r}^{1,\varepsilon}|^{2}dsdr+\int_{\Delta}^{t}\int_{r-\Delta}^{r}|h_{r}^{1,\varepsilon}|^{2}dsdr +\int_t^{t+\Delta}\int_{r-\Delta}^t|h_r^{1,\varepsilon}|^2dsdr \\
		=&\int_{0}^{\Delta}|h_{r}^{1,\varepsilon}|^{2}rdr+\Delta\int_{\Delta}^{t}|h_{r}^{1,\varepsilon}|^{2}dr+\int_{t}^{t+\Delta}(t-r+\Delta)|h_{r}^{1,\varepsilon}|^{2}dr  \\
		\leq&3\Delta\int_0^T|h_r^{1,\varepsilon}|^2dr.
	\end{align*}
	
	On the other hand, note that $II^\varepsilon(t)$ has the composition
	\begin{align*}
		II^{\varepsilon}(t) =&\int_{0}^{\Delta}\frac{1}{\Delta}\int_{0}^{r}\sigma(X_{r}^{\varepsilon,h^{\varepsilon}},\mathscr{L}_{X_{r}^{\varepsilon}},Y_{r}^{\varepsilon,h^{\varepsilon}},\mathscr{L}_{Y_{r}^{\varepsilon}})h_{r}^{1,\varepsilon}dsdr   \\
		&+\int_{\Delta}^{t}\frac{1}{\Delta}\int_{r-\Delta}^{r}\sigma(X_{r}^{\varepsilon,h^{\varepsilon}},\mathscr{L}_{X_{r}^{\varepsilon}},Y_{r}^{\varepsilon,h^{\varepsilon}},\mathscr{L}_{Y_{r}^{\varepsilon}})h_{r}^{1,\varepsilon}dsdr  \\
		&+\int_t^{t+\Delta}\frac{1}{\Delta}\int_{r-\Delta}^t\sigma(X_r^{\varepsilon,h^\varepsilon},\mathscr{L}_{X_r^\varepsilon},Y_r^{\varepsilon,h^\varepsilon},\mathscr{L}_{Y_{r}^{\varepsilon}})h_r^{1,\varepsilon}dsdr  \\
		=:&\sum_{i=1}^{3}II^\varepsilon_{i}(t).
	\end{align*}

As for the term $II^\varepsilon_{1}(t)$, due to (\ref{growthsigma}) and $h_{r}^{1,\varepsilon} \in \mathcal{A}_M$, it follows  that
\begin{align}
	\mathbb{E}\Big[\sup_{t\in[0,T]}|II^\varepsilon_{1}(t)|^2\Big]
\leq & \mathbb{E}\Big|\frac{1}{\Delta}\int_{0}^{\Delta}r\sigma(X_{r}^{\varepsilon,h^{\varepsilon}},\mathscr{L}_{X_{r}^{\varepsilon}},Y_{r}^{\varepsilon,h^{\varepsilon}},\mathscr{L}_{Y_{r}^{\varepsilon}})h_{r}^{1,\varepsilon}dr\Big|^2 \nonumber\\
\leq & \mathbb{E}\Big[\int_{0}^{\Delta}\|\sigma(X_{r}^{\varepsilon,h^{\varepsilon}},\mathscr{L}_{X_{r}^{\varepsilon}},Y_{r}^{\varepsilon,h^{\varepsilon}},\mathscr{L}_{Y_{r}^{\varepsilon}})\|^2dr\int_{0}^{\Delta}|h_{r}^{1,\varepsilon}|^2dr\Big] \nonumber\\
\lesssim&_{M} \mathbb{E}\int_{0}^{\Delta}\big(1+|X_{r}^{\varepsilon,h^{\varepsilon}}|^{2}+\mathbb{E}|X_{r}^{\varepsilon}|^{2}+|Y_{r}^{\varepsilon}|^{2}+\mathbb{E}|Y_{r}^{\varepsilon,h^{\varepsilon}}|^{2}\big)dr \nonumber\\
\lesssim&_{x,y,M,T}\Delta+\mathbb{E}\Big[\Delta^{\frac{q-2}{q}}\Big(\int_{0}^{\Delta}|Y_{r}^{\varepsilon,h^{\varepsilon}}|^{q}dr\Big)^{\frac{2}{q}}\Big]
\nonumber\\
\lesssim&_{x,y,M,T}\Delta+{\Delta}^{\frac{q-2}{q}}\to0,~~\text{as}~\varepsilon\to0, \label{4.100}
\end{align}
where we used (\ref{esX}), (\ref{esY}), ({\ref{3.1}}) in the fifth step, and ({\ref{3.2}}) in the last step. Similarly, as for 	$II^\varepsilon_{3}(t)$, we have
\begin{align}
&\mathbb{E}\Big[\sup_{t\in[0,T]}|II^\varepsilon_{3}(t)|^2\Big]\nonumber\\
&\leq \mathbb{E}\Big[\sup_{t\in[0,T]}\Big|\frac{1}{\Delta}\int_{t}^{t+\Delta}(t-r+\Delta)\sigma(X_{r}^{\varepsilon,h^{\varepsilon}},\mathscr{L}_{X_{r}^{\varepsilon}},Y_{r}^{\varepsilon,h^{\varepsilon}},\mathscr{L}_{Y_{r}^{\varepsilon}})h_{r}^{1,\varepsilon}dr\Big|^2\Big]\nonumber \\
&\leq \mathbb{E}\Big[\sup_{t\in[0,T]}\Big(\int_{t}^{t+\Delta}\|\sigma(X_{r}^{\varepsilon,h^{\varepsilon}},\mathscr{L}_{X_{r}^{\varepsilon}},Y_{r}^{\varepsilon,h^{\varepsilon}},\mathscr{L}_{Y_{r}^{\varepsilon}})\|^2dr\int_{t}^{t+\Delta}|h_{r}^{1,\varepsilon}|^2dr\Big)\Big] \nonumber\\
&\lesssim_{M}\Delta+\mathbb{E}\Big[\Delta^{\frac{q-2}{q}}\Big(\int_{0}^{T+\Delta}|Y_{r}^{\varepsilon,h^{\varepsilon}}|^{q}dr\Big)^{\frac{2}{q}}\Big]
\nonumber\\
&\lesssim_{x,y,M,T}\Delta+{\Delta}^{\frac{q-2}{q}}\to0,~~\text{as}~\varepsilon\to0. \label{4.1000}
\end{align}
Consequently, concerning the term $\mathcal{O}_1^\varepsilon(t)$, we infer that
\begin{align*}
	\mathbb{E}\Big[\sup_{t\in[0,T]}|\mathcal{O}_{1}^{\varepsilon}(t)|^{2}\Big]
	\lesssim\,\,& \mathbb{E}\Big|\int_{0}^{\Delta}\sigma(X_{r}^{\varepsilon,h^{\varepsilon}},\mathscr{L}_{X_{r}^{\varepsilon}},Y_{r}^{\varepsilon,h^{\varepsilon}},\mathscr{L}_{Y_{r}^{\varepsilon}})h_{r}^{1,\varepsilon}dr\Big|^{2}   \\	&+\mathbb{E}\Bigg[\sup_{t\in[0,T]}\left|\int_{\Delta}^{t}\sigma(X_{r}^{\varepsilon,h},\mathscr{L}_{X_{r}^{\varepsilon}},Y_{r}^{\varepsilon,h^{\varepsilon}},\mathscr{L}_{Y_{r}^{\varepsilon}})h_{r}^{1,\varepsilon}dr-II^{\varepsilon}_{2}(t)\right|^{2} \Bigg]  \\ &+\mathbb{E}\Big[\sup_{t\in[0,T]}|I^{\varepsilon}(t)|^{2}\Big]+\mathbb{E}\Big[\sup_{t\in[0,T]}|II^\varepsilon_{1}(t)|^{2}\Big]+\mathbb{E}\Big[\sup_{t\in[0,T]}|II^\varepsilon_{3}(t)|^{2}\Big].
\end{align*}
Thus, in view of (\ref{growthsigma}) and the fact that $h_{r}^{1,\varepsilon} \in \mathcal{A}_M$, as well as (\ref{4.9}), (\ref{4.100}), (\ref{4.1000}) and the definition of $II^\varepsilon_{2}(t)$, it follows that
\begin{equation}
	\mathbb{E}\Big[\sup_{t\in[0,T]}|\mathcal{O}_1^\varepsilon(t)|^2\Big]\to0,\quad\text{as}~\varepsilon\to0. \label{4.11}
\end{equation}
\noindent$\textbf{Step 2.}$	Now we turn to study the term $\mathcal{O}_2^\varepsilon(t)$, we denote
\begin{align*} &\mathscr{J}^{\varepsilon}(t):=\int_{\mathbb{R}^d\times\mathbb{R}^m\times\mathscr{P}_2(\mathbb{R}^m)\times[0,t]}\sigma(X_{s}^{\varepsilon,h^{\varepsilon}},\mathscr{L}_{X_{s}^{\varepsilon}},y,\rho)\pi_1h{\PPi}^{\varepsilon,\Delta}(dhdyd\rho ds),  \\
	&\mathscr{J}(t):=\int_{\mathbb{R}^d\times\mathbb{R}^m\times\mathscr{P}_2(\mathbb{R}^m)\times[0,t]}\sigma(X_s,\mathscr{L}_{\bar{X}_s},y,\rho)\pi_1h{\PPi}(dhdyd\rho ds).
\end{align*}
We shall establish the following convergence in distribution
\begin{equation}
	\mathcal{O}_2^\varepsilon(t)=\mathscr{J}^{\varepsilon}(t)-\mathscr{J}(t)\xrightarrow{\varepsilon\to0}0. \label{4.19}
\end{equation}
For any  bounded Lipschitz continuous function $f$, we note that
\begin{align}\label{Skoro}
	\mathbb{E}\big[f(\mathscr{J}^{\varepsilon}(t))-f(\mathscr{J}(t))\big]=\hat{\mathbb{E}}\big[f(\hat{\mathscr{J}}^{\varepsilon}(t))-f(\hat{\mathscr{J}}(t))\big]\leq \hat{\mathbb{E}}\big|\hat{\mathscr{J}}^{\varepsilon}(t)-\hat{\mathscr{J}}(t)\big|,
\end{align}
where
\begin{align*} &\hat{\mathscr{J}}^{\varepsilon}(t):=\int_{\mathbb{R}^d\times\mathbb{R}^m\times\mathscr{P}_2(\mathbb{R}^m)\times[0,t]}\sigma(\hat{X}_{s}^{\varepsilon,h^{\varepsilon}},\mathscr{L}_{X_{s}^{\varepsilon}},y,\rho)\pi_1h\hat{{\PPi}}^{\varepsilon,\Delta}(dhdyd\rho ds),  \\ &\hat{\mathscr{J}}(t):=\int_{\mathbb{R}^d\times\mathbb{R}^m\times\mathscr{P}_2(\mathbb{R}^m)\times[0,t]}\sigma(\hat{X}_s,\mathscr{L}_{\bar{X}_s},y,\rho)\pi_1h\hat{{\PPi}}(dhdyd\rho ds),
\end{align*}
and we used the property (\ref{444.16}) in the first step.

To establish (\ref{4.19}), it is sufficient to demonstrate that the right-hand side of (\ref{Skoro}) converges to zero as $\varepsilon\to 0$. Regarding the term $\hat{\mathscr{J}}^{\varepsilon}(t)-\hat{\mathscr{J}}(t)$, we observe the following decomposition
\begin{align*}
	&\hat{\mathscr{J}}^{\varepsilon}(t)-\hat{\mathscr{J}}(t)
	\\ =&\Bigg(\int_{\mathbb{R}^d\times\mathbb{R}^m\times\mathscr{P}_2(\mathbb{R}^m)\times[0,t]}\sigma(\hat{X}_{s}^{\varepsilon,h^{\varepsilon}},\mathscr{L}_{X_{s}^{\varepsilon}},y,\rho)\pi_1h\hat{{\PPi}}^{\varepsilon,\Delta}(dhdyd\rho ds) \\ &-\int_{\mathbb{R}^d\times\mathbb{R}^m\times\mathscr{P}_2(\mathbb{R}^m)\times[0,t]}\sigma(\hat{X}_{s},\mathscr{L}_{\bar{X}_{s}},y,\rho)\pi_1h\hat{{\PPi}}^{\varepsilon,\Delta}(dhdyd\rho ds)\bigg) \\ &+\Bigg(\int_{\mathbb{R}^d\times\mathbb{R}^m\times\mathscr{P}_2(\mathbb{R}^m)\times[0,t]}\sigma(\hat{X}_{s},\mathscr{L}_{\bar{X}_{s}},y,\rho)\pi_1h\hat{{\PPi}}^{\varepsilon,\Delta}(dhdyd\rho ds) \\
	&-\int_{\mathbb{R}^d\times\mathbb{R}^m\times\mathscr{P}_2(\mathbb{R}^m)\times[0,t]}\sigma(\hat{X}_{s},\mathscr{L}_{\bar{X}_{s}},y,\rho)\pi_1h\hat{{\PPi}}(dhdyd\rho ds)\bigg) \\
	=:&\sum_{i=1}^{2}\hat{\mathscr{J}}_{i}^{\varepsilon}(t).
\end{align*}
Concerning the term $\hat{\mathscr{J}}_{1}^{\varepsilon}(t)$, by using similar argument as in the proof of  (\ref{e01})  and the estimates (\ref{esX}), (\ref{esY}) and (\ref{3.2}), we derive the following estimates
\begin{align*}
	&\hat{\mathbb{E}}\Big[\sup_{t\in[0,T]}|\hat{\mathscr{J}}_{1}^{\varepsilon}(t)|^2\Big] \\
	\leq&\hat{\mathbb{E}}\Bigg[\operatorname*{sup}_{t\in[0,T]}\Big|\int_{\mathbb{R}^d\times\mathbb{R}^m\times\mathscr{P}_2(\mathbb{R}^m)\times[0,t]}\big(\sigma(\hat{X}_s^{\varepsilon,h^{\varepsilon}},\mathscr{L}_{X_s^{\varepsilon}},y,\rho)
	\\
	&-\sigma(\hat{X}_s,\mathscr{L}_{\bar{X}_{s}},y,\rho)\big)\pi_1h\hat{{\PPi}}^{\varepsilon,\Delta}(dhdyd\rho ds)\Big|^2 \Bigg] \\
	\leq&\hat{\mathbb{E}}\int_{\mathbb{R}^d\times\mathbb{R}^m\times\mathscr{P}_2(\mathbb{R}^m)\times[0,T]}\|\sigma(\hat{X}_s^{\varepsilon,h^{\varepsilon}},\mathscr{L}_{X_s^{\varepsilon}},y,\rho)-\sigma(\hat{X}_s,\mathscr{L}_{\bar{X}_{s}},y,\rho)\|^{2}\hat{{\PPi}}^{\varepsilon,\Delta}(dhdyd\rho ds) \\
	&\cdot\hat{\mathbb{E}}\int_{\mathbb{R}^d\times\mathbb{R}^m\times\mathscr{P}_2(\mathbb{R}^m)\times[0,T]}|h|^2\hat{{\PPi}}^{\varepsilon,\Delta}(dhdyd\rho ds) \\
	\lesssim&_{M,T} \hat{\mathbb{E}}\Bigg[\sup_{t\in[0,T]}|\hat{X}_{t}^{\varepsilon,h^{\varepsilon}}-\hat{X}_{t}|^{2}
	\\
	&\cdot\int_{\mathbb{R}^d\times\mathbb{R}^m\times\mathscr{P}_2(\mathbb{R}^m)\times[0,T]}\big(\mathbb{E}|X^{\varepsilon}_s|^{\kappa}+\mathbb{E}|\bar{X}_{s}|^{\kappa}+M_{\kappa}(\rho)\big)\hat{{\PPi}}^{\varepsilon,\Delta}(dhdyd\rho ds)\Bigg]
	\\
	&+\sup_{t\in[0,T]}\mathbb{E}|X_{t}^{\varepsilon}-\bar{X}_{t}|^{2}\cdot\hat{\mathbb{E}}\int_{\mathbb{R}^d\times\mathbb{R}^m\times\mathscr{P}_2(\mathbb{R}^m)\times[0,T]}\big(|\hat{X}_{s}^{\varepsilon,h^{\varepsilon}}|^\kappa+\mathbb{E}|\hat{X}_s|^{\kappa}
	\\[1mm]
	&+\mathbb{E}|X^{\varepsilon}_s|^{\kappa}+|\bar{X}_{s}|^{\kappa}+|y|^{\kappa}+M_{\kappa}(\rho)\big)\hat{{\PPi}}^{\varepsilon,\Delta}(dhdyd\rho ds)
	\\[1mm]
	\lesssim&_{x,y,M,T} \hat{\mathbb{E}}\Big[\sup_{t\in[0,T]}|\hat{X}_{t}^{\varepsilon,h^{\varepsilon}}-\hat{X}_{t}|^{2}\Big]+\sup_{t\in[0,T]}\mathbb{E}|X_{t}^{\varepsilon}-\bar{X}_{t}|^{2}.
\end{align*}
Consequently, utilizing the estimate (\ref{3.1}) together with the convergence results (\ref{t2.155555}) and (\ref{444.7}), we infer that
\begin{equation*}
	\hat{\mathbb{E}}\Big[\sup\limits_{t\in[0,T]}|\hat{\mathscr{J}}_{1}^{\varepsilon}(t)|^2\Big]\to0,\quad\mathrm{as}~\varepsilon\to0.
\end{equation*}

Once we can prove
\begin{equation}
	\hat{\mathbb{E}}\Big[\sup\limits_{t\in[0,T]}|\hat{\mathscr{J}}_{2}^{\varepsilon}(t)|\Big]\to0,\quad\mathrm{as}~\varepsilon\to 0,\label{4.20}
\end{equation}
then  the convergence (\ref{4.19}) follows.

\vspace{2mm}
\noindent$\textbf{Step 3.}$ Now, we turn to prove the convergence (\ref{4.20}).
Let us denote the integrand
\begin{equation*}
	\Psi(s, y, \rho, h) := \sigma(\hat{X}_{s},\mathscr{L}_{\bar{X}_{s}},y,\rho)\pi_1 h.
\end{equation*}
Since $\Psi$ is unbounded w.r.t.~$y$, $\rho$ and $h$, we employ a cut-off argument to deal with the convergence (\ref{4.20}). Define
$$\Psi_R(s, y, \rho, h):=\Psi(s, y, \rho, h) \chi_R(\Psi(s, y, \rho, h) ).$$
Here $\chi_R\in C^{\infty}_c(\mathbb{R}^n)$ is a cut-off function defined by
$$\chi_R(r)=\begin{cases} 1,~~~~|r|\leq R,&\quad\\
	0,~~~~|r|>R.&\quad\end{cases}$$
	We decompose the term $\hat{\mathscr{J}}_{2}^{\varepsilon}(t)$ into three parts:
	\begin{align}
	&\hat{\mathbb{E}}\Big[\sup\limits_{t\in[0,T]}|\hat{\mathscr{J}}_{2}^{\varepsilon}(t)|\Big] \nonumber \\
\leq&  \hat{\mathbb{E}}\int_{\mathbb{R}^d\times\mathbb{R}^m\times\mathscr{P}_2(\mathbb{R}^m)\times[0,T]} |\Psi - \Psi_R|(s, y, \rho, h) \hat{\PPi}^{\varepsilon,\Delta}(dhdyd\rho ds)  \nonumber \\
&+ \hat{\mathbb{E}}\bigg[\sup_{t\in[0,T]}\Big|\int_{\mathbb{R}^d\times\mathbb{R}^m\times\mathscr{P}_2(\mathbb{R}^m)\times[0,t]} \Psi_R(s, y, \rho, h) \hat{\PPi}^{\varepsilon,\Delta}(dhdyd\rho ds)  \nonumber \\
&- \int_{\mathbb{R}^d\times\mathbb{R}^m\times\mathscr{P}_2(\mathbb{R}^m)\times[0,t]} \Psi_R(s, y, \rho, h)  \hat{\PPi}(dhdyd\rho ds)\Big|\bigg]  \nonumber \\
&+ \hat{\mathbb{E}}\int_{\mathbb{R}^d\times\mathbb{R}^m\times\mathscr{P}_2(\mathbb{R}^m)\times[0,T]} |\Psi - \Psi_R|(s, y, \rho, h) \hat{\PPi}(dhdyd\rho ds)  \nonumber \\[1mm]
=:& \ I_1^{\varepsilon}(R) + I_2^{\varepsilon}(R) + I_3^{\varepsilon}(R). \label{e17}
	\end{align}
	To estimate the term $I_1^{\varepsilon}(R)$, we observe that $|\Psi - \Psi_R| \leq |\Psi| \mathbf{1}_{\{|\Psi| > R\}}$. Applying H\"{o}lder's inequality and using the assumption (\ref{growthsigma}), we have
	\begin{align*}
	I_1^{\varepsilon}(R) \leq& \hat{\mathbb{E}}\int_{\mathbb{R}^d\times\mathbb{R}^m\times\mathscr{P}_2(\mathbb{R}^m)\times[0,T]} |\Psi(s, y, \rho, h)| \mathbf{1}_{\{|\Psi(s, y, \rho, h)|> R\}} \hat{\PPi}^{\varepsilon,\Delta}(dhdyd\rho ds)
	\\
	\leq& \frac{1}{R^{1/4}} \hat{\mathbb{E}}\int_{\mathbb{R}^d\times\mathbb{R}^m\times\mathscr{P}_2(\mathbb{R}^m)\times[0,T]} |\Psi (s, y, \rho, h)|^{\frac{5}{4}} \hat{\PPi}^{\varepsilon,\Delta}(dhdyd\rho ds) \\ \lesssim&_T\frac{1}{R^{1/4}}\bigg(\hat{\mathbb{E}}\int_{\mathbb{R}^d\times\mathbb{R}^m\times\mathscr{P}_2(\mathbb{R}^m)\times[0,T]}\|\sigma(\hat{X}_{s},\mathscr{L}_{\bar{X}_{s}},y,\rho)\|^{\frac{10}{3}}\hat{{\PPi}}^{\varepsilon,\Delta}(dhdyd\rho ds)\bigg)^{\frac{3}{8}}\\
	&\cdot\bigg(\hat{\mathbb{E}}\int_{\mathbb{R}^d\times\mathbb{R}^m\times\mathscr{P}_2(\mathbb{R}^m)\times[0,T]}|h|^{2}\hat{{\PPi}}^{\varepsilon,\Delta}(dhdyd\rho ds)\bigg)^{\frac{5}{8}}\\
	\lesssim&_{M,T}\frac{1}{R^{1/4}}\bigg(\hat{\mathbb{E}}\int_{\mathbb{R}^d\times\mathbb{R}^m\times\mathscr{P}_2(\mathbb{R}^m)\times[0,T]}\big(1+|\hat{X}_{s}|^q+|\bar{X}_{s}|^{q}
	\\
	&+|y|^q+M_q(\rho)\big) \hat{{\PPi}}^{\varepsilon,\Delta}(dhdyd\rho ds)\bigg)^{\frac{3}{8}}.
	\end{align*}
	By (\ref{444.16}) and the uniform estimates established in Lemmas \ref{lem:uniformestimateofXY} and \ref{ll3.1}, we can get
	\begin{equation}\label{e14}
	\sup_{\varepsilon \in (0,1)} I_1^{\varepsilon}(R) \lesssim_{x,y,M,T} \frac{1}{R^{1/4}}.
	\end{equation}
	Similarly,  by (\ref{e01}) we obtain
	\begin{equation}\label{e15}
	\sup_{\varepsilon \in (0,1)} I_3^{\varepsilon}(R) \lesssim_{x,y,M,T} \frac{1}{R^{1/4}}.
	\end{equation}
	As for the term $I_2^{\varepsilon}(R)$, since $\Psi_R$ is bounded and continuous, by the convergence (\ref{444.7}) and using Vitali's convergence theorem, it follows that for any $R>0$,
	\begin{equation}\label{e16}
	\lim_{\varepsilon \to 0} I_2^{\varepsilon}(R) = 0.
	\end{equation}
	Consequently,	combining (\ref{e14})-(\ref{e16}) and taking $\varepsilon \to 0$ first and then $R\to \infty$ in (\ref{e17}), we conclude that (\ref{4.20}) holds. The proof is complete.
\end{proof}

\vspace{3mm}
Before presenting the convergence of $\mathcal{K}_{1}^{\varepsilon}(t)$, we first recall the controlled fast process $Y_t^{\varepsilon,h^{\varepsilon}}$ satisfying the equation
\begin{align*} dY_t^{\varepsilon,h^{\varepsilon}}=&\frac{1}{\delta}f(\mathscr{L}_{X^{\varepsilon}_t},Y_t^{\varepsilon,h^{\varepsilon}},\mathscr{L}_{Y^{\varepsilon}_t})dt+\frac{1}{\sqrt{\delta\varepsilon}}g(\mathscr{L}_{X^{\varepsilon}_t},Y_t^{\varepsilon,h^{\varepsilon}},\mathscr{L}_{Y^{\varepsilon}_t})h_t^{2,\varepsilon}dt \\&+\frac{1}{\sqrt{\delta}}g(\mathscr{L}_{X^{\varepsilon}_t},Y_t^{\varepsilon,h^{\varepsilon}},\mathscr{L}_{Y^{\varepsilon}_t})dW_t^2,~~~~~Y_0^{\varepsilon,h^{\varepsilon}}=y,
\end{align*}
and the uncontrolled fast process
\begin{align*} dY_t^{\varepsilon}=\frac{1}{\delta}f(\mathscr{L}_{X^{\varepsilon}_t},Y_t^{\varepsilon},\mathscr{L}_{Y^{\varepsilon}_t})dt+\frac{1}{\sqrt{\delta}}g(\mathscr{L}_{X^{\varepsilon}_t},Y_t^{\varepsilon},\mathscr{L}_{Y^{\varepsilon}_t})dW_t^2,~~Y_0^{\varepsilon}=y.
\end{align*}
The following lemma establishes the convergence  of the difference between $Y_t^{\varepsilon,h^{\varepsilon}}$ and $Y_t^{\varepsilon}$ in $L^2$.
\begin{lemma}\label{l4.3}
	Let $M>0$, $\{h^{\varepsilon}\}_{\varepsilon\in(0,1)}\subset\mathcal{A}_{M}$ and $\Delta$ be as in Definition $(\ref{2.11})$. Then
	\begin{equation}\label{4.17}
		\dfrac{1}{\Delta}\mathbb{E}\int_0^T|Y_t^{\varepsilon,h^\varepsilon}-Y_t^{\varepsilon}|^2dt\to0,~~as~\varepsilon\to0.
	\end{equation}
\end{lemma}
\begin{proof}
	Define $\rho^{\varepsilon} _{t}: = Y_t^{\varepsilon,h^{\varepsilon}}-Y_t^{\varepsilon}$. We apply It\^{o}'s formula to $|\rho^{\varepsilon} _{t}|^2$ and then take expectation to obtain that
	\begin{align*}
		\frac{d}{dt}\mathbb{E}|\rho^{\varepsilon}_{t}|^{2}=& \frac{2}{\delta}\mathbb{E}\Big[\langle f(\mathscr{L}_{X^{\varepsilon}_t},Y_t^{\varepsilon,h^{\varepsilon}},\mathscr{L}_{Y^{\varepsilon}_t})-f(\mathscr{L}_{X^{\varepsilon}_t},Y_t^{\varepsilon},\mathscr{L}_{Y^{\varepsilon}_t}),\rho^{\varepsilon}_{t}\rangle\Big]  \\&+\frac{1}{\delta}\mathbb{E}\|g(\mathscr{L}_{X^{\varepsilon}_t},Y_t^{\varepsilon,h^{\varepsilon}},\mathscr{L}_{Y^{\varepsilon}_t})-g(\mathscr{L}_{X^{\varepsilon}_t},Y_t^{\varepsilon},\mathscr{L}_{Y^{\varepsilon}_t})\|^{2}  \\&+\frac{2}{\sqrt{\delta\varepsilon}}\mathbb{E}\Big[\langle g(\mathscr{L}_{X^{\varepsilon}_t},Y_t^{\varepsilon,h^{\varepsilon}},\mathscr{L}_{Y^{\varepsilon}_t})h_{t}^{2,\varepsilon},\rho^{\varepsilon}_{t}\rangle\Big]\\
		=:&\sum_{i=1}^3I^{\varepsilon}_i(t).
	\end{align*}
	Based on the assumption  (\ref{con4.1}), we infer that
	\begin{equation*}
		I^{\varepsilon}_1(t)+I^{\varepsilon}_2(t)\leq-\frac{K_1}{\delta}\mathbb{E}|\rho^{\varepsilon}_t|^2.
	\end{equation*}
	Moreover, we have
	\begin{equation*} I_3^{\varepsilon}(t)\lesssim\dfrac{1}{\sqrt{\delta\varepsilon}}\mathbb{E}\Big[(1+\mathbb{E}|Y^{\varepsilon}_t|^2)^{\frac{1}{2}}|h_t^{2,\varepsilon}||\rho^{\varepsilon}_t|\Big]\leq\dfrac{C_y}{\varepsilon}\mathbb{E}|h_t^{2,\varepsilon}|^2+\dfrac{\epsilon_0}{\delta}\mathbb{E}|\rho^{\varepsilon}_t|^2,
	\end{equation*}
	where $\epsilon_{0}\in(0,K_1).$ Then we deduce that
	\begin{equation*}
		\frac{d}{dt}\mathbb{E}|\rho^{\varepsilon}_t|^2\leq-\frac{\beta}{\delta}\mathbb{E}|\rho^{\varepsilon}_t|^2+\frac{C_y}{\varepsilon}\mathbb{E}|h_t^{2,\varepsilon}|^2,
	\end{equation*}
	where $\beta:=K_1-\epsilon_{0}>0.$ The Gronwall's lemma leads to
	\begin{equation*}
		\mathbb{E}|\rho^{\varepsilon}_t|^2\lesssim\dfrac{1}{\varepsilon}\int_0^te^{-\frac{\beta}{\delta}(t-s)}\mathbb{E}|h_s^{2,\varepsilon}|^2ds.
	\end{equation*}
	Therefore, for $h^\varepsilon\in\mathcal{A}_M$ we have
	\begin{equation}\label{e02} \mathbb{E}\int_0^T|\rho^{\varepsilon}_t|^2dt\lesssim\frac{1}{\varepsilon}\mathbb{E}\int_0^T\int_0^te^{-\frac{\beta}{\delta}(t-s)}|h_s^{2,\varepsilon}|^2dsdt\lesssim_{M,T}\frac{\delta}{\varepsilon}.
	\end{equation}
	Since $\frac{\delta}{\varepsilon \Delta} \to 0$~as~$\varepsilon \to 0$, we get the desired result.
\end{proof}

We now state the convergence of $\mathcal{K}_{1}^{\varepsilon}(t)$ as follows.
\begin{lemma}\label{l4.2}
	Fix $t\in[0,T]$.	The following limit holds in distribution:
	\begin{align}\label{4.14}
		\mathcal{K}_{1}^{\varepsilon}(t)-\int_{\mathbb{R}^d\times\mathbb{R}^m\times\mathscr{P}_2(\mathbb{R}^m)\times[0,t]}\bar{b}(X_{s},\mathscr{L}_{\bar{X}_{s}}){{\PPi}}(dhdyd\rho ds) \xrightarrow{\varepsilon\to0}0.
	\end{align}
\end{lemma}
\begin{proof}
	In light of (\ref{2.18}), we have
	\begin{align*} \int_{\mathbb{R}^d\times\mathbb{R}^m\times\mathscr{P}_2(\mathbb{R}^m)\times[0,t]}\bar{b}(X_{s},\mathscr{L}_{\bar{X}_{s}}){{\PPi}}(dhdyd\rho ds)=\int_{0}^{t}\bar{b}(X_{s},\mathscr{L}_{\bar{X}_{s}})ds.
	\end{align*}
	Thus, it suffices to verify the following convergence in distribution:
	\begin{equation}\label{4.15}
\mathcal{K}_{1}^{\varepsilon}(t)-\int_{0}^{t}\bar{b}(X_{s},\mathscr{L}_{\bar{X}_{s}})ds \xrightarrow{\varepsilon\to0}0.
	\end{equation}
	We proceed with the proof of (\ref{4.15}) in two steps.
	
	\vspace{1mm}
	\noindent\textbf{Step 1.} Note that
	\begin{align*}
		&\mathcal{K}_{1}^{\varepsilon}(t)-\int_{0}^{t}\bar{b}(X_{s},\mathscr{L}_{\bar{X}_{s}})ds  \\=&\int_{0}^{t} b(X_{s}^{\varepsilon,h^{\varepsilon}},\mathscr{L}_{X_{s}^{\varepsilon}},Y_{s}^{\varepsilon,h^{\varepsilon}},\mathscr{L}_{Y_{s}^{\varepsilon}}) -b(X_{s(\Delta)}^{\varepsilon,h^{\varepsilon}},\mathscr{L}_{X_{s(\Delta)}^{\varepsilon}},Y_{s}^{\varepsilon},\mathscr{L}_{Y_{s}^{\varepsilon}})ds  \\
		&+\int_{0}^{t} b(X_{s(\Delta)}^{\varepsilon,h^{\varepsilon}},\mathscr{L}_{X_{s(\Delta)}^{\varepsilon}},Y_{s}^{\varepsilon},\mathscr{L}_{Y_{s}^{\varepsilon}})-\bar{b}(X_{s(\Delta)}^{\varepsilon,h^{\varepsilon}},\mathscr{L}_{X_{s(\Delta)}^{\varepsilon}})ds  \\
		&+\int_{0}^{t} \bar{b}(X_{s(\Delta)}^{\varepsilon,h^{\varepsilon}},\mathscr{L}_{X_{s(\Delta)}^{\varepsilon}})-\bar{b}(X_{s}^{\varepsilon,h^{\varepsilon}},\mathscr{L}_{X_{s}^{\varepsilon}}) ds \\
		&+\int_{0}^{t}
		\bar{b}(X_{s}^{\varepsilon,h^{\varepsilon}},\mathscr{L}_{X_{s}^{\varepsilon}})- \bar{b}(X_{s},\mathscr{L}_{\bar{X}_{s}})ds \\
		=:&I^{\varepsilon}(t) +II^{\varepsilon}(t) +III^{\varepsilon}(t) +IV^{\varepsilon}(t),
	\end{align*}
	where $s(\Delta):=[\frac{s}{\Delta}]\Delta$ and $[s]$ denotes the integer part of $s$. By the locally Lipschitz continuity of $b$ and $\bar{b}$ we have
	\begin{align}\label{zz4.17}
		&\mathbb{E}\Big[\sup_{t\in[0,T]}|I^{\varepsilon}(t)+III^{\varepsilon}(t)|^2\Big]  \nonumber\\
		\lesssim&_T \mathbb{E}\int_{0}^{T}\bigg[\big(1+\mathbb{E}|X_{t}^{\varepsilon}|^{\kappa}+\mathbb{E}|X_{t(\Delta)}^{\varepsilon}|^{\kappa}+\mathbb{E}|Y_{t}^{\varepsilon}|^{\kappa}\big) \big(|X_{t}^{\varepsilon,h^{\varepsilon}}-X_{t(\Delta)}^{\varepsilon,h^{\varepsilon}}|^{2}+|Y_{t}^{\varepsilon,h^{\varepsilon}}-Y_{t}^{\varepsilon}|^2\big)\bigg]dt \nonumber\\
		&+\mathbb{E}\int_{0}^{T}\bigg[\big(1+|X_{t}^{\varepsilon,h^{\varepsilon}}|^{\kappa}+|X_{t(\Delta)}^{\varepsilon,h^{\varepsilon}}|^{\kappa}+|Y_{t}^{\varepsilon,h^{\varepsilon}}|^{\kappa}+|Y_{t}^{\varepsilon}|^{\kappa}+\mathbb{E}|X_{t}^{\varepsilon}|^{\kappa}+\mathbb{E}|X_{t(\Delta)}^{\varepsilon}|^{\kappa}+\mathbb{E}|Y_{t}^{\varepsilon}|^{\kappa}\big) \nonumber\\
		&\cdot\big(1+\mathbb{E}|X_{t}^{\varepsilon}|^{\kappa}+\mathbb{E}|X_{t(\Delta)}^{\varepsilon}|^{\kappa}+\mathbb{E}|Y_{t}^{\varepsilon}|^{\kappa}\big)\mathbb{W}_{2}(\mathscr{L}_{X_{t}^{\varepsilon}},\mathscr{L}_{X_{t(\Delta)}^{\varepsilon}})^{2}\bigg]dt \nonumber \\
		\lesssim&_{x,y,M,T}\int_{0}^{T}\mathbb{E}|X_{t}^{\varepsilon,h^{\varepsilon}}-X_{t(\Delta)}^{\varepsilon,h^{\varepsilon}}|^{2}dt+\mathbb{E}\int_{0}^{T}\Big[\big(1+|Y_{t}^{\varepsilon,h^{\varepsilon}}|^{\kappa}\big)\mathbb{E}|X_{t}^{\varepsilon}-X_{t(\Delta)}^{\varepsilon}|^2\Big]dt\nonumber \\
		&+ \mathbb{E}\int_{0}^{T}|Y_{t}^{\varepsilon,h^{\varepsilon}}-Y_{t}^{\varepsilon}|^2dt \nonumber\\
		\lesssim&_{x,y,M,T}\Delta+\frac{\delta}{\varepsilon},
		\end{align}
	where we have used (\ref{F3.10}),  (\ref{3.2}), (\ref{3.12})  and (\ref{e02})  in the last step. Notice that we can invoke (\ref{t2.155555}), (\ref{3.1}) and (\ref{444.7}) to deduce
	\begin{align*}
		&\hat{\mathbb{E}}\Big|\int_{0}^{t}\big( 	\bar{b}(\hat{X}_{s}^{\varepsilon,h^{\varepsilon}},\mathscr{L}_{X_{s}^{\varepsilon}})- \bar{b}(\hat{X}_{s},\mathscr{L}_{\bar{X}_{s}})\big)ds\Big|^2  \\
		\lesssim&_{x,y,M,T}\hat{\mathbb{E}}\int_{0}^{t}|\hat{X}_{s}^{\varepsilon,h^{\varepsilon}}-\hat{X}_{s}|^{2}ds+\int_{0}^{t}\mathbb{W}_{2}(\mathscr{L}_{X_{s}^{\varepsilon}},\mathscr{L}_{\bar{X}_{s}})^{2}ds  \\
		\to&~ 0, ~~\text{as} ~\varepsilon \to 0.
	\end{align*}
	From the same argument as  (\ref{4.19}),  the following convergence holds in distribution
	\begin{equation}\label{e06}
		IV^{\varepsilon}(t) \xrightarrow{\varepsilon\to0}0.
	\end{equation}
	Consequently, combining (\ref{zz4.17}) with (\ref{e06}), in order to show (\ref{4.15}) it suffices to prove the following limit 	
	\begin{align}\label{e07}
		\mathbb{E}\Big[\sup_{t\in[0,T]}|II^{\varepsilon}(t)|^2\Big]\to 0&, ~~\text{as} ~\varepsilon \to 0.
	\end{align}

	\noindent\textbf{Step 2.} For the term $II^{\varepsilon}(t)$, we have
	\begin{align*}
		II^{\varepsilon}(t)=&\int_{0}^{t} b(X_{s(\Delta)}^{\varepsilon,h^{\varepsilon}},\mathscr{L}_{X_{s(\Delta)}^{\varepsilon}},Y_{s}^{\varepsilon},\mathscr{L}_{Y_{s}^{\varepsilon}})-b(X_{s(\Delta)}^{\varepsilon,h^{\varepsilon}},\mathscr{L}_{X_{s(\Delta)}^{\varepsilon}},\hat{Y}_{s}^{\varepsilon},\mathscr{L}_{\hat{Y}_{s}^{\varepsilon}})ds \\
		&+\int_{0}^{t} b(X_{s(\Delta)}^{\varepsilon,h^{\varepsilon}},\mathscr{L}_{X_{s(\Delta)}^{\varepsilon}},\hat{Y}_{s}^{\varepsilon},\mathscr{L}_{\hat{Y}_{s}^{\varepsilon}})-\bar{b}(X_{s(\Delta)}^{\varepsilon,h^{\varepsilon}},\mathscr{L}_{X_{s(\Delta)}^{\varepsilon}})ds \\
		=:&II^{\varepsilon}_{1}(t)+II^{\varepsilon}_{2}(t).
	\end{align*}
	where the process $\hat{Y}^{\varepsilon}$ is defined by (\ref{4.6a}). According to Lemma \ref{lem:uniformestimateofauxiliary}, we can easily derive that
	\begin{align}\label{e08}
		\mathbb{E}\Big[\sup_{t\in[0,T]}|II_1^{\varepsilon}(t)|^2\Big]\to 0&, ~~\text{as} ~\varepsilon \to 0.
	\end{align}	
	
	Now, for the term $II_2^{\varepsilon}(t)$ we  divide the time interval to get
	\begin{align*}
		|II_2^{\varepsilon}(t)|^{2} =&\Bigg|\sum_{k=0}^{[t/\Delta]-1}\int_{k\Delta}^{(k+1)\Delta}b(X_{s(\Delta)}^{\varepsilon,h^{\varepsilon}},\mathscr{L}_{X_{s(\Delta)}^{\varepsilon}},\hat{Y}_{s}^{\varepsilon},\mathscr{L}_{\hat{Y}_{s}^{\varepsilon}})-\bar{b}(X_{s(\Delta)}^{\varepsilon,h^{\varepsilon}},\mathscr{L}_{X_{s(\Delta)}^{\varepsilon}})ds \nonumber \\
		&+\int_{t(\Delta)}^{t}b(X_{s(\Delta)}^{\varepsilon,h^{\varepsilon}},\mathscr{L}_{X_{s(\Delta)}^{\varepsilon}},\hat{Y}_{s}^{\varepsilon},\mathscr{L}_{\hat{Y}_{s}^{\varepsilon}})-\bar{b}(X_{s(\Delta)}^{\varepsilon,h^{\varepsilon}},\mathscr{L}_{X_{s(\Delta)}^{\varepsilon}})ds \Bigg|^2 \nonumber \\
		\leq&\frac{C_{T}}{\Delta}\sum_{k=0}^{[t/\Delta]-1}\left|\int_{k\Delta}^{(k+1)\Delta}b(X_{s(\Delta)}^{\varepsilon,h^{\varepsilon}},\mathscr{L}_{X_{s(\Delta)}^{\varepsilon}},\hat{Y}_{s}^{\varepsilon},\mathscr{L}_{\hat{Y}_{s}^{\varepsilon}})-\bar{b}(X_{s(\Delta)}^{\varepsilon,h^{\varepsilon}},\mathscr{L}_{X_{s(\Delta)}^{\varepsilon}})ds\right|^{2} \nonumber\\
		&+2\left|\int_{t(\Delta)}^{t}b(X_{s(\Delta)}^{\varepsilon,h^{\varepsilon}},\mathscr{L}_{X_{s(\Delta)}^{\varepsilon}},\hat{Y}_{s}^{\varepsilon},\mathscr{L}_{\hat{Y}_{s}^{\varepsilon}})-\bar{b}(X_{s(\Delta)}^{\varepsilon,h^{\varepsilon}},\mathscr{L}_{X_{s(\Delta)}^{\varepsilon}})ds\right|^{2} \nonumber\\
		=:&\mathscr{J}^{\varepsilon}_{1}(t)+\mathscr{J}^{\varepsilon}_{2}(t).
	\end{align*}
	In view of term $\mathscr{J}^{\varepsilon}_{2}(t)$, it follows that
	\begin{align}
		\mathbb{E}|\mathscr{J}^{\varepsilon}_{2}(t)|
	\lesssim& \mathbb{E}\int_{t(\Delta)}^{t}|b(X_{s(\Delta)}^{\varepsilon,h^{\varepsilon}},\mathscr{L}_{X_{s(\Delta)}^{\varepsilon}},\hat{Y}_{s}^{\varepsilon},\mathscr{L}_{\hat{Y}_{s}^{\varepsilon}})|^{2}+|\bar{b}(X_{s(\Delta)}^{\varepsilon,h^{\varepsilon}},\mathscr{L}_{X_{s(\Delta)}^{\varepsilon}})|^{2}ds
	\nonumber\\
	\lesssim& \mathbb{E}\int_{t(\Delta)}^{t}\Big[1+|X_{s(\Delta)}^{\varepsilon,h^{\varepsilon}}|^{\kappa}+\mathbb{E}|{X_{s(\Delta)}^{\varepsilon}}|^{\kappa}+|\hat{Y}_{s}^{\varepsilon}|^{\kappa}+\mathbb{E}|\hat{Y}_{s}^{\varepsilon}|^{\kappa}\Big]ds 	\nonumber \\
	&+ \mathbb{E}\int_{t(\Delta)}^{t}\Big[\big(1+\mathbb{E}|{X_{s(\Delta)}^{\varepsilon}}|^{\kappa}\big)|X_{s(\Delta)}^{\varepsilon,h^{\varepsilon}}|^{2}
	\nonumber \\
	&
	+\big(1+|X_{s(\Delta)}^{\varepsilon,h^{\varepsilon}}|^{\kappa}+\mathbb{E}|{X_{s(\Delta)}^{\varepsilon}}|^{\kappa}\big)\big(1+\mathbb{E}|{X_{s(\Delta)}^{\varepsilon}}|^{\kappa}\big)  \mathbb{E}|{X_{s(\Delta)}^{\varepsilon}}|^{2}\Big]ds
 	\nonumber\\
	\lesssim&_{x,y,M,T}\Delta, \label{e09}
		\end{align}
	where we used the estimates (\ref{esX}), (\ref{3.13a}) and (\ref{3.1}) in the last step.

	For the term  $\mathscr{J}^{\varepsilon}_{1}(t)$, it can be estimated by
	\begin{align}
		&\mathbb{E}|\mathscr{J}^{\varepsilon}_{1}(t)|  \nonumber\\
	\lesssim&_{T} \frac{1}{\Delta}\mathbb{E}\sum_{k=0}^{[T/\Delta]-1}\bigg|\int_{k\Delta}^{(k+1)\Delta}b (X_{k\Delta}^{\varepsilon,h^{\varepsilon}},\mathscr{L}_{X_{k\Delta}^{\varepsilon}},\hat{Y}_{s}^{\varepsilon},\mathscr{L}_{\hat{Y}_{s}^{\varepsilon}})-\bar{b}(X_{k\Delta}^{\varepsilon,h^{\varepsilon}},\mathscr{L}_{X_{k\Delta}^{\varepsilon}})ds\bigg|^{2} \nonumber \\
	\lesssim&_{T}\frac{1}{\Delta^2}\max_{0\leq k\leq[T/\Delta]-1}\mathbb{E}\bigg|\int_{k\Delta}^{(k+1)\Delta}b (X_{k\Delta}^{\varepsilon,h^{\varepsilon}},\mathscr{L}_{X_{k\Delta}^{\varepsilon}},\hat{Y}_{s}^{\varepsilon},\mathscr{L}_{\hat{Y}_{s}^{\varepsilon}})-\bar{b}(X_{k\Delta}^{\varepsilon,h^{\varepsilon}},\mathscr{L}_{X_{k\Delta}^{\varepsilon}})ds \bigg|^2 \nonumber \\
	\lesssim&_{T} \frac{\delta^2}{\Delta^2}\max\limits_{0\leq k\leq[T/\Delta]-1}\mathbb{E}\bigg|\int_0^{\frac{\Delta}{\delta}}b (X_{k\Delta}^{\varepsilon,h^{\varepsilon}},\mathscr{L}_{X_{k\Delta}^{\varepsilon}},\hat{Y}_{s\delta+k\Delta}^\varepsilon,\mathscr{L}_{\hat{Y}_{s\delta+k\Delta}^\varepsilon})-\bar{b}(X_{k\Delta}^{\varepsilon,h^{\varepsilon}},\mathscr{L}_{X_{k\Delta}^{\varepsilon}})ds \bigg|^2  \nonumber\\
	\lesssim&_{T}\frac{\delta^2}{\Delta^2}\max_{0\leq k\leq[T/\Delta]-1}\left[\int_0^{\frac{\Delta}{\delta}}\int_r^{\frac{\Delta}{\delta}}\Psi_k(s,r)dsdr\right], \label{z4.21}
		\end{align}
	where for any $0\leq r\leq s\leq\frac{\Delta}{\delta},$
	\begin{align*}
		\Psi_{k}(s,r):=\mathbb{E}\Big[\langle &b (X_{k\Delta}^{\varepsilon,h^{\varepsilon}},\mathscr{L}_{X_{k\Delta}^{\varepsilon}},\hat{Y}_{s\delta+k\Delta}^\varepsilon,\mathscr{L}_{\hat{Y}_{s\delta+k\Delta}^\varepsilon})-\bar{b}(X_{k\Delta}^{\varepsilon,h^{\varepsilon}},\mathscr{L}_{X_{k\Delta}^{\varepsilon}}),\\
		&b (X_{k\Delta}^{\varepsilon,h^{\varepsilon}},\mathscr{L}_{X_{k\Delta}^{\varepsilon}},\hat{Y}_{r\delta+k\Delta}^\varepsilon,\mathscr{L}_{\hat{Y}_{r\delta+k\Delta}^\varepsilon})-\bar{b}(X_{k\Delta}^{\varepsilon,h^{\varepsilon}},\mathscr{L}_{X_{k\Delta}^{\varepsilon}})\rangle\Big].
	\end{align*}
	Recall the equalities (\ref{eq:ffor1}) and (\ref{eq:ffor2}). Following same argument as in (\ref{e10}), we have
	\begin{align*}
		&~~~~~\Psi_{k}(s,r) \\
		&=\mathbb{E}\bigg[\mathbb{E}\Big[\big\langle b (X_{k\Delta}^{\varepsilon,h^{\varepsilon}},\eta,\tilde{Y}_{s\delta+k\Delta}^{\gamma,\varepsilon,k\Delta,\eta,\hat{Y}_{k\Delta}^{\varepsilon}},\mathscr{L}_{\tilde{Y}_{s\delta+k\Delta}^{\gamma,\varepsilon,k\Delta,\eta,\hat{Y}_{k\Delta}^{\varepsilon}}})-\bar{b}(X_{k\Delta}^{\varepsilon,h^{\varepsilon}},\eta), \\
		&~~~~~~~~~~~b (X_{k\Delta}^{\varepsilon,h^{\varepsilon}},\eta,\tilde{Y}_{r\delta+k\Delta}^{\gamma,\varepsilon,k\Delta,\eta,\hat{Y}_{k\Delta}^{\varepsilon}},\mathscr{L}_{\tilde{Y}_{r\delta+k\Delta}^{\gamma,\varepsilon,k\Delta,\eta,\hat{Y}_{k\Delta}^{\varepsilon}}})-\bar{b}(X_{k\Delta}^{\varepsilon,h^{\varepsilon}},\eta)\big\rangle\big|\mathscr{F}_{k\Delta} \Big]\bigg] \\
		&=\mathbb{E}\bigg[\mathbb{E}\Big[\big\langle b (u,\eta,\tilde{Y}_{s\delta+k\Delta}^{\gamma,\varepsilon,k\Delta,\eta,y},\mathscr{L}_{\tilde{Y}_{s\delta+k\Delta}^{\gamma,\varepsilon,k\Delta,\eta,\hat{Y}_{k\Delta}^{\varepsilon}}})-\bar{b}(u,\eta), \\
		&~~~~~~~~~~~b (u,\eta,\tilde{Y}_{r\delta+k\Delta}^{\gamma,\varepsilon,k\Delta,\eta,y},\mathscr{L}_{\tilde{Y}_{r\delta+k\Delta}^{\gamma,\varepsilon,k\Delta,\eta,\hat{Y}_{k\Delta}^{\varepsilon}}})-\bar{b}(u,\eta)\big\rangle\Big]\big|_{(u,y)=(X_{k\Delta}^{\varepsilon,h^{\varepsilon}},\hat{Y}_{k\Delta}^{\varepsilon})} \bigg],
	\end{align*}
	where $\gamma=\mathscr{L}_{\hat{Y}_{k\Delta}^{\varepsilon}}, \eta=\mathscr{L}_{X_{k\Delta}^{\varepsilon}}$. Similar to (\ref{e11}), due to the Markovian property and time-homogeneous property of the lifted semigroup $\tilde{\bf{P}}^{\eta}_t$, it holds that for any $r\in [0,\frac{\delta}{s}]$ and  $s\in [r,\frac{\delta}{s}]$,
	\begin{align}
&\Psi_{k}(s,r)
 \nonumber\\
=\,\,\,& \EE\bigg[\tilde{\EE}\Big[ \langle \tilde{{\bf{P}}}_{s-r}^{\eta}\big[ b(u,\eta, ( Y_r^{\eta,\gamma, y}, \mathscr{L}_{Y_r}^{\eta,\gamma}))-\bar{b}(u,\eta)\big], \nonumber\\
&\;\;\;\;\;\;\;\;\;\;\;\;\;\;\;\;\;\;\;\;b(u,\eta, Y_{r}^{\eta,\gamma,y}, \mathscr{L}_{Y_{r}^{\eta, \gamma}})-\bar{b}(u,\eta)\rangle
\Big]\Big|_{u=X_{k\Delta}^{\varepsilon,h^\varepsilon}, y=\hat{Y}_{k\Delta}^{\varepsilon}}\bigg] \nonumber\\
=\,\,\,& \EE\bigg[\tilde{\EE}\Big[ \langle \tilde{\EE}\big[ b(u,\eta, Y_{s-r}^{\eta, \mathscr{L}_{Y_r}^{\eta,\gamma}, z}, \mathscr{L}_{Y_{s-r}^{\eta, \mathscr{L}_{Y_r}^{\eta,\gamma}}})-\bar{b}(u,\eta)\big]\mathbf{1}_{\{z=Y_r^{\eta,\gamma,y}\}}, \nonumber\\
&\;\;\;\;\;\;\;\;\;\;\;\;\;\;\;\;b(u,\eta, Y_{r}^{\eta,\gamma,y}, \mathscr{L}_{Y_{r}^{\eta, \gamma}})-\bar{b}(u,\eta)\rangle
\Big]\Big|_{u=X_{k\Delta}^{\varepsilon,h^\varepsilon}, y=\hat{Y}_{k\Delta}^{\varepsilon}}\bigg]
 \nonumber\\
\lesssim&_{\kappa}\EE\bigg[ \tilde{\EE}\Big[\big(1+|u|^{\kappa}+\big|Y_r^{\eta,\gamma,y}\big|^{\kappa}+M_{\kappa}(\eta)+M_{\kappa}(\mathscr{L}_{Y_r}^{\eta,\gamma})\big)e^{-\beta(s-r)} \nonumber\\
&\;\;\;\;\;\;\cdot \Big(\big(1+|u|^{\frac{\kappa}{2}}+|Y_r^{\eta,\gamma,y}|^{\frac{\kappa}{2}}+M_{\kappa}^{\frac{1}{2}}(\eta)+M_{\kappa}^{\frac{1}{2}}(\mathscr{L}_{Y_r^{\eta,\gamma}})\big) \nonumber\\
&\;\;\;\;\;\;\;\;\;\;\;\;+\big( 1+M_{\kappa}^{\frac{1}{2}}(\eta)\big)|u|+(1+|u|^{\frac{\kappa}{2}}+M_{\kappa}^{\frac{1}{2}}(\eta))(1+M_{\kappa}^{\frac{1}{2}}(\eta))^2\Big)\Big]\Big|_{u=X_{k\Delta}^{\varepsilon,h^\varepsilon}, y=\hat{Y}_{k\Delta}^{\varepsilon}}\bigg] \nonumber\\
\lesssim&_{x,y,M,T}e^{-\beta(s-r)}. \label{z4.25}
	\end{align}
	Inserting (\ref{z4.25}) into (\ref{z4.21}) and using (\ref{e09}), we obtain
	\begin{align*}
		\mathbb{E}\Big[\sup_{t\in[0,T]}|II^{\varepsilon}_{2}(t)|^2\Big]\lesssim&_{x,y,M,T} \frac{\delta}{\Delta}+\frac{\delta^2}{\Delta^2}+\Delta.
	\end{align*}
	Due to the definition (\ref{2.11}) and the convergence (\ref{e08}), we get  (\ref{e07}) holds.
\end{proof}

We proceed the proof of  \eqref{2.16}.    Based on the convergence claimed in (\ref{444.7}), we can obtain the equality $\hat{\mathbb{P}}$-a.s.
\begin{align*}
	\hat{X}_t &= x+\hat{\mathcal{K}}_1(t)+\hat{\mathcal{K}}_2(t),\quad t\geq 0.
\end{align*}
Furthermore, combining (\ref{444.16}) with the results from Lemmas \ref{l4.1} and \ref{l4.2}, we deduce that the limit pair $(X,\PPi)$ of the sequence $\{(X^{\varepsilon,h^{\varepsilon}},{\PPi}^{\varepsilon,\Delta})\}_{\varepsilon\in(0,1)}$ fulfills the following integral equation
\begin{align*}
	X_t =& x+\int_{\mathbb{R}^d\times\mathbb{R}^m\times\mathscr{P}_2(\mathbb{R}^m)\times[0,t]}\bar{b}(X_s,\mathscr{L}_{\bar{X}_{s}})+\sigma(X_s,\mathscr{L}_{\bar{X}_{s}},y,\rho)\pi_1h{\PPi}(dhdyd\rho ds)\\
	=&x+\int_{\mathbb{R}^d\times\RR^m\times\mathscr{P}_2(\mathbb{R}^m)\times[0,t]}\Theta(X_{s},\mathscr{L}_{\bar{X}_{s}},h,y,\rho){\PPi}(dhdyd\rho ds)\quad{\mathbb{P}}\text{-a.s.}.
\end{align*}
Consequently, $(X,{\PPi})$ satisfies (\ref{2.16}) in Definition \ref{d2.4}.    \hspace{\fill}$\Box$

\vspace{3mm}
\noindent\textbf{Proof of (\ref{2.17}).} In this part, we demonstrate that the second, third and fourth marginals of ${\PPi}$ are identified by the product of the invariant measure  $\nu^{\mathscr{L}_{\bar{X}_t}}$, Dirac measure $\delta_{\nu^{\mathscr{L}_{\bar{X}_t}}}$ and the Lebesgue measure, which will be verified in Lemma \ref{l4.5}.

Before doing that, for any $s \geq t$, we define the two parameter process $Y^\varepsilon (s;t)$ satisfying
\begin{align*}
	dY^\varepsilon(s;t)=&~\frac{1}{\delta}f(\mathscr{L}_{X^{\varepsilon}_t},Y^\varepsilon(s;t),\mathscr{L}_{Y^\varepsilon(s;t)})ds+\frac{1}{\sqrt{\delta}}g(\mathscr{L}_{X^{\varepsilon}_t},Y^\varepsilon(s;t),\mathscr{L}_{Y^\varepsilon(s;t)})dW_s^2, \\
	Y^\varepsilon (t;t)=&~{Y}_t^{\varepsilon}.
\end{align*}
The following lemma shows that the uncontrolled fast process ${Y}_t^{\varepsilon}$ is close to the process $Y^\varepsilon  (s;t)$ in $\mathit{L}^2$-sense on the interval $s \in [t,t+\Delta]$.

\begin{lemma}\label{l4.4}
	Let $M>0$, $\{h^{\varepsilon}\}_{\varepsilon\in(0,1)}\subset\mathcal{A}_{M}$, $\Delta$ as in Definition $(\ref{2.11})$. Then
	\begin{equation}\label{4.18}
		\dfrac{1}{\Delta}\mathbb{E}\int_t^{t+\Delta}|{Y}_s^\varepsilon-Y^\varepsilon  (s;t)|^2ds\to0,~~\text{as}~\varepsilon\to0.
	\end{equation}
\end{lemma}
\begin{proof}
	One can follow exactly similar argument as in Lemma \ref{l4.3} to get (\ref{4.18}). We omit the details.
\end{proof}

We  are now in the position to show that (\ref{2.17}) holds for the limit pair ${\PPi}$.
\begin{lemma}\label{l4.5}
	The limit pair ${\PPi}$ has the decomposition $(\ref{2.17})$, i.e., for any $F \in C_b(\mathbb{R}^m\times\mathscr{P}_2(\mathbb{R}^m)),$
	\begin{align*}
		&\int_{\mathbb{R}^d\times\mathbb{R}^m\times\mathscr{P}_2(\mathbb{R}^m)\times[0,T]}F(y,\rho){\PPi}(dhdyd\rho dt) \\
		=&\int_{0}^{T}\int_{\mathbb{R}^m\times\mathscr{P}_2(\mathbb{R}^m)}F(y,\rho )(\nu^{\mathscr{L}_{\bar{X}_t}}\times \delta_{\nu^{\mathscr{L}_{\bar{X}_t}}})(dyd\rho)dt.
	\end{align*}
\end{lemma}
\begin{proof}
	Without loss of generality, we suppose that $\mathit{F}\in \text{Lip}_b(\mathbb{R}^m\times\mathscr{P}_2(\mathbb{R}^m))$. We first consider the difference
	\begin{align*}
		&\int_{\mathbb{R}^d\times\mathbb{R}^m\times\mathscr{P}_2(\mathbb{R}^m)\times[0,T]}F(y,\rho){\PPi}^{\varepsilon,\Delta}(dhdyd\rho dt)
 \\
&-\int_{0}^{T}\int_{\mathbb{R}^m\times\mathscr{P}_2(\mathbb{R}^m)}F(y,\rho )(\nu^{\mathscr{L}_{\bar{X}_t}}\times \delta_{\nu^{\mathscr{L}_{\bar{X}_t}}})(dyd\rho)dt  \\
		=&\left(\int_0^T\frac{1}{\Delta}\int_t^{t+\Delta}F(Y_s^{\varepsilon,h^\varepsilon},\mathscr{L}_{Y^{\varepsilon}_s})dsdt-\int_0^T\frac{1}{\Delta}\int_t^{t+\Delta}F({Y}_s^\varepsilon,\mathscr{L}_{Y^{\varepsilon}_s})dsdt\right) \\
		&+\left(\int_{0}^{T}\frac{1}{\Delta}\int_{t}^{t+\Delta}F({Y}_{s}^{\varepsilon},\mathscr{L}_{Y^{\varepsilon}_s})dsdt-\int_{0}^{T}\frac{1}{\Delta}\int_{t}^{t+\Delta}F(Y ^\varepsilon (s;t),\mathscr{L}_{Y^\varepsilon (s;t)})dsdt\right)  \\
		&+\bigg(\int_{0}^{T}\frac{1}{\Delta}\int_{t}^{t+\Delta}F(\bar{Y}(s;t),\mathscr{L}_{\bar{Y}(s;t)})dsdt \\
		&-\int_{0}^{T}\int_{\mathbb{R}^m\times\mathscr{P}_2(\mathbb{R}^m)}F(y,\rho )(\nu^{\mathscr{L}_{\bar{X}_t}}\times \delta_{\nu^{\mathscr{L}_{\bar{X}_t}}})(dyd\rho)dt\bigg) \\
		&+\bigg(\int_{0}^{T}\frac{1}{\Delta}\int_{t}^{t+\Delta}F(Y^\varepsilon  (s;t),\mathscr{L}_{Y^\varepsilon (s;t)})dsdt \\
		&-\int_{0}^{T}\frac{1}{\Delta}\int_{t}^{t+\Delta}F(\bar{Y}(s;t),\mathscr{L}_{\bar{Y}(s;t)})dsdt\bigg)  \\
		=:&\sum_{i=1}^4 \mathcal{O}_i^\varepsilon,
	\end{align*}
	where the two parameter process $\bar{Y}(s;t)$ satisfies
	\begin{align*}
		d\bar{Y}(s;t)=&~\frac{1}{\delta}f(\mathscr{L}_{\bar{X}_t},\bar{Y}(s;t),\mathscr{L}_{\bar{Y}(s;t)})ds+\frac{1}{\sqrt{\delta}}g(\mathscr{L}_{\bar{X}_t},\bar{Y}(s;t),\mathscr{L}_{\bar{Y}(s;t)})dW_s^2, \\
		\bar{Y} (t;t)=&~{Y}_t^{\varepsilon}.
	\end{align*}	
	We claim that the terms $\mathcal{O}_{i}^{\varepsilon}, i=1,\ldots,4,$  converge to zero in probability, as $\varepsilon \to 0$. Firstly,
	by Lemma \ref{l4.3} and the dominated convergence theorem, it is straightforward that $\mathcal{O}_{1}^{\varepsilon}$ tends to
	zero in probability, as $\varepsilon \to 0$.  Similarly, Lemma \ref{l4.4} and the dominated convergence theorem yield that $\mathcal{O}_{2}^{\varepsilon}$ also tends to zero in probability, as $\varepsilon \to 0$.
	
	Now, let us address the term  $\mathcal{O}_{3}^{\varepsilon}$. We introduce the time-rescaled process  $\tilde Y_{s}:=\bar{Y} (t+\delta s;t)$ satisfying the equation
	\begin{equation*}
		d\tilde Y_{s}=f(\mathscr{L}_{\bar{X}_t},\tilde Y_{s},\mathscr{L}_{\tilde Y_{s}})ds+g(\mathscr{L}_{\bar{X}_t},\tilde Y_{s},\mathscr{L}_{\tilde Y_{s}})dW_{s}^{2},~~\tilde Y_{0}={Y}_t^{\varepsilon},\quad0\leq s\leq\frac{\Delta}{\delta}.
	\end{equation*}
	Notice that
	\begin{equation*}
		\dfrac{1}{\Delta}\int_t^{t+\Delta}F(\bar{Y}(s;t),\mathscr{L}_{\bar{Y}(s;t)})ds=\dfrac{\delta}{\Delta}\int_0^{\frac{\Delta}{\delta}}F(\tilde Y_s,\mathscr{L}_{\tilde Y_{s}})ds.
	\end{equation*}
	Define the process $\mathbf{Y}_{\cdot}:=(\tilde Y_{\cdot},\mathscr{L}_{\tilde Y_{\cdot}}) $. By the strong dissipativity assumption (\ref{con4.1}) and applying Lemma \ref{lem1}, we know that the process $\mathbf{Y}_s$ is exponentially ergodic with unique invariant measure $\nu^{\mathscr{L}_{\bar{X}_t}}\times \delta_{\nu^{\mathscr{L}_{\bar{X}_t}}}$.
	Thus invoking the classical Birkhoff ergodic theorem (cf.~\cite{B31} and see also \cite[Theorem 1.1]{D23}), due to the scale condition (\ref{2.11}), we can get the following almost surely convergence
	\begin{equation*}
		\lim_{\varepsilon\to0}\frac{\delta}{\Delta}\int_{0}^{\frac{\Delta}{\delta}}F(\mathbf{Y}_s)ds=\int_{\mathbb{R}^m\times\mathscr{P}_2(\mathbb{R}^m)}F(y,\rho )(\nu^{\mathscr{L}_{\bar{X}_t}}\times \delta_{\nu^{\mathscr{L}_{\bar{X}_t}}})(dyd\rho),
	\end{equation*}
	which together with the dominated convergence theorem yields that  $\mathcal{O}_{3}^{\varepsilon}$ tends to zero in probability.
	
	It remains to study the term
	$\mathcal{O}_{4}^{\varepsilon}.$
	Recalling the convergence (\ref{t2.155555}),
	we can use the same argument as we did in Lemma \ref{l4.3} to obtain the following convergence in probability
	$$\frac{1}{\Delta}\int_{t}^{t+\Delta}F(Y^\varepsilon (s;t),\mathscr{L}_{Y^\varepsilon(s;t)})ds-\frac{1}{\Delta}\int_{t}^{t+\Delta}F(\bar{Y} (s;t),\mathscr{L}_{\bar{Y}(s;t)})ds \to 0,~~\text{as}~ \varepsilon \to 0.$$
	By the 	dominated convergence theorem, $\mathcal{O}_{4}^{\varepsilon}$ tends to zero in probability. Therefore, we conclude the following convergence in probability,
	\begin{align}
	&\int_{\mathbb{R}^d\times\mathbb{R}^m\times\mathscr{P}_2(\mathbb{R}^m)\times[0,T]}F(y,\rho){\PPi}^{\varepsilon,\Delta}(dhdyd\rho dt) \nonumber\\
	\to& \int_{0}^{T}\int_{\mathbb{R}^m\times\mathscr{P}_2(\mathbb{R}^m)}F(y,\rho )(\nu^{\mathscr{L}_{\bar{X}_t}}\times \delta_{\nu^{\mathscr{L}_{\bar{X}_t}}})(dyd\rho)dt,~~\text{as}~ \varepsilon \to 0.\label{zhanweicedudeshoulian}
	\end{align}
	
	On the other hand, since ${\PPi}^{\varepsilon,\Delta}$ converges in distribution  to $\PPi$ in $\mathscr{P}(\mathbb{R}^d\times\mathbb{R}^m\times \mathscr{P}_2(\mathbb{R}^m)\times[0,T])$, then in view of (\ref{zhanweicedudeshoulian}) it is clear that
	\begin{align*}
		&\int_{\mathbb{R}^d\times\mathbb{R}^m\times\mathscr{P}_2(\mathbb{R}^m)\times[0,T]} F(y,\rho){\PPi}(dhdyd\rho dt) = \int_{0}^{T}\int_{\mathbb{R}^m\times\mathscr{P}_2(\mathbb{R}^m)}F(y,\rho )(\nu^{\mathscr{L}_{\bar{X}_t}}\times \delta_{\nu^{\mathscr{L}_{\bar{X}_t}}})(dyd\rho)dt.
	\end{align*}
	We finish the proof.
\end{proof}

\subsection{Laplace principle lower bound: Cases 1 and 2}\label{sub5.4}

In this subsection, we begin to establish the lower bound of Laplace principle. More specifically, our objective is to show that for any bounded and continuous function $\Lambda:C([0,T];\mathbb{R}^{n})\to\mathbb{R}$,
\begin{align*}
	&\liminf_{\varepsilon\to0}\left\{-{\varepsilon}\log\mathbb{E}\left[\exp\left\{-\frac{1}{\varepsilon}\Lambda(X^{\varepsilon})\right\}\right]\right\}  \\
	&\geq\operatorname*{inf}_{(\varphi,{\PPi})\in\mathscr{V}_{(x,\bar{X},\Theta,\nu^{\mathscr{L}_{\bar{X}}})}}\left[{\frac{1}{2}}\int_{\mathbb{R}^d\times\mathbb{R}^m\times\mathscr{P}_2(\mathbb{R}^m)\times[0,T]}|h|^{2}{\PPi}(dhdyd\rho dt)+\Lambda(\varphi)\right].
\end{align*}
First, by \cite[Theorem 3.17]{MR3967100} we justify the lower limit along any subsequence for which
\begin{equation*}
	-{\varepsilon}\log\mathbb{E}\left[\exp\left\{-\frac{1}{\varepsilon}\Lambda(X^{\varepsilon})\right\}\right]
\end{equation*}
is convergent. Such a subsequence exists due to the bound
\begin{align*}
	\bigg|-{\varepsilon}\log\mathbb{E}\left[\exp\left\{-\frac{1}{\varepsilon}\Lambda(X^{\varepsilon})\right\}\right]\bigg|\lesssim\|\Lambda\|_\infty.
\end{align*}
From (\ref{es14}) and the definition of the infimum, for every $\eta >0$, one can find a constant $M>0$ such that for all $\varepsilon\in(0,1)$, there exists $h^{\varepsilon}\in\mathcal{A}_{M}$ satisfying
\begin{equation*}
	-{\varepsilon}\log\mathbb{E}\left[\exp\left\{-\frac{1}{\varepsilon}\Lambda(X^\varepsilon)\right\}\right]\geq\mathbb{E}\left[\left(\frac{1}{2}\int_0^T|h_s^{\varepsilon}|^2ds+\Lambda(X^{\varepsilon,h^\varepsilon})\right)\right]-\eta.
\end{equation*}
By employing this specific control $h^{\varepsilon}$ and the associated processes $X^{\varepsilon,h^\varepsilon}$ to define the occupation measures ${\PPi}^{\varepsilon,\Delta}$, Lemma \ref{p4.1} guarantees the tightness of the family $\{(X^{\varepsilon,h^{\varepsilon}},{\PPi}^{\varepsilon,\Delta})\}_{\varepsilon\in(0,1)}$. As a result, for any given sequence $\{\varepsilon\}$, there exists a subsequence satisfying the weak convergence
\begin{equation*}
	(X^{\varepsilon,h^\varepsilon},{\PPi}^{\varepsilon,\Delta})\Rightarrow(X,{\PPi}),
\end{equation*}
where the limit pair  $(X,{\PPi})\in\mathscr{V}_{(x,\bar{X},\Theta,\nu^{\mathscr{L}_{\bar{X}}})}.$ By Fatou's lemma, we obtain
\begin{align*}
	&\liminf_{\varepsilon\to0}\left\{-\varepsilon\log\mathbb{E}\left[\exp\left\{-\frac{1}{\varepsilon}\Lambda(X^{\varepsilon})\right\}\right]\right\}  \\
	\geq&\liminf_{\varepsilon\to0}\mathbb{E}\left[\left(\frac{1}{2}\int_{0}^{T}|h_{s}^{\varepsilon}|^{2}ds+\Lambda(X^{\varepsilon,h^{\varepsilon}})\right)\right]-\eta  \\
	\geq&\liminf_{\varepsilon\to0}\mathbb{E}\left[\frac{1}{2}\int_{0}^{T}\frac{1}{\Delta}\int_{t}^{t+\Delta}|h_{s}^{\varepsilon}|^{2}dsdt+\Lambda(X^{\varepsilon,h^{\varepsilon}})\right]-\eta   \\
	=&\liminf_{\varepsilon\to0}\mathbb{E}\left[\frac{1}{2}\int_{\mathbb{R}^d\times\mathbb{R}^m\times\mathscr{P}_2(\mathbb{R}^m)\times[0,T]}|h|^{2}{\PPi}^{\varepsilon,\Delta}(dhdyd\rho dt)+\Lambda(X^{\varepsilon,h^{\varepsilon}})\right]-\eta  r \\
	\geq&\mathbb{E}\left[\frac{1}{2}\int_{\mathbb{R}^d\times\mathbb{R}^m\times\mathscr{P}_2(\mathbb{R}^m)\times[0,T]}|h|^{2}{\PPi}(dhdyd\rho dt)+\Lambda(X)\right]-\eta   \\
	\geq&\inf_{(\varphi,{\PPi})\in\mathscr{V}_{(x,\bar{X},\Theta,\nu^{\mathscr{L}_{\bar{X}}})}}\left[\frac{1}{2}\int_{\mathbb{R}^d\times\mathbb{R}^m\times\mathscr{P}_2(\mathbb{R}^m)\times[0,T]}|h|^{2}{\PPi}(dhdyd\rho dt)+\Lambda(\varphi)\right]-\eta   \\
	\geq&\operatorname*{inf}_{\varphi\in C([0,T];\mathbb{R}^{n})}\left[I(\varphi)+\Lambda(\varphi)\right]-\eta.
\end{align*}
Since $\eta>0$ is arbitrary, the lower bound of Laplace principle is proved.\hspace{\fill}$\Box$

\subsection{Laplace principle upper bound: Case 1}
Recall that in this case, $\sigma$ is independent of $y$ (i.e.~$\sigma(x,\mu,y,\nu)\equiv\sigma(x,\mu,\nu)$). Let $\Lambda: C([0,T]; \mathbb{R}^n) \to \mathbb{R}$ be bounded and continuous. Fix $\varepsilon_0 > 0$, and let $\psi \in C([0,T]; \mathbb{R}^n)$ with $\psi_0 = x$ such that
\begin{equation*}
	I(\psi) + \Lambda(\psi) \le \inf_{\varphi \in C([0,T]; \mathbb{R}^n)} [I(\varphi) + \Lambda(\varphi)] + \frac{\varepsilon_0}{2}.
\end{equation*}
According to the definition of the rate function $I(\psi)$, there exists a measure $\PPi\in \mathscr{P}(\mathbb{R}^d\times\mathbb{R}^m\times\mathscr{P}_2(\RR^m)\times[0,T])$ such that  $(\psi, \PPi) \in \mathscr{V}_{(x,\bar{X},\Theta,\nu^{\mathscr{L}_{\bar{X}}})}$ and the disintegration $\PPi(dhdyd\rho ds) = \eta(dh|y,\rho,s)(\nu^{\mathscr{L}_{\bar{X}_s}}\times \delta_{\nu^{\mathscr{L}_{\bar{X}_s}}})(dyd\rho)ds$ such that
\begin{equation*}
	\frac{1}{2}\int_0^T \int_{\mathbb{R}^d\times\mathbb{R}^m\times\mathscr{P}_2(\mathbb{R}^m)}|h|^{2} \eta(dh|y,\rho,s)(\nu^{\mathscr{L}_{\bar{X}_s}}\times \delta_{\nu^{\mathscr{L}_{\bar{X}_s}}})(dyd\rho)ds \le I(\psi) + \frac{\varepsilon_0}{2},
\end{equation*}
and $\psi$ satisfies
\begin{eqnarray}\label{es15}
	\psi_t =&& x + \int_0^t \int_{\mathbb{R}^d\times\mathbb{R}^m\times\mathscr{P}_2(\mathbb{R}^m)} \big[ \bar{b}(\psi_s, \mathscr{L}_{\bar{X}_s})
 \nonumber \\
&&+ \sigma(\psi_s, \mathscr{L}_{\bar{X}_s}, \nu^{\mathscr{L}_{\bar{X}_s}})\pi_1 h \big] \eta(dh|y,\rho,s)(\nu^{\mathscr{L}_{\bar{X}_s}}\times \delta_{\nu^{\mathscr{L}_{\bar{X}_s}}})(dyd\rho)ds.
\end{eqnarray} 

Define the  deterministic control $\bar{h} \in L^2([0,T]; \mathbb{R}^d)$ by
\begin{equation}\label{es16}
	\bar{h}_s := \int_{\mathbb{R}^d\times\mathbb{R}^m\times\mathscr{P}_2(\mathbb{R}^m)} h \, \eta(dh|y,\rho,s)(\nu^{\mathscr{L}_{\bar{X}_s}}\times \delta_{\nu^{\mathscr{L}_{\bar{X}_s}}})(dyd\rho).
\end{equation}
It follows from (\ref{es15}) and (\ref{es16}) that  $\psi$ solves the following deterministic equation
\begin{equation*}
	\psi_t = x + \int_0^t \bar{b}(\psi_s, \mathscr{L}_{\bar{X}_s}) ds + \int_0^t \sigma(\psi_s, \mathscr{L}_{\bar{X}_s}, \nu^{\mathscr{L}_{\bar{X}_s}}) \pi_1 \bar{h}_s ds.
\end{equation*}

Now, consider the  controlled process $X^{\varepsilon, \bar{h}}$ driven by this deterministic control $\bar{h}$. That is, we restrict the pre-limit infimum by choosing the specific admissible control $h^\varepsilon \equiv \bar{h}$ for all $\varepsilon > 0$. By the tightness of the family of processes $\{X^{\varepsilon, \bar{h}}\}_{\varepsilon\in(0,1)}$ in $C([0,T]; \mathbb{R}^n)$, it is easy to show that
$$X^{\varepsilon, \bar{h}} \Rightarrow \hat{X},$$ where
\begin{equation*}
	\hat{X}_t = x + \int_0^t \bar{b}(\hat{X}_s, \mathscr{L}_{\bar{X}_s}) ds + \int_0^t \sigma(\hat{X}_s, \mathscr{L}_{\bar{X}_s}, \nu^{\mathscr{L}_{\bar{X}_s}}) \pi_1 \bar{h}_s ds.
\end{equation*}
Uniqueness of this equation implies that $\hat{X}_t = \psi_t$ for every $t \in [0,T]$ with probability one. Then, applying the variational representation, we have
\begin{align*}
	&\limsup_{\varepsilon\to0}  \left\{-\varepsilon\log\mathbb{E}\left[\exp\left\{-\frac{1}{\varepsilon}\Lambda(X^{\varepsilon})\right\}\right]\right\} \\
	&= \limsup_{\varepsilon\to0} \inf_{h^\varepsilon \in \mathcal{A}} \mathbb{E}\left[\frac{1}{2}\int_{0}^{T}|h_{s}^\varepsilon|^{2}ds+\Lambda(X^{\varepsilon,h^\varepsilon})\right] \\
	&\leq \limsup_{\varepsilon\to0} \mathbb{E}\left[\frac{1}{2}\int_{0}^{T}|\bar{h}_{s}|^{2}ds+\Lambda(X^{\varepsilon,\bar{h}})\right]   \\
	&= \frac{1}{2} \int_0^T |\bar{h}_s|^2 ds + \Lambda(\psi) \\
	&\leq \frac{1}{2}\int_0^T \int_{\mathbb{R}^d\times\mathbb{R}^m\times\mathscr{P}_2(\mathbb{R}^m)}|h|^{2} \eta(dh|y,\rho,s)(\nu^{\mathscr{L}_{\bar{X}_s}}\times \delta_{\nu^{\mathscr{L}_{\bar{X}_s}}})(dyd\rho) ds + \Lambda(\psi) \\
	&\leq I(\psi) + \frac{\varepsilon_0}{2} + \Lambda(\psi) \\
	&\leq \inf_{\varphi\in C([0,T];\mathbb{R}^{n})}\left[I(\varphi)+\Lambda(\varphi)\right]+\varepsilon_0.
\end{align*}
Since $\varepsilon_0 > 0$ is arbitrary, the result is proven. \hfill $\square$

\subsection{Laplace principle upper bound: Case 2}\label{sub5.5}
Recall that in this case, $\sigma$ depends on $y$. The proof of the upper bound of Laplace principle in \textbf{Case 2} in Theorem \ref{t2.1} relies on  the nearly optimal controls to achieve the bound. Recalling definition (\ref{sulv}) from Theorem \ref{t2.1}, we rewrite the rate function as follows
\begin{equation*}
	I(\varphi)=\inf_{{\PPi}\in\Xi_{\varphi}}\frac{1}{2}\int_{0}^{T}\int_{\mathbb{R}^d\times\mathbb{R}^m\times\mathscr{P}_2(\mathbb{R}^m)}|h|^{2}{\PPi}_{t}(dhdyd\rho )dt,
\end{equation*}
where the set
\begin{align*}
	\left.\Xi_{\varphi}:=\left\{\begin{aligned}{\PPi}:&[0,T]\to\mathscr{P}(\mathbb{R}^d\times\mathbb{R}^m\times\mathscr{P}_2(\mathbb{R}^m)): \text{there exists a stochastic kernel}\, \eta\,\text{such that}\\
		&\,\,\,\,\,\,\,\,\,\,{\PPi}_t(A_1\times A_2 \times A_3)=\int_{A_2\times A_3}\eta(A_1|y,\rho,t)(\nu^{\mathscr{L}_{\bar{X}_t}}\times \delta_{\nu^{\mathscr{L}_{\bar{X}_t}}})(dyd\rho),\\
		&\,\,\,\,\,\,\,\,\,\,\int_0^T\int_{\mathbb{R}^d\times\mathbb{R}^m\times\mathscr{P}_2(\mathbb{R}^m)}\big[|h|^2+|y|^4+M_4(\rho)\big]{\PPi}_t(dhdyd\rho )dt<\infty,\\
		&\,\,\,\,\,\,\,\,\,\,\varphi_t=x+\int_0^t\int_{\mathbb{R}^d\times\mathbb{R}^m\times\mathscr{P}_2(\mathbb{R}^m)}\Theta(\varphi_s,\mathscr{L}_{\bar{X}_s},h,y,\rho){\PPi}_s(dhdyd\rho )ds.\end{aligned}\right.\right\}.
\end{align*}
Moreover, we also define
\begin{equation}\label{z4.35}
	\tilde{I}(\varphi):=\inf_{z\in\tilde{\Xi}_\varphi}\frac{1}{2}\int_0^T\int_{\mathbb{R}^m}|z_t(y)|^2\nu^{\mathscr{L}_{\bar{X}_t}}(dy)dt,
\end{equation}
where the set
\begin{align*}
	\left.\tilde{\Xi}_{\varphi}:=\left\{\begin{aligned}z:~&[0,T]\times\mathbb{R}^m\to\mathbb{R}^d: \int_0^T\int_{\mathbb{R}^m}\big[|z_t(y)|^2+|y|^4\big]\nu^{\mathscr{L}_{\bar{X}_t}}(dy)dt<\infty,\\
		&\,\,\,\,\,\,\,\,\,\,\varphi_t=x+\int_0^t\int_{\mathbb{R}^m\times\mathscr{P}_2(\mathbb{R}^m)}\Theta(\varphi_s,\mathscr{L}_{\bar{X}_s},z_s (y),y,\rho)(\nu^{\mathscr{L}_{\bar{X}_t}}\times \delta_{\nu^{\mathscr{L}_{\bar{X}_t}}})(dyd\rho )ds.\end{aligned}\right.\right\}.
\end{align*}

\begin{lemma}
	$I(\varphi)=\tilde{I}(\varphi), ~\varphi\in C([0,T];\mathbb{R}^n)$.
\end{lemma}
\begin{proof}
	To begin with, for any $z \in \tilde{\Xi}_{\varphi}$, we define ${\PPi} \in {\Xi}_{\varphi}$ as follows
	\begin{equation*}
		{\PPi}_t(dhdyd\rho ):=\delta_{z_{t}(y)}(dh)(\nu^{\mathscr{L}_{\bar{X}_t}}\times \delta_{\nu^{\mathscr{L}_{\bar{X}_t}}})(dyd\rho ).
	\end{equation*}
	Then it is clear that
	\begin{equation*}
		I(\varphi)\leq\tilde{I}(\varphi).
	\end{equation*}
	Conversely, for any given ${\PPi} \in {\Xi}_{\varphi}$, we define
	\begin{equation*}
		z_t(y):=\int_{\mathbb{R}^d}h\eta(dh|y,\rho,t),
	\end{equation*}
	where $\eta(dh|y,\rho,t)$ stands for the conditional distribution associated with ${\PPi}$. Given that the mapping $\Theta$ is affine in the control $h$, it follows that $z \in \tilde{\Xi}_{\varphi}$. Utilizing Jensen's inequality, we obtain
	\begin{align*}
		&\int_{0}^{T}\int_{\mathbb{R}^d\times\mathbb{R}^m\times\mathscr{P}_2(\mathbb{R}^m)}|h|^{2}{\PPi}_{t}(dhdyd\rho )dt\\
		\geq&\int_{0}^{T}\int_{\mathbb{R}^m\times\mathscr{P}_2(\mathbb{R}^m)}\Big|\int_{\mathbb{R}^d}h\eta(dh|y,\rho,t)\Big|^{2}(\nu^{\mathscr{L}_{\bar{X}_t}}\times \delta_{\nu^{\mathscr{L}_{\bar{X}_t}}})(dyd\rho)dt   \\
		=&\int_0^T\int_{\mathbb{R}^m\times\mathscr{P}_2(\mathbb{R}^m)}|z_t(y)|^2(\nu^{\mathscr{L}_{\bar{X}_t}}\times \delta_{\nu^{\mathscr{L}_{\bar{X}_t}}})(dyd\rho)dt.
	\end{align*}
	Then we can deduce that
	\begin{equation*}
		I(\varphi)\geq\tilde{I}(\varphi).
	\end{equation*}
	Thus the lemma follows.
\end{proof}

The following lemma offers an explicit characterization of the infimization problem (\ref{z4.35}), playing a vital role in establishing the upper bound of  Laplace principle.
\begin{lemma}
	The control $z:[0,T]\times\mathbb{R}^m\to\mathbb{R}^d$ defined by
	\begin{align*}
		&z_t (y):=\pi_1^*\sigma ^*(\varphi_{t},\mathscr{L}_{\bar{X}_{t}},y,\nu^{\mathscr{L}_{\bar{X}_t}})Q_2^{-1}(\varphi_{t},\mathscr{L}_{\bar{X}_{t}},\nu^{\mathscr{L}_{\bar{X}_t}})(\dot{\varphi}_t-\bar{b}(\varphi_t,\mathscr{L}_{\bar{X}_t}))
	\end{align*}
	attains the infimum in $(\ref{z4.35})$, where
	\begin{equation*}
		Q_{2}(\varphi_{t},\mathscr{L}_{\bar{X}_{t}},\nu^{\mathscr{L}_{\bar{X}_t}}):=\int_{\mathbb{R}^m}\sigma(\varphi_{t},\mathscr{L}_{\bar{X}_{t}},y,\nu^{\mathscr{L}_{\bar{X}_t}}) \pi_1\pi_1^{*}\sigma^* (\varphi_{t},\mathscr{L}_{\bar{X}_{t}},y,\nu^{\mathscr{L}_{\bar{X}_t}}) \nu^{\mathscr{L}_{\bar{X}_t}}(dy).
	\end{equation*}
	Moreover, the infimization problem $(\ref{z4.35})$ admits the following explicit solution
	
	\begin{align}\label{z4.36}
		{\tilde{I}(\varphi)=\begin{cases}\frac{1}{2}\int_0^T|Q_2^{-1/2}(\varphi_{t},\mathscr{L}_{\bar{X}_t},\nu^{\mathscr{L}_{\bar{X}_t}})(\dot{\varphi}_t-\bar{b}(\varphi_t,\mathscr{L}_{\bar{X}_t}))|^2dt,&\varphi \in \mathscr{H} ,\\+\infty,&otherwise,\end{cases}}
	\end{align}
	where
	$\mathscr{H}:=\big\{ \varphi : \varphi_0=x, \varphi~\text{is absolutely continuous}\big\}$.
\end{lemma}
\begin{proof}
	Note that for any $z \in \tilde{\Xi}_{\varphi}$, we have
	\begin{align*}
		&\dot{\varphi}_{t}=\int_{\mathbb{R}^m\times\mathscr{P}_2(\mathbb{R}^m)}\Theta(\varphi_{t},\mathscr{L}_{\bar{X}_{t}},z_t(y),y,\rho)(\nu^{\mathscr{L}_{\bar{X}_t}}\times \delta_{\nu^{\mathscr{L}_{\bar{X}_t}}})(dyd\rho) \\
		&~~~=\bar{b}(\varphi_{t},\mathscr{L}_{\bar{X}_{t}})+\int_{\mathbb{R}^m\times\mathscr{P}_2(\mathbb{R}^m)}\sigma(\varphi_{t},\mathscr{L}_{\bar{X}_{t}},y,\rho)z_t^1(y) (\nu^{\mathscr{L}_{\bar{X}_t}}\times \delta_{\nu^{\mathscr{L}_{\bar{X}_t}}})(dyd\rho) \\
		&~~~=\bar{b}(\varphi_{t},\mathscr{L}_{\bar{X}_{t}})+\int_{\mathbb{R}^m}\sigma(\varphi_{t},\mathscr{L}_{\bar{X}_{t}},y,\nu^{\mathscr{L}_{\bar{X}_t}})z_t^1(y) \nu^{\mathscr{L}_{\bar{X}_t}}(dy), \\
		&{\varphi}_{0}=x,
	\end{align*}
	which $z_t^1(y) := \pi_1 z_t(y)$.
	Then applying H{\"o}lder inequality for integrals of matrices (cf.~\cite[Lemma 5.1]{MR2914778}) yields that for any $z \in \tilde{\Xi}_{\varphi}$,
	\begin{align*}
		&\int_0^T\int_{\mathbb{R}^m}|z_t(y)|^2 \nu^{\mathscr{L}_{\bar{X}_t}}(dy)dt\\
		\geq&\int_0^T(\dot{\varphi}_t-\bar{b}(\varphi_t,\mathscr{L}_{\bar{X}_t}))^*Q_2^{-1}(\varphi_{t},\mathscr{L}_{\bar{X}_t},\nu^{\mathscr{L}_{\bar{X}_t}})(\dot{\varphi}_t-\bar{b}(\varphi_t,\mathscr{L}_{\bar{X}_t}))dt.
	\end{align*}
	Furthermore, for any $t \in [0,T]$, if we take
	\begin{equation}\label{z4.37}
		z_t(y):=\pi_1^*\sigma^* (\varphi_{t},\mathscr{L}_{\bar{X}_{t}},y,\nu^{\mathscr{L}_{\bar{X}_t}})Q_2^{-1}(\varphi_{t},\mathscr{L}_{\bar{X}_{t}},\nu^{\mathscr{L}_{\bar{X}_t}})(\dot{\varphi}_t-\bar{b}(\varphi_t,\mathscr{L}_{\bar{X}_t})),
	\end{equation}
	then $z\in\tilde{\Xi}_\varphi $ and
	\begin{align*}
		&\int_0^T\int_{\mathbb{R}^m}|z_t(y)|^2\nu^{\mathscr{L}_{\bar{X}_t}}(dy)dt
		\\
		=&\int_0^T(\dot{\varphi}_t-\bar{b}(\varphi_t,\mathscr{L}_{\bar{X}_t}))^*Q_2^{-1}(\varphi_{t},\mathscr{L}_{\bar{X}_t},\nu^{\mathscr{L}_{\bar{X}_t}})(\dot{\varphi}_t-\bar{b}(\varphi_t,\mathscr{L}_{\bar{X}_t}))dt,
	\end{align*}
	which implies that (\ref{z4.36}) holds. The infimum of (\ref{z4.35}) is achieved in $h$ defined by (\ref{z4.37}).
\end{proof}

\vspace{3mm}
\noindent\textbf{Proof of  Laplace principle upper bound.} Our objective is to show that for any bounded and continuous function  $\Lambda:C([0,T]; \mathbb{R}^n)\to\mathbb{R}$,
\begin{align}
	\limsup_{\varepsilon\to0}\left\{-{\varepsilon}\log\mathbb{E}\left[\exp\left\{-\frac{1}{\varepsilon}\Lambda(X^\varepsilon)\right\}\right]\right\}&\leq\inf_{\varphi\in C([0,T];\mathbb{R}^n)}\left[I(\varphi)+\Lambda(\varphi)\right] \nonumber\\ &=\inf_{\varphi\in C([0,T];\mathbb{R}^n)}\left[\tilde{I}(\varphi)+\Lambda(\varphi)\right].\label{z4.38}
\end{align}
Note that for any $\eta > 0$, there exists $\psi \in C([0,T];\mathbb{R}^n)$ with $\psi_0 = x$ such that
\begin{equation}\label{z4.39}
	\tilde{I}(\psi)+\Lambda(\psi)\leq\inf_{\varphi\in C([0,T];\mathbb{R}^n)}\left[\tilde{I}(\varphi)+\Lambda(\varphi)\right]+\eta<\infty,
\end{equation}
and for each $z \in \tilde{\Xi}_{\psi}$,
\begin{equation}\label{z4.40}
	\psi_t=x+\int_0^t\int_{\mathbb{R}^m\times\mathscr{P}_2(\mathbb{R}^m)}\Theta(\psi_s,\mathscr{L}_{\bar{X}_s},z_s(y),y,\rho)(\nu^{\mathscr{L}_{\bar{X}_t}}\times \delta_{\nu^{\mathscr{L}_{\bar{X}_t}}})(dyd\rho)ds.
\end{equation}
Due to the boundedness of $\Lambda$, it is clear that  $I(\psi) < \infty$, which implies the absolute continuity of $\psi$ in terms of  the definition of $I$. For such a function $\psi$, we introduce $\bar{h}_t(y)$ as follows
\begin{equation*}
	\bar{h}_t (y):=\pi_1^*\sigma^*(\psi_{t},\mathscr{L}_{\bar{X}_{t}},y,\nu^{\mathscr{L}_{\bar{X}_t}})Q_2^{-1}(\psi_{t},\mathscr{L}_{\bar{X}_{t}},\nu^{\mathscr{L}_{\bar{X}_t}})(\dot{\psi}_t-\bar{b}(\psi_t,\mathscr{L}_{\bar{X}_t})),
\end{equation*}
It follows that $\bar{h}_{\cdot}(y)\in L^{2}([0,T];\mathbb{R}^d)$ uniformly in $y \in \mathbb{R}^m$. Employing a standard mollification argument, one may assume, without loss of generality, that
\begin{align}\label{z4.42}
	\bar{h}\text{~is Lipschitz continuous in}~t\in[0,T].
\end{align}
To be more precise, choosing a non-negative function $\chi \in C_0^{\infty}(\mathbb{R})$ supported within $\{r:|r|\leq1\}$, which satisfies $\int_\mathbb{R}\chi(r)dr=1$. Then, for every $k \geq 1$, we set $\chi_{k}(r):=k\chi(kr)$ and define
\begin{equation*}
	\bar{h}_t^k(y):=\int_\mathbb{R}\bar{h}_r(y)\chi_k(t-r)dr.
\end{equation*}
According to Subsection 4.3 of \cite{MR4013879} (specifically, the fourth and seventh formulas on Page 4764), the properties of convolution imply that for any $t_1,t_2 \in  [0,T],$
\begin{equation*}
	|\bar{h}_{t_1}^k(y)-\bar{h}_{t_2}^k(y)|\leq c_k|t_1-t_2|,~~y\in\mathbb{R}^m
\end{equation*}
and
\begin{equation*}
	\|\bar{h}_.^k(y)-\bar{h}_.(y)\|_{L^2([0,T];\mathbb{R}^d)}\to0,~k\to\infty,\text{~uniformly in~}y\in\mathbb{R}^m.
\end{equation*}
Moreover, from the assumption $(\mathbf{H}_3)$, we can deduce that
\begin{align}\label{z4.43}
	\bar{h}\text{~is Lipschitz continuous and bounded in~}y\in\mathbb{R}^m,
\end{align}
Consequently, based on (\ref{z4.42}) and (\ref{z4.43}), we infer that the same properties also satisfied for the function
\begin{equation*}
	\varpi(\cdot,\cdot):=|\bar{h}_.(\cdot)|^2:[0,T]\times\mathbb{R}^m\to\mathbb{R}.
\end{equation*}

Now we define a control in feedback form by
\begin{equation*}
	\bar h_t^\varepsilon:=\bar h_t(\hat{Y}_t^\varepsilon),
\end{equation*}
where $\hat{Y}^\varepsilon$ is the solution of the following auxiliary equation
\begin{align*}
	d\hat{Y}_{t}^{\varepsilon}=\frac{1}{\delta}f(\mathscr{L}_{X^{\varepsilon}_{t(\Delta)}},\hat{Y}_{t}^{\varepsilon},\mathscr{L}_{\hat{Y}_{t}^{\varepsilon}})dt+\frac{1}{\sqrt{\delta}}g( \mathscr{L}_{X^{\varepsilon}_{t(\Delta)}},\hat{Y}_{t}^{\varepsilon},\mathscr{L}_{\hat{Y}_{t}^{\varepsilon}})dW_t^2,\;\;\hat{Y}_{0}^{\varepsilon}=y.
\end{align*}
Following from the proof of Theorem \ref{theo1}, we can easily derive the following convergence result
\begin{equation*}
	\lim\limits_{\varepsilon\to0}\mathbb{E}\int_0^TF(t,\hat{Y}_t^\varepsilon,\mathscr{L}_{\hat{Y}_t^\varepsilon})dt=\int_0^T\int_{\mathbb{R}^m\times\mathscr{P}_2(\mathbb{R}^m)}F(t,y,\rho)(\nu^{\mathscr{L}_{\bar{X}_t}}\times\delta_{\nu^{\mathscr{L}_{\bar{X}_t}}})(dyd\rho)dt,
\end{equation*}
where $F:[0,T]\times \mathbb{R}^m\times \mathscr{P}_2(\mathbb{R}^m)\to \mathbb{R}$ is Lipschitz continuous in all three variables. Therefore, choosing $F(t,y,\rho):= \varpi(t,y)$,  we  have
\begin{equation}\label{z4.44}
	\lim\limits_{\varepsilon\to0}\mathbb{E}\int_0^T\varpi(t,\hat{Y}_t^\varepsilon)dt=\int_0^T\int_{\mathbb{R}^m}\varpi(t,y)\nu^{\mathscr{L}_{\bar{X}_t}}(dy)dt.
\end{equation}
Furthermore, let $\psi\in C([0,T];\mathbb{R}^n)$ be the unique solution to the control problem (\ref{z4.40}) associated with the control $\bar h_t (y)$. It follows that
\begin{equation}\label{z4.45}
	X^{\varepsilon,\bar{h}^\varepsilon}\xrightarrow{\varepsilon\to0}\psi\quad\text{in~}C([0,T];\mathbb{R}^n)~~~\mathbb{P}\text{-a.s.},
\end{equation}
whose proof is left in Subsection \ref{s8.1} in Appendix.

Now we focus on proving (\ref{z4.38}). It follows
\begin{align}
&\limsup_{\varepsilon\to0}\left\{-{\varepsilon}\operatorname{log}\mathbb{E}\left[\operatorname{exp}\left\{-\frac{1}{\varepsilon}\Lambda(X^{\varepsilon})\right\}\right]\right\}  \nonumber\\
=&\limsup_{\varepsilon\to0}\operatorname*{inf}_{h\in\mathcal{A}}\mathbb{E}\left[\frac{1}{2}\int_{0}^{T}|h_{s}|^{2}ds+\Lambda(X^{\varepsilon,h})\right]  \nonumber\\
\leq&\limsup_{\varepsilon\to0}\mathbb{E}\left[\frac{1}{2}\int_{0}^{T}|\bar{h}_{s}^{\varepsilon}|^{2}ds+\Lambda(X^{\varepsilon,\bar{h}^{\varepsilon}})\right]   \nonumber\\
=&\mathbb{E}\left[\frac{1}{2}\int_{0}^{T}\int_{\mathbb{R}^m}|\bar{h}_{s}(y)|^{2}\nu^{\mathscr{L}_{\bar{X}_t}}(dy)ds+\Lambda(\psi)\right]  \nonumber\\
=&I(\psi)+\Lambda(\psi) \nonumber\\
\leq&\operatorname*{inf}_{\varphi\in C([0,T];\mathbb{R}^{n})}\left[I(\varphi)+\Lambda(\varphi)\right]+\eta, \label{z4.46}
	\end{align}
where we used (\ref{z4.44}) and (\ref{z4.45}) in the third step, (\ref{z4.36}) in the fourth step, and (\ref{z4.39}) in the last step.  Since $\eta$ is arbitrary, we complete the proof of the upper bound of Laplace principle.\hspace{\fill}$\Box$

\subsection{Compactness of level sets of $I(\cdot)$}\label{sub5.6}

Our aim is to verify that for each $s<\infty$, the level set
\begin{equation*}
	\Gamma(s):=\Big\{\varphi\in C([0,T];\mathbb{R}^n):I(\varphi)\leq s\Big\}
\end{equation*}
is a compact subset of $C([0,T];\mathbb{R}^n)$. More precisely, in Lemma $\ref{l4.6}$ we show the pre-compactness of $\Gamma(s)$, and in Lemma $\ref{l4.8}$ we demonstrate that it is closed. Then the claim follows.
\begin{lemma}\label{l4.6}
	Fix $K < \infty $. Consider any sequence $\{(\varphi^{k},{\PPi}^{k})\}_{k\in\mathbb{N}}\subset\mathscr{V}_{(x,\bar{X},\Theta,\nu^{\mathscr{L}_{\bar{X}}})}$ such that
	\begin{equation}\label{z4.31}
		\sup_{k\in \mathbb{N}}\int_{\mathbb{R}^d\times\mathbb{R}^m\times\mathscr{P}_2(\mathbb{R}^m)\times[0,T]}\big(|h|^2+|y|^4+M_4(\rho)\big){\PPi}^k(dhdyd\rho dt)\leq K.
	\end{equation}
	Then $\{(\varphi^{k},{\PPi}^{k})\}_{k\in\mathbb{N}}$ is pre-compact.
\end{lemma}
\begin{proof}
	Note that for any $0\leq t_{1}<t_{2}\leq T$ and  $k\in \mathbb{N}$,
	\begin{align*}
		&|\varphi_{t_2}^k-\varphi_{t_1}^k|
		  \\ =&\Big|\int_{\mathbb{R}^d\times\mathbb{R}^m\times\mathscr{P}_2(\mathbb{R}^m)\times[t_{1},t_{2}]}\Theta(\varphi_{s}^{k},\mathscr{L}_{\bar{X}_{s}},h,y,\rho){\PPi}^{k}(dhdyd\rho ds)\Big|   \\
		\lesssim& \int_{\mathbb{R}^d\times\mathbb{R}^m\times\mathscr{P}_2(\mathbb{R}^m)\times[t_1,t_2]}|\bar{b}(\varphi_s^k,\mathscr{L}_{\bar{X}_s})+\sigma(\varphi_s^k,\mathscr{L}_{\bar{X}_s},y,\rho)\pi_1h|{\PPi}^k(dhdyd\rho ds) \\ \lesssim&_{M} (t_2-t_1)^{\frac{1}{2}}+\bigg(\int_{\mathbb{R}^d\times\mathbb{R}^m\times\mathscr{P}_2(\mathbb{R}^m)\times[t_1,t_2]}\big(1+|\varphi_s^k|^2+|\bar{X}_{s}|^{2}+|y|^2+M_2(\rho)\big){\PPi}^k(dhdyd\rho ds)\bigg)^{\frac{1}{2}} \\
		\lesssim&_{M} (t_2-t_1)^{\frac{1}{2}}+(t_2-t_1)^{\frac{1}{4}}\bigg(\int_{\mathbb{R}^d\times\mathbb{R}^m\times\mathscr{P}_2(\mathbb{R}^m)\times[t_1,t_2]}\big(|y|^4+M_4(\rho)\big){\PPi}^k(dhdyd\rho ds)\bigg)^{\frac{1}{4}} \\
		\lesssim&_{M} (t_2-t_1)^{\frac{1}{2}}+(t_2-t_1)^{\frac{1}{4}}.
	\end{align*}
	This along with the fact that  $\varphi_0^k = x$ gives the pre-compactness of $\{\varphi^{k}\}_{k\in\mathbb{N}}$  according to the Arzel\`{a}-Ascoli theorem.
	
	The pre-compactness of $\{{\PPi}^{k}\}_{k\in\mathbb{N}}$ follows from  (\ref{z4.31}), employing exactly same argument as in Lemma \ref{p4.1}.
\end{proof}

To establish the closedness of the level set $\Gamma(s)$, we invoke a key lemma asserting that the limit of any sequence of viable pairs remains viable.
\begin{lemma}\label{l4.7}
	Fix $K < \infty $. Consider any sequence $\{(\varphi^{k},{\PPi}^{k})\}_{k\in\mathbb{N}}$ such that for every $k\in \mathbb{N}$, $(\varphi^{k},{\PPi}^{k})$ is viable and
	\begin{equation}\label{z4.32}
		\int_{\mathbb{R}^d\times\mathbb{R}^m\times\mathscr{P}_2(\mathbb{R}^m)\times[0,T]}\big(|h|^2+|y|^4+M_4(\rho)\big){\PPi}^k(dhdyd\rho dt)\leq K.
	\end{equation}
	Then the limit $(\varphi,{\PPi})$ is a viable pair.
\end{lemma}
\begin{proof}
	Since $(\varphi^{k},{\PPi}^{k})$ is viable, we know that for every $t \in [0,T]$,
	\begin{equation*}
		\varphi_{t}^k =x+\int_{\mathbb{R}^d\times\mathbb{R}^m\times\mathscr{P}_2(\mathbb{R}^m)\times[0,t]}\Theta(\varphi_{s}^{k},\mathscr{L}_{\bar{X}_{s}},h,y,\rho){\PPi}^{k}(dhdyd\rho ds)
	\end{equation*}
	where
	\begin{equation*}
		{\PPi}^k(dhdyd\rho dt)=\eta^k(dh|y,\rho,t)(\nu^{\mathscr{L}_{\bar{X}_t}}\times \delta_{\nu^{\mathscr{L}_{\bar{X}_t}}})(dyd\rho)dt,
	\end{equation*}
	and $\eta^k$ is a sequence of stochastic kernels.
	
	First, by applying Fatou's lemma, it is clear that ${\PPi}$ satisfies statement (i) in Definition \ref{d2.4}. Additionally, we note that the function $\Theta(\varphi,\mu,h,y,\rho)$ is continuous with respect to $\varphi,\mu,y,\rho$ and is affine in $h$. We consider the  difference
	\begin{align*}
		&\Big|\int_{\mathbb{R}^d\times\mathbb{R}^m\times\mathscr{P}_2(\mathbb{R}^m)\times[0,t]}\Theta(\varphi_{s}^{k},\mathscr{L}_{\bar{X}_{s}},h,y,\rho){\PPi}^{k}(dhdyd\rho ds)\\
		&-\int_{\mathbb{R}^d\times\mathbb{R}^m\times\mathscr{P}_2(\mathbb{R}^m)\times[0,t]}\Theta(\varphi_{s},\mathscr{L}_{\bar{X}_{s}},h,y,\rho){\PPi}(dhdyd\rho ds)\Big|\\
		\leq\,\,& \Big|\int_{\mathbb{R}^d\times\mathbb{R}^m\times\mathscr{P}_2(\mathbb{R}^m)\times[0,t]}\Theta(\varphi_{s}^{k},\mathscr{L}_{\bar{X}_{s}},h,y,\rho){\PPi}^{k}(dhdyd\rho ds)\\
		&-\int_{\mathbb{R}^d\times\mathbb{R}^m\times\mathscr{P}_2(\mathbb{R}^m)\times[0,t]}\Theta(\varphi_{s},\mathscr{L}_{\bar{X}_{s}},h,y,\rho){\PPi}^{k}(dhdyd\rho ds)\Big|\\
		&+\Big|\int_{\mathbb{R}^d\times\mathbb{R}^m\times\mathscr{P}_2(\mathbb{R}^m)\times[0,t]}\Theta(\varphi_{s},\mathscr{L}_{\bar{X}_{s}},h,y,\rho){\PPi}^{k}(dhdyd\rho ds)\\
		&-\int_{\mathbb{R}^d\times\mathbb{R}^m\times\mathscr{P}_2(\mathbb{R}^m)\times[0,t]}\Theta(\varphi_{s},\mathscr{L}_{\bar{X}_{s}},h,y,\rho){\PPi}(dhdyd\rho ds)\Big|\\
		\eqqcolon\,\,&I_{k}+II_{k}.
	\end{align*}
	For the term $I_{k}$, due to the fact that $\varphi^k \to \varphi$, together with the continuity of $\Theta(\varphi,\mu,h,y,\rho)$ with respect to $\varphi, \mu, y, \rho$ and its affine dependence on $h$, it follows that $I_{k} \to 0$ as $k \to \infty$. To verify the convergence of the term $II_{k}$, it suffices to establish the uniform integrability of $\Theta(\varphi,\mathscr{L}_{\bar{X}},h,y,\rho)$ w.r.t.~the family $\{{\PPi}^{k}\}_{k\in\mathbb{N}}$.  To this end, we consider
	\begin{align*}
		&\int_{\mathbb{R}^d\times\mathbb{R}^m\times\mathscr{P}_2(\mathbb{R}^m)\times[0,t]}|\Theta(\varphi_{s},\mathscr{L}_{\bar{X}_{s}},h,y,\rho)|^{\frac{5}{4}}{\PPi}^{k}(dhdyd\rho ds)\\
		\lesssim& \int_{\mathbb{R}^d\times\mathbb{R}^m\times\mathscr{P}_2(\mathbb{R}^m)\times[0,t]}|\bar{b}(\varphi_{s},\mathscr{L}_{\bar{X}_{s}})|^{\frac{5}{4}}{\PPi}^{k}(dhdyd\rho ds)\\
		&+\int_{\mathbb{R}^d\times\mathbb{R}^m\times\mathscr{P}_2(\mathbb{R}^m)\times[0,t]}|\sigma(\varphi_{s},\mathscr{L}_{\bar{X}_{s}},y,\rho)\pi_1h|^{\frac{5}{4}}{\PPi}^{k}(dhdyd\rho ds)\\
		\lesssim&1+\bigg(\int_{\mathbb{R}^d\times\mathbb{R}^m\times\mathscr{P}_2(\mathbb{R}^m)\times[0,t]}\|\sigma(\varphi_{s},\mathscr{L}_{\bar{X}_{s}},y,\rho)\|^4{\PPi}^{k}(dhdyd\rho ds)\bigg)^{\frac{5}{16}}\\
		&\cdot\bigg(\int_{\mathbb{R}^d\times\mathbb{R}^m\times\mathscr{P}_2(\mathbb{R}^m)\times[0,t]}|h|^{\frac{20}{11}}{\PPi}^{k}(dhdyd\rho ds)\bigg)^{\frac{11}{16}}\\
		\lesssim&_{M}1+\bigg(\int_{\mathbb{R}^d\times\mathbb{R}^m\times\mathscr{P}_2(\mathbb{R}^m)\times[0,t]} \big(|y|^4 +M_4(\rho)\big){\PPi}^{k}(dhdyd\rho ds)\bigg)^{\frac{5}{16}}\\
		<& \infty,
	\end{align*}
	where we utilized the assumption (\ref{z4.32}) in the last step. Consequently, due to the weak convergence ${\PPi}^k \Rightarrow {\PPi}$, by Vitali's convergence theorem we have $II_{k} \to 0$. We thus obtain (\ref{2.16}) for the limit pair $(\varphi,{\PPi})$. From the same reason,  it is straightforward  that ${\PPi}$ satisfies (\ref{2.17}).
\end{proof}

\begin{lemma}\label{l4.8}
	The rate function $I$ is lower semicontinuous.
\end{lemma}
\begin{proof}
	Consider a sequence $\{\varphi^k\}$ with limit $\varphi$. We shall prove
	\begin{equation*}
		\liminf_{k\to\infty}I(\varphi^k)\geq I(\varphi).
	\end{equation*}
	It suffices to consider the case that $I(\varphi^k)$ has a finite limit. That is, we assume there a constant $M < \infty$ such that $\lim_{k\to\infty}I(\varphi^k)\leq M$. Let us recall the definition
	\begin{equation*}
		I(\varphi):=\inf_{(\varphi,{\PPi})\in\mathscr{V}_{(x,\bar{X},\Theta,\nu^{\mathscr{L}_{\bar{X}}})}}\left\{\frac{1}{2}\int_{\mathbb{R}^d\times\mathbb{R}^m\times\mathscr{P}_2(\mathbb{R}^m)\times[0,T]}|h|^{2} {\PPi}(dhdyd\rho dt)\right\}.
	\end{equation*}
	Therefore, there exists a sequence $\{{\PPi}^{k}\}_{k\in\mathbb{N}}$ such that $(\varphi^{k},{\PPi}^{k})\subset\mathscr{V}_{(x,\bar{X},\Theta,\nu^{\mathscr{L}_{\bar{X}}})}$,
	\begin{align*}
		&\sup_{k\in\mathbb{N}}\frac{1}{2}\int_{\mathbb{R}^d\times\mathbb{R}^m\times\mathscr{P}_2(\mathbb{R}^m)\times[0,T]}|h|^{2}{\PPi}^{k}(dhdyd\rho dt) \\
		\leq\,\,\,& M+1+\frac{1}{2}\int_{\mathbb{R}^m}|y|^{4}\nu^{\mathscr{L}_{\bar{X}_t}} (dy)+\frac{1}{2}\int_{\mathscr{P}_2(\mathbb{R}^m)}M_4(\rho)\delta_{\nu^{\mathscr{L}_{\bar{X}_t}}}(d\rho ),
	\end{align*}
	and
	\begin{equation*}
		I(\varphi^k)\geq\frac{1}{2}\int_{\mathbb{R}^d\times\mathbb{R}^m\times\mathscr{P}_2(\mathbb{R}^m)\times[0,T]}|h|^2{\PPi}^k(dhdyd\rho dt)-\frac{1}{k}.
	\end{equation*}
	Due to (\ref{es:IPM}), we know that $\int_{\mathbb{R}^m}|y|^{4}\nu^{\mathscr{L}_{\bar{X}_t}}(dy)<\infty$ and $\int_{\mathscr{P}_2(\mathbb{R}^m)}M_4(\rho)\delta_{\nu^{\mathscr{L}_{\bar{X}_t}}}(d\rho) < \infty$. Hence there exists a constant $M'> 0$ such that
	\begin{equation*}
		\sup_{k\in\mathbb{N}}\frac{1}{2}\int_{\mathbb{R}^d\times\mathbb{R}^m\times\mathscr{P}_2(\mathbb{R}^m)\times[0,T]}|h|^{2}{\PPi}^{k}(dhdyd\rho dt)\leq M'.
	\end{equation*}
	According to Lemma \ref{l4.6}, we can consider a subsequence along which $(\varphi^{k},{\PPi}^{k})$ converges to a limit $(\varphi,{\PPi})$. From Lemma \ref{l4.7}, we know that $(\varphi,{\PPi})$ is viable. Therefore, by Fatou's lemma
	\begin{align*}
		\liminf_{k\to\infty} I(\varphi^{k}) \geq&\liminf_{k\to\infty}\left(\frac{1}{2}\int_{\mathbb{R}^d\times\mathbb{R}^m\times\mathscr{P}_2(\mathbb{R}^m)\times[0,T]}|h|^{2}{\PPi}^{k}(dhdyd\rho dt)-\frac{1}{k}\right)  \\
		\geq&\frac{1}{2}\int_{\mathbb{R}^d\times\mathbb{R}^m\times\mathscr{P}_2(\mathbb{R}^m)\times[0,T]}|h|^{2}{\PPi}(dhdyd\rho dt)  \\
		\geq&\inf_{(\varphi,{\PPi})\in\mathscr{V}_{(x,\bar{X},\Theta,\nu^{\mathscr{L}_{\bar{X}}})}}\left\{\frac{1}{2}\int_{\mathbb{R}^d\times\mathbb{R}^m\times\mathscr{P}_2(\mathbb{R}^m)\times[0,T]}|h|^{2}{\PPi}(dhdyd\rho dt)\right\}  \\
		=&I(\varphi),
	\end{align*}
	which concludes the proof of lower-semicontinuity of $I$.
\end{proof}

\begin{appendix}
\section{Proof of theorem~\ref{thm:wellposedness}}\label{appen1}
To prove the well-posedness, we treat $(X^{\varepsilon}, Y^{\varepsilon})$ as a system on product space $\RR^n \times \mathbb{R}^m$. Let
\begin{align*}
	Z_t^{\varepsilon}\coloneqq\begin{pmatrix}
		X_t^{\varepsilon}\\
		Y_t^{\varepsilon}
	\end{pmatrix}\,,\,\,\,\,\,\,
	\rho\coloneqq\begin{pmatrix}
		\mu\\
		\nu
	\end{pmatrix}\,,\,\,\,\,\,\,
	z\coloneqq\begin{pmatrix}
		x\\
		y
	\end{pmatrix}\,,\,\,\,\,\,\,
	\mathcal{W}_t\coloneqq\begin{pmatrix}
		W_t^1\\
		W_t^2
	\end{pmatrix}\,,
\end{align*}
and
\begin{align*}
	\tilde{b}^{\varepsilon}(z,\rho):=\begin{pmatrix}
		b(x,\mu,y,\nu)\\
		\frac{1}{\delta}f(\mu,y,\nu)
	\end{pmatrix}\,,\,\,\,\,\,\,\,\,\,
	\tilde{\sigma}^{\varepsilon}(z, \rho):=\text{diag}\Big(\sqrt{\varepsilon}\sigma(x,\mu,y,\nu),\frac{1}{\sqrt{\delta}}g(\mu,y,\nu)\Big).
\end{align*}
The system $(\ref{1.1a})$-$(\ref{1.1b})$  can be rewritten as the following equation
\begin{align*}
	dz_t= \tilde{b}^{\varepsilon}(Z_t^{\varepsilon}, \mathscr{L}_{Z_t^{\varepsilon}})dt+\tilde{\sigma}^{\varepsilon}(Z_t^{\varepsilon}, \mathscr{L}_{Z_t^{\varepsilon}})d\mathcal{W}_t\,,\,\,\,\,\,Z_0^{\varepsilon}=\begin{pmatrix}
		\xi\\
		\zeta
	\end{pmatrix}.
\end{align*}
We aim to prove the coefficients $\tilde{b}^{\varepsilon}$ and $\tilde{\sigma}^{\varepsilon}$ fulfills the conditions $(\mathbf{A_0})$-$(\mathbf{A_4})$ and $(\mathbf{A'_5})$ with $\mathbb{X}=\mathbb{V}=\mathbb{H}=\mathbb{R}^{n+m}$ given in \cite[Theorem 2.8]{hong2025meanfieldstochasticpartial}.  The conditions $(\mathbf{A_0})$ and $(\mathbf{A_1})$ in \cite[Theorem 2.8]{hong2025meanfieldstochasticpartial} are easily satisfied in the finite-dimensional space.

Under conditions $(\mathbf{A}_{\text{slow}})$ and $(\mathbf{A}_{\text{fast}})$, for any $z_1, z_2\in \RR^{n+m}$, $\rho_1, \rho_2\in \mathscr{P}_{\kappa}(\RR^{n+m})$ we have
\begin{align}
	&2\langle \tilde{b}^{\varepsilon}(z_1, \rho_1)-\tilde{b}^{\varepsilon}(z_2, \rho_2), z_1-z_2\rangle+\|\tilde{\sigma}^{\varepsilon}(z_1)-\tilde{\sigma}^{\varepsilon}(z_2)\|^2\nonumber\\
	\leq \,\,\,& 2 |b(x_1, \mu_1, y_1, \nu_1)-b(x_2, \mu_2, y_2, \nu_2)|\cdot |x_1-x_2|\nonumber\\
	&+\varepsilon \|\sigma(x_1, \mu_1, y_1, \nu_1)-\sigma(x_2, \mu_2, y_2, \nu_2)\|^2\nonumber\\
	& +\frac{2}{\delta} |f(\mu_1, y_1, \nu_1)-f(\mu_2, y_1, \nu_1)|\cdot |y_1-y_2|+\frac{1}{\delta} \|g(\mu_1, y_1, \nu_1)-g(\mu_2, y_1, \nu_1)\|^2\nonumber\\
	& +\frac{2}{\delta} \langle f(\mu_2, y_1, \nu_1)-f(\mu_2, y_2, \nu_2), y_1-y_2\rangle +\frac{1}{\delta} \|g(\mu_2, y_1, \nu_1)-g(\mu_2, y_2, \nu_2)\|^2\nonumber\\
	\lesssim \,\,\,& \big(1+\ca{M}_{\kappa}(\mu_1,\mu_2,\nu_1,\nu_2)\big)\big[|x_1-x_2|^2+|y_1-y_2|^2\big]
	\nonumber\\
	& + \big(1+\rho_{\kappa}(x_1, x_2,y_1,y_2)+\ca{M}_{\kappa}(\mu_1, \mu_2, \nu_1, \nu_2)\big)\big[\mathbb{W}_2(\mu_1, \mu_2)^2+\mathbb{W}_2(\nu_1, \nu_2)^2\big]\nonumber\\
	&+  \big(1+\rho_{\kappa}(0,0, y_1, y_1)+\ca{M}_{\kappa}(\mu_1,\mu_2, \nu_1, \nu_1)\big)
	\mathbb{W}_2(\mu_1, \mu_2)^2+|y_1-y_2|^2+ \bb{W}_2(\nu_1, \nu_2)^2\nonumber\\
	\lesssim \,\,\,& \big(1+\ca{M}_{\kappa}(\mu_1,\mu_2,\nu_1,\nu_2)\big)|z_1-z_2|^2
\nonumber\\
&+\big(1+\rho_{\kappa}(x_1, x_2,y_1,y_2)+\ca{M}_{\kappa}(\mu_1, \mu_2, \nu_1, \nu_2)\big)\mathbb{W}_2(\rho_1, \rho_2)^2\,,\nonumber
\end{align}
and
\begin{align}
	\langle b(z,\rho), z \rangle&\lesssim   1+|z|^2+M_2(\rho)\,,\nonumber\\
	|b(z, \rho)|^2&\lesssim 1+|z|^{\kappa}+M_{\kappa}(\rho)\,,\nonumber\\
	|\sigma(z, \rho)|^2&\lesssim 1+|z|^{2}+M_{2}(\rho).\nonumber
\end{align}
Therefore, the conditions $(\mathbf{A_2})$-$(\mathbf{A_4})$ and $(\mathbf{A'_5})$  in \cite[Theorem 2.8]{hong2025meanfieldstochasticpartial} are satisfied. \hspace{\fill}$\Box$
\section{Proof of (\ref{z4.45})}\label{s8.1}
We first recall
\begin{align*}
	X_t^{\varepsilon,\bar{h}^{\varepsilon}}=&x+\int_{0}^{t}b(X_s^{\varepsilon,\bar{h}^{\varepsilon}},\mathscr{L}_{X_s^{\varepsilon}},Y_s^{\varepsilon,\bar{h}^{\varepsilon}},\mathscr{L}_{Y_s^{\varepsilon}})ds+\int_{0}^{t}\sigma(X_s^{\varepsilon,\bar{h}^{\varepsilon}},\mathscr{L}_{X_s^{\varepsilon}},Y_s^{\varepsilon,\bar{h}^{\varepsilon}},\mathscr{L}_{Y_s^{\varepsilon}})\bar{h}_s^{1}(Y_s^\varepsilon)ds  \\
	&+\sqrt{\varepsilon}\int_{0}^{t}\sigma(X_s^{\varepsilon,\bar{h}^{\varepsilon}},\mathscr{L}_{X_s^{\varepsilon}},Y_s^{\varepsilon,\bar{h}^{\varepsilon}},\mathscr{L}_{Y_s^{\varepsilon}})dW_s^1  \\
	=:&x+\sum_{i=1}^3\mathcal{O}_i^\varepsilon(t)
\end{align*}
and
\begin{align*}
	\psi_{t}=&x+\int_0^t\bar{b}(\psi_{s},\mathscr{L}_{\bar{X}_{s}})ds \\
	&+\int_0^t\int_{\mathbb{R}^m\times\mathscr{P}_2(\mathbb{R}^m)}\sigma(\psi_{s},\mathscr{L}_{\bar{X}_{s}},y,\nu)\bar{h}_s^1(y)(\nu^{\mathscr{L}_{\bar{X}_t}}\times \delta_{\nu^{\mathscr{L}_{\bar{X}_t}}})(dyd\nu)ds,
\end{align*}
where $\bar{h}_s^1(y) := \pi_1 \bar{h}_s(y)$.

Then we have
\begin{align*}
	&X_t^{\varepsilon,\bar{h}^{\varepsilon}} - \psi_{t}  \\
	= &\mathcal{O}_1^\varepsilon(t) - \int_0^t\bar{b}(\psi_{s},\mathscr{L}_{\bar{X}_{s}})ds + \mathcal{O}_3^\varepsilon(t) \\
	&+\mathcal{O}_2^\varepsilon(t) - \int_0^t\int_{\mathbb{R}^m\times\mathscr{P}_2(\mathbb{R}^m)}\sigma(\psi_{s},\mathscr{L}_{\bar{X}_{s}},y,\nu)\bar{h}_s^1(y)(\nu^{\mathscr{L}_{\bar{X}_t}}\times \delta_{\nu^{\mathscr{L}_{\bar{X}_t}}})(dyd\nu)ds.
\end{align*}

On the one hand, it is  straightforward that
\begin{align*}
	\mathbb{E}\Big[\sup_{t\in[0,T]}|\mathcal{O}_3^\varepsilon(t)|\Big]\lesssim_T \varepsilon^{\frac{1}{2}}.
\end{align*}
Moreover, employing similar argument as in Lemma \ref{l4.2}, we derive
\begin{align*}
	\mathbb{E}\left[\sup_{t\in[0,T]}\Big|\mathcal{O}_1^\varepsilon(t)- \int_0^t\bar{b}(\psi_{s},\mathscr{L}_{\bar{X}_{s}})ds\Big|\right] \lesssim_{M,T}\mathbb{E} \int_0^T|X_s^{\varepsilon,\bar{h}^\varepsilon}-\psi_s|ds +\gamma_1(\varepsilon),
\end{align*}
where $\gamma_1(\varepsilon)$ is a function satisfying  $\gamma_1(\varepsilon) \to 0$, as $\varepsilon \to 0$.
On the other hand,  we have
\begin{align*}
	&\mathbb{E}\bigg[\sup_{t\in[0,T]}\bigg|\mathcal{O}_2^\varepsilon(t)-\int_0^t\int_{\mathbb{R}^m\times\mathscr{P}_2(\mathbb{R}^m)}\sigma(\psi_{s},\mathscr{L}_{\bar{X}_{s}},y,\nu)\bar{h}_s^1(y)(\nu^{\mathscr{L}_{\bar{X}_t}}\times \delta_{\nu^{\mathscr{L}_{\bar{X}_t}}})(dyd\nu)ds\bigg|\bigg] \nonumber\\
\lesssim& \mathbb{E}\bigg[\sup_{t\in[0,T]}\bigg|\mathcal{O}_2^\varepsilon(t)-\int_{0}^{t}\sigma(X_s^{\varepsilon,\bar{h}^{\varepsilon}},\mathscr{L}_{X_s^{\varepsilon}},Y_s^{\varepsilon},\mathscr{L}_{Y_s^{\varepsilon}})\bar{h}_s^{1}(Y_s^\varepsilon)ds\bigg|\bigg] \nonumber\\
&+\mathbb{E}\bigg[\sup_{t\in[0,T]}\bigg|\int_{0}^{t}\sigma(X_s^{\delta,\bar{h}^{\delta}},\mathscr{L}_{X_s^{\varepsilon}},Y_s^{\varepsilon},,\mathscr{L}_{Y_s^{\varepsilon}})\bar{h}_s^{1}(Y_s^\varepsilon)ds \nonumber\\
&-\int_0^t\int_{\mathbb{R}^m\times\mathscr{P}_2(\mathbb{R}^m)}\sigma(X_s^{\delta,\bar{h}^{\delta}},\mathscr{L}_{X_s^{\varepsilon}},y,\nu)\bar{h}_s^1(y)(\nu^{\mathscr{L}_{\bar{X}_t}}\times \delta_{\nu^{\mathscr{L}_{\bar{X}_t}}})(dyd\nu)ds\bigg|\bigg] \nonumber\\
&+\mathbb{E}\bigg[\sup_{t\in[0,T]}\bigg|\int_0^t\int_{\mathbb{R}^m\times\mathscr{P}_2(\mathbb{R}^m)}\sigma(X_s^{\delta,\bar{h}^{\delta}},\mathscr{L}_{X_s^{\varepsilon}},y,\nu)\bar{h}_s^1(y)(\nu^{\mathscr{L}_{\bar{X}_t}}\times \delta_{\nu^{\mathscr{L}_{\bar{X}_t}}})(dyd\nu)ds\bigg] \nonumber\\
&-\int_0^t\int_{\mathbb{R}^m\times\mathscr{P}_2(\mathbb{R}^m)}\sigma(\psi_{s},\mathscr{L}_{\bar{X}_{s}},y,\nu)\bar{h}_s^1(y)(\nu^{\mathscr{L}_{\bar{X}_t}}\times \delta_{\nu^{\mathscr{L}_{\bar{X}_t}}})(dyd\nu)ds\bigg|\nonumber \\
=:&\sum_{i=1}^3\mathcal{O}_{2i}^\varepsilon(T). \label{z8.2}
\end{align*}

According to Lemma \ref{l4.3}, we can easily get that
\begin{align*}
	\mathcal{O}_{21}^\varepsilon(T) \lesssim \gamma_2(\varepsilon),
\end{align*}
where $\gamma_2(\varepsilon)$ is a function satisfying  $\gamma_2(\varepsilon) \to 0$, as $\varepsilon \to 0$. Since $\sigma(x,\mu,y,\nu)$ is locally Lipschitz continuous and linear growth condition and $\bar{h}(y)$ is Lipschitz continuous
and bounded in $y \in \mathbb{R}^m$, we conclude that  $\sigma(x,\mu,y,\nu)\bar{h}(y)$ is also local Lipschitz continuous w.r.t.~$(x,\mu, y,\nu)$. Consequently, from the same argument as in Lemma \ref{l4.1} we can deduce
\begin{align*}
	\mathcal{O}_{22}^\varepsilon(T) \lesssim \gamma_3(\varepsilon),
\end{align*}
where $\gamma_3(\varepsilon)$ is a function satisfying  $\gamma_3(\varepsilon) \to 0$, as $\varepsilon \to 0$. As for the term $\mathcal{O}_{23}^\varepsilon(T)$, following the same argument as \cite[(3.38)]{MR4634338}, one can easily get
\begin{align*}
	\mathcal{O}_{23}^\varepsilon(T) \lesssim_{M,T}\mathbb{E}\int_0^T|X_s^{\varepsilon,\bar{h}^\varepsilon}-\psi_s|ds+\gamma_4(\varepsilon).
\end{align*}
where $\gamma_4(\varepsilon)$ is a function satisfying  $\gamma_4(\varepsilon) \to 0$, as $\varepsilon \to 0$.

Combining the arguments above, using Gronwall's inequality, we deduce that (\ref{z4.45}) holds.\hspace{\fill}$\Box$
\end{appendix}

\vspace{0.2cm}

\noindent\textbf{\large{Declarations}}

\vspace{0.2cm}

\noindent\textbf{Data availability} Data sharing is not applicable to this article as no datasets were generated or analysed during the current study.

\vspace{0.2cm}

\noindent\textbf{Conflict of interest} On behalf of all authors, the corresponding author states that there is no conflict of interest.
%
%
%
%

\end{document}